\documentclass[3p]{elsarticle}

\usepackage{enumitem}
\usepackage{amsmath, amssymb, amscd, amsthm, amsfonts}
\usepackage{physics}
\usepackage{braket}
\usepackage{lipsum}
\usepackage{graphicx}
\usepackage{epstopdf}
\usepackage{algorithmic}
\usepackage{multirow}
\usepackage{subcaption}
\usepackage{amsopn}
\usepackage{algorithm}
\usepackage{color}
\usepackage{hyperref}
\usepackage{bm}
\usepackage{mathtools}
\usepackage{nccmath}
\usepackage{subfiles}
\usepackage{array} 
\usepackage{multirow}
\usepackage{subcaption} 
\usepackage[export]{adjustbox}  
\usepackage{booktabs}
\usepackage[table]{xcolor}

\usepackage{bm}

\usepackage{siunitx}
\usepackage[capitalise, noabbrev]{cleveref}

\newcommand{\bmnum}[1]{{\textbf{\num[detect-weight=true]{#1}}}}

\newcolumntype{P}[1]{>{\centering\arraybackslash}p{#1}}
\newcolumntype{C}{>{\centering\arraybackslash}c}

\hypersetup{
    colorlinks,
    linkcolor={magenta!80!black},
    citecolor={blue!80!black},
    urlcolor={blue!80!black}
}

\newcommand{\bR}{{\mathbb{R}}}

\newcommand{\cB}{{\mathcal{B}}}
\newcommand{\cC}{{\mathcal{C}}}

\newcommand{\cF}{{\mathcal{F}}}
\newcommand{\cH}{{\mathcal{H}}}
\newcommand{\cM}{{\mathcal{M}}}
\newcommand{\cN}{{\mathcal{N}}}

\newcommand{\cR}{{\mathcal{R}}}

\newcommand{\bss}{{\boldsymbol{s}}}
\newcommand{\bst}{{\boldsymbol{t}}}
\newcommand{\bsu}{{\boldsymbol{u}}}
\newcommand{\bsv}{{\boldsymbol{v}}}
\newcommand{\bsw}{{\boldsymbol{w}}}

\newcommand{\bssigma}{{\boldsymbol{\sigma}}}

\newcommand{\bsalpha}{{\boldsymbol{\alpha}}}

\newcommand{\pp}{\mathfrak{p}}
\newcommand{\PP}{\mathfrak{P}}

\newcommand{\U}{\mathbb{U}}
\newcommand{\V}{\mathbb{V}}

\renewcommand{\div}{{\rm div\,}}
\newcommand{\uc}{u^\circ}
\renewcommand{\sc}{\sigma^\circ}
\newcommand{\e}{\varepsilon}
\newcommand{\bz}{\mathbf{z}}
\newcommand{\bbf}{\mathbf{f}}
 
  \renewcommand{\cH}{\mathcal{H}}
\newcommand{\cL}{\mathcal{L}}
\newcommand{\HH}{\mathbb{H}}
\newcommand{\LL}{\mathbb{L}}
\newcommand{\EE}{\mathbb{E}}

\newcommand{\N}{\mathbb{N}}
\newcommand{\R}{\mathbb{R}}
\newcommand{\X}{\mathbb{X}}
\newcommand{\Y}{\mathbb{Y}}

\newcommand{\rgrad}{{\mathrm{grad}\,}}
\newcommand{\rdiv}{{\mathrm{div}\,}}
\newcommand{\rtr}{{\mathrm{tr}}}

\newcommand{\mR}{\mathfrak{R}}
\newcommand{\prob}{{\rm Prob}}

\renewcommand{\div}{{\rm div\,}}

\newtheorem{theorem}{Theorem}
\newtheorem{lemma}{Lemma}
\newtheorem{proposition}{Proposition}
\newtheorem{assumption}{Assumption}
\newtheorem{remark}{Remark}
\newtheorem{corollary}{Corollary}

\newcommand\numberthis{\addtocounter{equation}{1}\tag{\theequation}}

\usepackage{titlesec}

\titleformat{\section}
  {\normalfont\Large\bfseries}{\thesection}{1em}{}

\titleformat{\subsection}
  {\normalfont\large\bfseries}{\thesubsection}{1em}{}

\titleformat{\paragraph}[runin]{\normalfont\bfseries}{\theparagraph}{1em}{}

\titleformat{\subsubsection}
  {\normalfont\normalsize\bfseries}{\thesubsubsection}{1em}{}

\usepackage{stackengine}
\stackMath
\newcommand\tenbar[2][1]{%
  \def\useanchorwidth{T}%
  \ifnum#1>1
    \stackunder[0.5pt]{\tenbar[\numexpr#1-1\relax]{#2}}{\scriptscriptstyle\rule{1.4ex}{0.6pt}}%
  \else
    \stackunder[2pt]{#2}{\scriptscriptstyle\rule{1.4ex}{0.6pt}}%
  \fi
}

\def\argmin{\mathop{\rm argmin}}

\biboptions{longnamesfirst}

\makeatletter
\def\ps@pprintTitle{%
 \let\@oddhead\@empty
 \let\@evenhead\@empty
 \def\@oddfoot{}%
 \let\@evenfoot\@oddfoot}
\makeatother

\begin{document}



\title{Variationally correct operator learning: Reduced basis neural operator with a posteriori error estimation}

\author[1]{Yuan Qiu}
\ead{yuan.qiu@gatech.edu}

\author[2]{Wolfgang Dahmen}
\ead{dahmen@math.sc.edu}

\author[1]{Peng Chen\corref{cor1}}
\ead{pchen402@gatech.edu}

\affiliation[1]{organization={School of Computational Science and Engineering},
                addressline={Georgia Institute of Technology},
                city={Atlanta},
                postcode={GA 30332},
                country={USA}}

\affiliation[2]{organization={Department of Mathematics},
                addressline={University of South Carolina},
                city={Columbia},
                postcode={SC 29208},
                country={USA}}
 
\cortext[cor1]{Corresponding author.}
\begin{abstract}
Minimizing PDE-residual losses is a common strategy to promote physical consistency in neural operators. However, standard formulations often lack \emph{variational correctness}, meaning that small residuals do not guarantee small solution errors due to the use of non-compliant norms or ad hoc penalty terms for boundary conditions. This work develops a variationally correct operator learning framework by constructing first-order system least-squares (FOSLS) objectives whose values are provably equivalent to the solution error in PDE-induced norms. We demonstrate this framework on stationary diffusion and linear elasticity, incorporating mixed Dirichlet-Neumann boundary conditions via \emph{variational lifts} to preserve norm equivalence without inconsistent penalties. To ensure the function space conformity required by the FOSLS loss, we propose a Reduced Basis Neural Operator (RBNO). The RBNO predicts coefficients for a pre-computed, conforming reduced basis, thereby ensuring variational stability by design while enabling efficient training. We provide a rigorous convergence analysis that bounds the total error by the sum of finite element discretization bias, reduced basis truncation error, neural network approximation error, and statistical estimation errors arising from finite sampling and optimization. Numerical benchmarks validate these theoretical bounds and demonstrate that the proposed approach achieves superior accuracy in PDE-compliant norms compared to standard baselines, while the residual loss serves as a reliable, computable \emph{a posteriori} error estimator.
\end{abstract}


\maketitle

\section{Introduction}
\label{sec:introduction}

Determining physical states of interest solely from observational data is generally intractable without invoking governing physical laws. These laws are typically modeled by systems of partial differential equations (PDEs) that depend on unspecified problem data, such as boundary conditions, source terms, or coefficient fields. The collection of solutions generated by varying these inputs forms the \emph{solution manifold}. In many-query tasks such as uncertainty quantification (UQ)~\cite{smith2024uncertainty,chen2017reduced}, Bayesian inversion~\cite{stuart2010inverse,chen2021stein} and optimal experimental design \cite{huan2024optimal,wu2023fast}, exploring this manifold requires evaluating the \emph{data-to-solution map}---the solution operator---thousands or millions of times.

The computational burden of traditional high-fidelity discretizations (e.g., finite elements) has driven a surge of interest in \emph{operator learning} for efficient surrogate modeling. Architectures such as Deep Operator Network (DeepONet)~\cite{lu2021learning} and Fourier Neural Operators (FNO)~\cite{li2020fourier} learn mappings between infinite-dimensional function spaces, typically via \emph{regression} on pre-computed high-fidelity snapshots. However, this supervised paradigm faces two critical limitations: the prohibitive cost of generating large training datasets and the reliance on standard $L^2$-type losses. Crucially, these regression norms are often not \emph{PDE-compliant}, meaning a small training error does not necessarily guarantee accuracy in the physically relevant norms.

To reduce reliance on high-fidelity data, \emph{residual minimization} (as used in Physics-Informed Neural Networks or PINNs~\cite{Kar1,Kar2}) offers an alternative by embedding the PDE directly into the loss function. However, standard residual losses often lack \emph{variational correctness}~\cite{bachmayr2025variationally}: they fail to provide uniform upper and lower bounds on the true error in a proper PDE norm. This typically occurs because bulk residuals are measured in norms that are too strong (e.g., $L^2$ for second-order operators), while boundary conditions are enforced via penalty terms in norms that are too weak or inconsistent~\cite{kharazmi2021hp}. Consequently, minimizing such losses does not rigorously control the solution error. While recent research has proposed \emph{a posteriori} error estimation strategies---including learning error certificates~\cite{fanaskov2024neural}, developing computable bounds~\cite{ernst2025posteriori}, and utilizing functional-type norms~\cite{cao2024spectral}---achieving variational correctness \emph{by design} requires a stable variational formulation. Specifically, one must identify a pair of \emph{PDE-compliant norms} for the trial and test spaces such that the residual in the \emph{dual test norm} is uniformly proportional to the error in the trial norm. This condition holds precisely when the variational formulation is stable in the sense of the Babu\v{s}ka-Ne\v{c}as Theorem~\cite{bachmayr2025variationally}.

A practical obstruction, well-known in adaptive finite elements, is that evaluating the dual norm is generally difficult; we refer to \cite{bachmayr2025variationally} for various approaches to overcome this. In this work, we focus on a strategy that seeks stable variational formulations where the test space is \emph{self-dual} (i.e., an $L^2$-space), allowing the residual to be computed simply as an $L^2$-norm. While standard higher-order PDE formulations typically do not admit such stability, reformulating them as a \emph{first-order system} often does. This approach, known as First-Order System Least-Squares (FOSLS)~\cite{bochev2009least,FOSLS}, has been successfully applied to space-time formulations of parabolic and wave equations~\cite{KF1,KF2}, as well as stationary diffusion and linear elasticity~\cite{opschoor2024first}. Deep learning methods utilizing FOSLS have already demonstrated significant promise for solving PDEs~\cite{cai2020deep, lyu2022mim, bersetche2023deep, li2024solving, opschoor2024first}. 
However, these studies have primarily focused on solving individual PDE instances, rather than addressing the more challenging task of learning the solution operator over a distribution of parameters.

In this work, we bridge this gap by developing a variationally correct operator learning framework for parametric PDEs. To demonstrate the methodology, we restrict our attention here to linear parametric PDEs---specifically, stationary diffusion and linear elasticity. Our approach diverges from standard methods in two key ways. First, we handle mixed Dirichlet-Neumann boundary conditions through \emph{variational lifts} rather than penalty terms, preserving rigorous norm equivalence. Second, to inherently satisfy the \emph{conformity requirements} of our loss function, we propose a \emph{Reduced Basis Neural Operator} (RBNO). While standard neural operators effectively learn mappings between generic function spaces, enforcing specific regularity constraints (such as flux continuity) remains non-trivial. The RBNO bypasses this difficulty by predicting coefficients for a basis that is conforming by construction, thereby ensuring variational stability by design while mitigating the complexity of high-dimensional outputs.

To practically construct this basis, we exploit the rapid decay of Kolmogorov $n$-widths for elliptic problems~\cite{CDacta,quarteroni2015reduced}. We project the solution manifold onto a low-dimensional linear space generated by Proper Orthogonal Decomposition (POD) of high-fidelity snapshots. Since these snapshots are computed using a conforming discretization, their linear combinations naturally preserve the required function space regularity. The neural operator then learns the mapping from parameters to the \emph{coefficients} of this reduced basis. This hybrid architecture functions analogously to an encoder-decoder with a fixed linear decoder~\cite{bhattacharya2021model,lu2022comprehensive}, yet distinctively, it is trained by minimizing the rigorous FOSLS residual rather than a regression loss. We note that while data-driven reduced basis neural networks have been explored elsewhere~\cite{dal2020data,fresca2021comprehensive,kutyniok2022theoretical, o2022derivative,franco2023deep,o2024derivative,zheng2025rebano,wang2025reduced}, our approach uniquely leverages this structure to enable variationally correct training.

\textbf{Contributions.} The specific contributions of this paper are as follows:
\begin{enumerate}
\item \textit{Rigorous construction of variationally correct FOSLS objectives.} We derive $L^2$-based FOSLS residual losses for stationary diffusion and linear elasticity. Crucially, we show how mixed Dirichlet/Neumann boundary data can be incorporated via \emph{variational lifts}, provably preserving the equivalence between the loss and the solution error in PDE-induced norms without relying on ad hoc boundary penalties.
\item \textit{Finite element realization and discrete error control.} We formulate conforming finite element discretizations that make the quadratic residual loss efficiently computable. We establish a theoretical link between this discrete loss and the discretization errors, proving convergence rates that depend on the mesh size, polynomial order, and PDE regularity. This ensures the training objective remains a reliable proxy for the true error even after discretization.
\item \textit{Reduced Basis Neural Operator (RBNO) with theoretical guarantees.} We propose the RBNO architecture, which predicts coefficients for a pre-computed, conforming POD basis. By construction, this ensures the network's output satisfies the function space conformity required by the FOSLS loss. We provide a convergence analysis that bounds the total error by the sum of finite element discretization bias, reduced basis projection error, neural network approximation error, statistical estimation errors arising from finite sampling, and optimization error, rigorously justifying the hybrid approach.
\item \textit{Numerical validation.} We demonstrate the method on diffusion benchmarks and a linear elasticity problem. We numerically validate the convergence analysis with respect to finite element discretization, reduced basis projection, statistical estimation and optimization errors in RBNO training. The results confirm that the proposed residual loss serves as a tight, computable \emph{a posteriori} error estimator and that the RBNO achieves superior accuracy in PDE-compliant norms compared to two baselines.
\end{enumerate}

The remainder of the paper is organized as follows. \cref{sec:conceptual_background_and_orientation} reviews the abstract framework of operator learning and variational correctness. \cref{sec:first_order_system_least_squares_loss_formulation} derives the variationally correct FOSLS formulations for diffusion and elasticity, including the variational lifting strategy for boundary conditions. \cref{sec:conforming_finite_element_approximations} details the finite element realization of the loss and the associated discrete error control. \cref{sec:reduced_basis_neural_operators} introduces the RBNO architecture and establishes the theoretical convergence analysis. \cref{sec:numerical_experiments} presents numerical validation of the error estimates and performance comparisons with standard neural operators. Finally, \cref{sec:conclusions_limitations_and_future_work} discusses limitations and future directions.

\section{Conceptual background and orientation}
\label{sec:conceptual_background_and_orientation}
We briefly outline the abstract problem formulation, upon which subsequent developments will be based. We consider a {\em parametric family} of linear partial differential equations (PDEs)
in residual form
\begin{equation}
\label{par}
\cR(u;\pp)= \mathcal{B}_\pp u - f = 0
\end{equation}
where $\pp$ stands for a scalar/vector-valued parameter---or more generally for a parameter field---that may range over a given parameter range or space $\PP$ with measure $\mu$. Given $\pp\in\PP$, we are interested in assessing
the solution $u=u(\pp)$ of \eqref{par} for the parameter-instance $\pp$. In what follows, we
assume that for each $\pp\in\PP$,  \eqref{par} is well-posed (in  a sense detailed below).
Hence, the {\em solution operator} or {\em parameter-to-solution map} $\pp \mapsto u(\pp)$ is 
well-defined and its range $\cM=\cM(\PP)$---the set of all solution-states, obtained by  traversing $\PP$---is often referred to as {\em solution manifold}. It will be important to specify for each $\pp\in\PP$  a suitable 
(infinite dimensional) trial space~$\U$ that is to accommodate $\cM$ and could, in principle, depend on $\pp$. 

We postpone a discussion on this issue for the moment and recall that a common approach to learning $\pp\mapsto  u(\pp)$ is completely data-driven and employs regression in $\U$ (or 
for convenience in $L_2$). This requires computing first a large number of high-fidelity solutions, e.g., for randomly chosen parameter samples. A large number of additional test samples is needed to assess the generalization error, as the inherent uncertainty in optimization success renders a priori expressivity results insufficient. In summary, given the at best achievable Monte Carlo rates, depending on the accuracy requirements, the associated computational cost may be prohibitive. 

Therefore, we focus in this paper on contriving {\em residual-type} loss functions. In this setting, the unknown solution is encoded by known problem data (such as source terms), allowing one to directly optimize over the trainable weights that describe the hypothesis class. The number of training samples scales with the number of residual evaluations, not with the number of high-fidelity solution snapshots. The price is that we can no longer assess directly the (generalization) error of the optimization outcome in the chosen model-compliant norm, {\em unless} we can show that the size of the loss is {\em uniformly proportional} to that error. We refer to such a loss as {\em variationally correct}, see \cite{bachmayr2025variationally}.

As shown in \cite{bachmayr2025variationally}, variational correctness is intimately related to a {\em stable variational formulation} for each {\em fiber problem} \eqref{par}. To be specific, suppose one has found for each $\pp\in\PP$ a pair of Hilbert spaces $\U_\pp,\V_\pp$ such that the family of bilinear forms $b_\pp(\cdot,\cdot)$, $\pp\in\PP$,  defined by $b_\pp(w,v)= (\cB_\pp w)(v)$, $w\in\U_\pp, v\in \V_\pp$. Then well-posedness of \eqref{par} is equivalent to saying that $b_\pp(\cdot,\cdot)$ satisfies for each $\pp$ a continuity condition, an inf-sup condition and a surjectivity condition according to the Babu\v{s}ka-Ne\v{c}as-Theorem, see e.g. \cite[Chapter 1]{DGActa} or \cite{Bab,Bab2}. This in turn means that $\cB_\pp$, defined weakly as above, is an {\em isomorphism} from $\U_\pp$ onto $\V'_\pp$. Noting that for any $w\in \U_\pp$,  $\cB_\pp(u(\pp)-w)= f - \cB_\pp w$, this is equivalent to saying that  there exist constants $0<c_\pp\le C_\pp<\infty$ such that
\begin{equation}
\label{err-res}
c_\pp\|u(\pp) -w\|_{\U_\pp}\le \|f- \cB_\pp w\|_{\V'_\pp}= \|\cR(w;\pp)\|_{\V'}\le C_\pp \|u(\pp) -w\|_{\U_\pp}
\quad \forall\, w\in \U_\pp.
\end{equation}
Thus, the residual, measured in the dual test-norm, is a lower and upper bound for
the error in the model-compliant norm $\|\cdot\|_{\U_\pp}$, albeit with proportionality constants that may depend on $\pp$, see e.g. \cite{CDG2025}. For elliptic (coercive) problems, the choice $\V_\pp=\U_\pp$ is natural, but for other problems, finding a pair $\U_\pp, \V_\pp$ that gives rise to a stable variational formulation (and hence to~\eqref{err-res}) is part of the homework.

First, the above notation indicates that the dependence of a stable pair of trial and test space may be {\em essential}, meaning when $\pp$ varies, then $\U_\pp$ or $\V_\pp$ or both
may vary even as sets, see~\cite{bachmayr2025variationally} for related results. For the problems studied here, this is not the case, i.e., the spaces agree as sets, and the norms are equivalent, so that we henceforth write  $\U_\pp=\U$ and $\V_\pp=\V$. In particular, it makes sense to view the solution manifold as a subset of $\U$. We refer to~\cite{bachmayr2025variationally} for scenarios where $\U_\pp$ and/or $\V_\pp$ depend on $\pp$ in an essential way, i.e., even as sets.

Note that learning the map $\pp\mapsto u(\pp)$ (respecting the correct metrics) is tantamount to approximating the function $u=u(x,\pp)$, as a function of spatial and parametric variables in a model-compliant norm. It is then natural to interpret this function as the solution of a {\em single} (lifted) variational problem, whose residual enters the training loss. Following \cite{bachmayr2025variationally,CDG2025}, consider the Bochner spaces 
$$
\X:= L_2(\PP;\U),\quad \Y:= L_2(\PP;\V) 
$$
and   the bilinear form
$$
b(w,v):= \int_\PP b_\pp(w(\pp),v(\pp))d\mu(\pp) ,\quad \X\times \Y \to \R,
$$
where $\mu$ is a (fixed) probability measure on $\PP$, reflecting the probabilistic nature
of the parameter dependence. It has been shown in \cite{bachmayr2025variationally} that the ``lifted'' problem:
given $f\in \Y'= L_2(\PP;\V')$, find $u\in \X$ such that
\begin{equation}
\label{lifted}
b(u,v)= f(v),\quad v\in \Y,
\end{equation}
is well-posed if and only if the family of fiber problems \eqref{par} is {\em uniformly} well-posed over $\PP$, which then means that for $\cB:\X\to\Y'$, defined by $(\cB w)(v):=
b(w,v)$, $w\in\X, v\in\Y$,
\begin{equation}
\label{err-res-lift}
c_0\|u-w\|_\X \le \|f-\cB w\|_{\Y'}=: \|\cR(w)\|_{\Y'}\le C_0 \|u-w\|_\X,\quad w\in \X,
\end{equation}
where $c_0=\inf_{\pp\in\PP}c_\pp$, $C_0=\sup_{\pp\in\PP}C_\pp$, for $c_\pp,C_\pp$ from
\eqref{err-res}. 

Therefore, $\|\cR(w)\|_{\Y'}$ can be viewed as an {\em ideal} residual loss. However, first, integration over $\PP$ is not practically feasible, and second, in view of the definition $\|\cR(w)\|_{\Y'}:= \sup_{v\in\Y}\frac{b(w,v)-f(v)}{\|v\|_\Y}$, the (exact) evaluation of a non-trivial ($\Y\neq\Y'$) dual norm is not practical.

Regarding the first issue, note that \eqref{err-res-lift} describes proportionalities between {\em expectations} because, by definition
\begin{equation}
\label{exp}
\|u-w\|^2_\X= \mathbb{E}_{\pp\sim \mu}\big[ \|w(\pp)- u(\pp)\|^2_\U\big],\quad 
\|\cR(w)\|^2_{\Y'} = \mathbb{E}_{\pp\sim \mu}\big[\|\cR(w;\pp)\|^2_{\V'}\big].
\end{equation}
A common response to the first issue is to approximate the expectation by its empirical counterpart
\begin{equation}
\label{emploss}
\cL(w;\widehat\PP):= \frac{1}{\#\widehat\PP}\sum_{\pp\in \widehat\PP}\cL(w;\pp)=: \|\cR(w)\|^2_{\ell_2(\widehat\PP;\V')},\quad \text{where}\;\cL(w;\pp):= \|\cR(w;\pp)\|^2_{\V'}.
\end{equation}
Here $\widehat\PP$ denotes a collection of finite random samples of $\pp$ sampled from $\mu$ and $\#\widehat\PP$ denotes the sample size.
Of course, \eqref{err-res-lift} implies its discrete counterpart
\begin{equation}
\label{discr}
c_0\|u-w\|_{\ell_2(\widehat\PP;\U)} \le 
\cL(w;\widehat\PP)\le C_0 
\|u-w\|_{\ell_2(\widehat \PP;\U)},\quad w\in \X.
\end{equation}
Specifically, when working with a hypothesis class $\cH=\cH(\Theta)= \{w(\cdot,\pp;\theta):\theta\in\Theta\}$ determined by
 a given budget of trainable weights $\theta\in\Theta$, the regression ansatz
 \begin{equation}
\label{emprisk}
\theta^*\in \argmin_{\theta\in\Theta}\frac{1}{\# \widehat\PP} \sum_{\pp \in \widehat\PP} \| u(\pp)- w(\pp;\theta)\|^2_\U
\end{equation}
is replaced by the ``equivalent'' task of finding
\begin{equation}
\label{res}
\theta^*\in \argmin_{\theta\in\Theta}\cL(w(\theta);\widehat\PP),
\end{equation}
that spares us from computing the snapshots $u(\pp)$, which could be done approximately by minimizing the corresponding {\em empirical risk} with finite training samples. The deviation between $\cL(w;\widehat\PP)$ and $\mathbb{E}_{\pp\sim \mu}\big[\|\cR(w;\pp)\|^2_{\V'}\big]$, when $\#\widehat\PP\to \infty$, the generalization error, so to speak,  is a matter of statistical estimation.

In summary, $ \|f- \cB w\|_{\V' }= \|\cR(w)\|_{\V'}$, or its discrete counterpart \eqref{emploss}, is ``ideal'' residual loss; however, the second issue remains how to evaluate \eqref{discr} for training purposes. A central message from~\cite{bachmayr2025variationally}, which we reinforce here, is that one should reformulate \eqref{par} (if it is not already in this form) as a \emph{first-order system} of PDEs, since this often provides greater flexibility in constructing stable variational formulations. In particular, it is often possible to find a formulation in which the test space $\V$ is an $L_2$-type space (a direct product of $L_2$-spaces) and hence self-dual $\V'=\V$. This is, for instance, the case for space-time variational formulations of parabolic problems or the wave equation, \cite{KF1,KF2,GS}, as well as for parametric diffusion problems and static elasticity models, which are our focus here; see~\cite{bochev2009least,FOSLS,bachmayr2025variationally,CDG2025,opschoor2024first}.

When $\V=\V'$ is an $L_2$-space, the loss \eqref{emploss} is a least squares $L_2$-residual and can therefore (up to quadrature errors) be evaluated.  Nevertheless, the ensuing optimization task \eqref{res}, obtained when $w(\cdot,\cdot,\theta)$ belongs to a hypothesis class comprising deep neural networks with input variables $(x,\pp)\in \Omega\times \PP$, is, in general, hard to handle. This optimization task can be substantially alleviated when using a {\em hybrid} format of low-rank type, where we separate dependence on spatial and parametric variables. This has been done, for instance, in~\cite{bachmayr2025variationally,CDG2025}, where approximation systems are employed that contain elements, which are finite elements as functions of spatial variables with $\pp$-dependent expansion coefficients that can be represented, for instance, by DNNs. In the present paper, we also employ a hybrid format that differs in terms of the spatially dependent modes.  

The underlying rationale is that for elliptic models and their close relatives, the map $\pp\to u(\pp)$ is (under mild assumptions on the parametric coefficient fields) even holomorphic (see \cite{CDacta,opschoor2024first}) and the so-called  {\em Kolmogorov} $n$-widths decay robustly in the parametric dimension. Recall that for a compact subset $\mathcal{K}$ in a Banach space $\mathbb{X}$, the Kolmogorov $n$-widths (as a measure of
thickness of $\mathcal{K}$ in $\mathbb{X}$) are given by
\begin{equation}
    \label{Kolm}
d_n(\mathcal{K})_{\mathbb{X}}:= \inf_{{\rm dim}\,\V_n=n}\sup_{v\in \mathcal{K}}\inf_{v_n\in
\V_n}\|v-v_n\|_{\mathbb{X}},\quad n\in\mathbb{N}.
\end{equation} 
For $\mathcal{K} = \cM$ and $\X=L^2_\mu(\PP;\U)$ the $d_n(\mathcal{K})_{\mathbb{X}}$ measure how well the solution manifold can be approximated from a single linear space in a worst-case sense, i.e., in $L_\infty(\PP;\U)$. Thinking of $\pp$ as a random variable distributed according to some measure $\mu$ on $\PP$, the Bochner space $L^2_\mu(\PP;\U)$ is more akin to a machine learning approach, and \eqref{exp} suggests measuring accuracy in a mean-squared sense. The corresponding optimality benchmark then reads
\begin{equation}
\label{deltan}
\delta_n(\cM,\mu)^2_\U:= \inf_{{\rm dim}\,\U_n=n}\int_\PP \min_{w\in \U_n}\|u(\pp)-w\|^2_\U
d\mu(\pp).
\end{equation}
It is well-known that for this metric the {\em best linear approximation spaces} result from the Hilbert-Schmidt decomposition of the operator $M_u: L^2_\mu(\PP)\to \U$ defined by
$$
M_u w := \int_\PP u(\pp)w(\pp)d\mu(\pp),
$$
has a Hilbert-Schmidt decomposition
$$
M_u:= \sum_{k=1}^\infty s_k \langle \cdot,\phi_k\rangle_{L^2_\mu(\PP)}u_k,
$$
where $\mathbf{s}= (s_k)_{k\in\N}\in \ell_2(\N)$ is non-negative and non-increasing while $(u_k)_{k\in\N}$ and $(\phi_k)_{k\in\N}$ are orthonormal systems in $\U$ and $L^2_\mu(\PP)$, respectively. In particular, the spaces $\U_n:= {\rm span}\,\{u_k: k\le n\}$ realize
the benchmark \eqref{deltan} with 
\begin{equation}
\label{how}
\delta_n(\cM,\mu)^2_\U= \sum_{k>n}s^2_k,\quad n\in\N.
\end{equation}
While it is impossible to determine $\U_n$ exactly, one can approximate the $u_k$ with the aid of the well-established strategy of {\em Proper Orthogonal Decomposition} (POD). Roughly, this is done by choosing a sufficiently large ``truth space'' $\U_h\subset \U$, typically a finite element space with mesh size $h$, compute for a sufficiently large number of random samples $\pp^i$, $i=1,\ldots,N_{\text{POD}}$, corresponding (approximate) solutions $u^i=\tilde u(\pp^i)$ of $\cR(u;\pp^i)$ in $\U_h$, and compute a Singular Value Decomposition of $u^i$, $i=1,\ldots,N_{\text{POD}}$, in $\U$. When $\mu$ is a probability measure it readily follows that $\delta_n(\cM,\mu)_\U\le
d_n(\cM)_\U$ so that a rapid decay of the Kolmogorov $n$-widths implies a rapid decay of the $\delta_n(\cM,\mu)_\U$. In such a situation, we expect that for a sufficiently large training set $\PP_{N_{\text{POD}}}=\{\pp^i:i=1,\ldots, N_{\text{POD}}\}$, the decay of the approximate singular values $\tilde s_i$ closely reflects the decay of the exact singular values $s_i$ and
that the tail $\sum_{k=n+1}^{N_{\text{POD}}} \tilde s^2_k$ becomes smaller than a given target accuracy $\e$ already for some $n_\e\ll {\rm dim}\,\U_h$ of moderate size. In this case
\begin{equation}
\label{Un}
\U_{n_\e}:= {\rm span}\,\big\{u_k:k=1,\ldots,n_\e\big\},
\end{equation}
with the POD bases $u_k$, $k = 1, \ldots, n_\e$, serves as a {\em reduced space}. More details will be provided later below for the respective examples. 

One should keep in mind that a rapid decay of \eqref{deltan} requires a significantly smaller number
of high-fidelity approximate solutions of \eqref{par} than a purely data-driven approach, based on \eqref{emprisk}.

Given $\U_{n_\e}$ we will then seek approximations to the solutions $u$ of \eqref{par} in the ``hybrid'' form
\begin{equation}
\label{hybrid}
u(x,\pp;\theta):= \sum_{k=1}^{n_\e} \phi_k(\pp;\theta)u_k(x).
\end{equation}
Aside from facilitating optimization, this format offers advantages in handling boundary conditions. As mentioned earlier, one major objective is to incorporate non-trivial mixed Dirichlet-Neumann conditions in the training objective in a variationally correct way. First, due to the nature of the POD basis functions $u_k$, it is easy to incorporate essential Dirichlet conditions. As detailed below for both models, the stationary diffusion equation and linear elasticity, the first-order formulations allow us, in particular, to incorporate Neumann conditions by an $L_2$ source in the least-squares residual, which is determined via solving a single auxiliary second-order elliptic problem.

\section{First-order system least-squares loss formulation}
\label{sec:first_order_system_least_squares_loss_formulation}

 \subsection{Stationary diffusion equation with mixed boundary conditions}

\subsubsection{Second order formulation}

Let $\Omega\subset \bR^{d}$ be a bounded domain. The strong form of a stationary diffusion equation with heterogeneous  diffusivity $\pp(x)\in \bR^{d\times d}$ and mixed boundary conditions (Dirichlet and Neumann) is given by
\begin{equation}\label{eq:poisson_strong_primal}
    \begin{cases}
        -\div{(\pp\nabla{u})} = f, &\qq{in $\Omega$,} \\
        u = u_{0}, & \qq{on $\Gamma_D$,} \\
        \pp\nabla{u}\cdot n = g, & \qq{on $\Gamma_N$,}
    \end{cases}
\end{equation}
where $u(x)\in\bR$ is the unknown scalar field and $f$ is a given source term with properties specified below. The Dirichlet boundary is denoted by $\Gamma_D\subset\partial\Omega$ with prescribed boundary data $u_0$. The  Neumann boundary is given by $\Gamma_N=\partial\Omega\setminus\Gamma_D$, where $n=n(x)$ is (for almost all $x\in \partial\Omega$) the outward unit normal at $x\in\partial\Omega$. Finally, $g$ stands for the prescribed normal flux.

To explain in which sense a first-order formulation is equivalent, we need to refer in both cases to an appropriate weak formulation. To that end, we recapitulate for the
convenience of the reader, a few classical facts, and define as usual
$$ 
\U:= H^1_{0,\Gamma_D}(\Omega)= {\rm clos}_{H^1(\Omega)}\,\{\phi\in C^\infty(\Omega): {\rm supp}\, \phi \cap \Gamma_D=\emptyset\},
$$ 
where we adopt common terminology when abbreviating $H^1_0(\Omega):= H^1_{0,\partial\Omega}(\Omega)$ when the Dirichlet boundary $\Gamma_D$ is all of $\partial\Omega$. Moreover, let $\U':= (H^1_{0,\Gamma_D}(\Omega))'$ denote its normed dual, i.e., the space of all bounded linear functionals on $H^1_{0,\Gamma_D}(\Omega)$. The classical weak formulation of \eqref{eq:poisson_strong_primal} then reads: given $f\in \U'$ find $u\in H^1(\Omega)$ satisfying
\begin{equation}
\label{weak2nd}
u|_{\Gamma_D}= u_0 \quad\mbox{and}\quad b_\pp(u,v):= \int_\Omega \pp\nabla u\cdot \nabla v\,dx = f(v)+ \langle g,v\rangle_{\Gamma_N}, \quad v\in \U,
\end{equation}
where the boundary constraint is understood in the sense of traces. Moreover, $\langle\cdot,\cdot\rangle_{\Gamma_N}$ denotes the duality pairing on $H^{-1/2}(\Gamma_N)\times H^{1/2}_{00}(\Gamma_N)$, where $H^{-1/2}(\Gamma_N):=
(H^{1/2}_{00}(\Gamma_N))'$. Here, $H^{1/2}_{00}(\Gamma_N)$ is the subspace of $H^{1/2}(\Gamma_N)$ that consists of those elements whose extension by zero to the rest of $\partial\Omega$ belongs to $H^{1/2}(\partial\Omega)$. Note that $v\in \U$ implies $v|_{\Gamma_N}\in H^{1/2}_{00}(\Gamma_N)$, 
so that $\langle g,\cdot\rangle_{\Gamma_N}: \U \to \bR$ indeed defines a 
bounded linear functional on $\U$, hence belongs to $\U'$. The operator induced by this weak formulation is affine. 

A proper weak first-order formulation requires a different representation of boundary conditions, which will follow from a different but equivalent alternate weak formulation of \eqref{eq:poisson_strong_primal}.
This will be obtained by solving two auxiliary parameter-independent elliptic problems representing harmonic extensions for both types of boundary data and taking a Riesz representation of the right-hand side functional. 

Regarding the Dirichlet conditions, choose a fixed $w\in H^{1}(\Omega)$ such that $w|_{\Gamma_D}=u_0$ and $\nabla w \cdot n|_{\Gamma_N}=0$ (in the sense of traces), by solving the parameter-independent problem: find $w\in H^1(\Omega)$ such that 
\begin{equation}\label{eq:DiriLiftPoisson}
(\nabla w,\nabla v)_\Omega=0,\quad v\in H^1(\Omega), \,\,w|_{\Gamma_D} = u_0.    
\end{equation}
Then consider the equivalent variational problem:  find $\uc\in \U$ such that
\begin{equation}
\label{weak2nd2}
b_\pp(\uc,v) = -b_\pp(w,v) + f(v) + \langle g,v\rangle_{\Gamma_N},\quad v\in \U,
\end{equation} 
so that $u:= \uc+w$ obviously satisfies \eqref{weak2nd}.  Note that $-b_\pp(w,\cdot) + f + \langle g,\cdot\rangle_{\Gamma_N}$ defines a bounded linear functional on $\U$ - as needed to render \eqref{weak2nd}  well-posed - provided that $f\in \U'$. 

As such, $f$ could, in principle, incorporate an additional flux-condition on $\Gamma_N$, violating the 
explicit specification in terms of $g$. To avoid this we confine $f$ to be  {\em flux-free} which in the present context means that
\begin{equation}
\label{flfree}
f(v) = (f_2,v)_\Omega + (\div\,f_1)(v) = (f_2,v)_\Omega - (f_1, \nabla v)_\Omega,\quad v\in \U.
\end{equation}
where $f_2\in L_2(\Omega)$ and $f_1\in L_2(\Omega;\R^d)$. We refer to a detailed discussion in \ref{app:decomposition} that explains in which sense the implied condition $\langle  f_1 \cdot n,v\rangle_{\Gamma_N}=0$ for $v\in \U$, is to be understood.

Regarding the Neumann condition, consider the second similar auxiliary problem: find $q\in\U$ such that
\begin{equation}
\label{harm}
(\nabla q,\nabla v)_\Omega = \langle g,v\rangle_{\Gamma_N}
,\,\, v\in \U,
\end{equation}
which has a unique solution $q\in \U$ satisfying the distributional relation 
\begin{equation}
\label{eq:aux-z}
-(\Delta q)(v) =0,\qquad  \langle  \nabla q \cdot n,v\rangle_{\Gamma_N} = \langle g,v\rangle_{\Gamma_N}, \quad v\in\U. 
\end{equation}
 Setting $z := \nabla q\in L^2(\Omega;\bR^d)$, we have $\langle g,v\rangle_{\Gamma_N}
= (z,\nabla v)= -(\div\,z)(v) +\langle  z \cdot n,v\rangle_{\Gamma_N}= \langle z \cdot n,v\rangle_{\Gamma_N}$. 
  
In summary, we obtain an equivalent weak formulation to \eqref{weak2nd}: find $\uc\in \U$ such that
\begin{equation}
\label{finally}
(\pp\nabla \uc + \pp\nabla w - z + f_1, \nabla v)_\Omega = (f_2,v)_\Omega ,\quad v\in \U,
\end{equation} 
where $u= \uc + w$ is the solution \eqref{weak2nd} and the normal trace term is realized through the field $z$.

\subsubsection{First-order system least-squares formulation}\label{sssec:FOSLS}
Introducing the auxiliary flux variable
\begin{equation}
\label{smc}
\sc:= \pp\nabla \uc + \pp\nabla w - z + f_1,
\end{equation}
\eqref{finally} takes the simple form
\begin{equation}
\label{simple}
(\sc,\nabla v)_\Omega = -(\div\,\sc)(v) + \langle \sc\cdot n,v\rangle_{\Gamma_N}= (f_2,v)_\Omega,\quad v\in\U.
\end{equation}
Testing first with $v\in H^1_0(\Omega)$,
shows that $-(\div\,\sc)(v)=(f_2,v)_\Omega$, $v\in H^1_0(\Omega)$. Since $H^1_0(\Omega)$ is
dense in $L_2(\Omega)$, we conclude that $- \div\,\sc = f_2$ in $L_2(\Omega)$, hence 
$\sc\in H(\div;\Omega)$. Testing subsequently with all $v\in\U$, yields $\langle \sc\cdot n,v\rangle_{\Gamma_N} =0$ for all $v\in\U$, which says that $\sc\cdot n|_{\Gamma_N}=0$ in $H^{-1/2}(\Gamma_N)$.

This shows that we have to seek $\sc$ in $\Sigma$ and $(\div\,\sc + f_2,v)_\Omega$ is well-defined for $v\in L_2(\Omega)$, where 
\begin{equation}
\label{divsc}
\Sigma:= \{\eta\in H(\div;\Omega): \eta \cdot n|_{\Gamma_N} =0\}.
\end{equation}
Hence, we arrive at the well-known FOSLS formulation from \cite{bochev2009least,FOSLS}: find $(\sc,\uc)\in \HH := \Sigma \times \U$ such that
\begin{equation}
\label{fosls}
(\sc - \pp\nabla\uc ,\tau)_\Omega = (F,\tau)_\Omega ,\quad -(\div \,\sc, v)_\Omega = (f_2,v)_\Omega,\quad (\tau,v)\in \LL_2:= L_2(\Omega;\R^{d+1}),
\end{equation}
where $F:= \pp\nabla w -z +f_1\in L_2(\Omega;\R^d)$. Defining the operator
\begin{equation}\label{eq:stableB}
(\cB_\pp[\sc,\uc])([\tau,v]):= (\sc - \pp\nabla \uc, \tau)_\Omega - (\div \sc,v)_\Omega,
\quad [\tau,v]\in \LL_2= L_2(\Omega;\bR^{d+1}),
\end{equation}
it has been shown in \cite{bochev2009least,FOSLS} that $\cB_\pp$ is an {\em isomorphism} from $\HH$ onto $\LL_2:=L_2(\Omega;\bR^{d+1})$, i.e.,
\begin{equation}
\label{stability}
\|\cB_\pp[\sigma',u']\|^2_{\LL_2}\eqsim \|[\sigma',u']\|^2_\HH := \|\sigma'\|^2_{\Sigma} + \|u'\|^2_{\U},\quad \forall [\sigma',u']\in \HH, 
\end{equation}
where the proportionality constants depend on lower and upper bounds on $\pp\in\PP$ in $\Omega$. For the convenience of the reader, we provide a proof adapted to the parametric case in \ref{app:PoissonNormEquiv}.

For any approximation $[\tilde{\sigma}^\circ,\tilde{u}^\circ ]\in \HH$ of the exact solution $[\sc,\uc]\in \HH$ at any parameter $\pp \in \PP$, applying this to the error $[\sigma',u']= [\tilde{\sigma}^\circ - \sc,\tilde{u}^{\circ}-\uc]$
yields the desired error-residual relation
\begin{equation}
\label{err-res}
\|[\tilde{\sigma}^\circ, \tilde{u}^\circ] -[\sc, \uc]\|^2_\HH
\eqsim \cL([\tilde{\sigma}^\circ, \tilde{u}^\circ]; \pp),\quad [\tilde\sigma^\circ,\tilde{u}^\circ]\in \HH, \, \pp\in \PP,
\end{equation} 
where the residual fiber-loss $\cL([\tilde\sigma^\circ, \tilde{u}^\circ]; \pp)$ of the approximate solution $[\tilde\sigma^\circ, \tilde{u}^\circ]$ at parameter $\pp$ is given by 
\begin{equation}\label{eq:lossPoisson}
    \cL([\tilde\sigma^\circ, \tilde{u}^\circ]; \pp) := \|\tilde\sigma^\circ - ( \pp\nabla \tilde{u}^\circ + \pp\nabla w - z + f_1) \|^2_{L_2(\Omega;\bR^d)} +
\|\div \tilde\sigma^\circ+ f_2\|^2_{L_2(\Omega)}.
\end{equation}
Note that the loss function involves only $L_2$-norms and no explicit boundary term in a broken or dual norm. It is therefore  computable and can be used to form the mean squared loss
$\cL([\tilde\sigma^\circ, \tilde{u}^\circ];\widehat\PP) $ in \eqref{emploss}, on which   our surrogate models will be trained:
\begin{equation}
\label{trainPoisson}
\theta^*\in \argmin_{\theta\in\Theta}\frac{1}{\#\widehat\PP}\sum_{\pp\in \widehat\PP}
\| \sc(\pp;\theta) - ( \pp\nabla \uc (\pp;\theta) + \pp\nabla w - z + f_1) \|^2_{L^2(\Omega;\bR^d)} +
\|\div  \sc(\pp;\theta)+ f_2\|^2_{L^2(\Omega)},
\end{equation}

To put the above developments into perspective, while the training loss does not contain any explicit boundary term, the homogeneous boundary conditions on $\sc$ and $\uc$ need to be built into the trial system. Since we will be using a hybrid representation system, the functions $\sc(\cdot,\pp;\theta)$ and $\uc(\cdot,\pp;\theta)$, as functions of the spatial variables $x$, are piecewise polynomial, such boundary conditions are easy to realize.
Second, the loss is variationally correct, i.e., its size bounds the error in the trial norm from below and above.  

Alternatively, recombining $\uc$ and $w$, one can omit the source term $w$ and absorb the data term $\pp\nabla w$ back in $u= \uc +w$. In this case, one has to enforce the inhomogeneous boundary condition $u|_{\Gamma_D}=u_0$ in the trial space so that the POD reduced space needs to be affine. Likewise, $\sigma:=\sc +z$ would satisfy the correct Neumann conditions so that one might consider a corresponding affine subspace of $H(\div;\Omega)$ as well.

Instead of incorporating any boundary conditions into the trial system, it is common practice in machine learning to enforce boundary conditions by penalizing boundary residuals. Again, to preserve variational correctness (the loss is a sharp error bound), one could consider the residual with respect to an ``extended'' operator 
\begin{equation}
\label{extended}
\hat\cB_\pp : H(\div;\Omega)\times H^1(\Omega) \to \LL_2 \times H^{-1/2}(\Gamma_N) \times H^{1/2}_{00}(\Gamma_D),
\quad \hat\cB_\pp[\sigma',u']:= \begin{pmatrix}\cB_\pp[\sigma',u']\\ \sigma' \cdot n|_{\Gamma_N}\\
u'|_{\Gamma_D} 
\end{pmatrix} .
\end{equation}
If one can show that this is also an isomorphism, a corresponding
variationally correct fiber-residual relation would now read for each $\pp\in\PP$
\begin{align}
\label{trace}
\|[\tilde\sigma,\tilde u] - [\sigma,u]\|^2_{H(\div;\Omega)\times H^1(\Omega)} 
& \eqsim  \|\tilde\sigma -   \pp\nabla \tilde u   - f_1 \|^2_{L_2(\Omega;\bR^d)} 
  +
\|\div \tilde\sigma + f_2\|^2_{L_2(\Omega)}\nonumber\\
&\qquad + \|\tilde u- u_0\|^2_{H^{1/2}(\Gamma_D)} + \|\tilde\sigma \cdot n-g\|^2_{H^{-1/2}(\Gamma_N)},\nonumber\\
&\qquad \qquad [\tilde\sigma,\tilde u]\in H(\div;\Omega)\times H^1(\Omega).
\end{align}
Thus, the residual now involves a broken Sobolev and dual Sobolev trace-norm whose evaluation is far from straightforward (although possible for the Dirichlet conditions using the intrinsic definitions as a double integral over $\Gamma_D$). Simply incorporating the Dirichlet boundary conditions by adding the term 
$\|\tilde u- u_0\|^2_{L^2(\Gamma_D)} $ instead of $\|\tilde u- u_0\|^2_{H^{1/2}(\Gamma_D)}$ is {\bf not} variationally correct, i.e., does {\bf not} provide a rigorous error bound. Replacing $\|\tilde\sigma \cdot n-g\|^2_{H^{-1/2}(\Gamma_N)}$ by $\|\tilde\sigma \cdot n-g\|^2_{L_2(\Gamma_N)}$ would at least yield an upper bound for the error, but may still, depending on the data, ask too much regularity.

In summary, the above findings will be used as follows. Given a hypothesis class $\mathcal{H}(\theta)$,
comprised of functions $[\sigma(\cdot;\theta),u(\cdot;\theta)]$ that depend on trainable weights 
$\theta\in \Theta$, we can now follow the lines outlined at the end of Section \ref{sec:conceptual_background_and_orientation}. Most importantly,
in the present situation, we have $\V= L_2(\Omega;\R^{d+1})$ so that the losses in \eqref{emploss} and \eqref{trainPoisson}, respectively, are computable.

\subsection{Linear elasticity equation} 
We turn now to the classical model for linear elasticity. We consider a convex polygonal domain $\Omega\subset\bR^{d}$ (typically $d=1,2,3$ for practical scenarios). While plain letters denote scalar-valued functions, we distinguish vectors and (rank-2) tensors by single or double underscores such as 
$$
\tenbar[1]{f}=(f_1, f_2, \ldots, f_d) \text{ and } \tenbar[2]{f} = 
\begin{psmallmatrix}
    f_{11} & f_{12} & \cdots & f_{1d} \\
    f_{21} & f_{22} & \cdots & f_{2d} \\
    \vdots & \vdots & \ddots & \vdots \\
    f_{d1} & f_{d2} & \cdots & f_{dd}
\end{psmallmatrix}.
$$

\begin{equation*}
    \tenbar[1]{\rgrad} f = 
    \begin{pmatrix}
        \frac{\partial f}{\partial x_1}, & \frac{\partial f}{\partial x_2}, & \cdots, & \frac{\partial f}{\partial x_d}
    \end{pmatrix}, \quad 
    \rdiv \tenbar[1]{f} = \sum_{i=1}^{d}\frac{\partial f_i}{\partial x_i}, 
\quad 
    \tenbar[2]{\rgrad}{\tenbar[1]{f}} = 
    \begin{pmatrix}
        \tenbar[1]{\rgrad}{f_1} \\ 
        \tenbar[1]{\rgrad}{f_2} \\ 
        \vdots \\
        \tenbar[1]{\rgrad}{f_d}
    \end{pmatrix}, \quad 
     \tenbar[1]{\rdiv} \tenbar[2]{f} = 
        \begin{pmatrix}
           \rdiv \tenbar[1]{f_1} \\
           \rdiv \tenbar[1]{f_2} \\
           \vdots \\
           \rdiv \tenbar[1]{f_d} \\
        \end{pmatrix}.
\end{equation*}
The vector-vector and tensor-tensor inner products ``$\cdot$'' respectively  ``$\colon$'' are then
given  by
\begin{equation*}
    \tenbar[1]{f} \cdot \tenbar[1]{g} = \sum_{i=1}^{d}f_ig_i, 
    \quad 
    \tenbar[2]{f}\colon\tenbar[2]{g} = \sum_{i=1}^{d}\sum_{j=1}^{d} f_{ij}g_{ij}, 
\end{equation*}
and the tensor-vector inner product reads 
\begin{equation*}
    \tenbar[2]{f} \cdot \tenbar[1]{g} = 
    \begin{pmatrix}
        \tenbar[1]{f_1} \cdot \tenbar[1]{g}, \tenbar[1]{f_2} \cdot \tenbar[1]{g}, \ldots,  \tenbar[1]{f_d} \cdot \tenbar[1]{g}
    \end{pmatrix}^{T}.
\end{equation*}

In these terms, the strong formulation of a linear elasticity equation with mixed boundary conditions (Dirichlet and Neumann) can be written as 
\begin{equation}\label{eq:elasticity_strong_primal}
    \begin{cases}
       -\tenbar[1]{\rdiv} \tenbar[2]{\sigma}(\tenbar[1]{u})= \tenbar[1]{f},  &\qq{in $\Omega$,} \\
        \tenbar[1]{u} = \tenbar[1]{u_0},  &\qq{on $\Gamma_D$,} \\
        \tenbar[2]{\sigma}(\tenbar[1]{u}) \cdot \tenbar[1]{n} = \tenbar[1]{t}, &\qq{on $\Gamma_N$,} 
    \end{cases}
\end{equation}
where the Cauchy stress tensor $\tenbar[2]{\sigma}$ is a function of the displacement vector field $\tenbar[1]{u}$
\begin{equation}\label{eq:stress-tensor}
    \tenbar[2]{\sigma}(\tenbar[1]{u})= 2\mu \, \tenbar[2]{\varepsilon}(\tenbar[1]{u}) + \lambda \, \rtr(\tenbar[2]{\varepsilon}(\tenbar[1]{u})) \tenbar[2]{I_d}.
\end{equation}
Here $\lambda, \mu\in\bR$ are (possibly spatially dependent) parameters associated with the material, $\tenbar[2]{I_d}$ is the $d\times d$ identity matrix, and $\tenbar[2]{\varepsilon}(\tenbar[1]{u})=\frac{1}{2}(\tenbar[2]{\rgrad}\tenbar[1]{u} + (\tenbar[2]{\rgrad}\tenbar[1]{u})^{T})$ is the strain tensor.
Note that in these terms, we have
\begin{equation}
\label{Green}
\int_\Omega (\tenbar[1]{{\rm div}}\, \tenbar[2]{\sigma})\cdot \tenbar[1]{v}\,dx = - \int_\Omega
\tenbar[2]{\sigma} : {\rgrad}\tenbar[1]{v}\,dx + \int_{\partial\Omega}(\tenbar[2]{\sigma}\cdot\tenbar[1]{n})\cdot\tenbar[1]{v}\,d\gamma.
\end{equation}

As in the previous section, to properly handle boundary conditions, we start with the standard weak formulation of \eqref{eq:elasticity_strong_primal}: find $\tenbar[1]{u}\in H^1 (\Omega;\R^d)$ such that 
\begin{equation}
\label{elast-weak}
b(\tenbar[1]{u},\tenbar[1]{v}):= \int_\Omega {\tenbar[2]\sigma(\tenbar[1]{u})}:\tenbar[2]{\rgrad}\tenbar[1]{v}\,dx =\int_{\Gamma_N} 
\tenbar[1]{t}\cdot 
\tenbar[1]{v}\, d\gamma + \int_{\Omega} \tenbar[1]{f}\cdot \tenbar[1]{v}\, dx,
\quad v\in \U= H^1_{0,\Gamma_D}(\Omega;\R^d),
\end{equation}
where we assume here for convenience that $\tenbar[1]{f}\in L^2(\Omega;\R^d)$. Otherwise, we would need to consider a decomposition of flux-free $\tenbar[1]{f}\in (H^1_{0,\Gamma_D}(\Omega;\R^d))'$, by applying the scalar-valued decompositions \eqref{flfree} componentwise,
\begin{equation}
\label{tendeco}
\tenbar[1]{f}= \tenbar[1]{f_2} + \tenbar[1]{{\rm div}}\,\tenbar[2]{f_1},
\end{equation}
where $\tenbar[1]{f_2}\in L_2(\Omega;\R^d)$ and $\tenbar[2]{f_1}\in L_2(\Omega;\R^{d\times d})$.

As in the previous section, we wish to represent the boundary conditions as source terms. Regarding the displacement condition, let $\tenbar[1]{w} \in H^1(\Omega;\R^d)$ be any fixed function satisfying $\tenbar[1]{w}|_{\Gamma_D}= \tenbar[1]{u_0}|_{\Gamma_D}$ and $\tenbar[2]{\rgrad}{\tenbar[1]w}\cdot\tenbar[1]{n} = \tenbar[1] 0$ on $\Gamma_N$ (componentwise in the sense of traces). For instance, one can solve
\begin{equation}
\label{eq:DiriLiftElasticity}
a(\tenbar[1]{w},\tenbar[1]{v}):=\int_\Omega \tenbar[2]{\rgrad}{\tenbar[1]{w}}:\tenbar[2]{\rgrad}{\tenbar[1]{v}}\,dx = 0,\quad \tenbar[1]{w}|_{\Gamma_D}= \tenbar[1]{u_0}|_{\Gamma_D},
\end{equation}
approximately in some finite element space. Then \eqref{elast-weak} is equivalent to finding $\tenbar[1]{u}^\circ\in \U$ such that
\begin{equation}
\label{elast-weak2}
b(\tenbar[1]{u}^\circ,\tenbar[1]{v}) = -b(\tenbar[1]{w},\tenbar[1]{v}) + 
\int_{\Gamma_N} 
\tenbar[1]{t}\cdot 
\tenbar[1]{v}\, d\gamma + \int_{\Omega} \tenbar[1]{f}\cdot \tenbar[1]{v}\, dx,
\quad v\in \U.
\end{equation}
As in the previous section, $\int_{\Gamma_N} 
\tenbar[1]{t}\cdot 
\tenbar[1]{v}\, d\gamma$ is to be understood as a dual pairing, denoted by $\langle\tenbar[1]{t},\tenbar[1]{v}\rangle_{\Gamma_N}$. Likewise, we abbreviate for convenience $(\tenbar[1]{f},\tenbar[1]{v})_\Omega := \int_{\Omega} \tenbar[1]{f}\cdot \tenbar[1]{v}\, dx$. We can view $-b(\tenbar[1]{w},\cdot)$ as a functional in $\U'$ acting on $v\in \U$, namely, for $\tenbar[1]{w}$ as above we have (see \eqref{Green})
\begin{equation}
\label{act}
-b(\tenbar[1]{w},\tenbar[1]{v})= \tenbar[1]{{\rm div}}(  \tenbar[2]{\sigma}(\tenbar[1]{w}) )(\tenbar[1]{v}),
\end{equation}
and hence is the divergence of an $L_2(\Omega;\R^{d\times d})$-tensor field. 
To represent the normal-trace integral in $\U'$, we proceed as before and consider the auxiliary problem: find $\tenbar[1]{q}\in \U=H^1_{0,\Gamma_D}(\Omega;\R^{d})$ such that
\begin{equation}
\label{aux-elast}
a(\tenbar[1]{q}, \tenbar[1]{v}) = \int_{\Gamma_N} \tenbar[1]{t}\cdot 
\tenbar[1]{v}\, d\gamma.
\end{equation}
In strong form, this solves $-\tenbar[1]{{\rm div}}\,\tenbar[2]{\rgrad}(\tenbar[1]{q}) = 
0$ with $\tenbar[2]{\rgrad} \tenbar[1]{q}\cdot \tenbar[1]{n} = \tenbar[1]{t} $ on $\Gamma_N$ and $\tenbar[1]{q}=0$ on $\Gamma_D$.

Hence, (recalling that $\tenbar[1]{f}\in L_2(\Omega;\R^d)$), in complete analogy to the previous section, a weak formulation of \eqref{eq:elasticity_strong_primal} can be stated with $\tenbar[2]{z}:= \tenbar[2]{\rgrad}\tenbar[1]{q}$ as
\begin{equation}
\label{restate}
 b(  \tenbar[1]{u}^\circ+ \tenbar[1]{w},\tenbar[1]{v})= (\tenbar[1]{f},\tenbar[1]{v})_\Omega  
 +a( 
 \tenbar[1]{q},\tenbar[1]{v}),
 \quad \tenbar[1]{u}^\circ|_{\Gamma_D}=0,
 \end{equation}
which is a linear elasticity system with homogeneous displacement boundary conditions on boundary $\Gamma_D$.
In subsequent experiments, we consider the parametric linear elasticity equation with the following parametrization: $(\lambda, \mu): \Omega \to\mathbb{R}^2$ 
are the Lambda parameters depending on the Young's modulus $E: \Omega \to \mathbb{R}$ and Poisson ratio $\nu\in\mathbb{R}$ as
\begin{align*}
    \mu(x)=\frac{E(x)}{2(1+\nu)}, \quad \lambda(x)=\frac{\nu E(x)}{(1+\nu)(1-2\nu)}.
\end{align*}
We assume the Young's modulus is a random field parameter given by $E(x)=\exp{\pp(x)} + 1$, where $\pp$ is a random field with Gaussian measure $\cN(\bar{\pp}, \cC)$ with mean $\bar{\pp}$ and covariance operator $\cC :=(\delta I -\gamma \Delta)^{-\alpha}$, with $\delta, \gamma, \alpha > 0$ collectively determining the correction, variance, and smoothness. 
We restrict our analysis to random field parameters satisfying $\pp \leq \hat{\pp}$ for a sufficiently large constant $\hat{\pp}$. This ensures that the Young's modulus $E$ is uniformly bounded, with $1 \leq E \leq \exp(\hat{\pp}) + 1$. We denote the parametric linear mapping $\cC_\pp$, a rank-4 stiffness tensor, that maps the strain $\tenbar[2]{\varepsilon}$ to the stress $\tenbar[2]{\sigma}$ as 
\begin{equation*}
    \tenbar[2]{\sigma} = \cC_\pp\tenbar[2]{\varepsilon} := 2\mu\,\tenbar[2]{\varepsilon} + \lambda\, \rtr (\tenbar[2]{\varepsilon})\tenbar[2]{I_d}.
\end{equation*}
Under our assumptions on the range of parameters,
its inverse $\cC_\pp^{-1}$ is given by 
\begin{equation}\label{eq:Cinv}
    \cC^{-1}_\pp\tenbar[2]{\sigma} = \frac{1}{2\mu}\tenbar[2]{\sigma} - \frac{\lambda}{2\mu(\lambda d+2\mu)}\rtr (\tenbar[2]{\sigma})\tenbar[2]{I_d}.
\end{equation}

The following first-order formulation of Equation~\eqref{restate} is obtained by treating the stress tensor $\tenbar[2]{\sigma}$ as an  independent unknown variable and (in its preconditioned form) reads (recall that
$\tenbar[2]{z}= \tenbar[2]{\rgrad}\tenbar[1]{q}$)
\begin{equation}\label{eq:elasticity_strong_first_order}
    \begin{cases}
       -\tenbar[1]{\rdiv} \tenbar[2]{\sigma}^\circ = \tenbar[1]{f},  &\qq{in $\Omega$,} \\[4pt]
       \cC_\pp^{-1/2} \tenbar[2]{\sigma}^\circ= \cC_\pp^{1/2}\tenbar[2]{\varepsilon}(\tenbar[1]{u}^\circ +\tenbar[1]{w})-   \cC_\pp^{-1/2}\tenbar[2]{z}, & \qq{in $\Omega$,}\\[4pt]
        \tenbar[1]{u}^\circ = 0
        ,  &\qq{on $\Gamma_D$,} \\[4pt]
       \tenbar[2]{\sigma}^\circ \cdot \tenbar[1]{n} = \tenbar[1]{0}, &\qq{on $\Gamma_N$.} 
    \end{cases}
\end{equation}
In this form we can apply the results in \cite{bochev2009least,FOSLS,opschoor2024first} to arrive at
 the weak (fiber) formulation:
 find $\tenbar[2]{\sigma}^\circ \in \Sigma := H_{0, \Gamma_N}(\rdiv; \Omega;\R^{d\times d})$ and 
$\tenbar[1]{u}^\circ\in \U := H^1_{0,\Gamma_D}(\Omega;\R^d)$ such that
\begin{align}
\label{1stweak}
b_\pp([\tenbar[2]{\sigma}^\circ,\tenbar[1]{u}^\circ], [\tenbar[2]{\tau},\tenbar[1]{v}]) &:= (
-\tenbar[1]{\rdiv}\tenbar[2]{\sigma}^\circ,\tenbar[1]{v})_\Omega + ( \cC_\pp^{-1/2} (\tenbar[2]{\sigma}^\circ- \cC_\pp\tenbar[2]{\varepsilon}(\tenbar[1]{u}^\circ)),   \tenbar[2]{\tau})_\Omega\nonumber\\
& = (\tenbar[1]{f},\tenbar[1]{v})_\Omega + ( \cC_\pp^{1/2}\tenbar[2]{\varepsilon}(\tenbar[1]{w}) - \cC_\pp^{-1/2}\tenbar[2]{z},
\tenbar[2]{\tau})_\Omega,\quad \tenbar[2]{\tau}\in L^2(\Omega;\R^{d\times d}),\, 
\tenbar[1]{v}\in L^2(\Omega;\R^d).
\end{align}
We define as in the previous section the operator
$\cB_\pp: \HH:= \Sigma \times \U \to \LL_2
:= L_2(\Omega;\R^{d\times d})\times L_2(\Omega;\R^d)$ by $(\cB_\pp[\tenbar[2]{\sigma}^\circ,\tenbar[1]{u}^\circ])([\tenbar[2]{\tau},\tenbar[1]{v}])= b_\pp([\tenbar[2]{\sigma}^\circ,\tenbar[1]{u}^\circ],[\tenbar[2]{\tau},\tenbar[1]{v}])$ for all $([\tenbar[2]{\sigma}^\circ,\tenbar[1]{u}^\circ],[\tenbar[2]{\tau},\tenbar[1]{v}])\in \HH \times \LL_2$, i.e.,
$$
\cB_\pp[\tenbar[2]{\sigma}^\circ,\tenbar[1]{u}^\circ] =  {\cC_\pp^{-1/2}(\tenbar[2]{\sigma}^\circ -
\cC_\pp\tenbar[2]{\e}(\tenbar[1]{u}^\circ))\choose -\tenbar[1]{\rdiv}\tenbar[2]{\sigma}^\circ}.
$$
It has been shown in \cite{bochev2009least,FOSLS,opschoor2024first} that
$\cB_\pp$ is for each $\pp\in\PP$ an isomorphism. Corresponding inf-sup- or continuity constants depend in general on $\pp$ so that any uniform stability depends on the range permitted in~$\PP$. 
Similarly to the error-residual relation \eqref{err-res} for the case of the stationary diffusion problem, we obtain, for any approximation $[\tenbar[2]{\tilde\sigma}^\circ, \tenbar[1] {\tilde u}^\circ] $ of the exact solution $[\tenbar[2]{\sigma}^\circ, \tenbar[1]{u}^\circ] $ at any given parameter $\pp$, the fiber-error-residual relation
\begin{equation}
\label{err-res-el}
\|[\tenbar[2]{\tilde\sigma}^\circ, \tenbar[1] {\tilde u}^\circ]- [\tenbar[2]{\sigma}^\circ, \tenbar[1]{u}^\circ]
\|^2_\HH \eqsim \cL([\tenbar[2]{\tilde \sigma}^\circ, \tenbar[1]{\tilde u}^\circ]; \pp)\quad \forall\, [\tenbar[2]{\tilde\sigma}^\circ, \tenbar[1] {\tilde u}^\circ] \in \HH, \, \pp\in \PP, 
\end{equation}
where the residual $\cL([\tenbar[2]{\tilde \sigma}^\circ, \tenbar[1]{\tilde u}^\circ]; \pp)$ of the approximation $[\tenbar[2]{\tilde \sigma}^\circ, \tenbar[1]{\tilde u}^\circ]$ at parameter $\pp$ is given by 
\begin{equation}\label{eq:lossElasticity}
   \cL([\tenbar[2]{\tilde \sigma}^\circ, \tenbar[1]{\tilde u}^\circ]; \pp) := \|\cC_\pp^{-1/2}\tenbar[2]{\tilde\sigma}^\circ - \cC_\pp^{1/2}\tenbar[2]{\varepsilon}(\tenbar[1]{\tilde u}^\circ+ \tenbar[1]{w}) + \cC_\pp^{-1/2}\tenbar[2]{z}\|_{L^2(\Omega;\R^{d\times d})}^2 + \| \tenbar[1]{\rdiv} \tenbar[2]{\tilde \sigma}^\circ + \tenbar[1]{f} \|_{L^2(\Omega;\R^d)}^2.
\end{equation}
This is for each $\pp\in\PP$  an easily computable loss function that enters
the corresponding mean-squared loss in \eqref{emploss} which is used to train the surrogate models.

Training yields then $[\tenbar[2]{\sigma}^\circ(\theta^*),\tenbar[1]{u}^\circ(\theta^*)]
\in  \X = L_\mu^2(\PP;\HH)$
from which we recover an approximation to the solution of \eqref{eq:elasticity_strong_first_order} by
\begin{equation}
\label{ptou}
\tenbar[1]{\tilde u}  (\pp) = \tenbar[1]{u}^\circ (\pp;\theta^*)) + \tenbar[1]{w},\quad 
\tenbar[2]{\tilde \sigma}(\pp) = \tenbar[2]{\sigma}^\circ(\pp;\theta^*) +  \tenbar[2]{z}.
\end{equation}
We immediately infer from \cite{bachmayr2025variationally} that the exact
solution $[\sc ,\uc]\in \X$ of \eqref{eq:elasticity_strong_first_order} satisfies
\begin{equation}
\label{lifted-err-res}
\|\tenbar[2]{\sigma}^\circ(\theta^*),\tenbar[1]{u}^\circ(\theta^*)] - [\tenbar[2]{\sigma}^\circ ,\tenbar[1]{u}^\circ]\|^2_{\ell_2(\widehat\PP;\HH)}
\eqsim \cL([\tenbar[2]{\sigma}^\circ(\theta^*),\tenbar[1]{u}^\circ(\theta^*)];\widehat\PP).
\end{equation}

\section{Conforming finite element approximations}
\label{sec:conforming_finite_element_approximations}

To evaluate the residual losses \eqref{eq:lossPoisson} (diffusion) and \eqref{eq:lossElasticity} (elasticity), we employ conforming Galerkin FOSLS discretizations with mixed finite-element (FE) pairs \(\Sigma_h\times\U_h\) and discrete boundary lifts. We formulate the discrete problems with lifted boundary conditions and, under uniform ellipticity, establish (i)~variational error–residual equivalence and (ii)~convergence rates for the FE losses evaluated at the FE solutions.

\subsection{Conforming finite element approximation of the diffusion problem}

For the diffusion problem and each parameter $\pp\in\PP$, we compute the conforming FE approximation $[\tilde{\sigma}^\circ, {\tilde u}^\circ] = [\sc_h, \uc_h] \in \HH_h$ of the lifted fiber solution as 
\begin{equation}\label{eq:FEApproxPoisson}
    \sc_h(x, \pp) = \sum_{n = 1}^{N_h^\sigma} \sigma^\circ_n(\pp) \varphi_n^\sigma(x) \text{ and } u^\circ_h(x, \pp) = \sum_{n = 1}^{N_h^u}  u_n^\circ(\pp) \varphi_n^u(x),
\end{equation}
where $\boldsymbol{\sigma}^\circ(\pp) = ( \sigma^\circ_1(\pp), \dots, \sigma^\circ_{N_h^\sigma}(\pp))^\top$ and $\boldsymbol{ u^\circ}(\pp) = ( u_1^\circ(\pp), \dots, u^\circ_{N_h^u}(\pp))^\top$ are the parametric coefficient vectors, $(\varphi_n^\sigma)_{n =1}^{N_h^\sigma}$ and $(\varphi_n^u)_{n =1}^{N_h^u}$ are the basis functions of the FE space $\HH_h = \Sigma_h \times \U_h\subset H(\rdiv;\Omega) \times H^1(\Omega)$, defined over a mesh with element size $h$. Here $N_h^\sigma$ and $N_h^u$ are the number of degrees of freedom (DoFs) of the FE spaces $\Sigma_h$ and $\U_h$, respectively. Specifically, we employ the conforming FE pairs 
$$
\Sigma_h \times \U_h = \text{RT}_k^\circ \times \text{CG}_m^\circ,\quad k\ge0,\; m\ge1,
$$ 
where $\text{RT}_k$ is the \emph{Raviart–Thomas element}~\cite{logg2012automated} of order $k$ (which has polynomial degree $\leq k+1$), $H(\rdiv;\Omega)$-conforming vector fields with continuous normal fluxes (see, e.g., \href{https://defelement.org/elements/raviart-thomas.html}{DefElement} for details), and CG$_m$ is the \emph{continuous Galerkin element} (or Lagrange element) of degree $m$, with piecewise polynomials of degree $\leq m$ on each cell and $C^0$-continuous across cell interfaces, thus $H^1(\Omega)$-conforming scalar fields. Here $\text{RT}^\circ_k$ and $\text{CG}^\circ_m$ denote the subspaces with the essential boundary conditions incorporated, i.e., $\tau\cdot n = 0$ on $\Gamma_N$ for any $\tau\in \text{RT}^\circ_k$ and $v=0$ on $\Gamma_D$ for any $v\in \text{CG}^\circ_m$, practically by setting the corresponding DoFs to zero.

For each $\pp$ we compute the FE approximation of the solution $[\sc(\pp),\uc(\pp)]$ by solving the Galerkin system
\begin{align}
\label{gal}
(\cB_\pp [\sc_h,u^\circ_h],\cB_\pp[\tau_h,v_h])_\Omega &= ( [\pp\nabla w_h - z_h + f_1,f_2],
\cB_\pp[\tau_h,v_h])_\Omega, \quad \forall \, [\tau_h,v_h]\in \HH_h,
\end{align}
which reads in explicit terms as: find $[\sc_h,u_h^\circ]\in \Sigma_h\times\U_h$ such that
\begin{align*}
&  \quad (\sc_h -   \pp\nabla u^\circ_h, \tau_h- \pp\nabla v_h)_\Omega + ( \text{div}\,\sc_h, \text{div}\,\tau_h)_\Omega\\
& = 
(\pp\nabla w_h - z_h + f_1, \tau_h-\pp\nabla v_h)_\Omega - (f_2, \text{div}\,\tau_h)_\Omega,
\quad \forall \,  [\tau_h,v_h]\in \HH_h.
\numberthis\label{eq:normal_poisson}
\end{align*}
Here the auxiliary variables $w_h \in \text{CG}_m$ and $q_h \in \text{CG}^\circ_m$ are conforming Galerkin approximations of the lifts $w$ and $q$ from \eqref{eq:DiriLiftPoisson} and \eqref{harm}, respectively, on the same mesh and with the same polynomial degree $m$ as $u_h^\circ$; we set $z_h := \nabla q_h$. The FE loss is assembled with these discrete lifts $(w_h, z_h)$. Note that the solution of the Galerkin problem \eqref{eq:normal_poisson} satisfies the minimal residual property
\begin{equation}\label{eq:FELSPoisson}
    [\sc_h(\pp), u^\circ_h(\pp)] = \argmin_{[\tau_h, v_h] \in \HH_h} \cL([\tau_h,  v_h]; \pp).
    \end{equation}
By construction, $[\sc_h, u_h^\circ]$ minimizes the fiber loss over $\Sigma_h\times\U_h$, i.e., it is the orthogonal projection in the FOSLS inner product induced by $(\tau-\pp\nabla v,\eta-\pp\nabla s) + (\div\tau,\div\eta)$.

Under standard regularity assumptions for the exact fiber solution $[\sc, \uc]$ and uniform ellipticity of the parameter $\pp$, the FE loss evaluated at the FE solution $[\sc_h, u_h^\circ]$ satisfies the following properties; a proof is given in \ref{app:proofFELSPoisson}.

\begin{theorem}\label{thm:FELSPoisson}
Assume $\Omega$ is Lipschitz and $\pp\in L^\infty(\Omega)$ is uniformly bounded with $0<\alpha\le \pp(x)\le \beta<\infty$ independent of $\pp\in\PP$. Let $\Sigma_h=\mathrm{RT}^\circ_k$ and $\U_h=\mathrm{CG}^\circ_m$ on a shape-regular mesh of size $h$ with $k\ge0$, $m\ge1$. Let $[\sc,\uc]$ be the exact solution and $[\sc_h,\uc_h]$ the Galerkin solution in $\Sigma_h\times\U_h$. Then, for each fixed $\pp\in\PP$, one has the error-residual equivalence
\begin{equation}\label{eq:FE-loss-equiv}
  \mathcal{L}([\sc_h,\uc_h];\pp) \eqsim \|\sc-\sc_h\|_{H(\div;\Omega)}^2 + \|\uc-\uc_h\|_{H^1(\Omega)}^2,
\end{equation}
with equivalence constants depending only on $\alpha,\beta$ and the domain.

Moreover, if $\sc\in H^{s_{\sigma}}(\Omega;\R^d)$, $\div\sc\in H^{s_{\div}}(\Omega)$ with $s_{\sigma},s_{\div}\ge0$, and $\uc\in H^{s_u}(\Omega)$ with $s_u\ge1$, then
\begin{equation}\label{eq:FELS-bound-Poisson}
  \mathcal{L}([\sc_h,\uc_h];\pp)
  \lesssim h^{2\min(k+1,s_{\sigma})}\,\|\sc\|_{H^{s_{\sigma}}}^2
        + h^{2\min(k+1,s_{\div})}\,\|\div\sc\|_{H^{s_{\div}}}^2
        + h^{2\min(m,s_u-1)}\,\|\uc\|_{H^{s_u}}^2.
\end{equation}
In particular, if $\sc\in H^{k+1}$, $\div\sc\in H^{k+1}$, and $\uc\in H^{m+1}$, then
\(\mathcal{L}([\sc_h,\uc_h];\pp)\lesssim h^{2(k+1)} + h^{2m}\), and the balanced choice $m=k+1$ yields the optimal scaling
\(\mathcal{L}([\sc_h,\uc_h];\pp)\lesssim h^{2(k+1)}\).

\end{theorem}

\subsection{Conforming finite element approximation of the elasticity problem}

For the elasticity problem and each parameter $\pp\in\PP$, we approximate the lifted fiber solution $[\tenbar[2]{\sigma}^\circ(\pp),\tenbar[1]{u}^\circ(\pp)]$ by a conforming Galerkin FOSLS discretization in the mixed FE space
\[
  \HH_h = \Sigma_h\times\U_h \subset H_{0,\Gamma_N}(\rdiv;\Omega;\R^{d\times d})\times H^1_{0,\Gamma_D}(\Omega;\R^d),
\]
where we take
\[
  \Sigma_h \times \U_h = (\mathrm{RT}_k^{\circ})^d \times (\mathrm{CG}_m^{\circ})^d,\quad k\ge0,\; m\ge1.
\]
Note the essential boundary conditions are built into the spaces: $\tenbar[2]{\tau}_h\cdot\tenbar[1]{n}=\tenbar[1]0$ on $\Gamma_N$ for $\tenbar[2]{\tau}_h\in\Sigma_h$ and $\tenbar[1]{v}_h=\tenbar[1]0$ on $\Gamma_D$ for $\tenbar[1]{v}_h\in\U_h$.
We parametrize the FE approximations $[\tenbar[2]{\sigma}^\circ_h(\pp),\tenbar[1]{u}^\circ_h(\pp)]$ in the same form as in \eqref{eq:FEApproxPoisson} with the tensor- and vector-valued basis functions of $\Sigma_h$ and $\U_h$, respectively, with corresponding coefficient vectors $\boldsymbol{\sigma}^\circ(\pp)$ and $\boldsymbol{u}^\circ(\pp)$.
Discrete boundary lifts are computed analogously to the diffusion case. Specifically, we compute $\tenbar[1]{w}_h\in (\mathrm{CG}_m)^d$ as the vector-valued harmonic lift of the Dirichlet data by solving the auxiliary problem \eqref{eq:DiriLiftElasticity}, and $\tenbar[1]{q}_h\in (\mathrm{CG}^\circ_m)^d$ as the solution of the auxiliary problem \eqref{aux-elast}; we then set the tensor lift $\tenbar[2]{z}_h := \tenbar[2]{\rgrad}\,\tenbar[1]{q}_h$. The FE loss and right-hand side are assembled with these discrete lifts $(\tenbar[1]{w}_h,\tenbar[2]{z}_h)$.

For each $\pp \in \PP$, we compute the FE solution $[\tenbar[2]{\sigma}^\circ_h,\tenbar[1]{u}^\circ_h]\in\Sigma_h\times\U_h$ by solving the Galerkin system
\begin{align}
\label{eq:normal_elasticity}
  (\cB_\pp[\tenbar[2]{\sigma}^\circ_h,\tenbar[1]{u}^\circ_h],\,\cB_\pp[\tenbar[2]{\tau}_h,\tenbar[1]{v}_h])_\Omega
  &= ([\cC_\pp^{1/2}\tenbar[2]{\e}(\tenbar[1]{w}_h) - \cC_\pp^{-1/2}\tenbar[2]{z}_h, \tenbar[1]{f}],\,\cB_\pp[\tenbar[2]{\tau}_h,\tenbar[1]{v}_h])_\Omega, \quad \forall[\tenbar[2]{\tau}_h,\tenbar[1]{v}_h]\in\Sigma_h\times\U_h,
\end{align}
which in expanded form reads: find $[\tenbar[2]{\sigma}^\circ_h,\tenbar[1]{u}^\circ_h]\in\Sigma_h\times\U_h$ such that 
\begin{align*}
  &\quad (\cC_\pp^{-1/2}(\tenbar[2]{\sigma}^\circ_h- \cC_\pp\tenbar[2]{\e}(\tenbar[1]{u}^\circ_h)),\,\cC_\pp^{-1/2}(\tenbar[2]{\tau}_h- \cC_\pp\tenbar[2]{\e}(\tenbar[1]{v}_h)))_\Omega
   + (\tenbar[1]{\rdiv}\,\tenbar[2]{\sigma}^\circ_h,\,\tenbar[1]{\rdiv}\,\tenbar[2]{\tau}_h)_\Omega\\
  &= (\cC_\pp^{1/2}\tenbar[2]{\e}(\tenbar[1]{w}_h) - \cC_\pp^{-1/2}\tenbar[2]{z}_h,\,\cC_\pp^{-1/2}(\tenbar[2]{\tau}_h- \cC_\pp\tenbar[2]{\e}(\tenbar[1]{v}_h)))_\Omega
   - (\tenbar[1]{f},\,\tenbar[1]{\rdiv}\,\tenbar[2]{\tau}_h)_\Omega, \quad \forall[\tenbar[2]{\tau}_h,\tenbar[1]{v}_h]\in\Sigma_h\times\U_h.  
   \numberthis
\end{align*}
Equivalently, $[\tenbar[2]{\sigma}^\circ_h,\tenbar[1]{u}^\circ_h]$ minimizes the FE fiber loss over $\Sigma_h\times\U_h$, i.e.,
\begin{equation}\label{eq:FELSElasticity}
  [\tenbar[2]{\sigma}^\circ_h(\pp),\tenbar[1]{u}^\circ_h(\pp)] = \argmin_{[\tenbar[2]{\tau}_h,\tenbar[1]{v}_h]\in\HH_h}
  \cL([\tenbar[2]{\tau}_h,\tenbar[1]{v}_h];\pp).
\end{equation}

Note that the Cauchy stress tensor $\tenbar[2]{\sigma}$ in \eqref{eq:stress-tensor} is symmetric, while the FE approximation $\tenbar[2]{\sigma}_h = \tenbar[2]{\sigma}^\circ_h + \tenbar[2]{z}_h$, with $\tenbar[2]{\sigma}^\circ_h \in (\text{RT}_k^\circ)^d$ and $\tenbar[2]{z}_h = \tenbar[2]{\rgrad}\,\tenbar[1]{q}_h$ with $\tenbar[1]{q}_h \in (\text{CG}_m^\circ)^d$, is not necessarily symmetric. To address this, we can impose weak symmetry by adding a penalty term $\|\tenbar[2]{\sigma}_h - \tenbar[2]{\sigma}_h^\top\|_{L^2(\Omega)}^2$ to the FE loss, leading to corresponding modifications in the Galerkin system \eqref{eq:normal_elasticity}. We omit this for simplicity. 

Under standard regularity assumptions for the exact fiber solution $[\tenbar[2]{\sigma}^\circ,\tenbar[1]{u}^\circ]$ and uniform bounds for the elasticity tensor $\cC_\pp$, the FE loss evaluated at the FE solution satisfies similar properties as in Theorem \ref{thm:FELSPoisson}. We present the theorem and proof in \ref{app:proofFELSElasticity} for completeness.

\section{Reduced basis neural operators}
\label{sec:reduced_basis_neural_operators}

\subsection{Reduced basis approximations}\label{ssec:errors}
Given the high-dimensional nature of the FE coefficient vectors $\bssigma^\circ_h(\pp) \in \bR^{N_h^\sigma}$ and $\bsu^\circ_h(\pp) \in \bR^{N_h^u}$ in \eqref{eq:FEApproxPoisson}, directly parameterizing these full DoF representations with a neural network leads to very large models and expensive residual evaluations, especially for large-scale FE discretizations. Instead, we seek a low-dimensional representation of the parametric solution manifold that is conformal with the underlying variational structure and allows for efficient evaluation of the loss.

The following facts and notational conventions will be frequently used in what follows.
For the parametric coefficient vector $\bss_h(\pp) = (\bssigma^\circ_h(\pp);\bsu^\circ_h(\pp)) \in \bR^{N_h^s}$ of the FE functions $s_h(\pp) = [\sc_h(\pp), u^\circ_h(\pp)]$ in the diffusion case, we first equip $\bR^{N_h^s}$ with the discrete $H(\rdiv) \times H^1$-norm 
\begin{equation}\label{eq:X-norm-Poisson}
    ||\bss_h||_{{X_h}}^2 := \|[\sc_h,\uc_h]\|^2_{\HH_h} = (\sc_h, \sc_h)_{H(\rdiv)} + (u^\circ_h, u^\circ_h)_{H^1}, 
\end{equation}
where each term is assembled using the same FE pair as in the high-fidelity discretization. The corresponding Gram matrix ${X_h}$ defines the discrete norm $||\bss_h||_{{X_h}}^2 = \bss_h^\top {X_h} \bss_h$.
We then perform Proper Orthogonal Decomposition (POD) in this ${X_h}$-inner product. Specifically, given $N_s$ snapshot solutions for a training sample set $\PP_{N_s}=\{\pp_1,\ldots,\pp_{N_s}\}\subset \PP$, stored in the matrix $\mathbf{S}= \mathbf{S}_{\PP_{N_s}} = (\bss_h(\pp_{1}), \dots, \bss_h(\pp_{N_s})) \in \bR^{N_h^s \times N_s}$,
we solve the eigenvalue problem
\begin{equation}
    \mathbf{C} \bsv_k = \lambda_k \bsv_k, \quad k = 1, \dots, N_s,
\end{equation}
where $\mathbf{C} := \mathbf{S}^\top {X_h} \mathbf{S}/N_s \in \bR^{N_s \times N_s}$. The eigenpairs $(\lambda_k, \bsv_k)$, with $\lambda_1 \geq \lambda_2 \geq \cdots$, yield the POD basis vectors $\pi_k = \mathbf{S} \bsv_k/\sqrt{\lambda_k} \in \bR^{N_h^s}$, $k = 1, \dots, N_s$, which are orthonormal with respect to the ${X_h}$-inner product; see \cite[Chapter 6]{quarteroni2015reduced}.

The functions $\phi_k \in \HH_h$, $ k=1,\ldots,r$, whose FE coefficients are given by the POD basis vectors $\pi_k$, thus form an $\HH$-orthonormal basis
of an $r$-dimensional subspace $\HH_r = \text{span}\{\phi_1, \dots, \phi_r\}$, which we refer to as the POD-RB subspace of $\HH_h$. Hence, the matrix $\Pi_r = (\pi_1, \dots, \pi_r) \in \bR^{N^s_h \times r}$  with columns $\pi_k$, defines for any $\bss_h\in\R^{N_h^s}$ via $\bss_r = \Pi_r^\top {X_h} \bss_h \in \bR^r$ a projection onto an $r$-dimensional subspace of $\R^{N_h^s}$. Specifically, 
 $$
 s_r = \Phi_r \bss_r \in \HH_r, \text{ with } \Phi_r = (\phi_1, \dots, \phi_r) \text{ and } \bss_r = \Pi_r^\top {X_h} \bss_h \in \bR^r,
 $$
is the $\HH$-orthogonal projection of the FE function $s_h$ with coefficient vector $\bss_h$. Conversely, given a coefficient vector $\bss_r\in \R^r$ of a function $s_r$ in $\HH_r$, the coefficient vector of $s_h$ is given by
\begin{equation}\label{eq:bssr}
    \bss_h = \Pi_r  \bss_r \in \R^{N_h^s}.
\end{equation}
 In practice, we  take the projection dimension $r$ such that \cite{quarteroni2015reduced} 
\begin{equation}\label{eq:low_rank_approximation_via_trailing_eigenvalues}
    ||s_h - s_r||^2_{L^2(\PP; \HH_h)} = ||\bss_h - \Pi_r\bss_r||_{L^2(\PP; {X_h})}^2 \approx \sum_{k = r+1}^{N_s} \lambda_k \leq \tau \sum_{k = 1}^{N_s} \lambda_k,
\end{equation}
for a desired tolerance $\tau \ll 1$. Equivalently, $s_r$ approximates the ``truth-space'' solution
$s_h$ in $\HH_h$ with relative error~$\sqrt{\tau}$. 
When the solution manifold {has rapidly decaying Kolmogorov $r$-widths},  we have $r \ll N_h^s$.  

 We will learn reduced basis neural operator (RBNO) {\em surrogates} $s_r(\pp;\theta) = [\sc_r,\uc_r](\pp;\theta) \in \HH_r$ to the  parameter-to-solution map $\pp\mapsto s(\pp) = [\sc(\pp),\uc(\pp)]\in \HH$ in terms of elements from the {\em hypothesis class}
\begin{equation}
\label{hypo}
\mathcal{H}_r(\Theta):= \Big\{s_r(\pp;\theta)
= \Phi_r \bss_r(\pp;\theta), \; \bss_r(\pp;\theta) \in \bR^r, \pp \in \PP, \theta \in \Theta \Big\}\subset L_{\mu}^2(\PP,\HH_r),
\end{equation}
where the expansion coefficient vector $\bss_r(\pp;\theta)$ is represented by neural networks with input variables $\pp\in\PP$ and trainable parameters $\theta\in \Theta$ in a suitable parameter space $\Theta$.

On account of FOSLS stability in \eqref{Hnormeq} and \eqref{elasteq}, the loss $\mathcal{L}$
quantifies the error
$$
\|s(\pp)- s_r\|^2_{\HH}\eqsim \cL(s_r;\pp)
= \|\cB_\pp s_r - S\|^2_{\LL_2}
$$
uniformly in $\pp\in\PP$, where $S$ is the source term, e.g., $S = [\pp\nabla w - z + f_1,f_2]$ for the diffusion problem. 

In order to see how errors of this type depend on the two constituents $\HH_h, \HH_r$, we derive first corresponding decompositions of the loss. To that end, we make standard use of the fact that loss minimizers over any subspace of $\HH$ are solutions of the normal equation posed on this subspace (with source data projected accordingly) and hence enjoy Galerkin orthogonality.
\begin{lemma}
\label{lem:gal}
Assume that $\tilde\HH$ is a closed subspace of $\HH$ and that $\tilde s(\pp)= [\tilde\sigma^\circ(\pp),\tilde u^\circ(\pp)]
\in\tilde\HH$ solves
\begin{equation}
\label{tFOSLS}
(\cB_\pp \tilde s(\pp),\cB_\pp \tilde s')_\Omega= (S,\cB_\pp \tilde s')_\Omega, \quad \forall\tilde s'\in\tilde\HH.
\end{equation}
Then we have
\begin{equation}
\label{tildedec}
\cL(\tilde s';\pp) = \|\cB_\pp (\tilde s'-\tilde s(\pp))\|^2_{\LL_2} + \|\cB_\pp\tilde s(\pp)- S\|^2_{\LL_2},\quad \forall\tilde s'\in\tilde \HH.
\end{equation}
In other words,
\begin{equation}
\label{tildemin}
\tilde s(\pp)=\argmin_{\tilde s'\in \tilde\HH} \cL(\tilde s';\pp)= \argmin_{\tilde s'\in \tilde\HH}\|\cB_\pp(\tilde s'- \tilde s(\pp))\|^2_{\LL_2}.
\end{equation}
\end{lemma}
\begin{proof} By definition, we have for the exact FOSLS solution $s(\pp)=[\sc(\pp),\uc(\pp)]\in \HH$
\begin{align*}
\cL(\tilde s';\pp)&= \|\cB_\pp \tilde s' -S\|^2_{\LL_2}=\|\cB_\pp (\tilde s' -\tilde s(\pp)) +\cB_\pp(\tilde s(\pp)- s(\pp))\|_{\LL_2}^2\\
& = \|\cB_\pp (\tilde s' -\tilde s(\pp))\|_{\LL_2}^2
+ \|\cB_\pp(\tilde s(\pp)- s(\pp))\|_{\LL_2}^2 + 2(\cB_\pp (\tilde s' -\tilde s(\pp)),
\cB_\pp(\tilde s(\pp)- s(\pp)))_\Omega.
\end{align*}
Since by \eqref{tFOSLS}, $(\cB_\pp(\tilde s(\pp)-s(\pp)),\cB_\pp\tilde s')_\Omega =0$ for all $\tilde s'\in\tilde \HH$, and since $\tilde s' -\tilde s(\pp)\in \tilde\HH$, we conclude that $(\cB_\pp (\tilde s' -\tilde s(\pp)),
\cB_\pp(\tilde s(\pp)- s(\pp)))_\Omega=0$ for all $\tilde s'\in\tilde\HH$, whence the assertion follows. 
\end{proof}

We are now in a position to interrelate best approximation errors for later purposes.

\begin{theorem}
\label{prop:reduced_loss_error_equivalence}
As before, let $\HH_h=\Sigma_h\times\U_h\subset H(\rdiv;\Omega)\times H^1(\Omega)$ be the $H(\div)\times H^1$-conforming FE space and, for each fixed parameter $\pp\in\PP$, let $s_h(\pp) = [\sc_h(\pp),\uc_h(\pp)]\in\HH_h$  be the FE FOSLS solution and
\begin{equation}\label{eq:RBmin}
  s_r(\pp) = [\sc_r(\pp),\uc_r(\pp)] = \argmin_{[\tau_r,v_r]\in\HH_r} \cL([\tau_r,v_r];\pp)
\end{equation} 
be the RB FOSLS solution in the POD subspace $\HH_r \subset \HH_h$.
Then the following hold:
\begin{enumerate}
\item
For any $s_r \in \HH_r$, there holds
\begin{equation}
\label{eq:RB-fiber-equivalence}
  \cL(s_r;\pp)
  \eqsim \|\, s_r(\pp) - s_r \,\|_{\HH}^2 + \cL(s_r(\pp); \pp).
\end{equation}
In addition, the reduced FOSLS minimizer $s_r(\pp)$ in \eqref{eq:RBmin} satisfies the quasi-optimality bound \begin{equation}\label{eq:RB-best-approx}
  \|\,s_r(\pp) - s(\pp) \,\|_{\HH}
  \eqsim \min_{t_r\in\HH_r}\|\,t_r - s(\pp) \,\|_{\HH}.
\end{equation}

\item
For any $s_r \in \HH_r$, there holds
\begin{equation}\label{eq:RB-fiber-equivalence-t_r}
  \cL(s_r;\pp) \eqsim \|\, s_h(\pp) - s_r \,\|_{\HH}^2 + \cL(s_h(\pp); \pp).
\end{equation}
In addition, the reduced FOSLS minimizer $s_r(\pp)$ in \eqref{eq:RBmin}
satisfies the quasi-optimality bound
\begin{equation}\label{eq:RB-best-approx-t_r}
  \|\,s_r(\pp) - s_h(\pp) \,\|_{\HH}
  \eqsim \min_{t_r\in\HH_r}\|\,t_r - s_h(\pp) \,\|_{\HH}.
\end{equation}
\item
The function $s_r \in \X_r:= L_{\mu}^2(\PP;\HH_r)$ that for each $\pp\in\PP$ solves \eqref{tFOSLS} for
$\tilde \HH = \HH_r$,    
is the unique minimizer
\begin{equation}
\label{rmin}
s_r = \argmin_{t\in \X_r}\EE_{\pp\sim\mu}\big[\cL(t(\pp);\pp)\big]=: \argmin_{t\in \X_r}\cL(t;\PP)
\end{equation}
and  
\begin{equation}
\begin{aligned}
\label{quasiopt}
&\EE_{\pp\sim\mu}\big[ \cL(s_r(\pp);\pp) \big] \eqsim \EE_{\pp\sim\mu}\big[\|s_r(\pp)- s(\pp)\|^2_{\HH}\big] \eqsim
\min_{t_r\in \X_r}\EE_{\pp\sim\mu}\big[\|t_r(\pp)-s(\pp) \|^2_{\HH}\big],
\end{aligned}
\end{equation}
with a constant that depends only on $\PP$.
\end{enumerate}
\end{theorem}
\begin{proof}
\begin{enumerate}
    \item Taking $\tilde\HH := \HH_r$, \eqref{eq:RB-fiber-equivalence} follows from
        \eqref{tildedec} in Lemma \ref{lem:gal} with $\tilde s(\pp)= s_r(\pp)$ and
        $\tilde s'= s_r$, combined with FOSLS stability \eqref{Hnormeq} and \eqref{elasteq}
        for the diffusion and elasticity model, respectively. We obtain \eqref{eq:RB-best-approx} by the optimality of $s_r(\pp)$ and FOSLS stability, i.e., for any $t_r \in \HH_r$,
        \begin{equation}
            \|\,s_r(\pp) - s(\pp) \,\|_{\HH} \eqsim  \|\cB_\pp s_r(\pp) - S\|^2_{\LL_2} = \cL(s_r(\pp);\pp) \leq \cL(t_r;\pp) = \|\cB_\pp t_r - S\|^2_{\LL_2} \eqsim \|\,t_r - s(\pp) \,\|_{\HH}.
        \end{equation} 

    \item The same reasoning for $\tilde\HH= \HH_h$, $\tilde s(\pp) =[\sc_h(\pp),\uc_h(\pp)]$, confirms \eqref{eq:RB-fiber-equivalence-t_r}, noting that $s_r\in\HH_r$ are also elements in $\HH_h$. Applying minimization of $s_r\in \HH_r$ over \eqref{eq:RB-fiber-equivalence-t_r}, we have 
    \begin{equation}
        \min_{s_r\in \HH_r} \cL(s_r; \pp)\simeq \min_{s_r\in \HH_r}\|s_h(\pp)-s_r \|_{\HH}^2 + \cL(s_h(\pp);\pp), 
    \end{equation}
    or equivalently, 
    \begin{equation}\label{eq:proof_2_RB-best-approx-t_r}
        \cL(s_r(\pp)\; \pp) \simeq \min_{t_r\in \HH_r} \|t_r-s_h(\pp)\|_{\HH}^2 + \cL(s_h(\pp); \pp).
    \end{equation}
    Then \eqref{eq:RB-best-approx-t_r} holds when we combing~\eqref{eq:proof_2_RB-best-approx-t_r} with $\cL(s_r(\pp);\pp) \eqsim \|\, s_h(\pp) - s_r(\pp) \,\|_{\HH}^2 + \cL(s_h(\pp); \pp)$ which is obtained by plugging $s_r=s_r(\pp)$ in~\eqref{eq:RB-fiber-equivalence-t_r}. 

    \item This is an immediate consequence of 1. and uniform FOSLS stability \eqref{Hnormeq} and \eqref{elasteq}.
\end{enumerate}

\end{proof}

\begin{remark}
\label{rem:deco}
Theorem \ref{prop:reduced_loss_error_equivalence} says that $s_r(\pp)$ minimizes for each $\pp$ the loss over $\HH_r$ and (near-)best approximates $s(\pp)$.
Statement 2 says that $s_r(\pp)$ (near-)best approximates the FE-FOSLS solution from 
$\HH_r$.  Overall this states that the error incurred by $s_r(\pp)$
with respect to the exact solution is uniformly proportional to the sum of the 
best approximation errors of $s(\pp)$ from the FE space $\HH_h$ and
the error of the approximation of the FE-FOSLS solution by the RB-FOSLS solution.
\end{remark}

\begin{corollary}
\label{cor:err}
Adhering to the above assumptions, 
let
\begin{equation}
\label{theta*}
s_r(\theta^*)\in \argmin_{t_r\in \cH_r(\Theta)}\EE_{\pp\sim\mu}\big[
\cL(t_r(\pp);\pp)\big].
\end{equation}
Then, with a proportionality constant, depending only on uniform FOSLS-stability
\eqref{Hnormeq} and \eqref{elasteq},  there holds
\begin{align}
\label{error-dec}
\EE_{\pp\sim \mu}\big[\|s(\pp)- s_r(\pp; \theta^*)\|^2_{\HH}\big]\lesssim h^{2\eta} + \sum_{j>r}\lambda_j^2 +
\EE_{\pp\sim\mu}\Big[  
\|s_r(\pp)- s_r(\pp;\theta^*)\|^2_{\HH}\Big],
\end{align}
where $\eta$ depends on the polynomial orders of the FE approximation and the regularity of $s(\pp)\in \HH$, $\pp\in \PP$, as in Theorem \ref{thm:FELSPoisson}. Note that the last summand reflects the best approximation from the hypothesis class to the best approximation of the exact FOSLS solutions $s(\pp)$ from $\HH_r$. 
\end{corollary}
\begin{proof} We infer from Theorem \ref{prop:reduced_loss_error_equivalence}, 
\eqref{eq:RB-fiber-equivalence}, that, with $s_r(\theta^*) \in \cH_r(\Theta)$,  and FOSLS stability  \eqref{Hnormeq} and \eqref{elasteq},
\begin{align}
\label{steps}
& \EE_{\pp\sim \mu}\big[\|s(\pp)- s_r(\pp;\theta^*)\|^2_{\HH}\big]\nonumber \\
& \eqsim\EE_{\pp\sim \mu}\Big[ 
 \|s_r(\pp;\theta^*)-s_r(\pp)\|^2_{\HH}\Big]
+ \EE_{\pp\sim\mu}\Big[\cL(s_r(\pp);\pp) \Big]\nonumber\\
& \eqsim 
\EE_{\pp\sim \mu}\Big[ 
  \|s_r(\pp;\theta^*)-s_r(\pp)\|^2_{\HH}\Big] + \EE_{\pp\sim\mu}\Big[\|s_h(\pp)- s_r(\pp)\|^2_{
 \HH}\Big] + \EE_{\pp\sim\mu}\Big[\cL(s_h(\pp);\pp)\Big],
\end{align}
where we have invoked Theorem \ref{prop:reduced_loss_error_equivalence}, Statement 2 in the last step.
Bounding the first term in the last inequality by \eqref{eq:low_rank_approximation_via_trailing_eigenvalues}, using again that, uniformly in 
$\pp\in\PP$, on account of FOSLS-stability and the best approximation property of the FE-FOSLS solutions, we conclude that
$$
\cL(s_h(\pp);\pp)\eqsim  \inf_{t_h(\pp)\in \HH_h}\|
t_h(\pp)- s(\pp)\|^2_{\HH}.
$$
Finally, invoking the error bounds in   \ref{app:proofFELSPoisson} and \ref{app:proofFELSElasticity}, finishes the proof.
\end{proof}

\subsection{Loss evaluation}\label{ssec:losseval}
We learn an approximation to the optimal approximation $s_r(\pp) \in \HH_r$ from the hypothesis class $\mathcal{H}_r(\Theta)\subset \X_r$ (see \eqref{hypo}) by minimizing an empirical loss
over $\Theta$. Since this will require frequent evaluations of fiber losses $\cL(\cdot;\pp)$ at randomly drawn samples $\pp\in\PP$, the cost of such evaluations is an issue. In fact, substituting the FE approximation \eqref{eq:FEApproxPoisson} in the residual fiber loss function \eqref{eq:lossPoisson}, we obtain 
\begin{equation}\label{eq:FE-loss-Poisson}
    \cL(s_h(\pp); \pp) = \bsw(\pp)^\top \, W_\pp \, \bsw(\pp) = \bss_h(\pp)^\top W_\pp^\circ \bss_h(\pp) + 2 \bss_h(\pp)^\top \bsalpha_\pp + \beta_\pp,
\end{equation}
where $\bsw(\pp) = (\bss_h(\pp); 1) \in \bR^{N_h^s+1}$ is a concatenation of the parametric coefficient vectors $\bss_h(\pp)\in \bR^{N_h^s}$ and scalar $1$. Here $W_\pp \in \bR^{(N_h^s+1) \times (N_h^s+1)}$ is a sparse symmetric matrix assembled from the FE approximation at each parameter $\pp \in \PP$, which has the block structure  
\begin{align*}
    W_\pp=\begin{pmatrix}
        W^\circ_\pp & \bsalpha_\pp \\[4pt]
        \bsalpha^\top_\pp & \beta_\pp
    \end{pmatrix}, 
\end{align*}
with $W^\circ_\pp \in \mathbb{R}^{N_h^s\times N_h^s}$, $\bsalpha_\pp \in \mathbb{R}^{N_h^s}$, and $\beta_\pp \in\mathbb{R}$ assembled from the bilinear, linear, and constant forms in the loss function \eqref{eq:lossPoisson}. Note that evaluation of the loss function \eqref{eq:FE-loss-Poisson} has complexity $O(N_h^s)$ {when employing sparse matrix vector products, which is expensive} for large-scale FE discretizations. 

To reduce the {computational complexity}, we introduce a RB computation of the loss function \eqref{eq:FE-loss-Poisson}, {when evaluated at elements in the RB space $\HH_r$}. Let $s_r$ denote the RB approximation of the FE solution $s_h$,
with corresponding approximation of the coefficient vectors by the projection in \eqref{eq:bssr}. 
 Then the RB loss function can be evaluated as 
\begin{equation}\label{eq:lossPoissonReduced}
    \cL(s_r; \pp) = \bss_r(\pp)^\top W_\pp^r \bss_r(\pp) + 2 \bss_r(\pp)^\top \bsalpha_\pp^r + \beta_\pp,
\end{equation}
where reduced weights $W_{\pp}^{r}:=\Pi^\top_r W^\circ_\pp \Pi_r \in\mathbb{R}^{r\times r}$ and vectors $\bsalpha^{r}_\pp:=\Pi^\top_r \bsalpha_\pp\in\mathbb{R}^{r}$ can be precomputed and stored for the training samples of the parameter, allowing a fast evaluation of the loss function with complexity $O(r^2)$. The RB FOSLS minimizer $s_r(\pp)\in \HH_r$ from \eqref{eq:RBmin} can then be computed with the RB coefficient vector $\bss_r(\pp) \in \bR^r$ obtained by solving the reduced normal equation
\begin{equation}\label{eq:RB_normal}
    W_\pp^r\, \bss_r(\pp) = -\bsalpha_\pp^r.
\end{equation}

For the elasticity problem, we follow the same computation of the RB approximation as in the diffusion case. Specifically, we adopt the following discrete $H(\rdiv;\Omega;\R^{d\times d})\times H^1(\Omega;\R^d)$-norm as in \cite{opschoor2024first} 
\begin{equation}\label{eq:X-norm-elasticity}
    ||\bss_h||_{{X_h}}^2 := (\tenbar[1]{\rdiv} \tenbar[2]{\sc_h}, \tenbar[1]{\rdiv} \tenbar[2]{\sc_h})_{L^2} + (\tenbar[2]{\sc_h}, \cC_{\bar{\pp}}^{-1} \tenbar[2]{\sc_h})_{L_2} + 
(\tenbar[2]{\e}(\tenbar[1]{\uc_h}), \cC_{\bar{\pp}}\tenbar[2]{\e}(\tenbar[1]{ \uc_h}))_{L^2}, 
\end{equation}
where $\cC_{\bar{\pp}}$ and $\cC_{\bar{\pp}}^{-1}$ are the stiffness tensor and its inverse at the mean of the parameter $\bar{\pp} \in \PP$. The same statements as in Proposition \ref{prop:reduced_loss_error_equivalence} hold for the RB approximation of the elasticity problem with the loss function \eqref{eq:lossElasticity} and the discrete norm \eqref{eq:X-norm-elasticity}.

We conclude this section with one further prerequisite that will be needed in the next section.
\begin{proposition}
  \label{prop:uniform-bounds}
Assume $\pp\in \PP \subset L_\infty(\Omega)$ is uniformly bounded as in Theorem \ref{thm:FELSPoisson}. Then for all $\pp\in\PP$, the reduced weights $W_\pp^r$, vectors $\bsalpha_\pp^r$, and scalars $\beta_\pp$, defined in \eqref{eq:lossPoissonReduced}, satisfy the uniform bounds
\begin{equation*}
 \gamma_- I \preceq W_\pp^r \preceq \gamma_+ I,\qquad \|\bsalpha_\pp^r\|_2 \le \sqrt{\gamma_+}\, S_{\max},\qquad |\beta_\pp| \le S_{\max}^2.
\end{equation*}
Here $c \le \gamma_- \le \gamma_+ \le C $ with the stability constants $0<c<C<\infty$ in Lemma \ref{lem:PoissonNormEquiv}, and $S_{\max}$ is the uniform bound for the input source $\|S(\pp)\|_{\LL^2} \le S_{\max}$, $\pp \in \PP$. For instance, for diffusion one has  $S=[\pp\nabla w - z + f_1, f_2]$.
\end{proposition}
\begin{proof}
Since $W_\pp^r=\Pi_r^\top W_\pp^\circ \,\Pi_r$, for any $\bss_r\in\R^r$ we have
$\bss_r^\top W_\pp^r \bss_r \,=\, (\Pi_r \bss_r)^\top W_\pp^\circ (\Pi_r \bss_r).$
Uniform FOSLS stability (see Lemma~\ref{lem:PoissonNormEquiv}) implies there exist $0<\gamma_-\le\gamma_+ <\infty$ such that the stability estimates \eqref{Hnormeq} hold in FE spaces $\Sigma_h \times \U_h \subset \Sigma \times \U$ with constants $c \le \gamma_-$ and $\gamma_+ \le C$, i.e., for all FE coefficient vectors $\bss_h\in\R^{N_h^s}$,
\[
 \gamma_-\, \bss_h^\top X_h \bss_h \;\le\; \bss_h^\top W_\pp^\circ \bss_h \;\le\; \gamma_+\, \bss_h^\top X_h \bss_h,\qquad \pp\in\PP.
\]
Therefore, by setting $\bss_h = \Pi_r \bss_r$, there holds
\[
 \gamma_-\, \bss_r^\top (\Pi_r^\top X_h \Pi_r) \bss_r \;\le\; \bss_r^\top W_\pp^r \bss_r \;\le\; \gamma_+\, \bss_r^\top (\Pi_r^\top X_h \Pi_r) \bss_r.
\]
By the ${X_h}$-orthonormality of $\Pi_r$, $\Pi_r^\top X_h \Pi_r = I_r$, hence
$\gamma_-\,\|\bss_r\|_2^2 \le \bss_r^\top W_\pp^r \bss_r \le \gamma_+\,\|\bss_r\|_2^2$, proving the spectral bounds.
Next, from the quadratic expansion $\cL(\bss_h;\pp)= \bss_h^\top W_\pp^\circ \bss_h + 2 \,\bss_h^\top \bsalpha_\pp + \beta_\pp$, the linear term obeys, for any FE coefficient vector $\bss_h$ with corresponding FE functions $s_h$,
\[
 |\bss_h^\top \bsalpha_\pp| \;=\; (\cB_\pp s_h, S(\pp))\;\le\;  \|\cB_\pp s_h\|_{\LL^2} \, \|S(\pp)\|_{\LL^2}\;\le\; \sqrt{\gamma_+}\,\|S(\pp)\|_{\LL^2}\, (\bss_h^\top X_h \bss_h)^{1/2},
\]
where we have used Cauchy--Schwarz in the first inequality and the uniform stability bound in the second. Taking $\bss_h=\Pi_r \bss_r$ with $\|\bss_r\|_2=1$ and using $\bss_r^\top (\Pi_r^\top X_h \Pi_r) \bss_r=1$, yields
$\|\bsalpha_\pp^r\|_2 = \sup_{\|\bss_r\|_2=1} (\Pi_r \bss_r)^\top \bsalpha_\pp \le \sqrt{\gamma_+}\,\|S(\pp)\|_{\LL^2} \le \sqrt{\gamma_+}\,S_{\max}.$
Finally, $\beta_\pp=\|S(\pp)\|_{\LL^2}^2 \le S_{\max}^2$. Uniformity follows from the hypotheses.
\end{proof}

\subsection{{Learning the} reduced basis neural operator}
Recall that the elements of the hypothesis class $\mathcal{H}_r(\Theta)$, defined in \eqref{hypo}, belong for each $\pp\in\PP$, as functions of $x\in\Omega$, to the $r$--dimensional RB space constructed in~\cref{ssec:errors}. As indicated in~\cref{sec:conceptual_background_and_orientation}, learning a parametric map will be based on minimizing an empirical risk of the type 
\begin{equation}\label{eq:empirical-loss}
  \min_{\theta\in\Theta} \hat R_{\widehat\PP}(\theta),\quad \mbox{where}\quad \hat R_{\widehat\PP}(\theta):= \hat R_{\widehat\PP}(s_r(\cdot;\theta)):=\frac{1}{\#{\widehat\PP}}\sum_{\pp\in{\widehat\PP}}\cL\big(s_r(\pp;\theta);\pp\big).
\end{equation}
In what follows 
$$
{\widehat\PP} = \{\pp_1,\ldots,\pp_N\}\subset \PP,
$$
stands for a collection of finite i.i.d.\ random samples from $\PP$. In computations, the evaluation of the RB loss $\cL(s_r(\pp;\theta);\pp)$ is always based on the right hand side expression \eqref{eq:lossPoissonReduced} for $\bss_r(\pp) = \bss_r(\pp;\theta)$. The coefficients in $\bss_r(\pp;\theta)$ need to approximate for every $\pp\in\PP$ the coefficient vector $\bss_r(\pp)$ of function $s_r(\pp)$. Recall that those are obtained when minimizing the ``ideal continuous'' loss which we also abbreviate, for convenience, as follows: for any $\theta\in\Theta$ we write $s_r(\cdot;\theta)=[\sc_r(\cdot;\theta),\uc_r(\cdot;\theta)]\in \mathcal{H}_r(\Theta)$ and set
$$
R(s_r(\cdot;\theta))=:R(\theta):= \EE_{\pp\sim\mu}\Big[\cL(s_r(\pp;\theta);\pp)\Big].
$$
We refer to functions $s_r(\cdot;\theta)=[\sc_r(\cdot; \theta),\uc_r(\cdot;\theta)]\in \mathcal{H}_r(\Theta)$ as RBNO approximations to the solution of the RB-FOSLS problem.

In order to estimate the accuracy of empirical loss minimizers with respect to the Bochner norm $\|w\|^2_{\X}= \int_\PP \|w(\pp)\|^2_\HH d\mu(\pp)= \EE_{\pp\sim\mu}\big[\|w(\pp)\|^2_\HH\big]$, we combine the previous error bounds with   (rather standard) concepts from Learning Theory. Here we are content with a relatively simple version that may not be the best possible. Nevertheless, their applicability imposes some conditions on the hypothesis class $\mathcal{H}_r(\Theta)$ that need to be ensured. \\[2mm]
{\em Boundedness:} Since $s_r(\pp)$ minimizes the (ideal) loss over $\HH_r$ and
hence solves the normal equations in  $\HH_r$, uniform ellipticity of the normal equations
implies boundedness of the solution by the data, which in our case means, on account of 
the orthogonality of the reduced POD basis $\Phi_r$, that
$$
\|\bss_r(\pp)\|_{\ell_2}= \| [\sc_r(\pp),\uc_r(\pp)]\|_{\HH}\le \sqrt{\gamma_+}S_{\max},\quad \pp\in \PP,
$$
provided that $\PP$ remains bounded in a suitable domain. Therefore, restraining the neural network approximations to the $\pp$-dependent coefficients to remain bounded as well will not impede their expressivity.

\begin{assumption}\label{ass:range-capacity}
There exists a uniform bound $B>0$ such that the neural network output $\|\bss_r(\pp;\theta)\|_{\ell_2}\le B$ for all $(\pp,\theta)\in\PP\times\Theta$ (enforced, e.g., by range clipping or weight penalties). 
\end{assumption}

The second important property is the uniform bound and Lipschitz continuity of the loss $\cL(t_r;\pp)$ with respect to the first argument $t_r \in \mathcal{H}_r(\Theta)$.
\begin{lemma}
\label{lem:Lip}
Assume that Assumption \ref{ass:range-capacity} is valid. Then, there exist constants   $M,L<\infty$ such that
\begin{equation}
\label{bounded}
\cL(t_r;\pp)\le M,\quad (t_r,\pp)\in \mathcal{H}_r(\Theta)\times \PP,
\end{equation}
and
\begin{equation}
\label{Lip}
\big| \cL(t_r;\pp)- \cL(t_r';\pp)\big| \le L \|t_r- t_r'\|_{\HH},\quad t_r, t_r' \in \mathcal{H}_r(\Theta), \,\pp\in \PP.
\end{equation}
\end{lemma}
\begin{proof}
By definition,
\begin{align*}
\cL(t_r;\pp)^{\frac 12}&= \|\cB_\pp t_r - S\|_{\LL_2}\le \|\cB_\pp t_r\|_{\LL_2}+ S_{\max} 
 \le \sqrt{\gamma_+}\|t_r\|_{\HH}+S_{\max}.
\end{align*}
Since $t_r\in \mathcal{H}_r(\Theta)$ is of the form $t_r(x;\pp)= \Phi_r \bst_r(\pp;\theta)$, orthogonality of the reduced bases $\Phi_r$, gives $\|t_r\|_{\HH}= \|\bst_r(\pp;\theta)\|_{\ell_2}\le B$, $(\pp,\theta)\in \PP\times \Theta$, by Assumption \ref{ass:range-capacity},
confirming \eqref{bounded} with
\begin{equation}
\label{M}
M\le \Big(\sqrt{\gamma_+}B + S_{\max}\Big)^2.
\end{equation}

Similarly,
\begin{align*}
\big|\cL(t_r;\pp)- \cL(t_r';\pp)\big| &= \big|\|\cB_\pp t_r-S\|^2_{\LL_2}- \|\cB_\pp t_r'-S\|^2_{\LL_2}\big|\\
&=  \big|\|\cB_\pp t_r-S\|_{\LL_2}+ \|\cB_\pp t_r'-S\|_{\LL_2}\big|\,
\big|\|\cB_\pp t_r-S\|_{\LL_2}- \|\cB_\pp t_r'-S\|_{\LL_2}\big|\\
& \le 2\Big\{\sqrt{\gamma_+}B + S_{\max}\Big\}\|\cB_\pp t_r- \cB_\pp t_r'\|_{\LL_2}\\
&\le 4 \Big\{\sqrt{\gamma_+}B + S_{\max}\Big\}\sqrt{\gamma_+} \|t_r-t_r'\|_{\HH},
\end{align*}
which finishes the proof with $L=4 \Big\{\sqrt{\gamma_+}B + S_{\max}\Big\}\sqrt{\gamma_+}$.
\end{proof}

As a final prerequisite, we recall the notion of {\em pseudo-dimension} of a function class
$\cF$. The pseudo-dimension is the VC-dimension of the set of their epi-graphs. For precise definitions, see, for instance, \cite[Chapter 3]{mohri2018foundations}.

\begin{remark}
\label{rem:VC}
When $\cH$ is a linear space, one has ${\rm Pdim}\,\cH= {\rm dim}\,\cH +1$.
When $\cH$ is a fully connected neural network of   depth $D$, it has
been shown in \cite{bartlett2019nearly} that ${\rm Pdim}\,\cH \lesssim \#\Theta \,D \log(\#\Theta)$.
\end{remark}
\bigskip

To highlight the roles of the involved constituents in our particular setting, we present the following uniform convergence result in high probability.

\begin{theorem}\label{thm:final-H-error}
We assume the validity of Assumption \ref{ass:range-capacity}.  For any $\pp \in \PP$, let $s(\pp) = [\sc(\pp), \uc(\pp)]\in\HH$ be the solution of the FOSLS problem \eqref{fosls} and let $s_r(\pp, \hat \theta) = [\sc_r,\uc_r](\pp;\hat \theta) \in \HH_r$ denote the RBNO approximation, obtained by minimizing the empirical risk \eqref{eq:empirical-loss} over $\Theta$. More precisely, $s_r(\pp, \hat \theta)$ satisfies
\begin{equation}
  \label{empmin}
\hat R_{\widehat\PP}(\hat\theta) \le \inf_{\theta\in\Theta}\hat R_{\widehat\PP}(\theta)+ \varepsilon_{\rm opt},
\end{equation}
where $\varepsilon_{\rm opt}\ge 0$ accounts for a possible optimization error.
Then,  for any $\delta\in(0,1)$, with probability at least $1-\delta$,  there holds
\begin{align}
\label{mainest}
 \EE_{\pp\sim\mu}\!\left[\,\|s(\pp)-s_r(\pp;\hat\theta)\|^2_{\HH}\right]
 \;\le c^{-1}\inf_{\theta\in \Theta}\EE_{\pp\sim\mu}\big[\cL(s_r(\pp;\theta);\pp)\big] + M\sqrt{\frac{\log \frac{1}{\delta}}{2N}} + C L   \sqrt{\frac{2P\log\frac{eN}{P}}{N}}
 + \varepsilon_{\mathrm{opt}},
 \end{align}
 where $c$ is the constant from \eqref{Hnormeq} or \eqref{elasteq}, $C$ is an absolute constant,   $P$ is the   pseudo-dimension of the hypothesis class $\cH_r(\Theta)$,
 and $M,L$ are from Lemma \ref{lem:Lip}.
 Moreover, there exist constants $C_1$ and $C_2$, depending on FE discretization and uniform FOSLS stability, respectively, such that the approximation error can be decomposed
as follows
\begin{equation}
\label{apprerr}
 \inf_{\theta\in \Theta}\EE_{\pp\sim\mu}\big[\cL(s_r(\pp;\theta);\pp)\big]\le C_1 \Big\{
  h^{2\eta} 
 +
 \sum_{k>r}\lambda_k\Big\}
\,\,\, 
 + 
 C_2\inf_{\theta\in\Theta}\EE_{\pp\sim\mu}\Big[\|s_r(\pp)- s_r(\pp;\theta)\|^2_{\HH}\Big].
 \end{equation}
Here the exponent $\eta$ depends on the regularity of the exact FOSLS solution $s(\pp)$, see \ref{sec:FE-Error-Estimate}.
\end{theorem}
Before turning to the proof, a few comments are in order. \eqref{apprerr} says that the first term on the right hand side of \eqref{mainest} could be reduced to an approximation error within the subspace $\HH_r$ by choosing a sufficiently fine FE discretization and a 
correspondingly accurate reduced basis. The second group of terms in \eqref{mainest}, representing the estimation error, is only meaningful for $N> P$, hence in an {\em under-parametrized} regime, which is in accordance with the fact that resulting surrogates are supposed to act as reduced models.
For large $N$, the estimation error behaves like $O(N^{-1/2})$, which is a slow rate.
Faster rates would also require variance information, which is not required here.

Bounds of the above type are, in essence, standard. For the convenience of the reader
we devote the remainder of this section to highlighting the
main ingredients of the proof.

\begin{proof} 
Note first that in the present case, the minimal risk 
$$R^*= \min_{s\in L_{\mu}^2(\PP;\HH)}
\EE_{\pp\sim \mu}\big[\cL(s;\pp)\big]=0.$$ 
Thus, by uniform FOSLS stability
$$
\EE_\mu\! \left[\|s(\pp)-s_r(\pp;\hat\theta)\|^2_{\HH}\right] \eqsim \EE_{\pp\sim\PP}\big[\cL(s_r(\pp;\hat\theta);\pp)\big]= R(\hat\theta)=R(\hat\theta)-R^*,
$$
so that, in the above terms, our task is to bound $R(\hat\theta)=R(s_r(\cdot;\hat\theta))$.
We follow the first standard lines and define
$$
s_r^*\in \argmin_{s_r\in \cH_r(\Theta)} R(s_r),
$$
and decompose the excess risk
\begin{equation}
\label{deco1}
R(\hat\theta) = R(\hat\theta)- R(s_r^*) + R(s_r^*).
\end{equation}
By FOSLS-stability, 
$$
R(s_r^*)\eqsim \min_{s_r\in \cH_r(\Theta)}\EE_{\pp\sim\mu}
\big[\|s(\pp)- s_r(\pp)\|^2_{\HH}\big]
$$
represents the {\em best-approximation error} from the hypothesis  class $\cH_r(\Theta)$, which is a deterministic quantity. Instead, the first summand, $R(\hat\theta)- R(s_r^*)$, is usually referred to as {\em estimation error}, a stochastic quantity, which we further have to
analyze. We employ  a further decomposition to obtain
\begin{align*}
R(\hat\theta)- R(s_r^*)&= \{R(\hat\theta)- \hat R_{\widehat\PP}(\hat\theta)\} + \{\hat R_{\widehat\PP}(\hat\theta)- \hat R_{\widehat\PP}(s_r^*)\} + \{\hat R_{\widehat\PP}(s_r^*)-  R(s_r^*)\}
\nonumber\\
&\le 2 \sup_{s_r\in \cH_r(\Theta)}\big\{R(s_r)- \hat R_{\widehat\PP}(s_r)\big\} + \varepsilon_{\rm opt},  
\end{align*}
where we have used \eqref{empmin}. This gives
\begin{equation}
\label{deco2}
R(\hat\theta)\le  2 \sup_{s_r\in \cH_r(\Theta)}\big\{R(s_r)- \hat R_{\widehat\PP}(s_r)\big\} + R(s_r^*) + \varepsilon_{\rm opt},
\end{equation}
To further estimate the only stochastic quantity on the right hand side of \eqref{deco2},
note that the estimation error
\begin{equation}
\label{esterror}
E(\pp_1,\ldots,\pp_N):= \sup_{s_r\in \cH_r(\Theta)}\big\{R(s_r)- \hat R_{\widehat\PP}(s_r)\big\},
\end{equation}
satisfies, on account of Lemma \ref{lem:Lip}, \eqref{bounded},
$$
\big| E(\pp_1,\ldots,\pp_{i-1},\pp_i,\pp_{i+1},\ldots,\pp_N)-E(\pp_1,\ldots,\pp_{i-1},\pp'_i,\pp_{i+1},\ldots,\pp_N)\big| \le \frac{M}N, \quad \pp_i,\,\pp'_i\in\PP.
$$
Applying Mc'Diarmid's inequality (see \cite[Propositio 1.3]{FBach}), yields
\begin{align}
\label{McDia}
\prob\Big\{ \big| E(\pp_1,\ldots,\pp_N)- \EE_{{\widehat\PP}\sim \mu^N}\big[E(\pp_1,\ldots,\pp_N)\big]\big|&\ge t\Big\}\le 2 \exp\Big\{-2t^2/ (N( M/N)^2)\Big\}\nonumber\\
&= 2\exp\Big\{-\frac{2Nt^2}{M^2}\Big\}.
\end{align}
Hence, for any $\delta\in (0,1)$, with probability at least $1-\delta$, one has
\begin{equation}
\label{inprob}
\sup_{s_r\in \cH_r(\Theta)}\big\{R(s_r)- \hat R_{\widehat\PP}(s_r)\big\} \le \EE_{{\widehat\PP}\sim \mu^N}\Big[\sup_{s_r\in \cH_r(\Theta)}\big\{R(s_r)- \hat R_{\widehat\PP}(s_r)\big\}\Big]   + M \sqrt{\frac{\log \frac{1}{\delta}}{2N}}.
\end{equation}
It remains to bound expectations of the above type on the right of \eqref{inprob}.
To that end, one can resort to
the so-called (empirical) {\em Rademacher-Complexity}, defined as follows (see e.g. \cite[Definition 3.1]{mohri2018foundations}). For ${\widehat\PP}$ as above
let $\cF$ denote a class of functions mapping $\PP$ into $[0,M]$.  The empirical Rademacher Complexity of $\cF$ is defined by
$$
\hat\mR_{\widehat\PP}(\cF):= \EE_\epsilon\Big[\sup_{f\in \cF}\frac 1N\sum_{i=1}^N \epsilon_i f(\pp_i)\Big],
$$
where $\epsilon\in \{-1,1\}^N$ are independent random variables taking values $\pm 1$ with
equal probability (Rademacher variables). 
The expectation of $\hat\mR_{\widehat\PP}(\cF)$
$$
\mR_N(\cF):= \EE_{{\widehat\PP}\sim \mu^N}\Big[\hat\mR_{\widehat\PP}(\cF)\Big]
$$
is referred to as {\em Rademacher Complexity}.
Specifically
consider functions of the form
$$
f(\pp;\theta):= \cL(s_r(\pp;\theta);\pp), 
$$
and let the class $\cF$ be comprised of all such functions obtained when 
$s_r$ ranges over $\cH_r(\Theta)$, and $\pp\in \PP,\,\theta\in \Theta$.
In view of Assumption \ref{ass:range-capacity} and Lemma \ref{lem:Lip},
each $f\in\cF$ maps 
$ \PP$ into $[0,M]$. Notice next that, in these terms \eqref{esterror}
can be rewritten as
$$
\sup_{s_r\in \cH_r}\big\{R(s_r)- \hat R_{\widehat\PP}(s_r)\big\} = \sup_{f\in\cF}\Big\{\EE_{\pp\sim\mu}[f(\pp)]- \frac 1N \sum_{\pp\in {\widehat\PP}} f(\pp)\Big\}.
$$ 
It is well-known (see for instance, \cite{BartlettMendel} or \cite[Proposition 4.3]{FBach}) that the expectation of this quantity can be bounded in terms of the Rademacher Complexity as
\begin{equation}
\label{Radbound}
\EE_{{\widehat\PP}\sim\mu^N}\Big[ \sup_{f\in\cF}\Big|\EE_{\pp\sim\mu}[f(\pp)]- \frac 1N \sum_{\pp\in {\widehat\PP}} f(\pp)\Big|\Big]\le 4\mR_N(\cF).
\end{equation}
Moreover, since Lemma \ref{lem:Lip}, \eqref{Lip}, applies to the class $\cF$, so that 
the ``contraction principle'' yields (see \cite[Proposition 4.3.]{FBach})
\begin{equation}
\label{contract}
\mR_N(\cF)\le L \mR_N(\cH_r(\Theta)),
\end{equation}
where $L$ is the Lipschitz constant. 
Combining this with \eqref{inprob} and \eqref{Radbound}, provides that with probability
at least $1-\delta$ we have
\begin{equation}
\label{RadboundH}
\sup_{s_r\in \cH_r}\big\{R(s_r)- \hat R_{\widehat\PP}(s_r)\big\} \le 4L \mR_N(\cH_r(\Theta)) 
+ M \sqrt{\frac{\log \frac{1}{\delta}}{2N}},
\end{equation}
which now shows a dependence on the complexity of the hypothesis class $\cH_r(\Theta)$.
This can be alternatively described in terms of the  {pseudo-dimension} of 
$\cH_r(\Theta)$. 
Relating Rademacher Complexity to covering numbers and using Dudley's integral, it can eventually be shown (see \cite{BartlettMendel,AnthonyBartlett}) that
\begin{equation}
\label{VCbound}
\mR_N(\cH_r(\Theta))\le C \sqrt{\frac{2P\log\frac{eN}{P}}{N}},
\end{equation}
where $P={\rm Pdim\,\cH_r(\Theta)}$ is the pseudo-dimension of $\cH_r(\Theta)$ and $C$ is an absolute constant.
Substituting this bound into \eqref{RadboundH}, combining the result with \eqref{deco2},
we conclude that with probability at least $1-\delta$
\begin{equation}
\label{collect}
R(\hat\theta) \le R(s_r^*)+  4L C \sqrt{\frac{2P\log\frac{eN}{P}}{N}} +   M \sqrt{\frac{\log \frac{1}{\delta}}{2N}}.
\end{equation} 
The assertion follows now from bounding the approximation error $R(s_r^*)$ with the aid of  Corollary \ref{cor:err}.\end{proof}

\section{Numerical experiments}
\label{sec:numerical_experiments}

All computations were run on a Linux x86\_64 workstation equipped with dual AMD EPYC 9334 CPUs 
and an NVIDIA L40S GPU. For all finite element solves and weight-matrix assembly we used Python~3.13 with the libraries \href{https://github.com/FEniCS/dolfinx/}{DOLFINx~0.9.0}~\cite{Baratta_DOLFINx_the_next_2023} and \href{https://github.com/scientificcomputing/scifem}{scifem~0.7.0}, together with \href{https://hippylib.github.io/}{hIPPYlib~3.1.0}~\cite{villa2021hippylib} (built on legacy \href{https://fenicsproject.org/}{FEniCS~2019.1.0}~\cite{logg2012automated}) to generate samples of the Gaussian random parameter fields. Neural networks were implemented and trained in \href{https://docs.pytorch.org/docs/stable/user_guide/index.html}{PyTorch~2.6.0+cu124}~\cite{paszke2019pytorch}. 

\subsection{Problem setup}
We consider two stationary diffusion problems---a heat conduction model with a piecewise-constant conductivity field and a Darcy flow model with a lognormal random permeability---and one linear elasticity problem describing a clamped beam under traction. Together, these examples cover both smooth and rough random parameter fields, as well as scalar-valued (diffusion) and vector-valued (elasticity) PDEs, and include mixed Dirichlet–Neumann boundary conditions typical of the three applications.

\subsubsection{Heat conduction}  
We consider a steady-state heat conduction problem with a piecewise-constant conductivity field and a nonzero external heat source, motivated by electronics thermal management. The conductivity field $\pp(x)$ is assumed to be constant over a $4\times 4$ arrangement of mini-squares uniformly spread across the domain $\Omega=(0,1)\times (0,1)$, that is,
\begin{align*}
    \pp(x) = \mathbf{1}_{\Omega\setminus (\cup \bar{\Omega}_i)}(x) + \sum_{i = 1}^{16} 10^{\mu_i} \mathbf{1}_{\bar{\Omega}_i}(x), 
\end{align*}
where $\mathbf{1}_{\bar{\Omega}_i}(x)$ denotes the indicator function of subset $\bar{\Omega}_i\subset \Omega$. Each $\bar{\Omega}_i$ is defined as  
\begin{align*}
    \bar{\Omega}_i = \Big[\frac{m}{8}-\frac{1}{16}, \frac{m}{8}+\frac{1}{16}\Big] \times \Big[\frac{n}{8}-\frac{1}{16}, \frac{n}{8}+\frac{1}{16}\Big], \qquad m, n \in \{1,3,5,7\},
\end{align*}
so that the 16 pairs $(m,n)$ define 16 mini-squares. Each $\mu_i\in\bR$ is independently sampled from the uniform distribution $\mathcal{U}(-1, 1)$, so that the parameter vector $\bm\mu=(\mu_1,\dots,\mu_{16})$ lies in $\mathbb{R}^{16}$. The boundary data are prescribed as 
\begin{align*}
    u_0(x)&=0.1(1-x_1)\sin(4\pi x_2), \qq{$x\in\Gamma_{\text{left}}\cup\Gamma_{\text{right}}=:\Gamma_D$,} \\
    g(x)  &= 0.1(1-x_2)\cos(2\pi x_1), \qq{$x\in \Gamma_{\text{top}}\cup\Gamma_{\text{bottom}}:=\Gamma_N$,}
\end{align*}
where  $\Gamma_{\text{left}}, \Gamma_{\text{right}}, \Gamma_{\text{top}}, \Gamma_{\text{bottom}}$ denote the left, right, top and bottom boundary of the domain $\Omega$. The source term $f$ is defined by its action on any test function $v\in H^{1}_{0, \Gamma_D}(\Omega)$ 
\begin{align*}
     \langle f, v\rangle = (\bbf_1, \grad{v})_{L^2} + (f_2, v)_{L^2}, 
 \end{align*}
where 
\begin{align*}
 \bbf_1(x) = (0.5\times\mathbf{1}_{\cup\bar{\Omega}_i}(x), -0.5\times\mathbf{1}_{\cup\bar{\Omega}_i}(x))^{\top} \text{\;and\;} f_2(x) = 1.
\end{align*}
This representation fits the flux-free decomposition of the source used in the loss formulation. Note that  $f\in H^{-1}(\Omega)\subset (H^1_{0,\Gamma_D}(\Omega))'$, while $f\notin L^2(\Omega)$, in accordance with our setup \eqref{flfree}. The low regularity of the source term, combined with the piecewise discontinuous conductivity field, results in limited solution regularity, see Figure \ref{fig:visualization_parameter_solution} (top row) for an illustration of one parameter-solution pair at a random parameter sample. In fact,  the finite element solution $[\uc_h, \sc_h]$ (with elements RT$_1^\circ\times$CG$_2^\circ$ on a mesh of size $128\times128$, where the edges of the mini-squares are aligned with the mesh edges) exhibits steep gradients around the mini-squares.

\begin{figure}[!htb]
    \centering
    \includegraphics[width=0.235\linewidth]{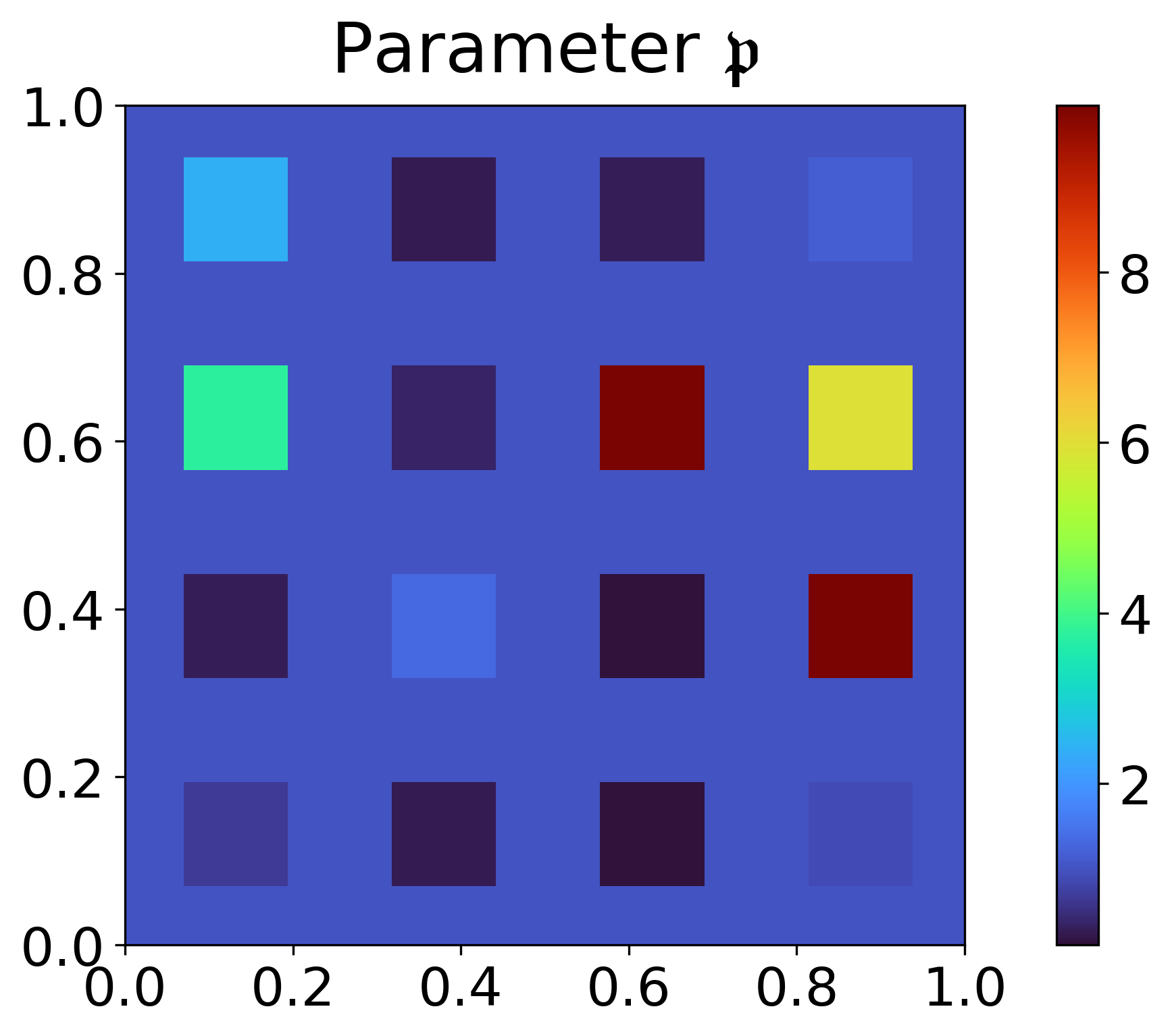}
    \includegraphics[width=0.24\linewidth]{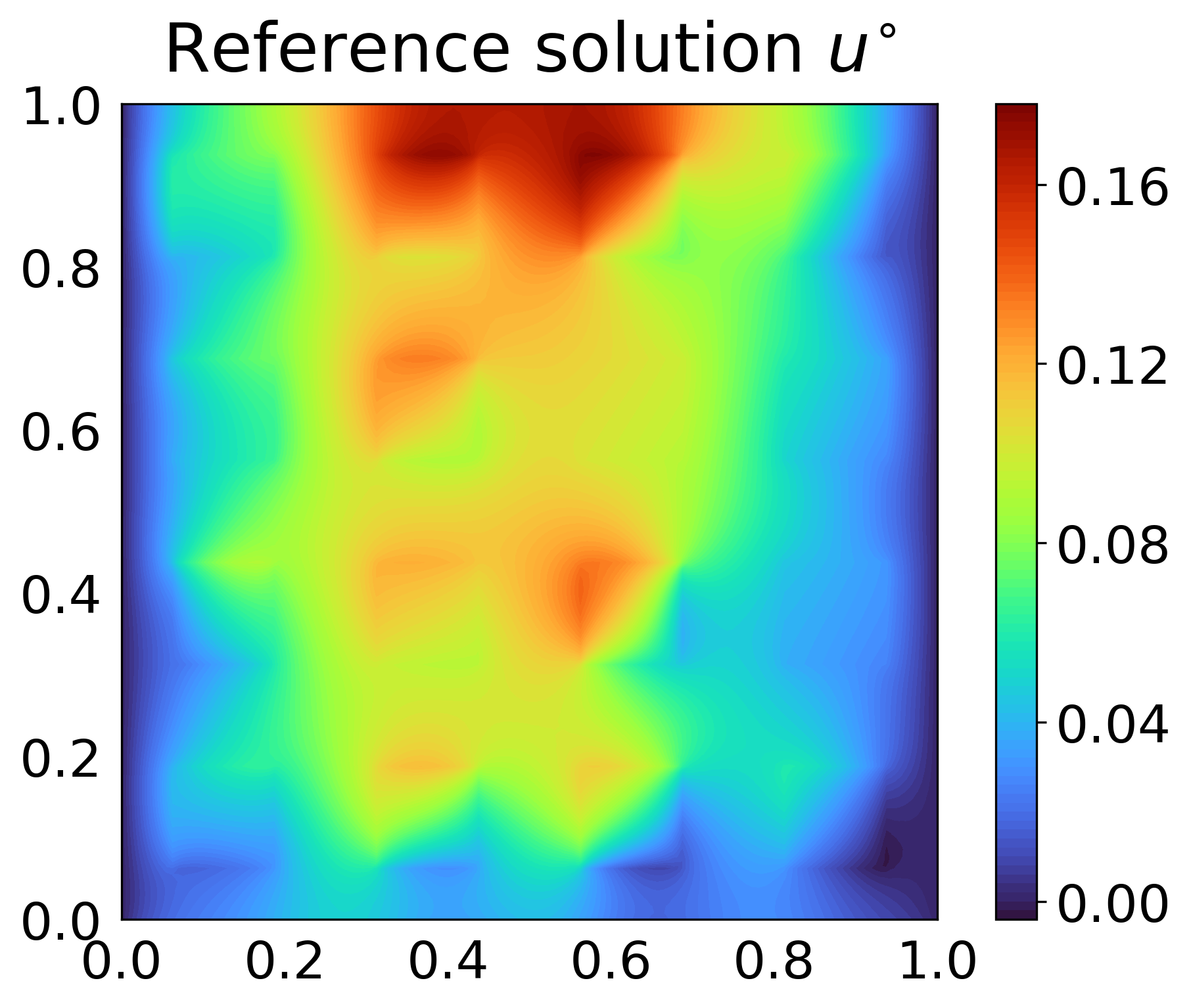}
    \includegraphics[width=0.24\linewidth]{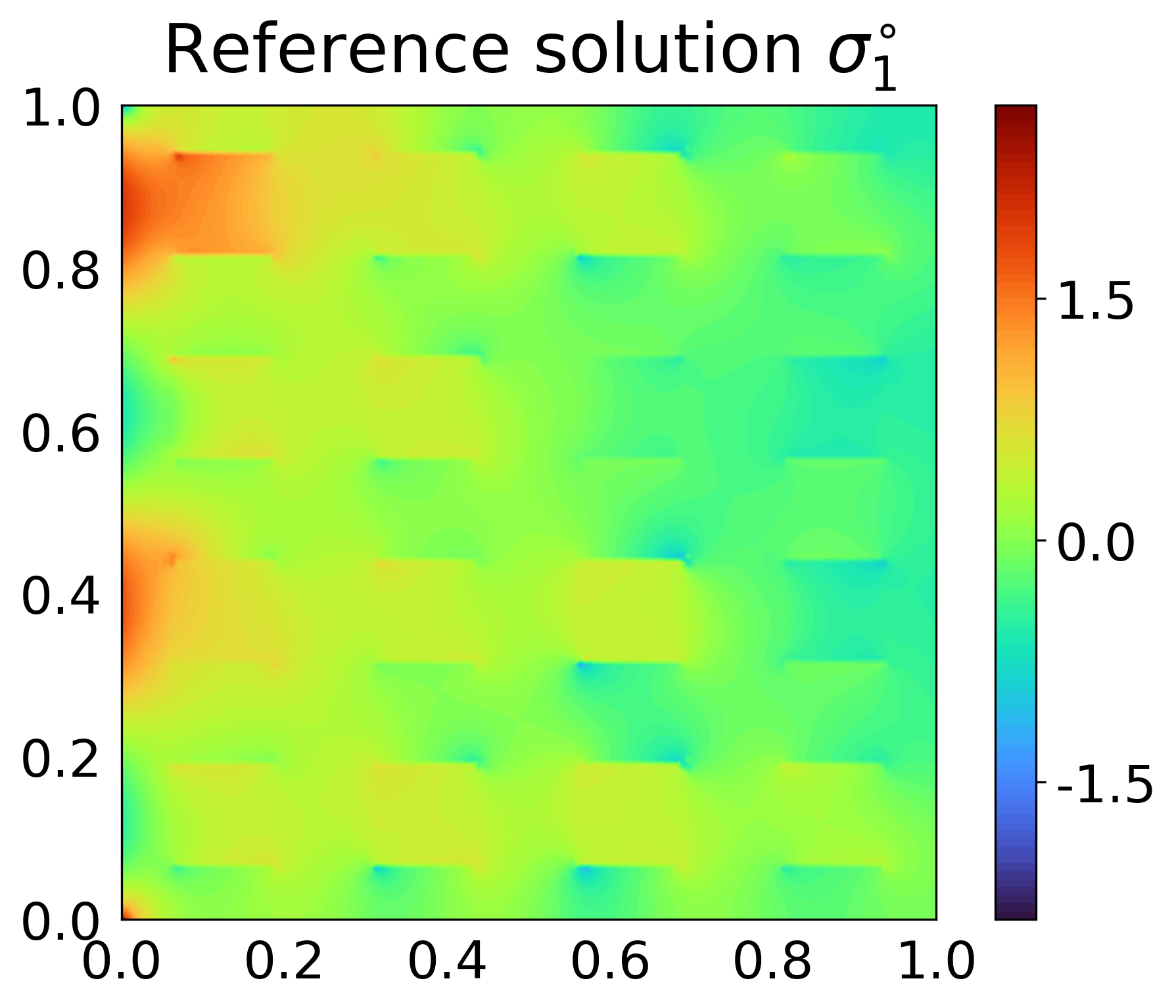}
    \includegraphics[width=0.24\linewidth]{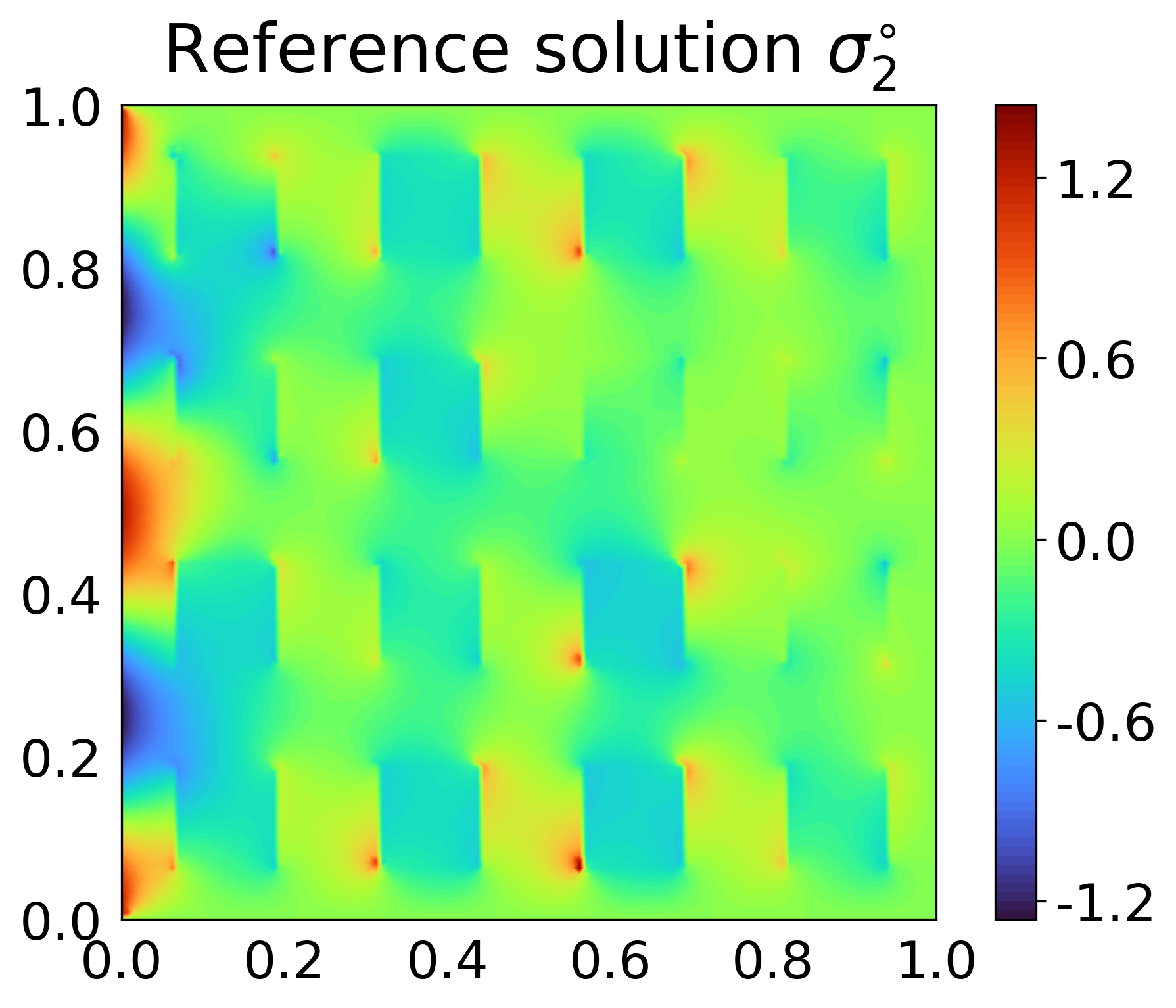}

    \vspace{1em}
    \includegraphics[width=0.24\linewidth]{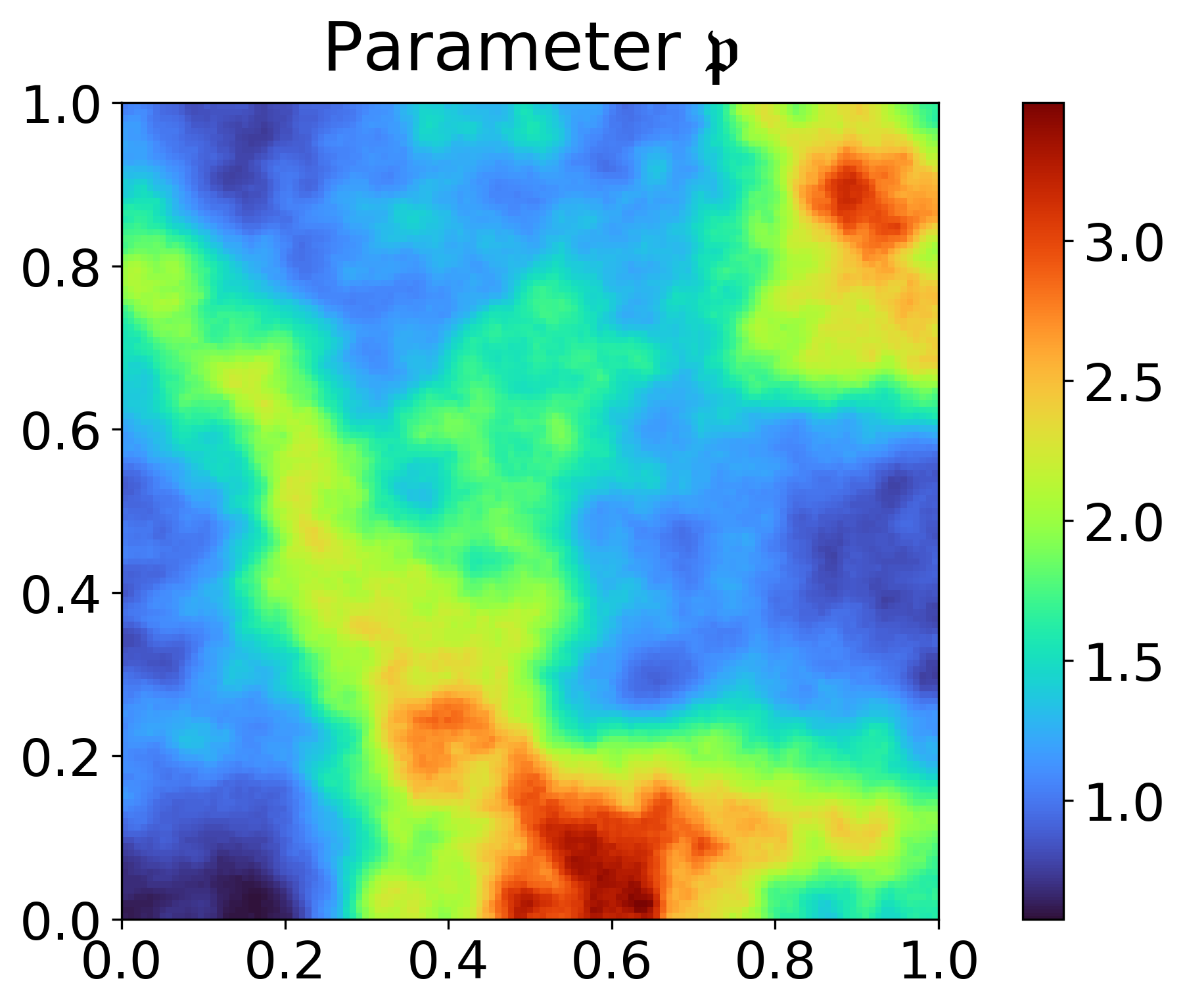}
    \includegraphics[width=0.24\linewidth]{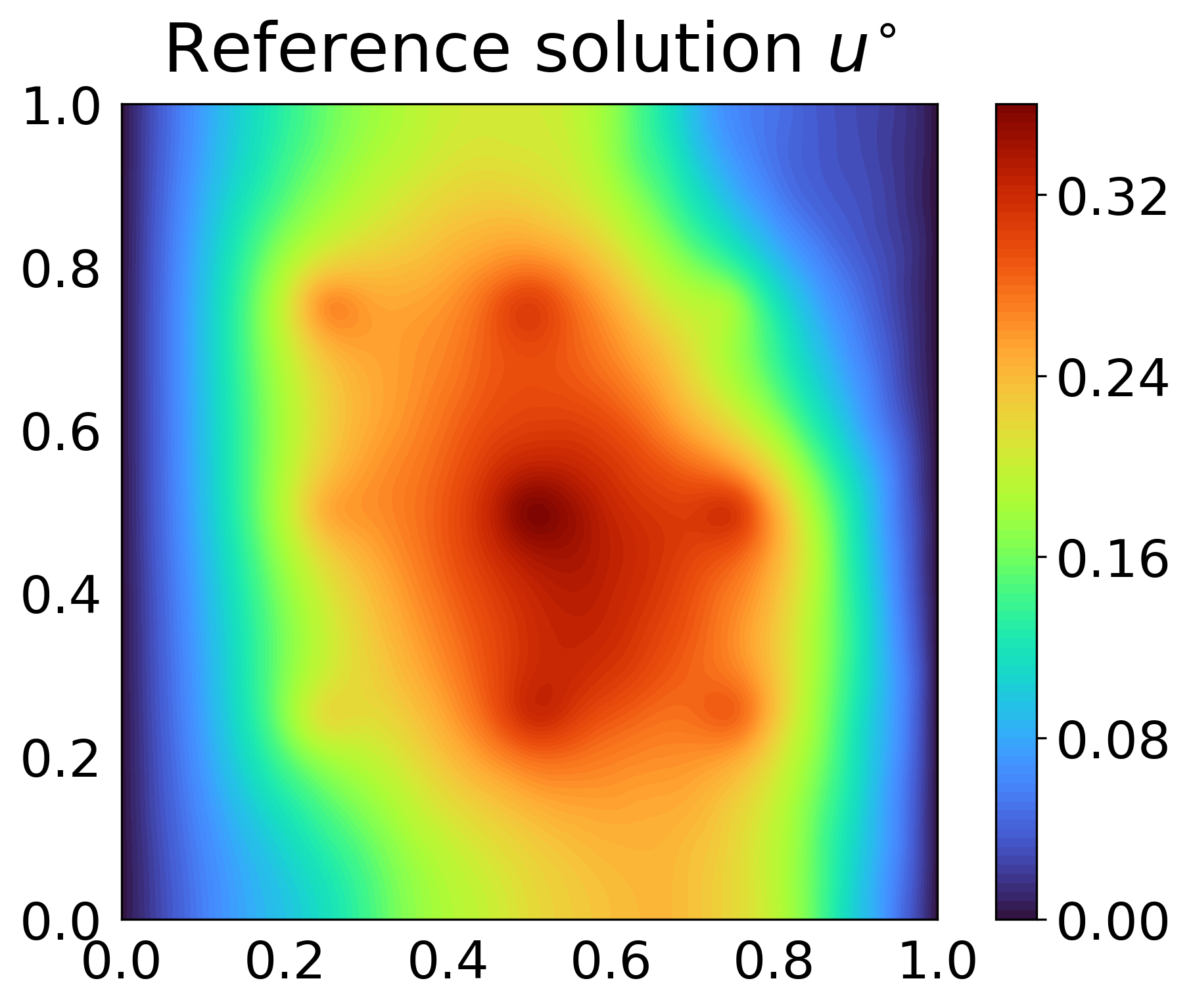}
    \includegraphics[width=0.24\linewidth]{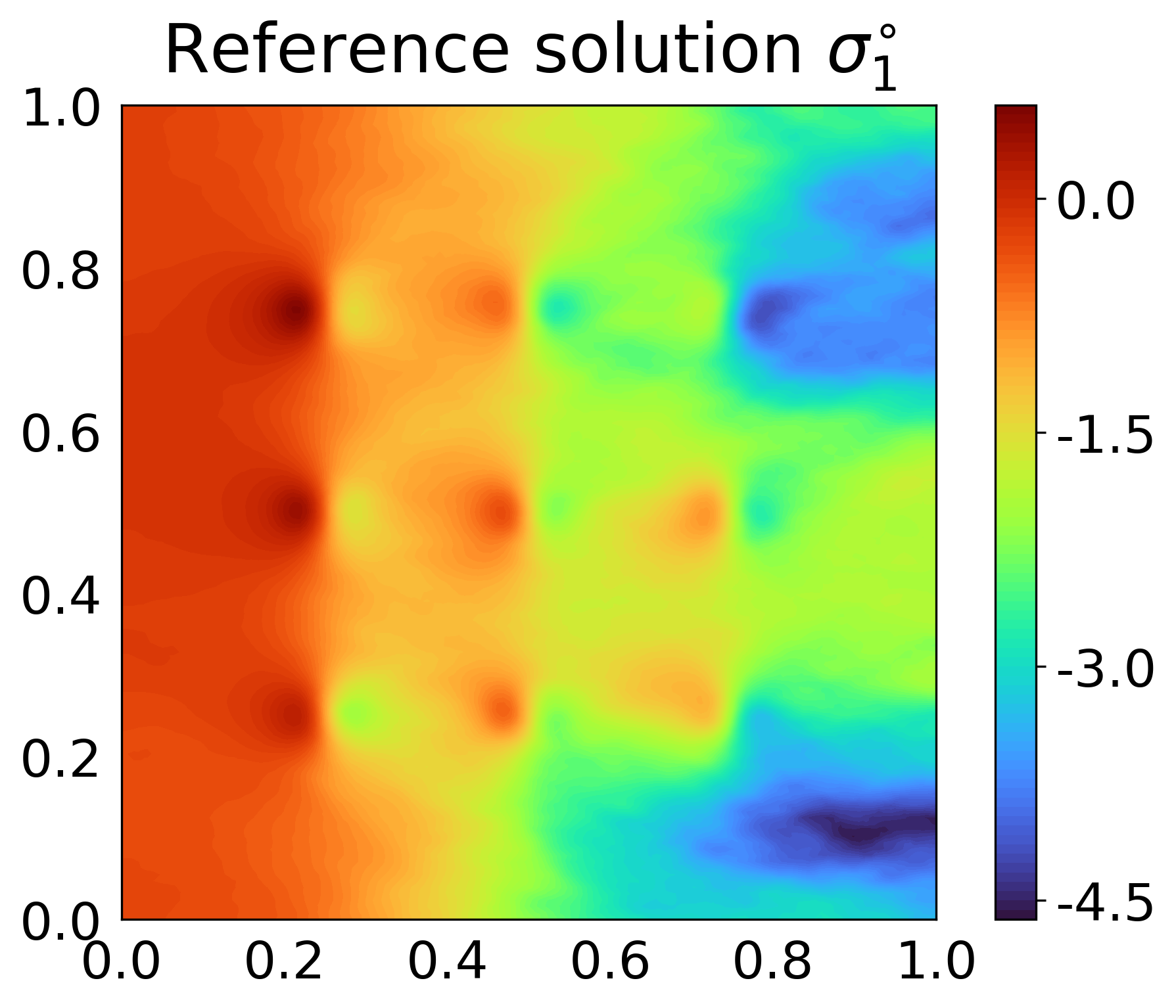}
    \includegraphics[width=0.24\linewidth]{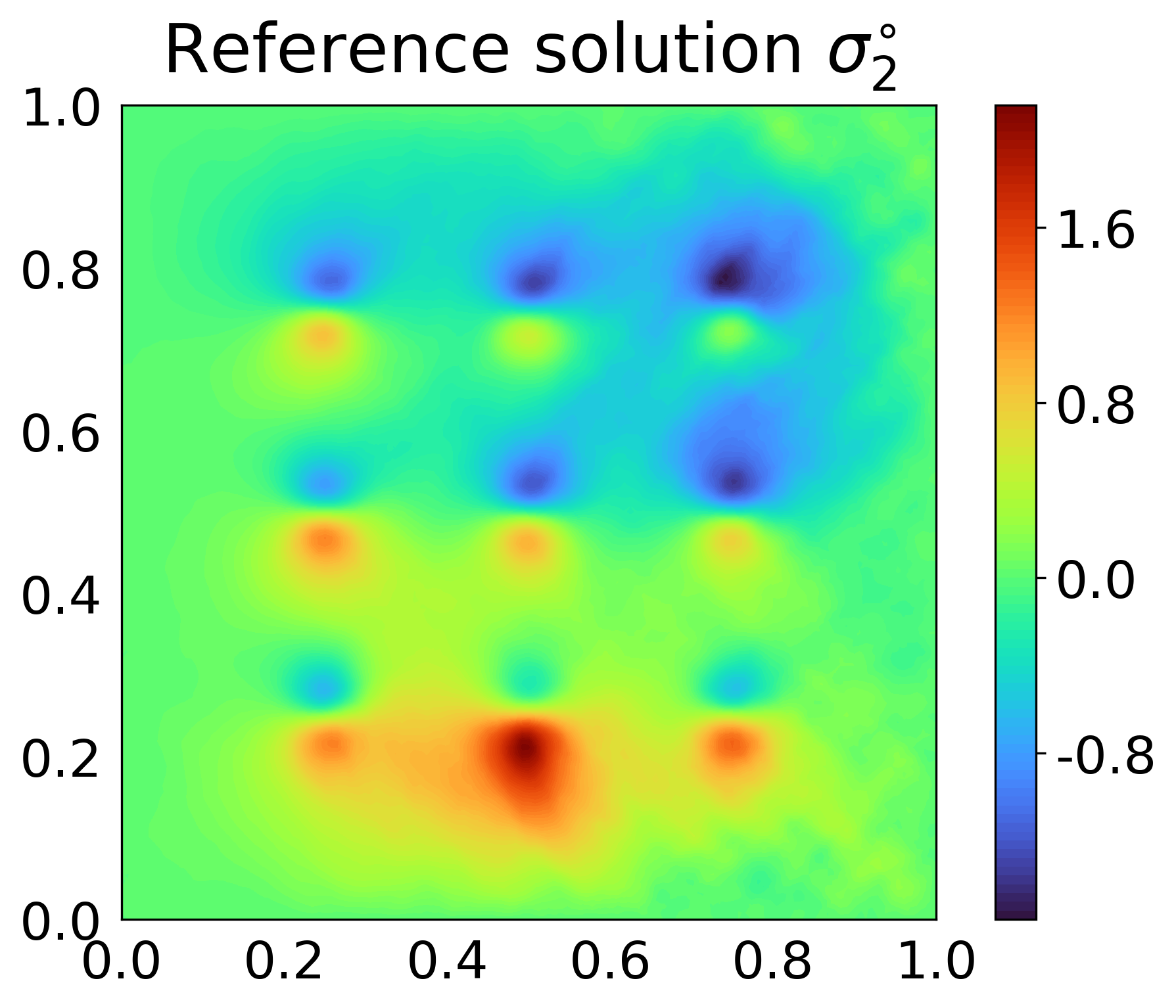}

    \vspace{1em}
    \includegraphics[width=0.32\linewidth]{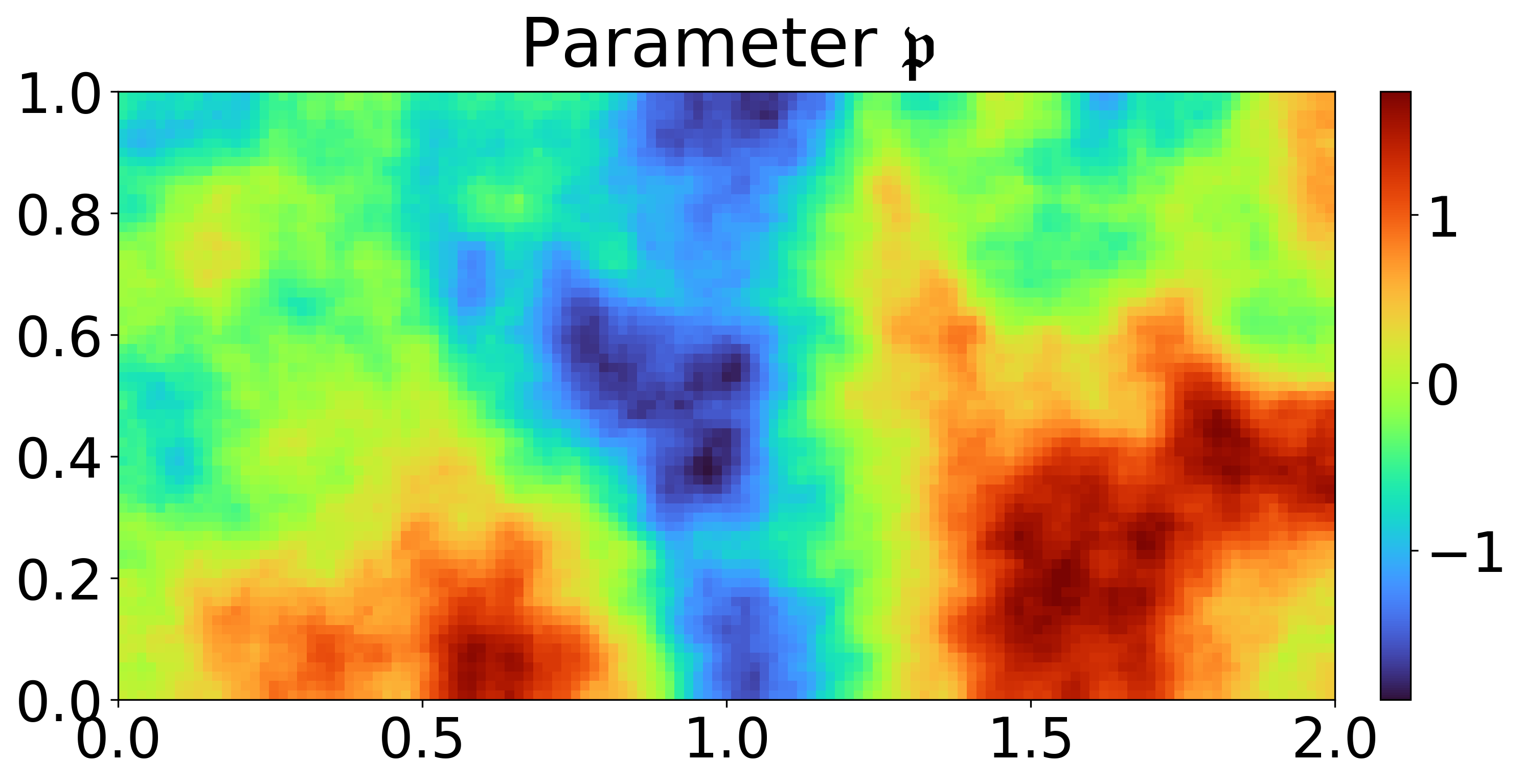}
    \includegraphics[width=0.32\linewidth]{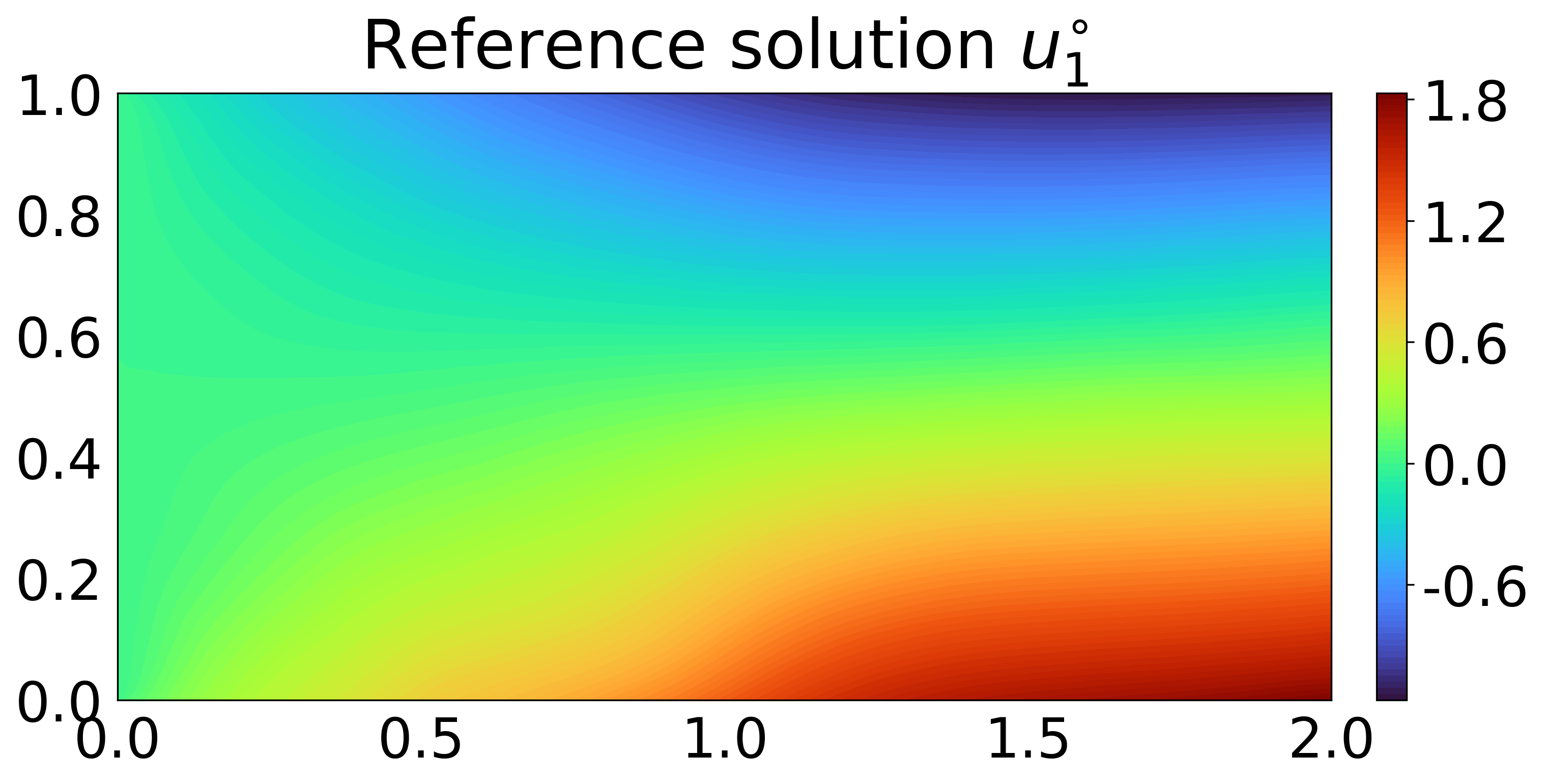}
    \includegraphics[width=0.32\linewidth]{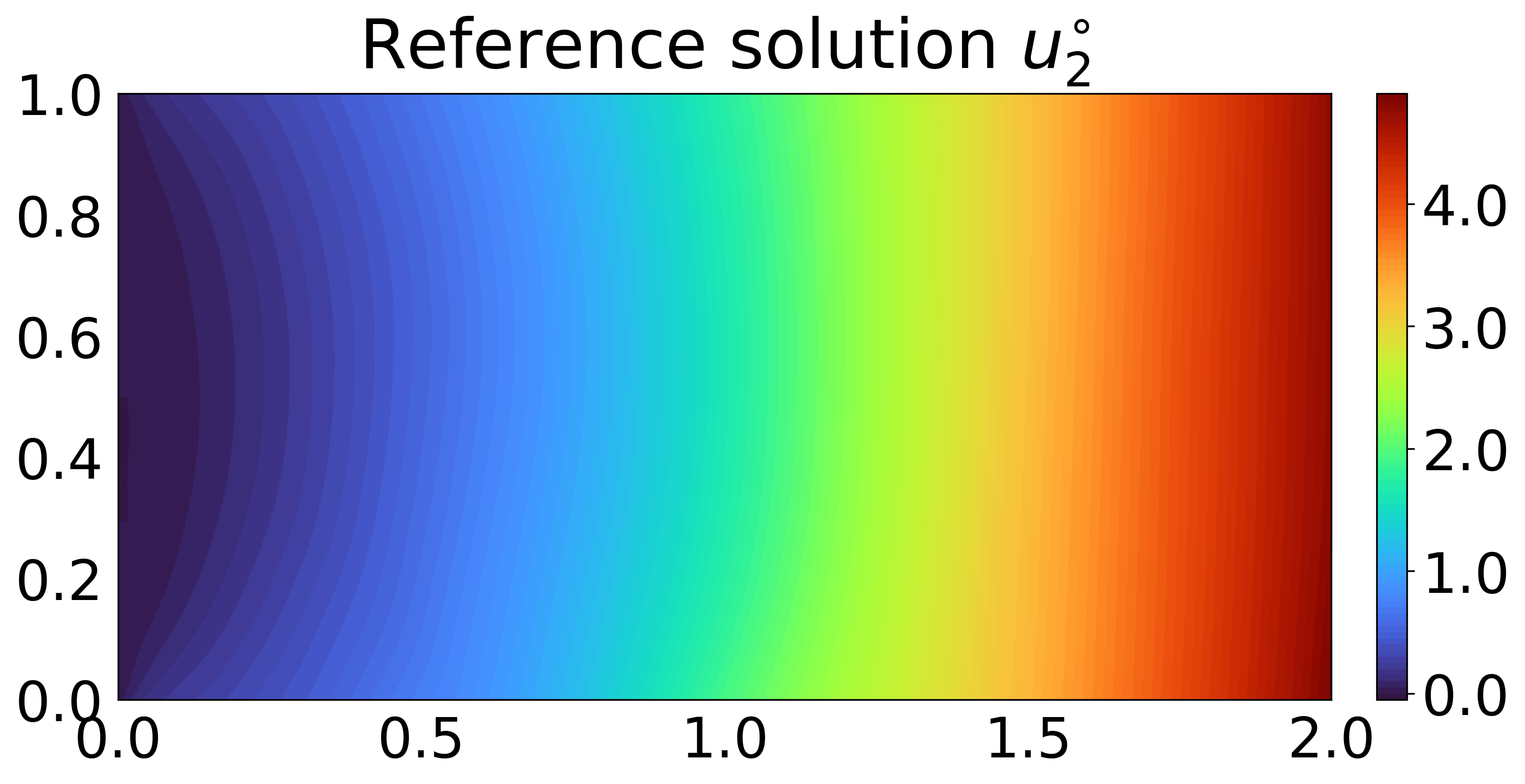}

    \vspace{1em}
    \includegraphics[width=0.24\linewidth]{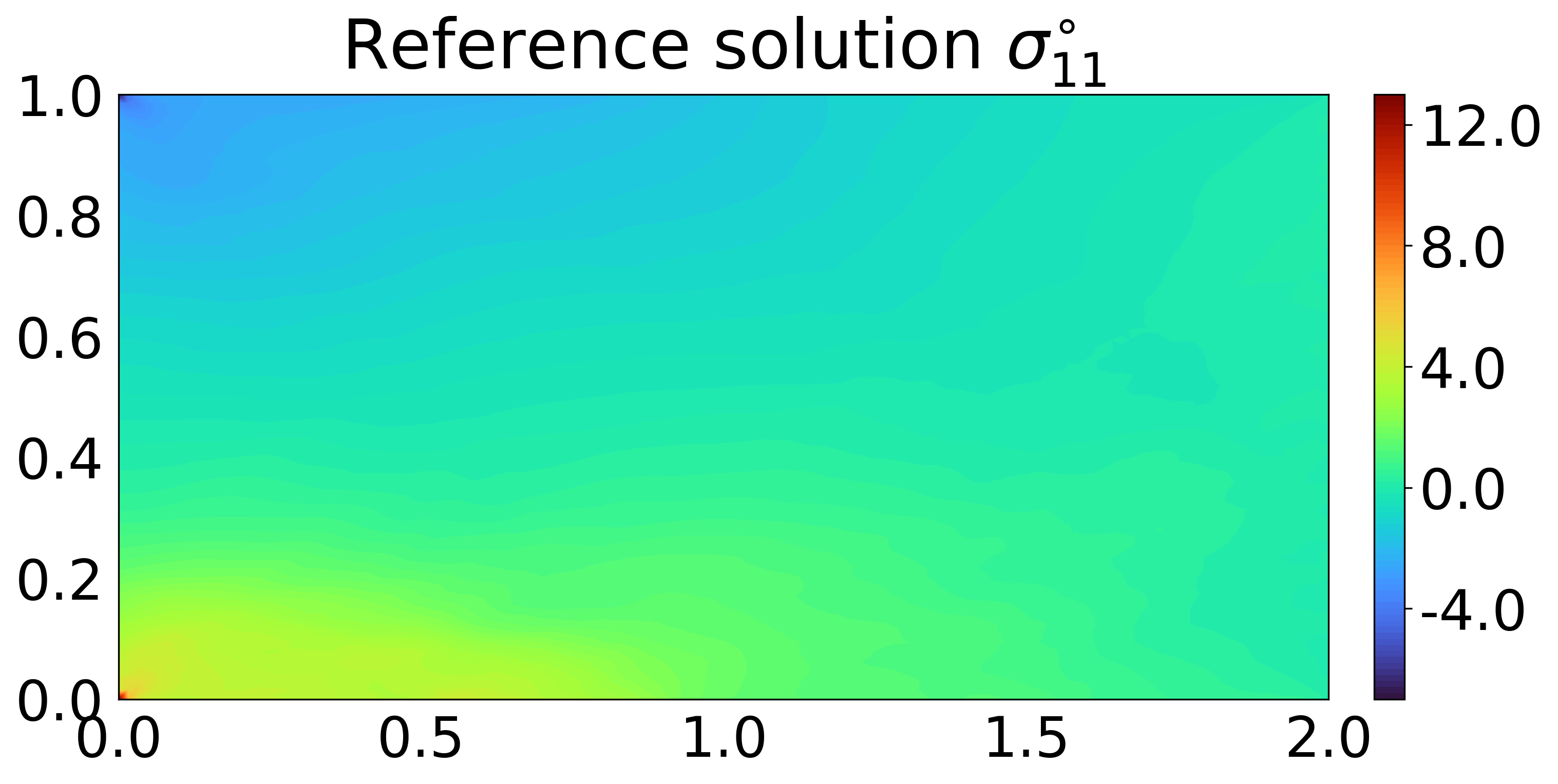}
    \includegraphics[width=0.24\linewidth]{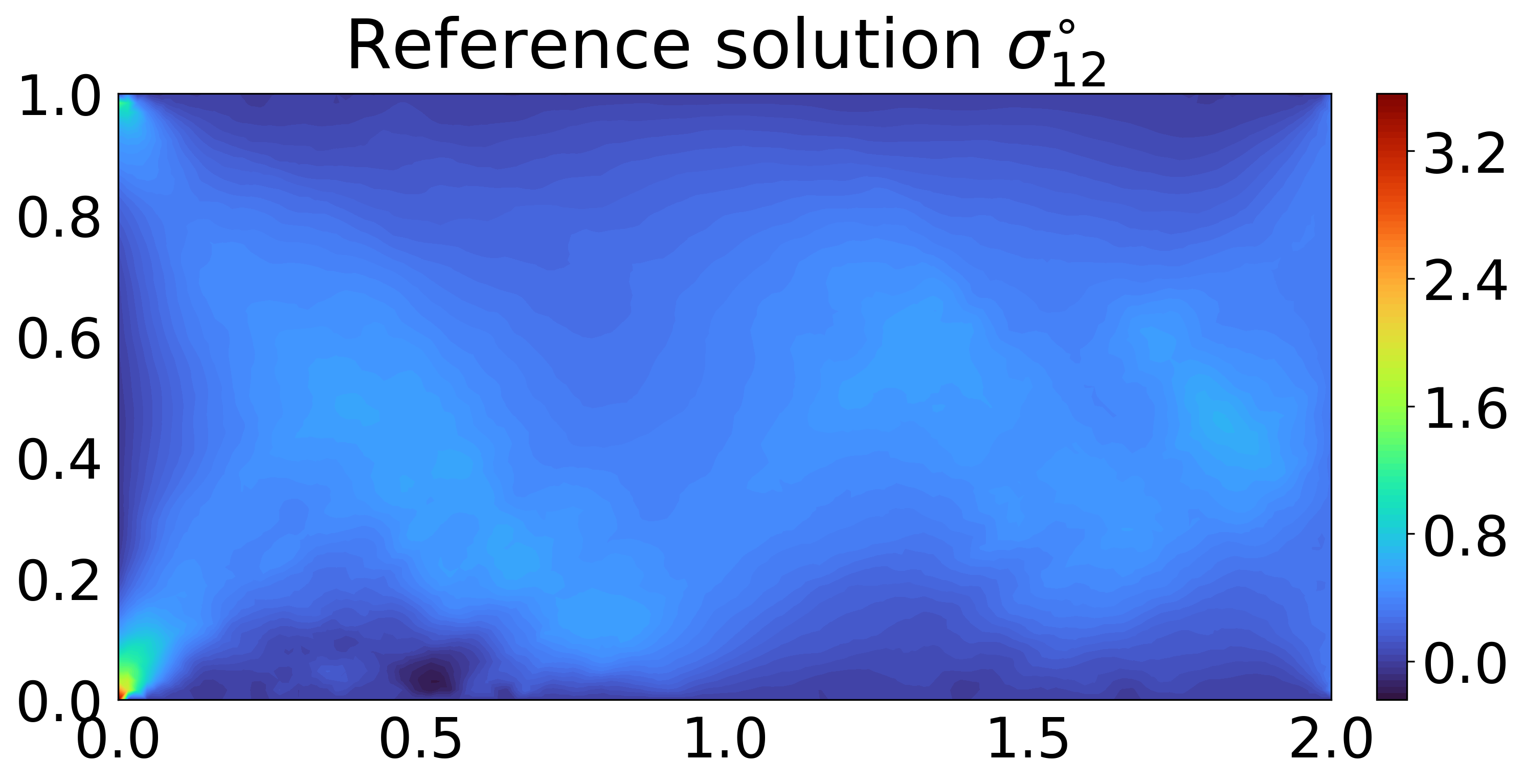}
    \includegraphics[width=0.24\linewidth]{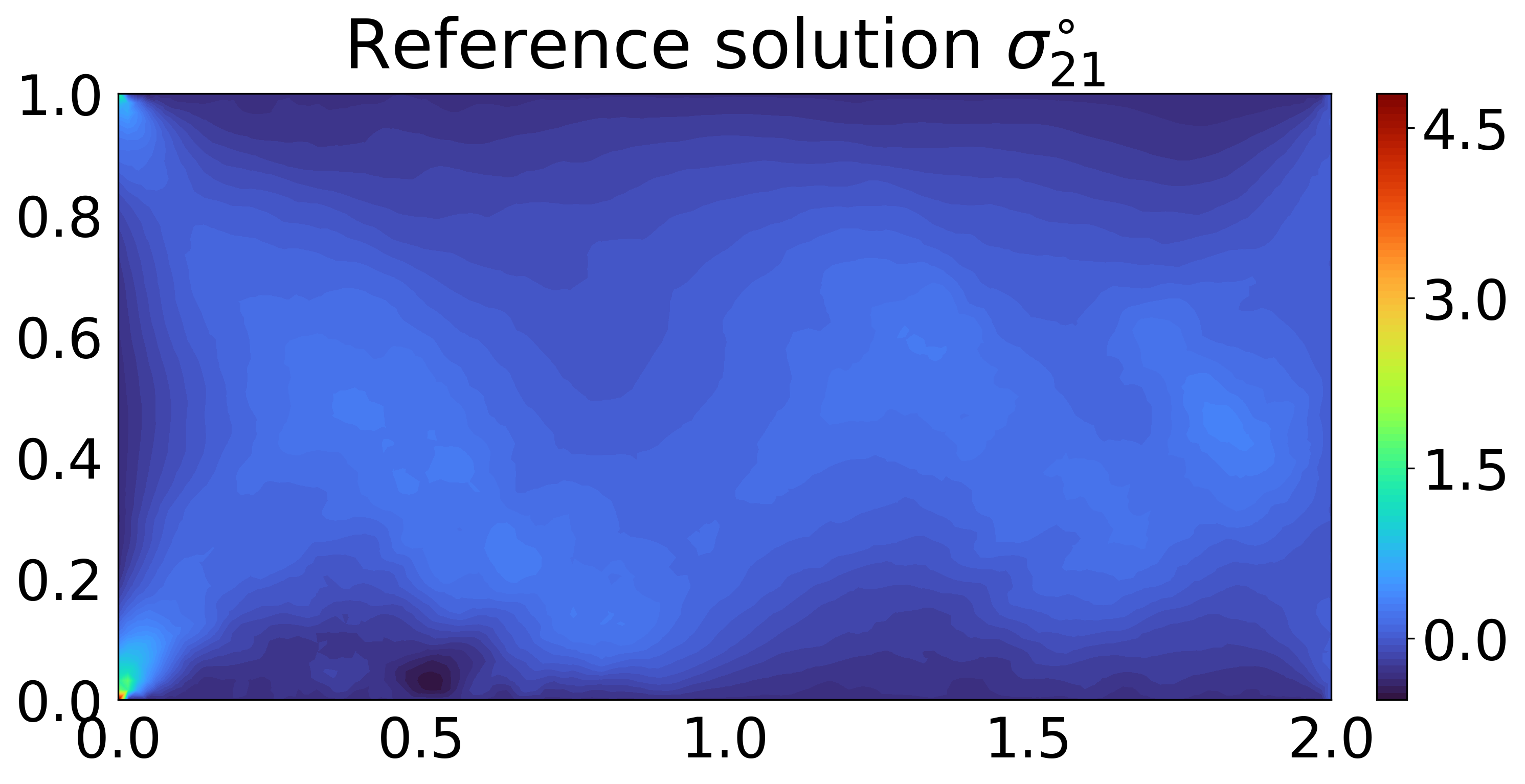}
    \includegraphics[width=0.24\linewidth]{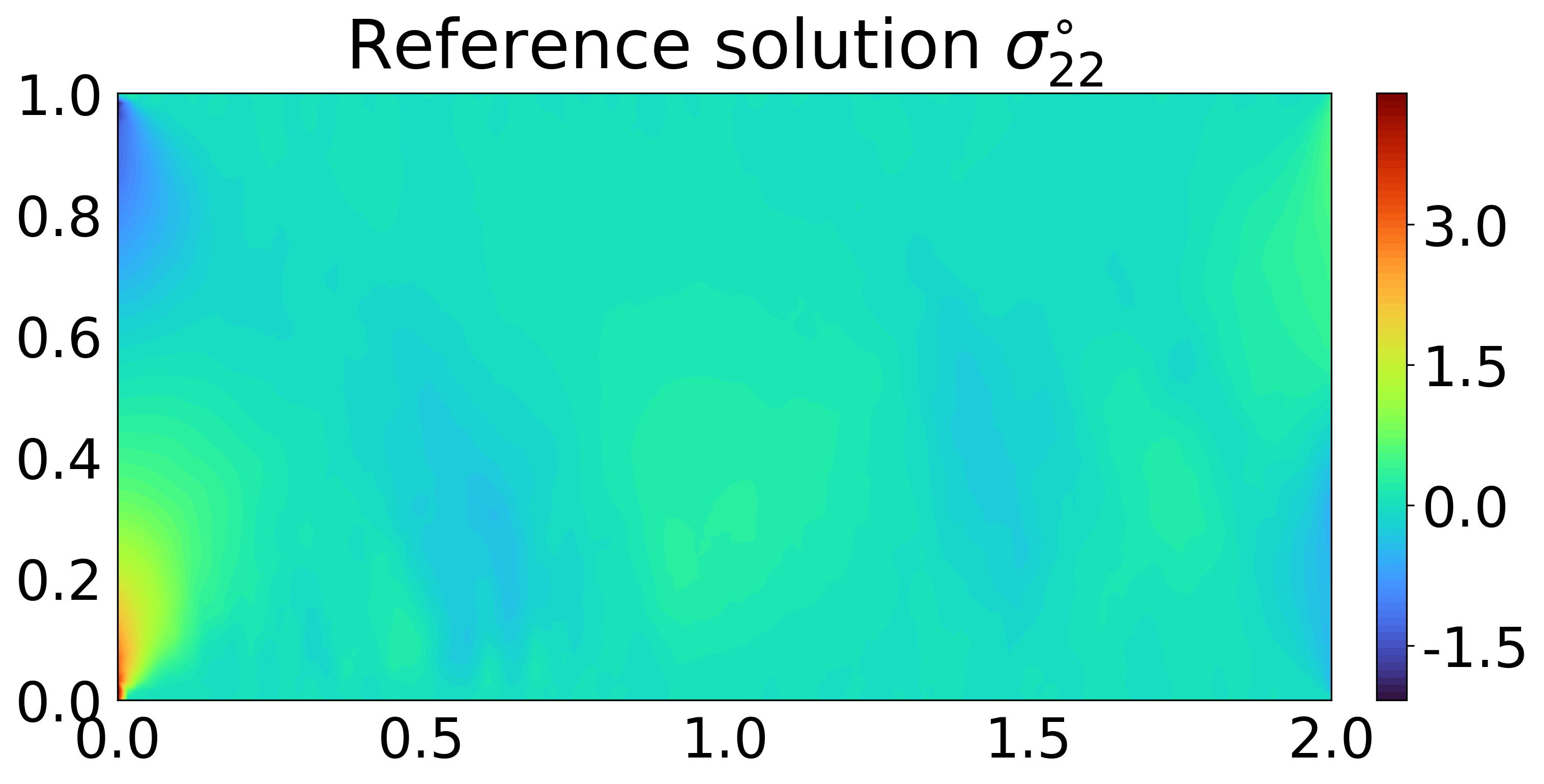}

    \caption{Visualization of parameter-to-solution map $\pp_h \mapsto [u_h^{\circ}(\pp_h), \sigma_h^{\circ}(\pp_h)]$ at a random parameter sample $\pp_h$ (left) for the heat conduction (first row),  Darcy flow (second row), and linear elasticity setup (third and fourth rows). 
    }
    \label{fig:visualization_parameter_solution}
\end{figure}

\subsubsection{Darcy flow} In this case, we consider a Darcy flow problem with a lognormal diffusion field to model subsurface flow with a random permeability field. Specifically, we assume $\Omega=(0,1)\times (0,1)$ and 
\begin{equation}
    \pp(x) = 0.01 + \exp(m(x)), 
\end{equation}
where $m \sim \mathcal{N}(\bar{m}, \mathcal{C})$ with mean $\bar{m} = 0$ and covariance operator $\cC :=(\delta I -\gamma \Delta)^{-\alpha}$. The parameters $\delta, \gamma, \alpha > 0$ collectively control the correlation, variance, and smoothness of the random field. For demonstration, we set $\delta=1.5$, $\gamma=0.15$, and $\alpha=2$. To sample the Gaussian random field $m$ we use the hIPPYlib implementation~\cite{villa2021hippylib}, which realizes the covariance operator $\cC$ via an auxiliary elliptic problem with Robin boundary conditions. 

The boundary data for the Darcy solution are specified as 
\begin{align*}
    u_0(x) &= 1-x_1, \qq{$x\in\Gamma_{\text{left}}\cup\Gamma_{\text{right}}=:\Gamma_D$,} \\
    g(x) &= 0, 
    \qq{$x\in \Gamma_{\text{top}}\cup\Gamma_{\text{bottom}}:=\Gamma_N$.}
\end{align*}
The source term, representing water extraction at nine wells, is defined as  
\begin{equation}
    f(x) = \sum_{i = 1}^9 100 \exp\left(-\left( \frac{\|x - c_i\|_2}{w}\right)^2\right), 
\end{equation}
where the centers are given by $c_i = (m/4, n/4)$ for $m,n = 1, 2, 3$ and the width is $w = 1/32$. The second row of \cref{fig:visualization_parameter_solution} displays one parameter-solution pair at a random parameter sample, with the finite element solution $[\uc_h, \sc_h]$ computed with elements RT$_1^\circ\times$CG$_2^\circ$ on a mesh of size $128\times128$, and the finite element parameter sample computed using CG$_1$ elements on the same mesh. 

\subsubsection{Linear elasticity} 
We consider a rectangular elastic body clamped along its left edge, with an upward force applied to its right edge. In this experiment, the physical domain is set to $\Omega=(0,2)\times (0,1)$. The Gaussian measure for the random parameter field $\pp$ is characterized by $\bar{\pp}=0, \delta=1.5$, $\gamma=0.15$, and $\alpha=2$, and it enters the model through spatial variations of the stiffness tensor. The Poisson ratio is set to $\nu=0.4$. The boundary data are prescribed as  
\begin{align*}
    \tenbar[1] u_0(x) &=(0, 0)^{\top}, \qq{$x\in\Gamma_{\text{left}}$,} \\
    \tenbar[1] t(x) &=
    \begin{cases}
        (0, 0)^{\top}, &\qq{$x\in\Gamma_{\text{top}}\cup\Gamma_{\text{bottom}}$,} \\
        (0.6 \exp(-(x_2-0.5)^2/4), 0.3(1+x_2/10))^{\top}, &\qq{$x\in\Gamma_{\text{right}}$,} 
    \end{cases}
\end{align*}
That is, the left edge is clamped, while on the right edge we apply a horizontal traction given by a Gaussian bump centered at $x_2=0.5$ and a vertical traction that increases linearly with $x_2$; all other boundary segments are traction-free. Equivalently, $\Gamma_D:= \Gamma_{\text{left}}$ and $\Gamma_N:= \Gamma_{\text{top}}\cup\Gamma_{\text{bottom}}
\cup \Gamma_{\text{right}}$. One parameter-solution pair at a random parameter sample is illustrated in the third and fourth rows of \cref{fig:visualization_parameter_solution}, where the finite element solution $[\tenbar[2]{\sigma}^\circ_h,\tenbar[1]{u}^\circ_h]$ is computed with elements $(\text{RT}_1^\circ)^2\times(\text{CG}^\circ_2)^2$ on a mesh of size $256\times128$. The finite element parameter sample is computed using CG$_1$ elements on the same mesh. We can observe a corner singularity of the stress tensor at the clamped left edge, where the boundary conditions change from Dirichlet to Neumann type, yielding low solution regularity.

\subsubsection{Variational lift of boundary data}
The variational lift of the boundary data by harmonic extension to the domain is shown in \cref{fig:auxiliary_variables}. These fields are precomputed by solving \eqref{eq:DiriLiftPoisson} and~\eqref{harm} for the diffusion problems using CG$_m$ elements (matching the order $m$ of $\uc_h$) and by solving~\eqref{eq:DiriLiftElasticity} and \eqref{aux-elast} for the linear elasticity problem using CG$_m^2$ elements. 

\begin{figure}[!ht]
    \centering
        \includegraphics[width=0.24\linewidth]{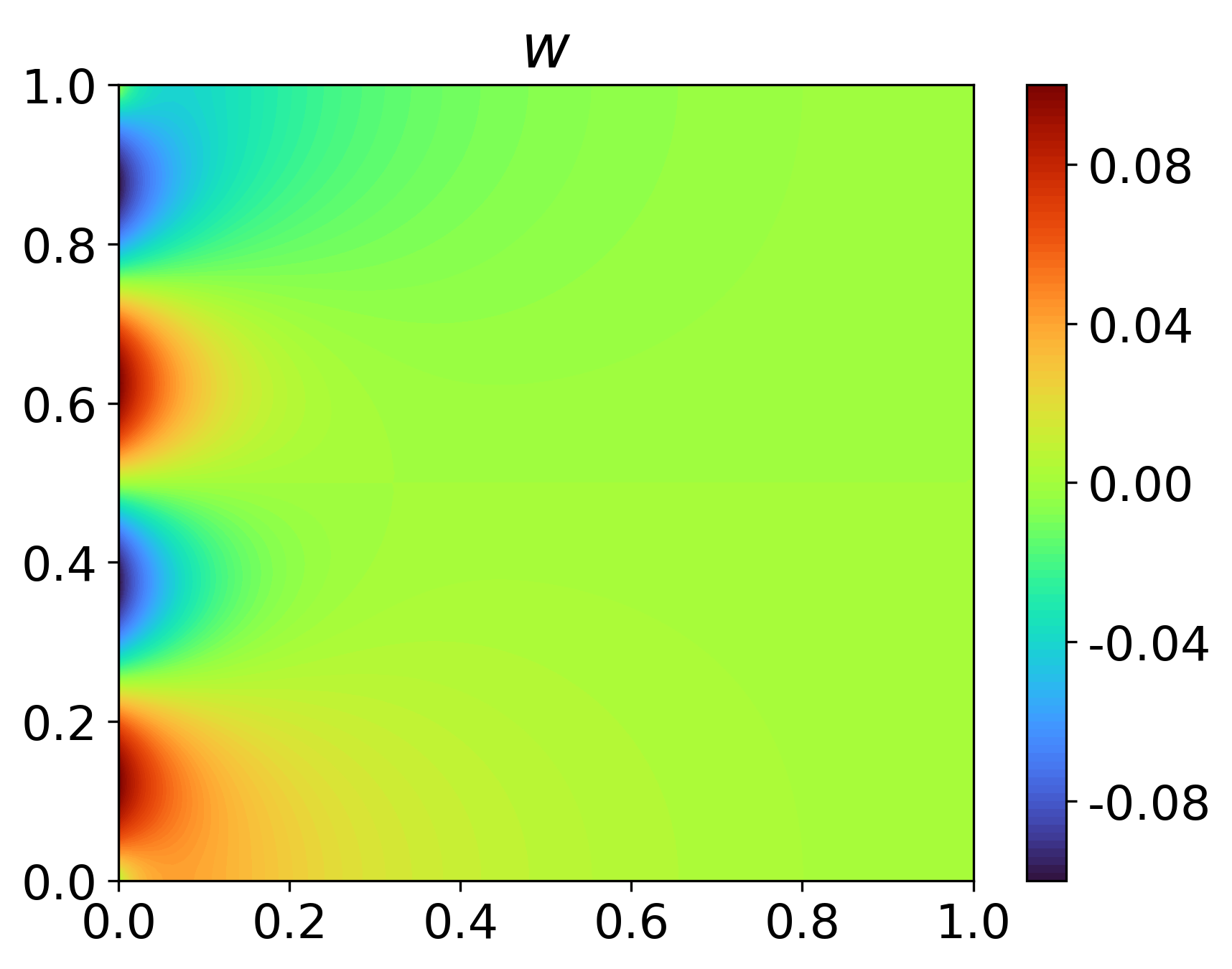}
        \includegraphics[width=0.24\linewidth]{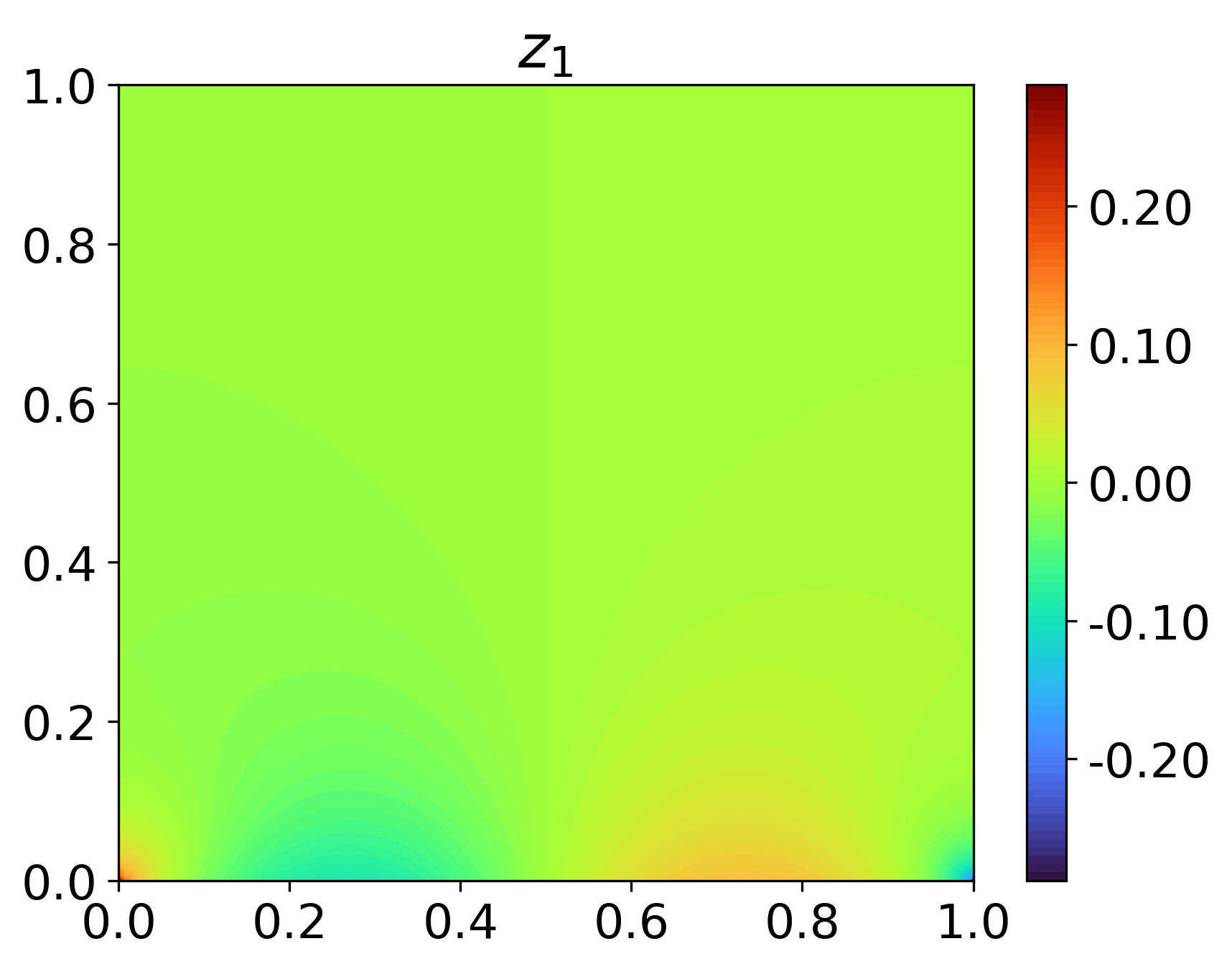}
        \includegraphics[width=0.24\linewidth]{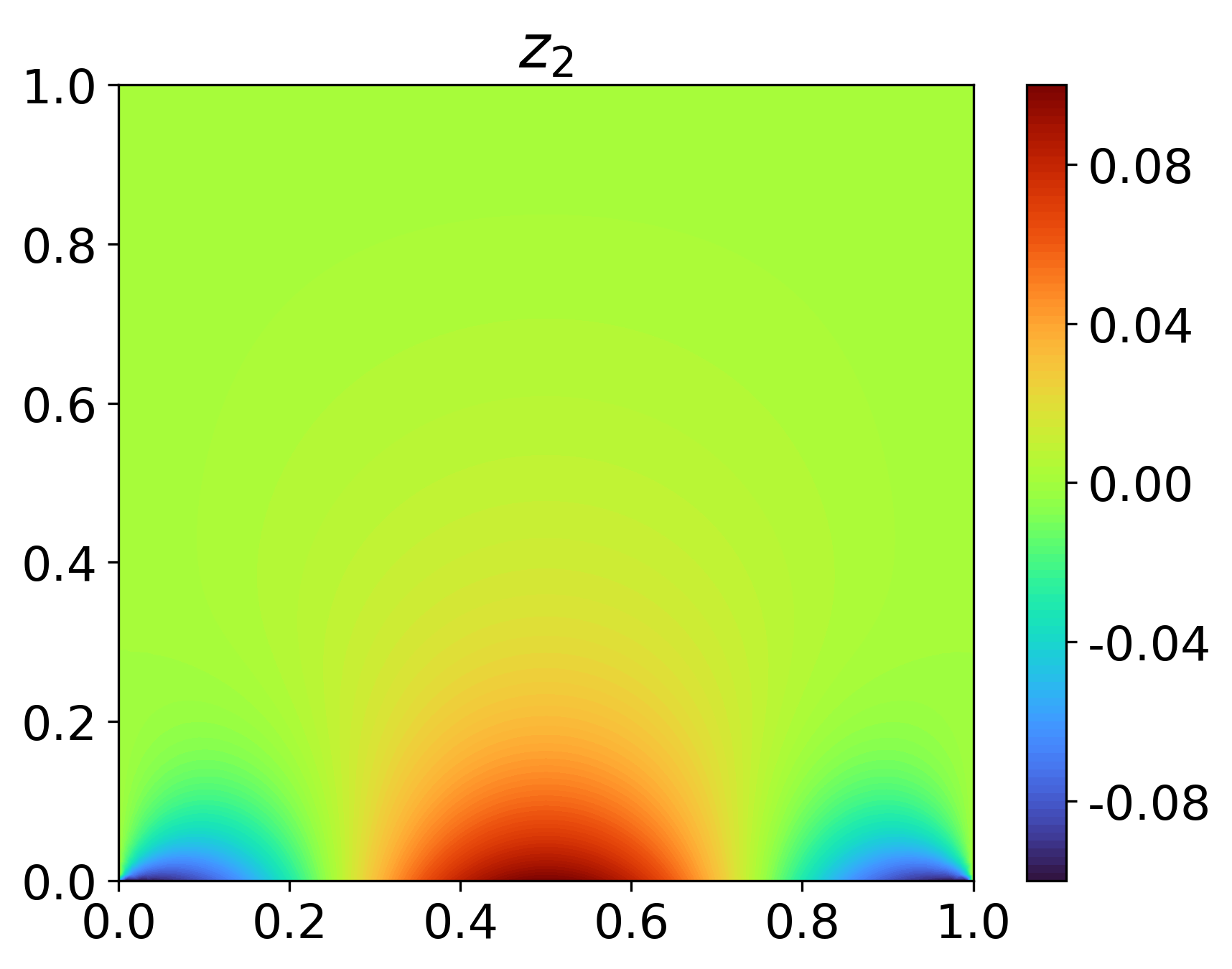}
        \includegraphics[width=0.225\linewidth]{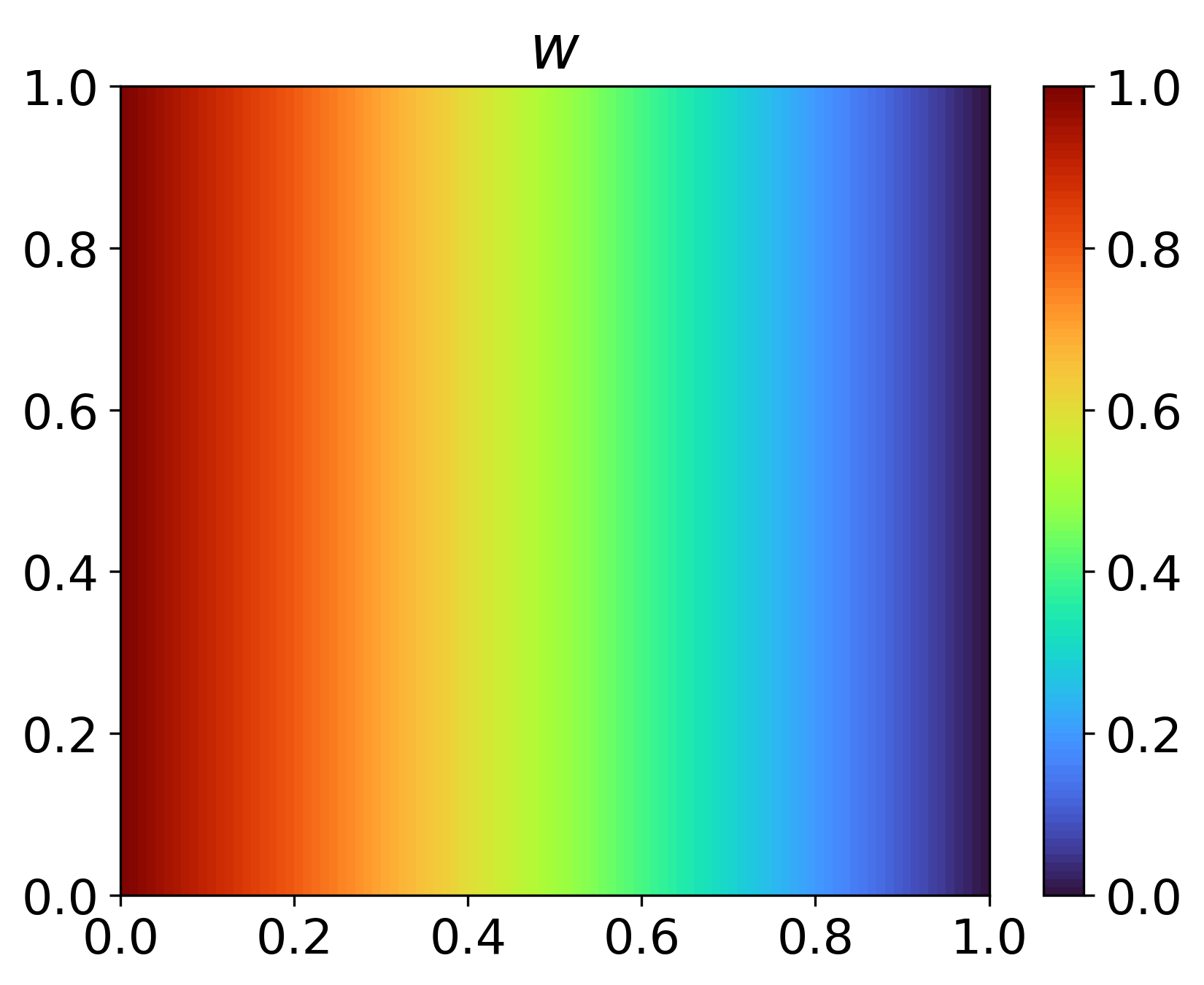}

        \includegraphics[width=0.24\linewidth]{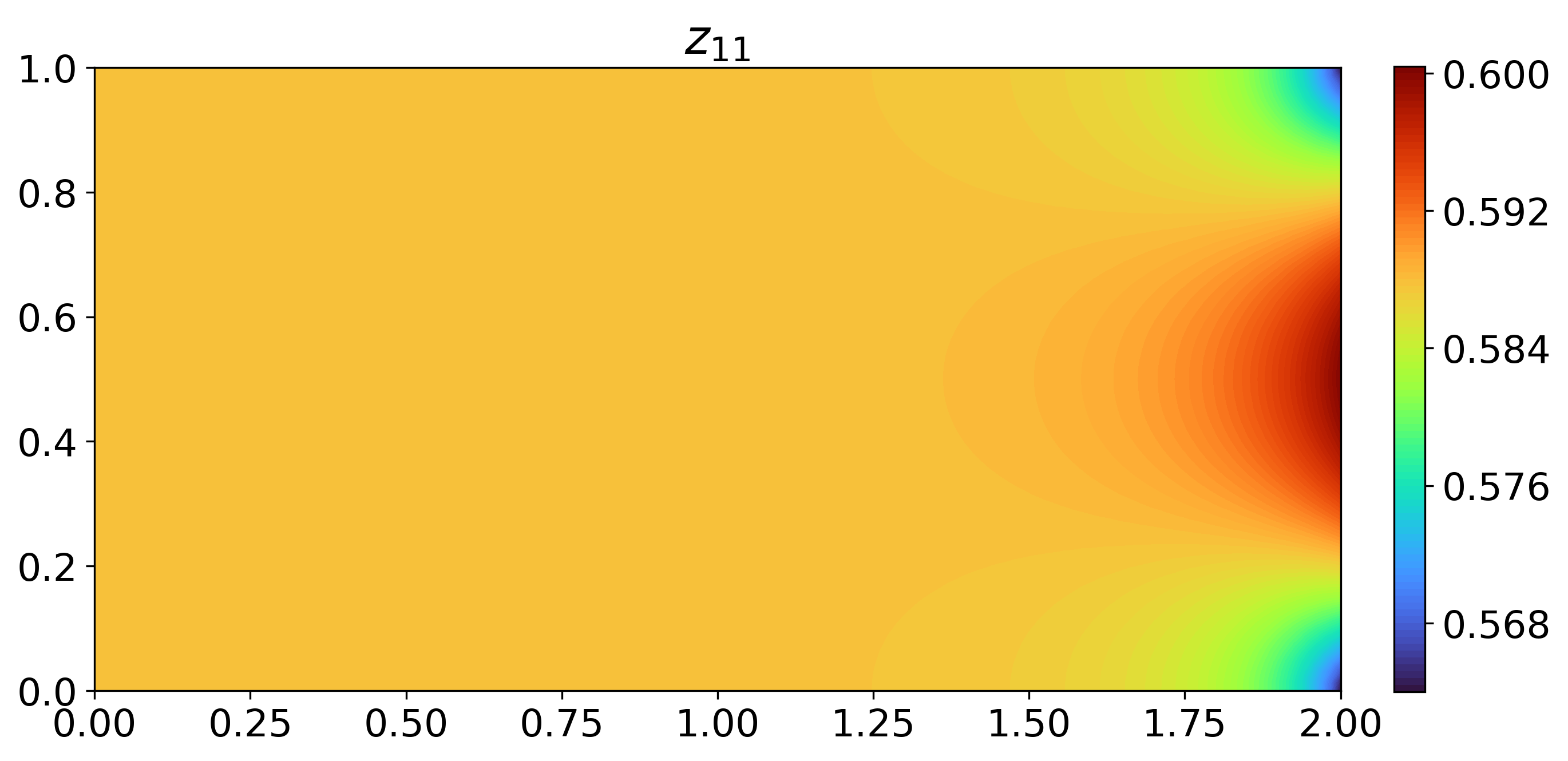}
        \includegraphics[width=0.24\linewidth]{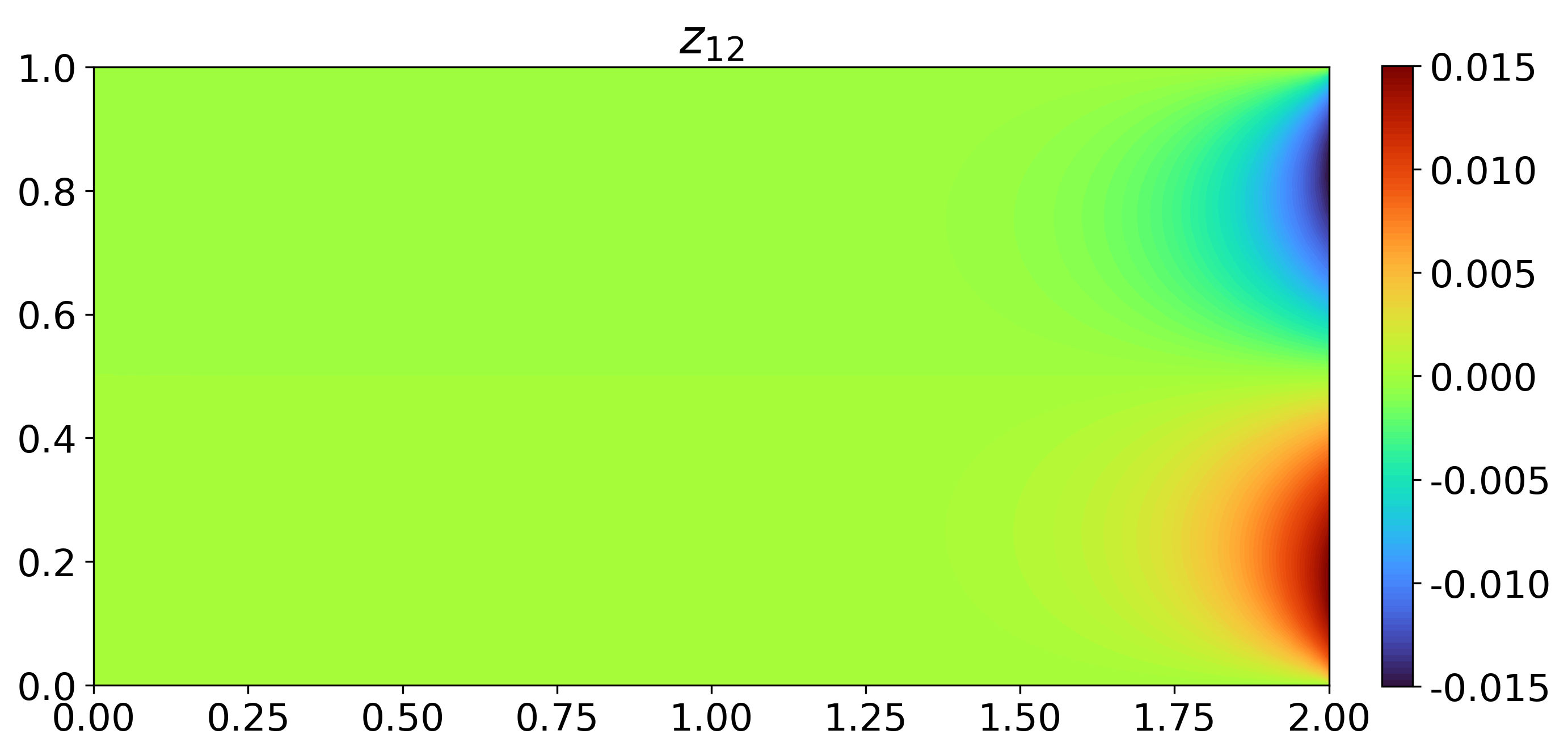}
        \includegraphics[width=0.24\linewidth]{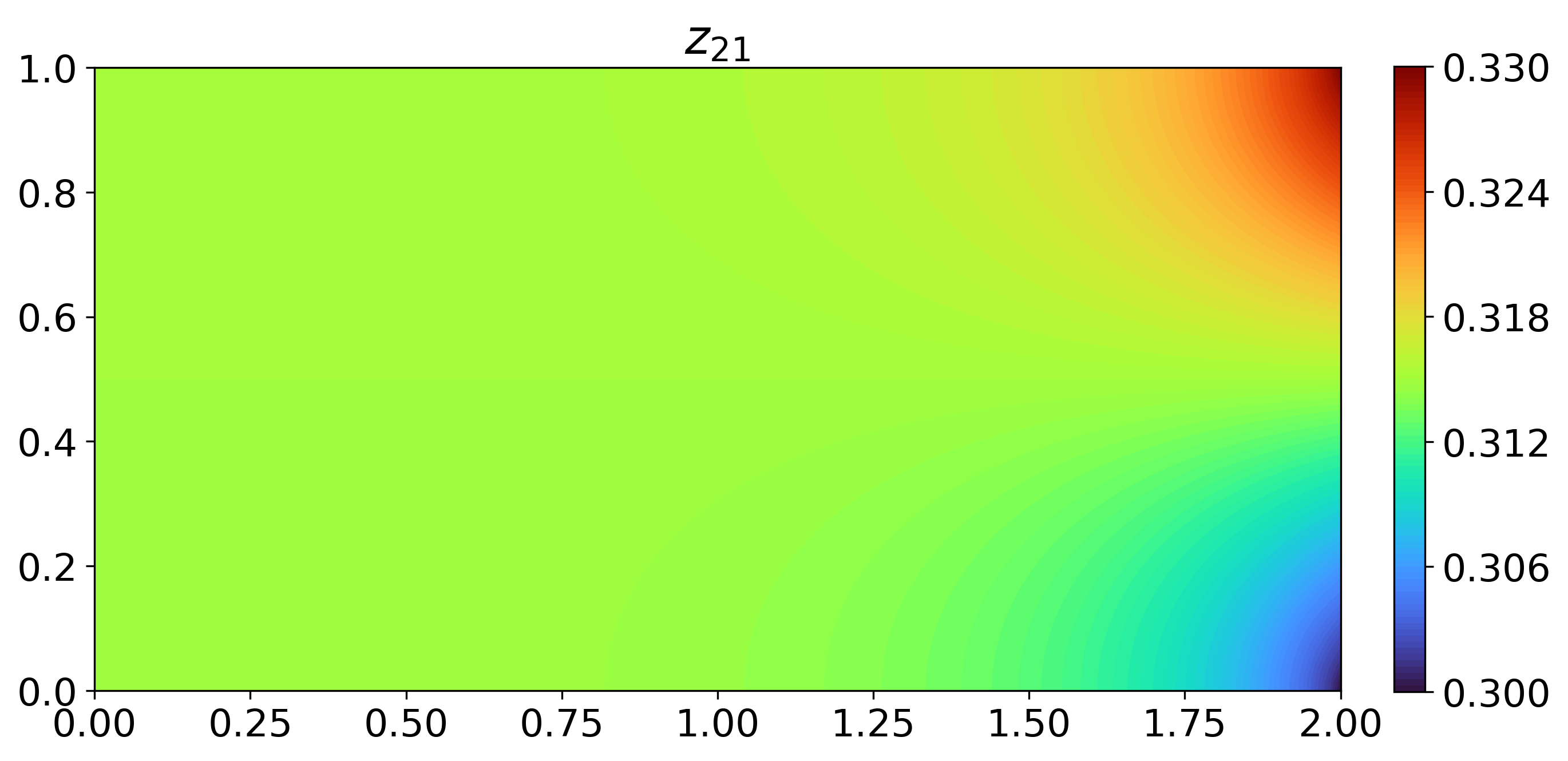}
        \includegraphics[width=0.24\linewidth]{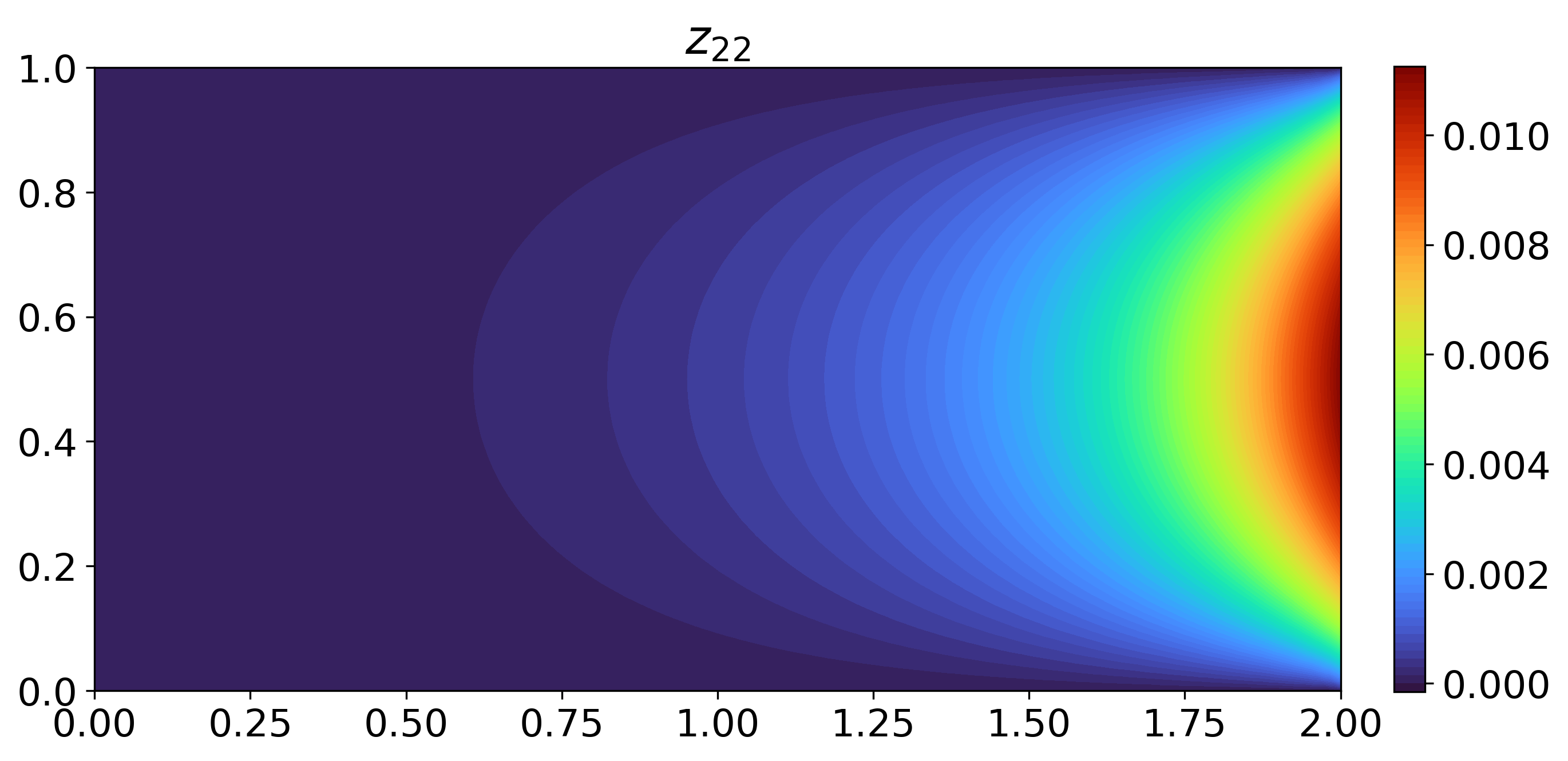} 
    \caption{Auxiliary variables $w$ and $z=(z_1,z_2)$ for heat conduction (top left three plots), $w$ for Darcy flow (top right) with $z=(0,0)$, obtained by solving auxiliary problems \eqref{eq:DiriLiftPoisson} and~\eqref{harm}. Bottom: $\tenbar[2] z = (z_{11}, z_{12}; z_{21}, z_{21})$ by solving \eqref{aux-elast} for linear elasticity with $\tenbar[1] w = (0, 0)$. These variables encode the inhomogeneous Dirichlet and Neumann boundary data.}
    \label{fig:auxiliary_variables}
\end{figure}

\subsection{Conforming finite element approximations}
We verify that the conforming FOSLS FE discretization yields sufficiently small FE loss and exhibits the expected convergence rates as in \cref{thm:FELSPoisson} for the diffusion problem and \cref{thm:FELSElasticity} for the linear elasticity problem. Moreover,   we compare the cost of assembling the loss weights to that of solving the underlying PDEs. 

To assess the FE loss, we solve the FOSLS normal equations \eqref{gal} for the diffusion problem using elements $\Sigma_h \times \U_h = \mathrm{RT}_k^{\circ} \times \mathrm{CG}_{k+1}^{\circ}$ and \eqref{eq:normal_elasticity} for the linear elasticity problem using elements $\Sigma_h \times \U_h = (\mathrm{RT}_k^{\circ})^d \times (\mathrm{CG}_{k+1}^{\circ})^d$, both with different mesh sizes $h$ and orders $k = 0, 1$, for 100 independent parameter samples, and evaluate the resulting FE loss by inserting the FE solutions into the corresponding residual loss. The mean FE loss over the 100 samples, for various mesh sizes and FE spaces, is recorded in~\cref{tab:FE_loss}. 
{For all cases, we also compute the mean FE error (squared error) in the $\HH = H(\rdiv)\times H^1$ norm compared to the reference solutions computed using higher order elements ($k = 2$). The close FE losses and FE errors confirm the error-residual equivalence \eqref{eq:FE-loss-equiv} for the diffusion problem and \eqref{eq:FE-loss-equiv-el} for the linear elasticity problem.} The computational cost of assembling the sparse weight matrix $W_\pp$ (with sparsity pattern shown in Figure~\ref{fig:sparsity_patterns}) scales linearly (especially for the same FE orders) with the number of degrees of freedom (DoFs) $N_h^s$ and is much smaller than the cost of solving the algebraic system of size $N_h^s \times N_h^s$ arising from the FE discretization of the PDE; see the rightmost columns of~\cref{tab:FE_loss}. As expected, refining the mesh and/or increasing the FE order reduces the FE loss and improves the FE approximations, at the expense of larger assembly and solve times. 

\begin{table}[!htb]
\centering
\begin{tabular}{|c|c c r S[table-format=1.2e-2, scientific-notation=true, table-column-width=2cm] S[table-format=1.2e-2, scientific-notation=true, table-column-width=2cm] S[table-format=1.3] S[table-format=2.2]|}
\hline
\textbf{Problem} & \textbf{Mesh} & \textbf{FE space} & \textbf{\# DoFs} & \textbf{FE loss} & \textbf{FE error} & \textbf{$W_\pp$} & \textbf{Solve} \\ \hline
 \multirow{4}{*}{Heat Conduction} 
 & \multirow{2}{*}{64$\times$64}  
   & RT$_0\times$CG$_1$ &  16,641  & 4.86e-03  & 5.83e-03 & 0.015 & 0.10 \\ 
 &  & RT$_1\times$CG$_2$ & 57,857  & 3.52e-04  & 4.55e-04 & 0.051 & 0.38 \\ \cline{2-8}
 & \multirow{2}{*}{128$\times$128} 
   & RT$_0\times$CG$_1$ &  66,049 & 1.55e-03  & 1.93e-03 & 0.04  & 0.45  \\ 
    &  & RT$_1\times$CG$_2$& 230,401 & 1.09e-04  & 1.46e-04 & 0.22  & 1.76\\ \hline
 \multirow{4}{*}{Darcy Flow} 
 & \multirow{2}{*}{128$\times$128}  
   & RT$_0\times$CG$_1$ &  66,049  & 9.54e-01 & 9.55e-01  &  0.27 & 0.52  \\ 
 &  & RT$_1\times$CG$_2$ &  230,401   & 4.00e-03 & 4.01e-03  & 0.49  & 1.88  \\ \cline{2-8}
 & \multirow{2}{*}{256$\times$256} 
   & RT$_0\times$CG$_1$ & 263,169  & 2.40e-01  & 2.40e-01 & 1.03  &  2.38 \\ 
 &  & RT$_1\times$CG$_2$ & 919,553  & 2.53e-04 & 2.53e-04 & 2.10  & 8.36 \\ \hline
\multirow{4}{*}{Elasticity} 
 & \multirow{2}{*}{128$\times$64} 
   & RT$_0^2\times$CG$_1^2$ & 66,306   & 6.51e-03 & 2.77e-02  & 0.20 & 0.63 \\ 
 &  & RT$_1^2\times$CG$_2^2$ & 230,914  & 7.37e-04 & 9.29e-04 & 1.41 & 3.22\\ \cline{2-8}
 & \multirow{2}{*}{256$\times$128} 
   & RT$_0^2\times$CG$_1^2$ & 263,682  & 2.46e-03 & 5.96e-03 & 0.77 & 3.00 \\ 
 &  & RT$_1^2\times$CG$_2^2$ & 920,578  & 3.00e-04 & 3.32e-04 &  5.59 & 14.41 \\ \hline
\end{tabular}
\caption{Mean FE loss and mean FE error (squared) in $\HH$-norm over 100 parameter samples, and wall-clock time (in seconds) to assemble the weight matrix $W_\pp$ in \eqref{eq:FE-loss-Poisson} and to solve the PDE using a direct (LU) solver for the diffusion \eqref{eq:lossPoisson} and linear elasticity \eqref{eq:lossElasticity} problems, with different mesh sizes, FE spaces, and the corresponding number of degrees of freedom (DoFs) $N_h^s$. 
} 
\label{tab:FE_loss}
\end{table}

\begin{figure}[!htb]
    \vspace{-1em}
    \centering
    \includegraphics[width=0.32\linewidth]{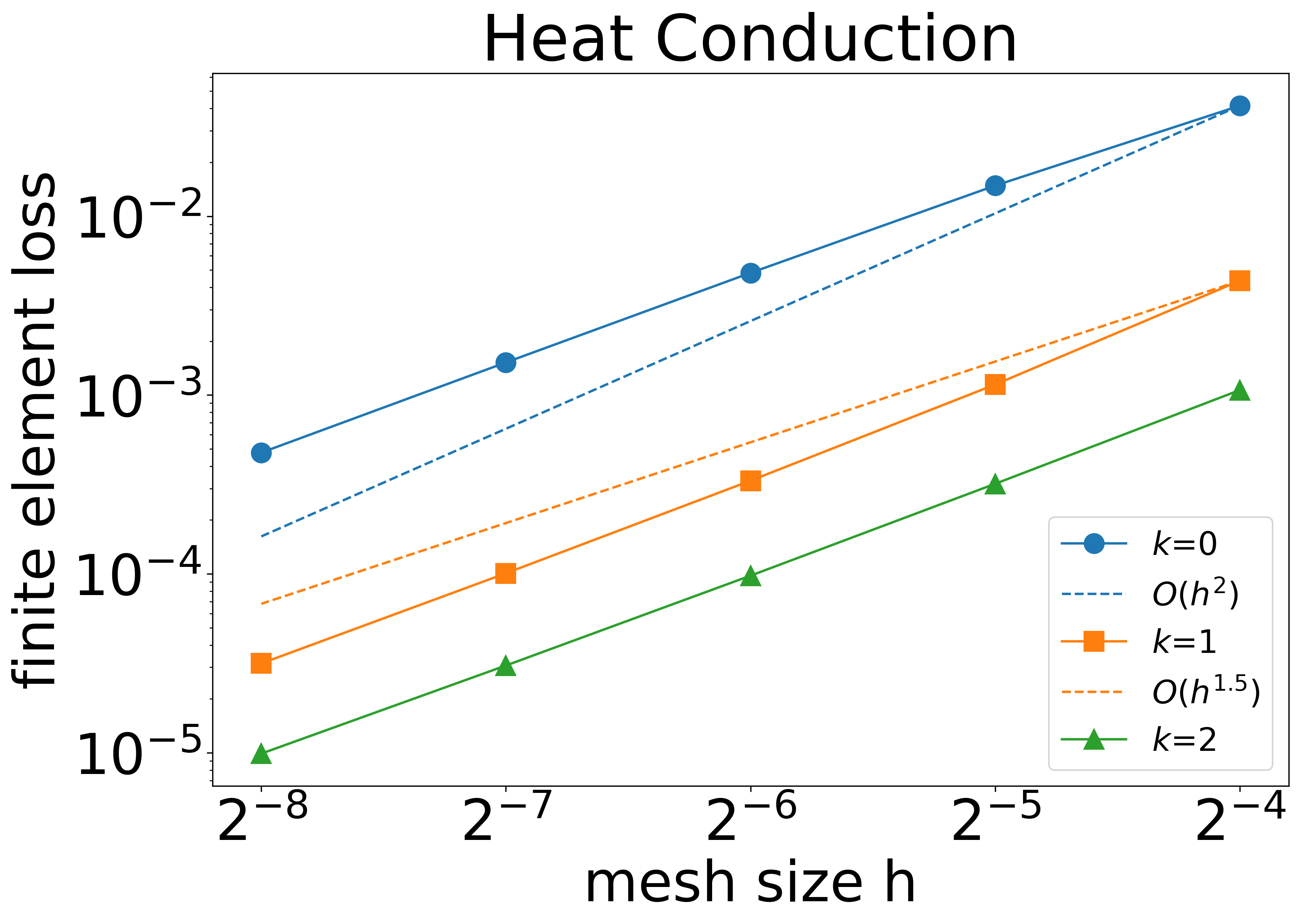}
    \includegraphics[width=0.32\linewidth]{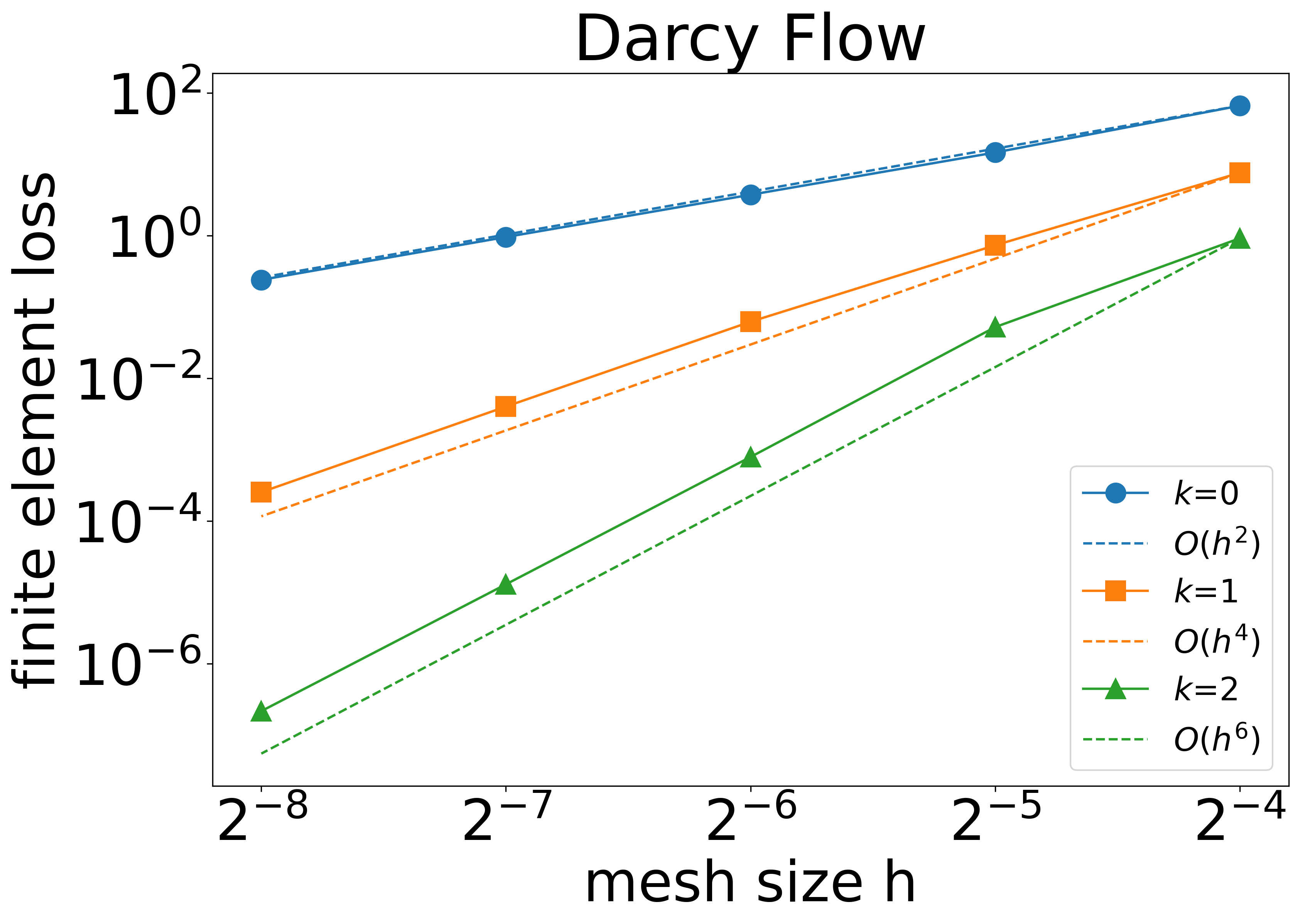}
    \includegraphics[width=0.32\linewidth]{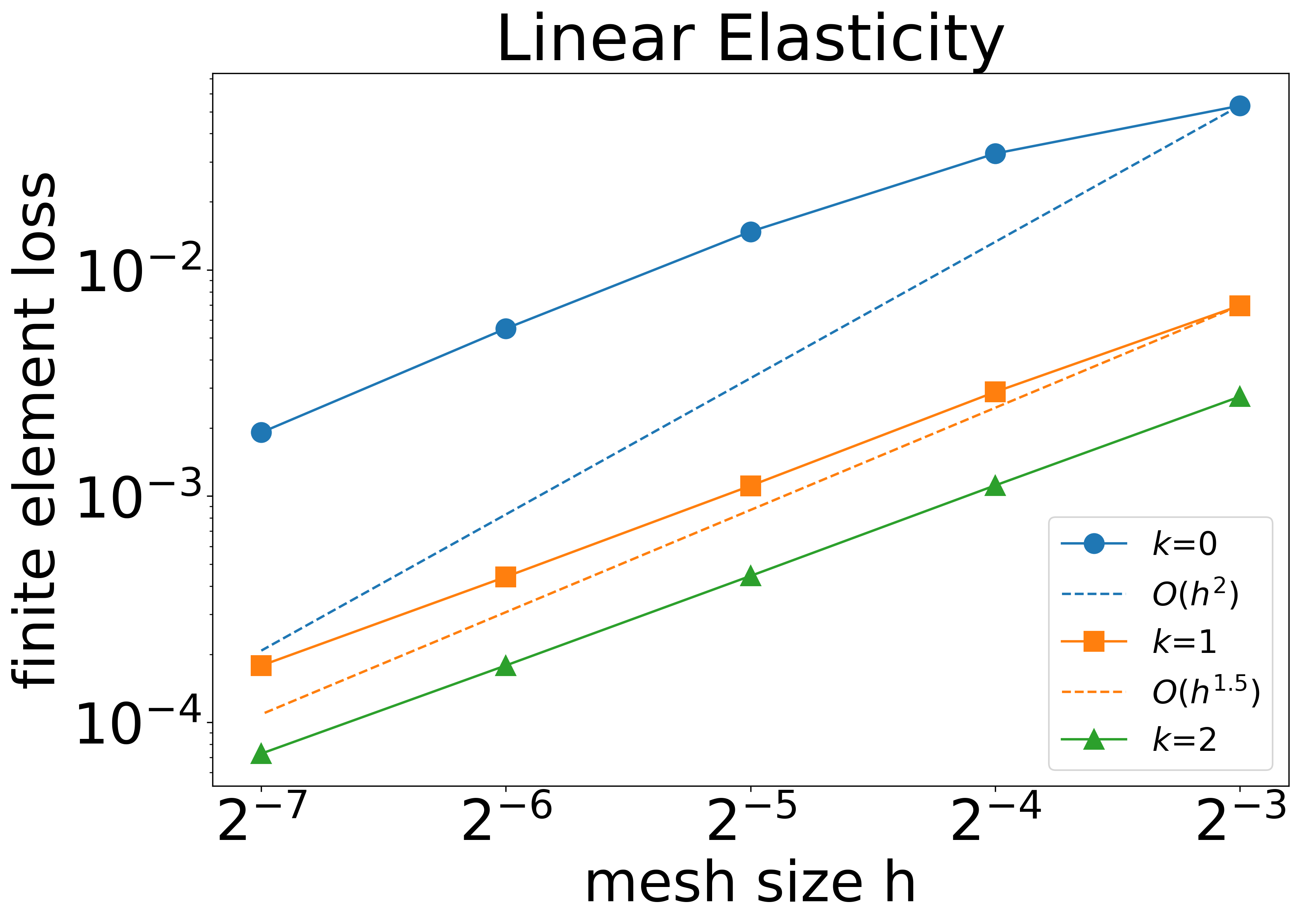}
    \caption{Convergence of the FE loss with respect to mesh size ($h$) and FE order ($k=0,1,2$) for a representative parameter sample. Solid lines show the measured losses and dashed lines the reference rates ($O(h^{2(k+1)})$); Darcy flow follows the predicted asymptotic behavior, while convergence in heat conduction and linear elasticity is limited by reduced solution regularity due to discontinuous coefficients and boundary-induced corner singularities, confirming the analysis in Theorem \ref{thm:FELSPoisson} and \ref{thm:FELSElasticity}.
    }
    \label{fig:finite_element_loss_wrt_mesh_size_and_order}
\end{figure}

\cref{fig:finite_element_loss_wrt_mesh_size_and_order} shows, for a representative random parameter sample,
the convergence of FE loss with respect to the mesh size $h$ and the FE order $k$ across the three problem setups. In all cases, the loss decreases monotonically as the mesh is refined and as the order increases. The measured losses (solid lines) follow closely the expected asymptotic convergence rates $O(h^{2(k+1)})$ (dashed lines) for the Darcy flow, confirming that the FE discretization exhibits the theoretically predicted convergence behavior with respect to both $h$ and $k$ as in \eqref{eq:FELS-bound-Poisson} of Theorem \ref{thm:FELSPoisson}. The observed convergence rates are limited by the lower regularity of the solutions due to the discontinuous conductivity field in heat conduction, and the mixed boundary conditions in the linear elasticity problem that lead to corner singularities. 

\subsection{Reduced basis approximations}
We now investigate the equivalence of the RB loss and errors in \cref{prop:reduced_loss_error_equivalence} and their dependence on the number of RB functions, and how this choice affects computational cost. We first construct the RB spaces as in Section \ref{ssec:errors} from $N_{\text{POD}}=500$ and $1000$ independent parameter samples and the corresponding high-fidelity FE solutions (using RT$_1\times$CG$_2$ elements on a $128\times128$ mesh) for the heat conduction and Darcy flow problem, respectively. For the linear elasticity problem, we use $N_{\text{POD}}=1000$ independent parameter samples and the corresponding high-fidelity FE solutions (using RT$_1^2\times$CG$_2^2$ elements on a $128\times64$ mesh). Then for each of $500$ test random parameter samples, we compute the FE weight matrix $W_\pp$ in the FE loss \eqref{eq:FE-loss-Poisson}, the corresponding RB weight matrix $W_\pp^r$ and vector $\bsalpha_\pp^r$ in the RB loss \eqref{eq:lossPoissonReduced} by projection, and solve the reduced normal equation \eqref{eq:RB_normal} for the coefficient vector $\bss_r(\pp)$ of the RB solution $s_r(\pp)$. We also compute the FE solutions $s_h(\pp)$ and $\bar{s}_h(\pp)$, using the same mesh and elements of RT$_1\times$CG$_2$ and RT$_3\times$CG$_4$, respectively, where we take the latter as the ``ground truth'' solution. Finally, we compute the RB loss \eqref{eq:lossPoissonReduced} of the RB solution and the squared error of the RB solution in $\HH = H(\rdiv)\times H^1$-norm with respect to the corresponding FE solutions. 

By the decomposition of the RB loss \eqref{eq:RB-fiber-equivalence-t_r} in~\cref{prop:reduced_loss_error_equivalence}, we expect the gap between the RB loss and the FE loss to be equivalent (up to problem-dependent constants) to the squared error of the RB solution relative to the FE solution, i.e.,
$\cL(s_r(\pp);\pp) - \cL(s_h(\pp); \pp) \;\eqsim\;
\|s_h(\pp) - s_r(\pp)\|_{\HH}^2$. This relationship is demonstrated in \cref{fig:analysis_loss_diff_and_mse_error}. {We observe this equivalence by comparing an empirical mean squared error approximation to $
\EE_{\pp\sim\mu}\big[||s_r(\pp) - s_h(\pp)||^2_{\HH}\big]$ against the empirical mean loss difference, approximating $\EE_{\pp\sim\mu}\big[\cL(s_r(\pp);\pp) -
\cL(s_h(\pp); \pp)\big]$, for an increasing number of basis functions across the three problems. Both quantities are evaluated via sample average approximation using a shared set of 500 test samples, which was found to be sufficient for low-variance estimation of the mean.}

\begin{figure}[!htb]
    \centering
    \includegraphics[width=0.32\linewidth]{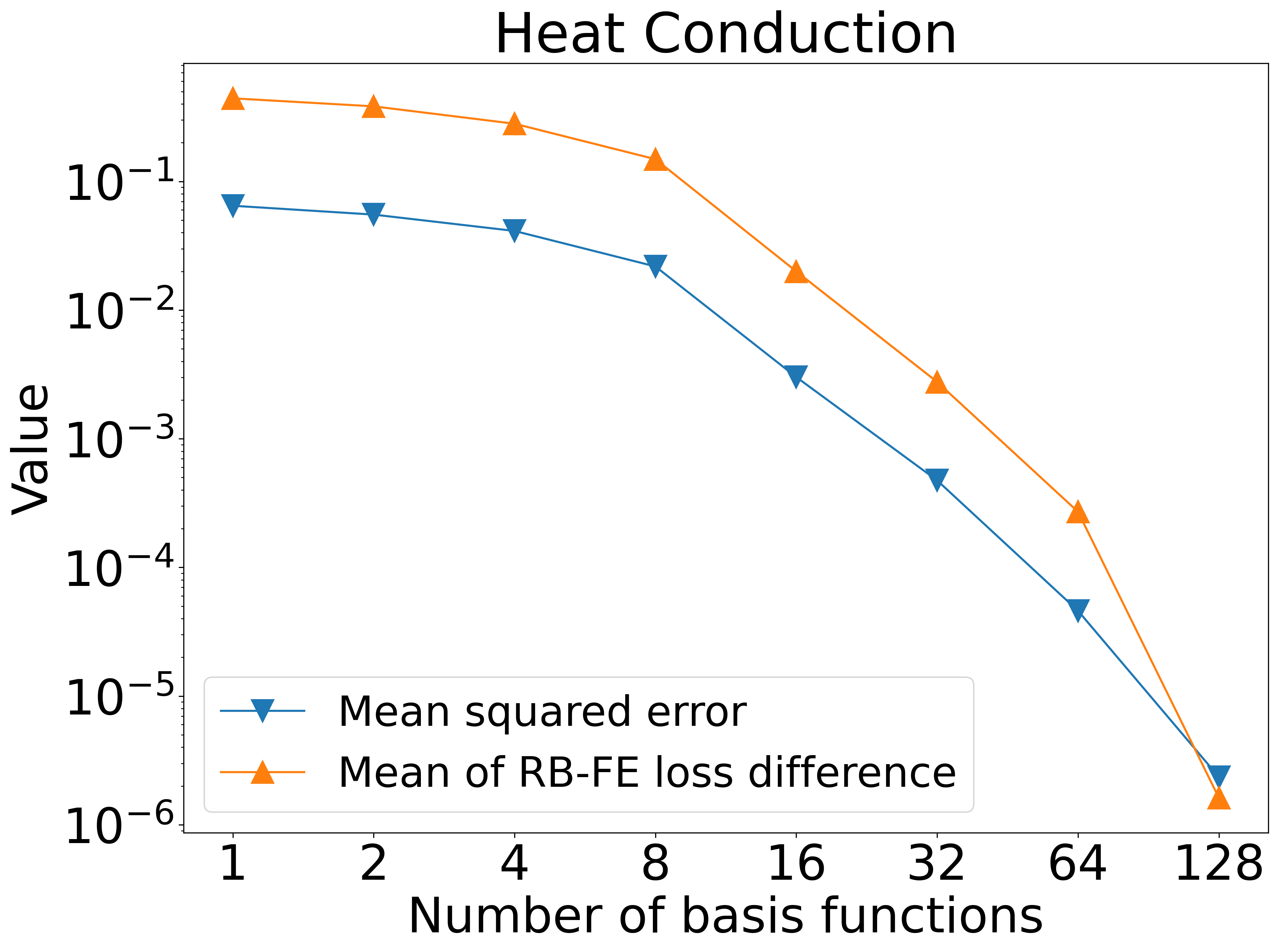}
    \includegraphics[width=0.32\linewidth]{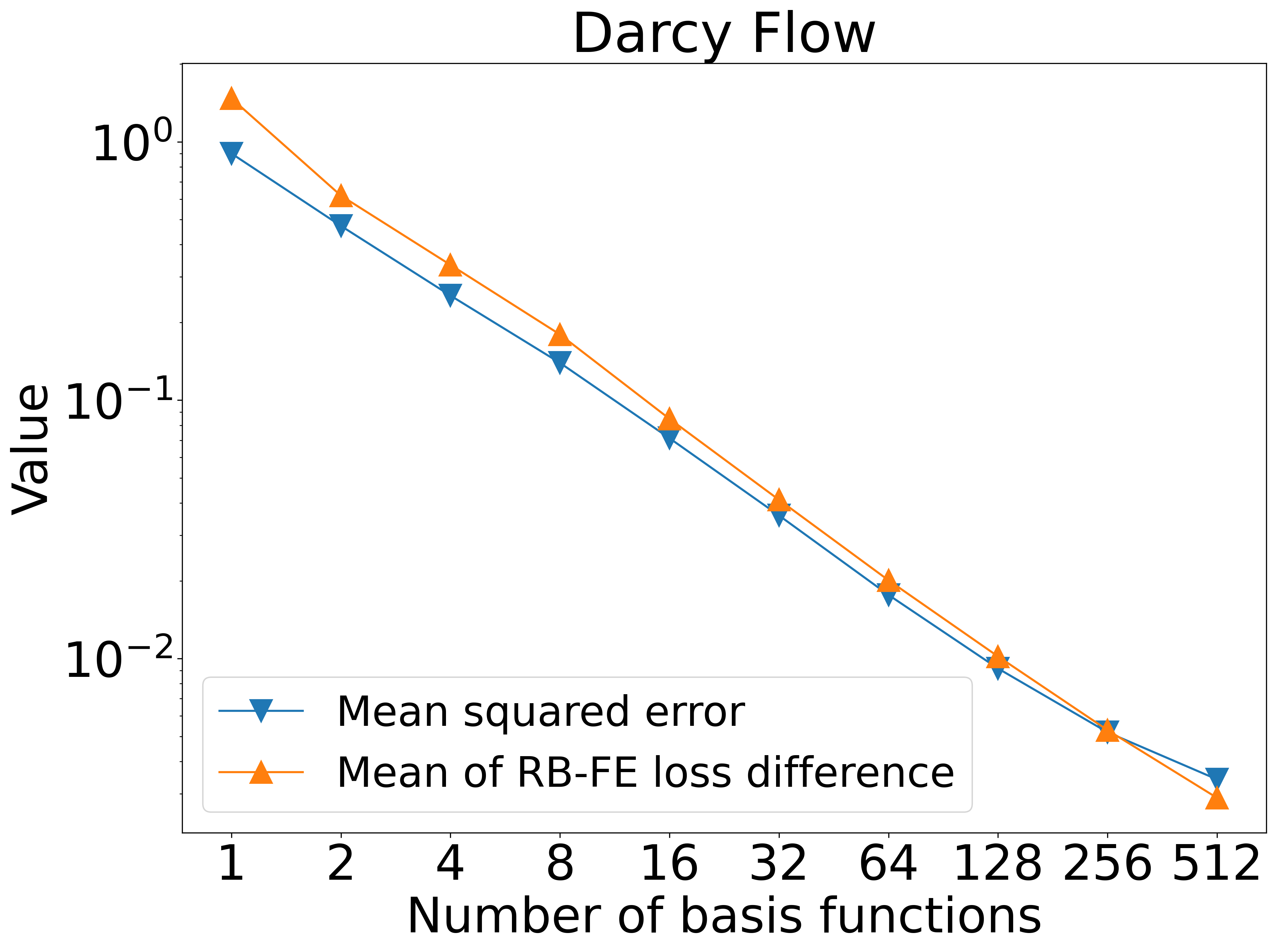}
    \includegraphics[width=0.32\linewidth]{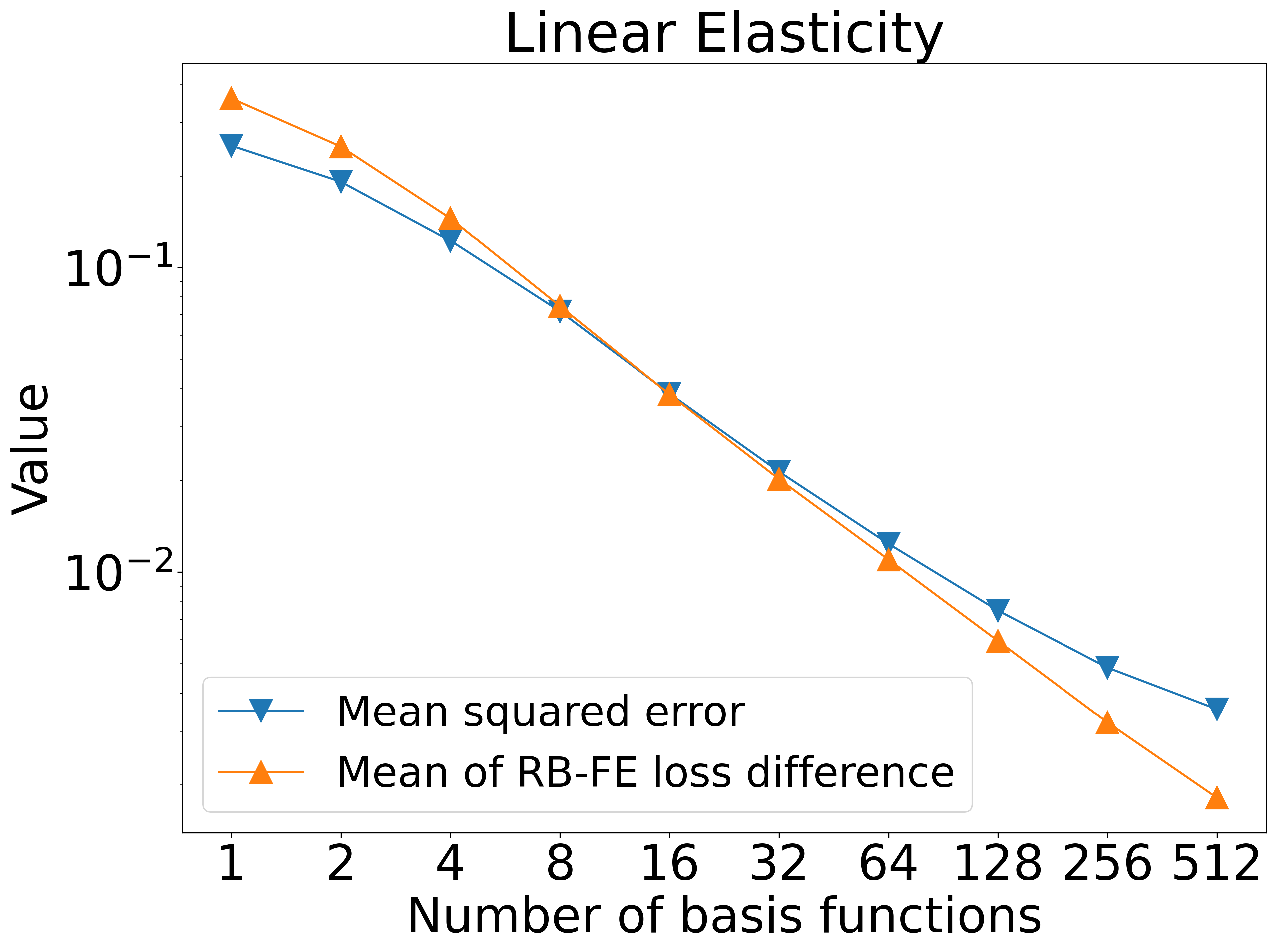}
    \caption{{Comparison between the empirical mean squared error $\EE_{\pp\sim\mu}\big[||s_r(\pp) - s_h(\pp)||^2_{\HH}\big]$ and the empirical mean loss difference $\EE_{\pp\sim\mu}\big[\cL(s_r(\pp);\pp) - \cL(s_h(\pp); \pp)\big]$. Both quantities are estimated over $500$ random samples (using RT$_1\times$CG$_2$ elements for $s_h$), confirming the RB loss decomposition \eqref{eq:RB-fiber-equivalence-t_r} in~\cref{prop:reduced_loss_error_equivalence}.}}
    \label{fig:analysis_loss_diff_and_mse_error}
\end{figure}

The comparison of the mean squared error of the RB solution $s_r$ (compared to the FE solution $s_h$), the $X_h$-projection error of the FE solution $s_h \in \HH_h$ onto the RB space $\HH_r$, and the error estimate by the trailing eigenvalues in \eqref{eq:low_rank_approximation_via_trailing_eigenvalues}, is shown in \cref{fig:error_compared_to_RT1XCG2_solution}, confirming the quasi-optimality of the RB solution in \eqref{eq:RB-best-approx-t_r} in \cref{prop:reduced_loss_error_equivalence} and the tight error estimate. The faster decay of the error estimate by the trailing eigenvalues at a large number of RB basis functions is due to the limited number $N_{\text{POD}}$ of samples in computing the eigenvalues (which leads to inaccurate estimation of small eigenvalues). 

\begin{figure}[!htb]
    \centering
    \includegraphics[width=0.32\linewidth]{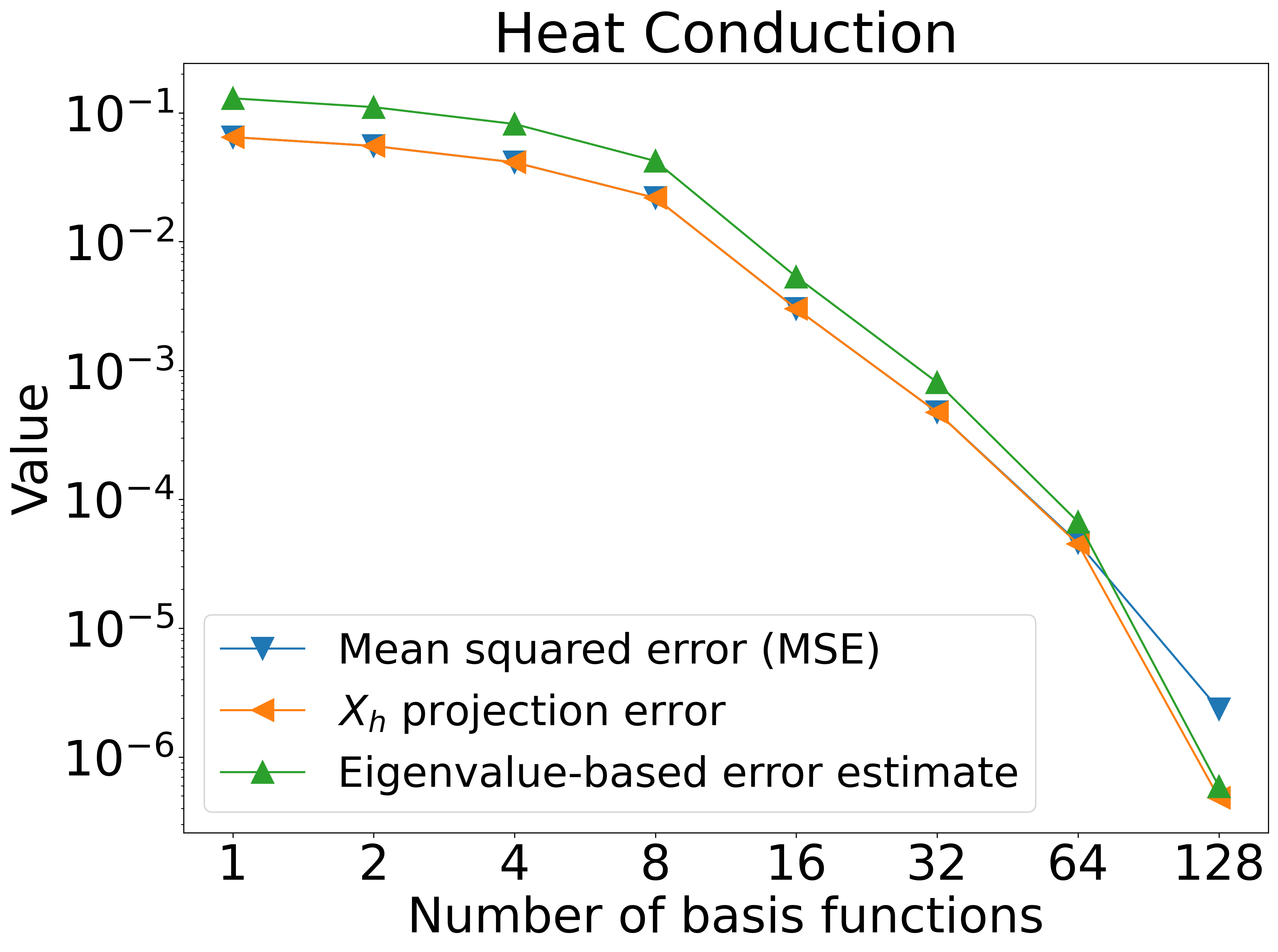}
    \includegraphics[width=0.32\linewidth]{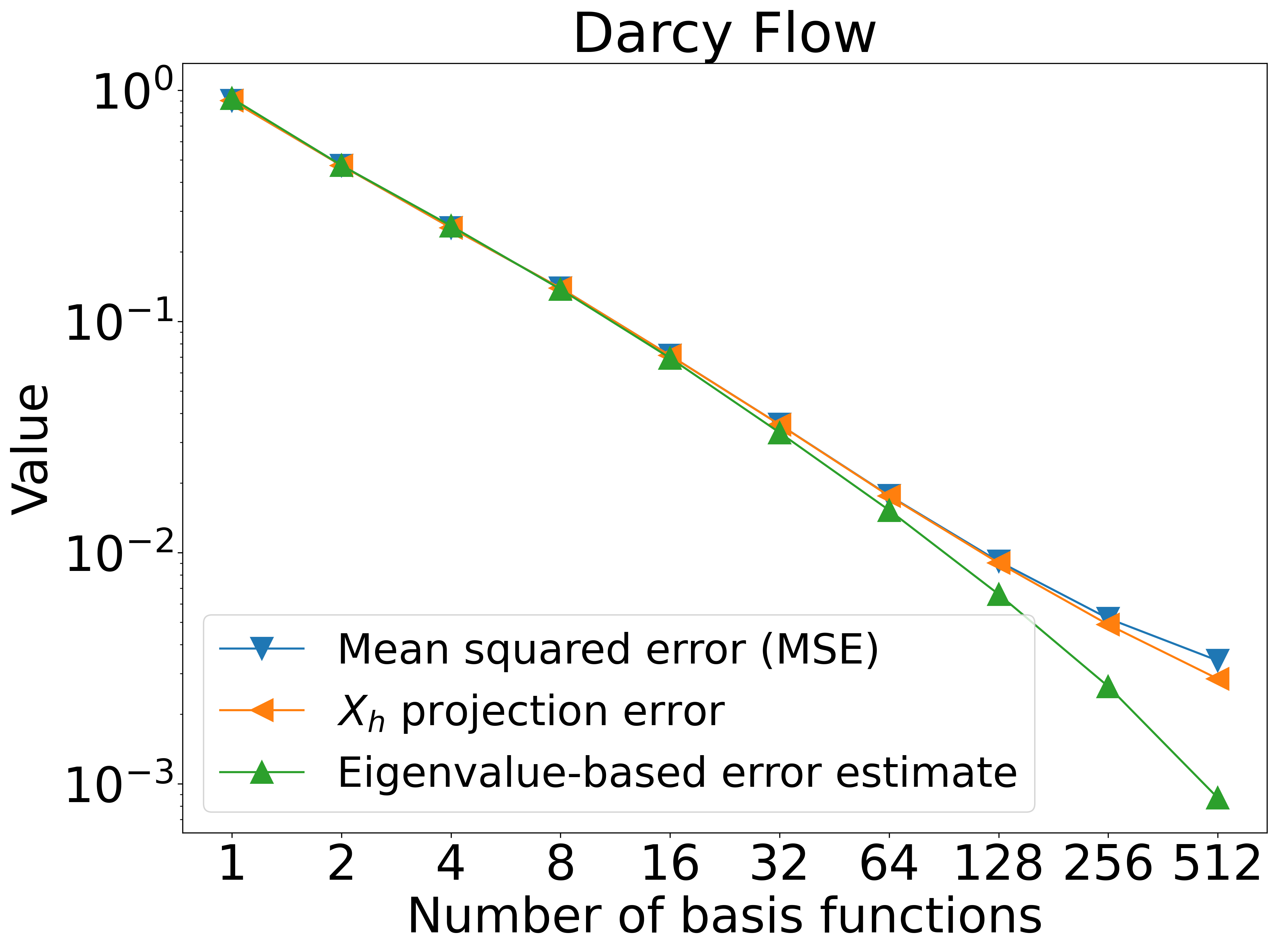}
    \includegraphics[width=0.32\linewidth]{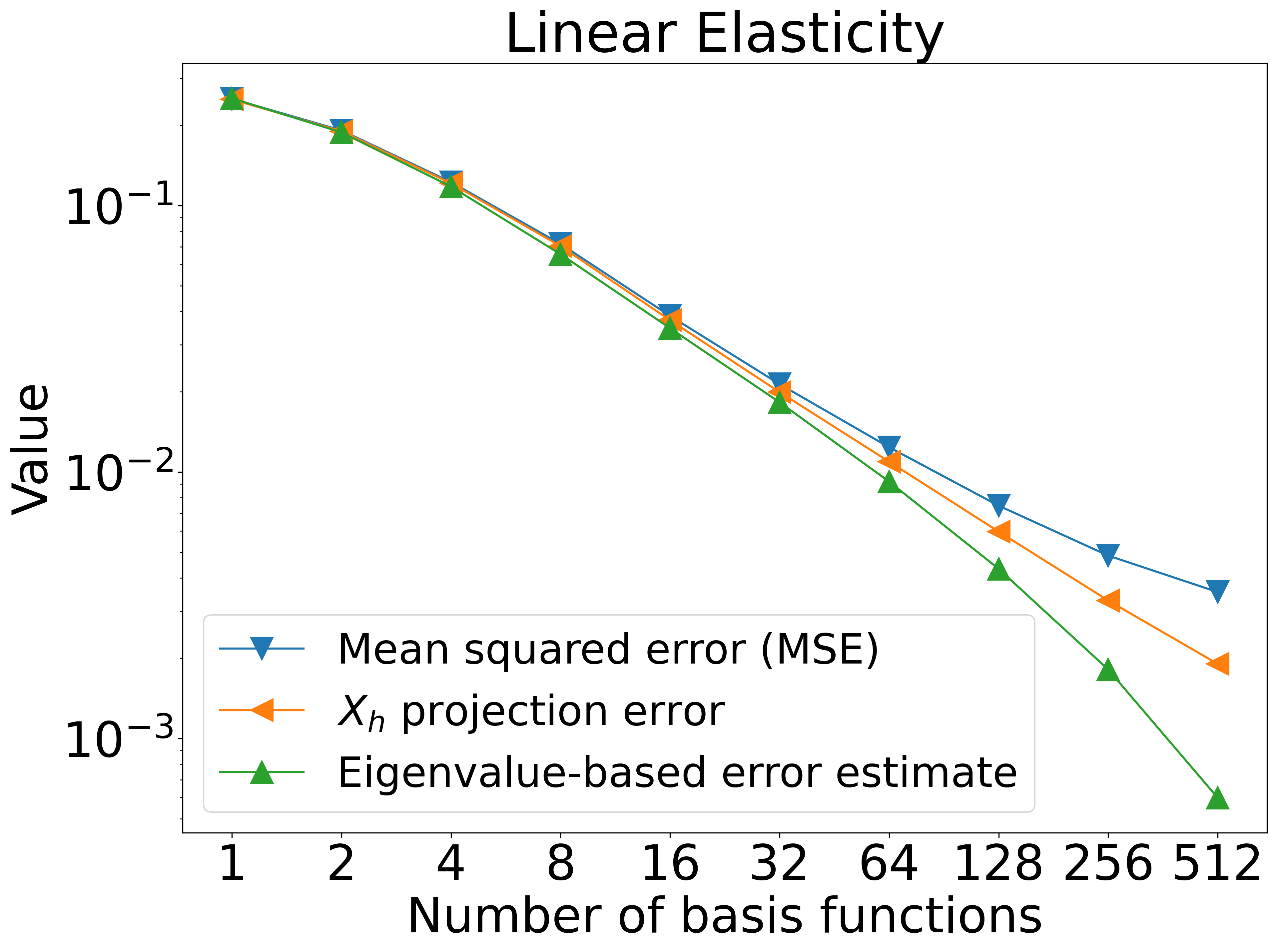}
    \caption{{Comparison between the empirical mean squared error $\EE_{\pp\sim\mu}\big[||s_r(\pp) - s_h(\pp)||^2_{\HH}\big]$ of the RB solution $s_r$ (approximated over 500 samples with respect to the FE solution $s_h$ using RT$_1\times$CG$_2$ elements), the square of the $X_h$-projection error of $s_h$ onto the RB space $\HH_r$, and the eigenvalue-based error estimate \eqref{eq:low_rank_approximation_via_trailing_eigenvalues}. These results illustrate the quasi-optimality \eqref{eq:RB-best-approx-t_r} of the RB approximation in Theorem~\ref{prop:reduced_loss_error_equivalence} and the tightness of the error estimate.}}
    \label{fig:error_compared_to_RT1XCG2_solution}
\end{figure}

Moreover, the comparison of the mean RB loss and the mean squared error of the RB solution $s_r$ (compared to the  ``ground truth'' solution $\bar{s}_h$) with respect to the number of RB functions is shown in \cref{fig:error_compared_to_RT3XCG4_solution}, demonstrating the residual-error equivalence in \eqref{quasiopt} in \cref{prop:reduced_loss_error_equivalence}. The slower decay of the mean squared error of the RB solution for a large number of RB basis functions is due to the FE discretization errors (using RT$_1\times$CG$_2$ elements) compared to the ``ground truth'' (FE discretization using RT$_3\times$CG$_4$ elements). We observe a noticeably faster decay of both loss and error in heat conduction compared to Darcy flow and the linear elasticity problem, reflecting the smaller Kolmogorov $n$-widths of the manifold of the parameter-to-solution map in heat conduction.

\begin{figure}[!htb]
    \centering
    \includegraphics[width=0.32\linewidth]{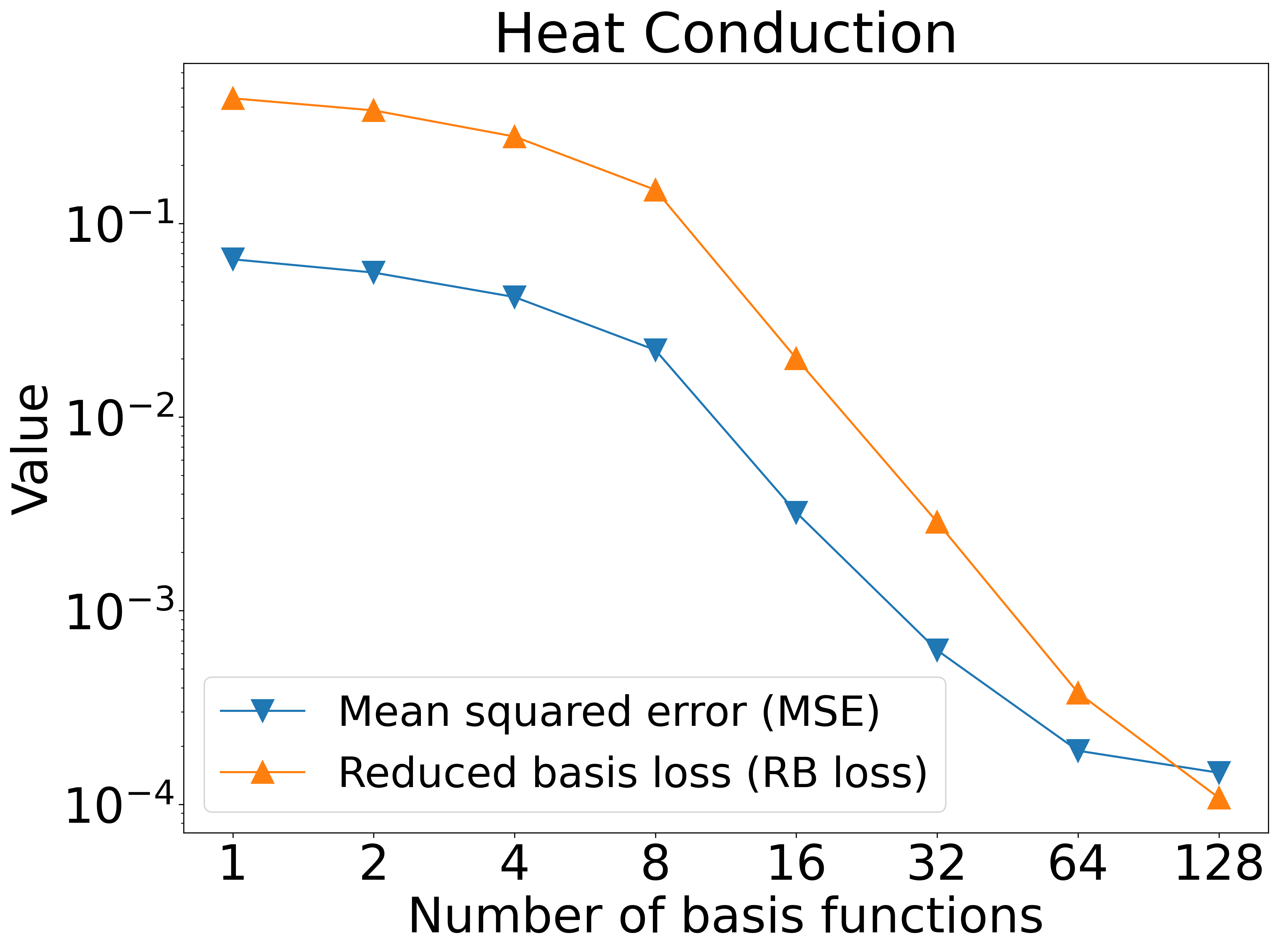}
    \includegraphics[width=0.32\linewidth]{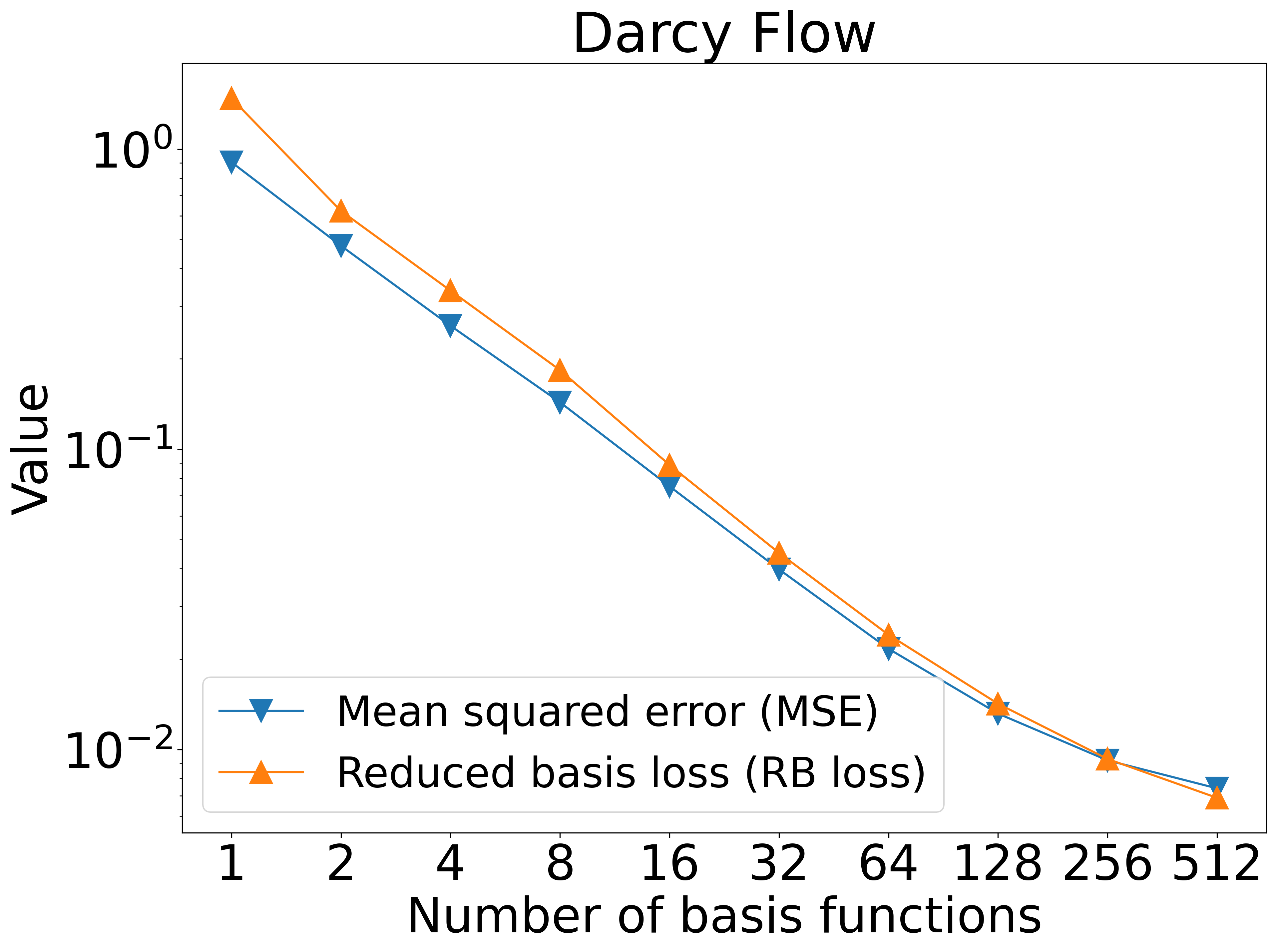}
    \includegraphics[width=0.32\linewidth]{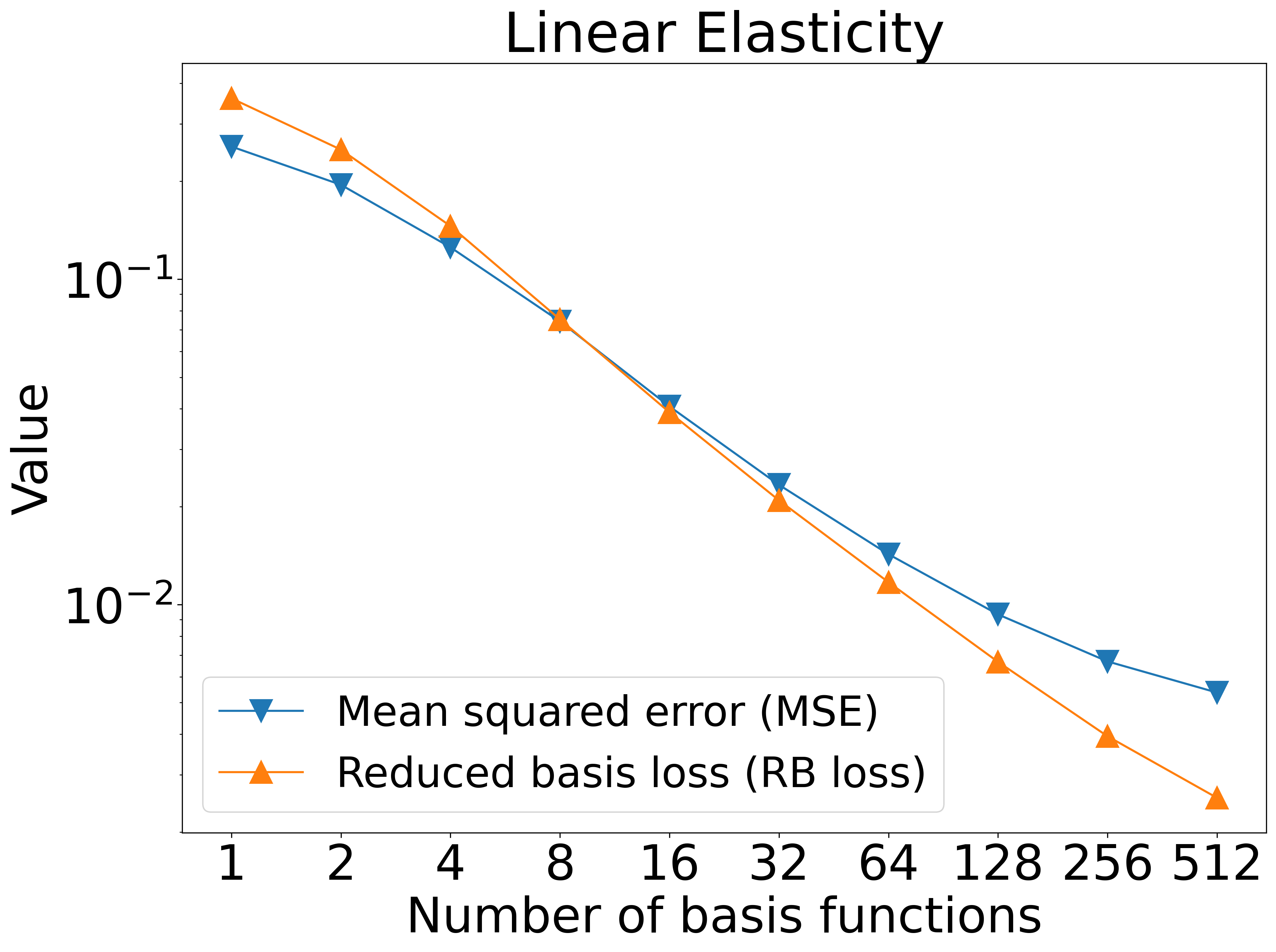}
    \caption{{Comparison between the empirical mean squared error $\EE_{\pp\sim\mu}\big[||s_r(\pp) - \bar{s}_h(\pp)||^2_{\HH}\big]$ of the RB solution $s_r$ (relative to the ``ground truth'' FE solution $\bar{s}_h$ with RT$_3\times$CG$_4$ elements) and the empirical mean RB loss $\EE_{\pp\sim\mu}\big[ \cL(s_r(\pp);\pp) \big]$. Both averages are computed using $500$ random samples, confirming the residual-error equivalence \eqref{quasiopt} in \cref{prop:reduced_loss_error_equivalence}.}}
    \label{fig:error_compared_to_RT3XCG4_solution}
\end{figure}

\subsection{Reduced basis neural operator}

As defined in hypothesis class \eqref{hypo}, the reduced basis neural operator (RBNO) consists of two parts. In the first part, a neural network maps the parameter $\pp$ (represented as a grid-based image of size $\mathbb{R}^{H\times W}$, e.g., $H=W=129$ for the diffusion problem) to the RB coefficients $\bss_r(\pp;\theta) \in \mathbb{R}^r$ of the approximate solution; this stage is computed in single precision (float32). In the second part, the RBNO solution is reconstructed as a linear combination of the precomputed RB basis functions with the predicted RB coefficients, i.e., $s_r(\pp;\theta) = \Phi_r \bss_r(\pp;\theta)$, and this reconstruction is evaluated in double precision (float64). Note that the FE and RB weight matrices are also computed in double precision to ensure accurate loss evaluation. 

We approximate the parameter-to-RB coefficient map by a deep neural network 
\begin{align*}
    \bss_r(\cdot;\theta) = L_2 \circ \sigma \circ L_1\circ \mathrm{Flatten}\circ C_L\circ\cdots\circ C_1: \mathbb{R}^{H\times W} \to \mathbb{R}^r, 
 \end{align*}
where $L_1, L_2$ are fully-connected linear layers, and each $C_l$ is a convolutional block followed by a nonlinear activation function $\sigma$, i.e., $C_l = \sigma \circ \mathrm{Conv}_{l}$. We use the LeakyReLU activation function $\sigma(x)=\max(0, x) + 0.01\min(0, x)$, and initialize all network parameters $\theta$ using the Xavier uniform method~\cite{smith2024uncertainty}. We train the neural network by minimizing the empirical RB loss \eqref{eq:empirical-loss} using the SOAP optimizer~\cite{vyas2024soap}, which we find to outperform the commonly used AdamW optimizer for faster convergence and higher accuracy. More details on the architecture and training can be found in \ref{sec:training-details}. 

The decay of the empirical loss on the training set (with 1000 and 3000 samples, respectively) and validation set (with 500 samples) is shown in \cref{fig:loss_history} for the three problems. During training, we monitor the validation loss and retain the network weights that achieve a minimum validation loss (early stopping), rather than those minimizing the training loss, to prevent overfitting and improve generalization to unseen parameter samples. From 1000 to 3000 training samples, we can observe a larger training loss at the end of training, a smaller validation loss at the minimum values, and a closer gap between them.

\begin{figure}[!ht]
    \centering
    \includegraphics[width=0.32\linewidth]{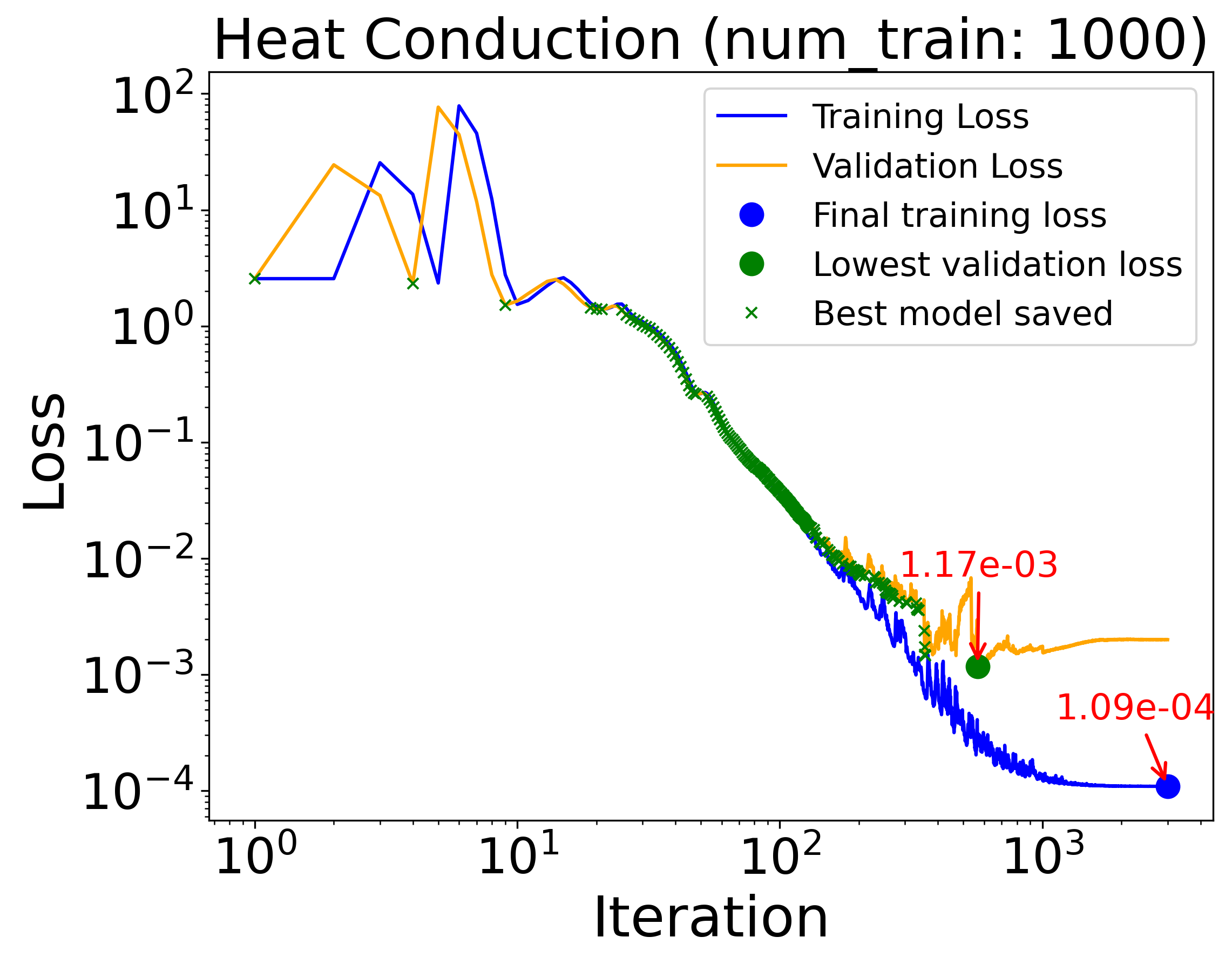}
    \includegraphics[width=0.32\linewidth]{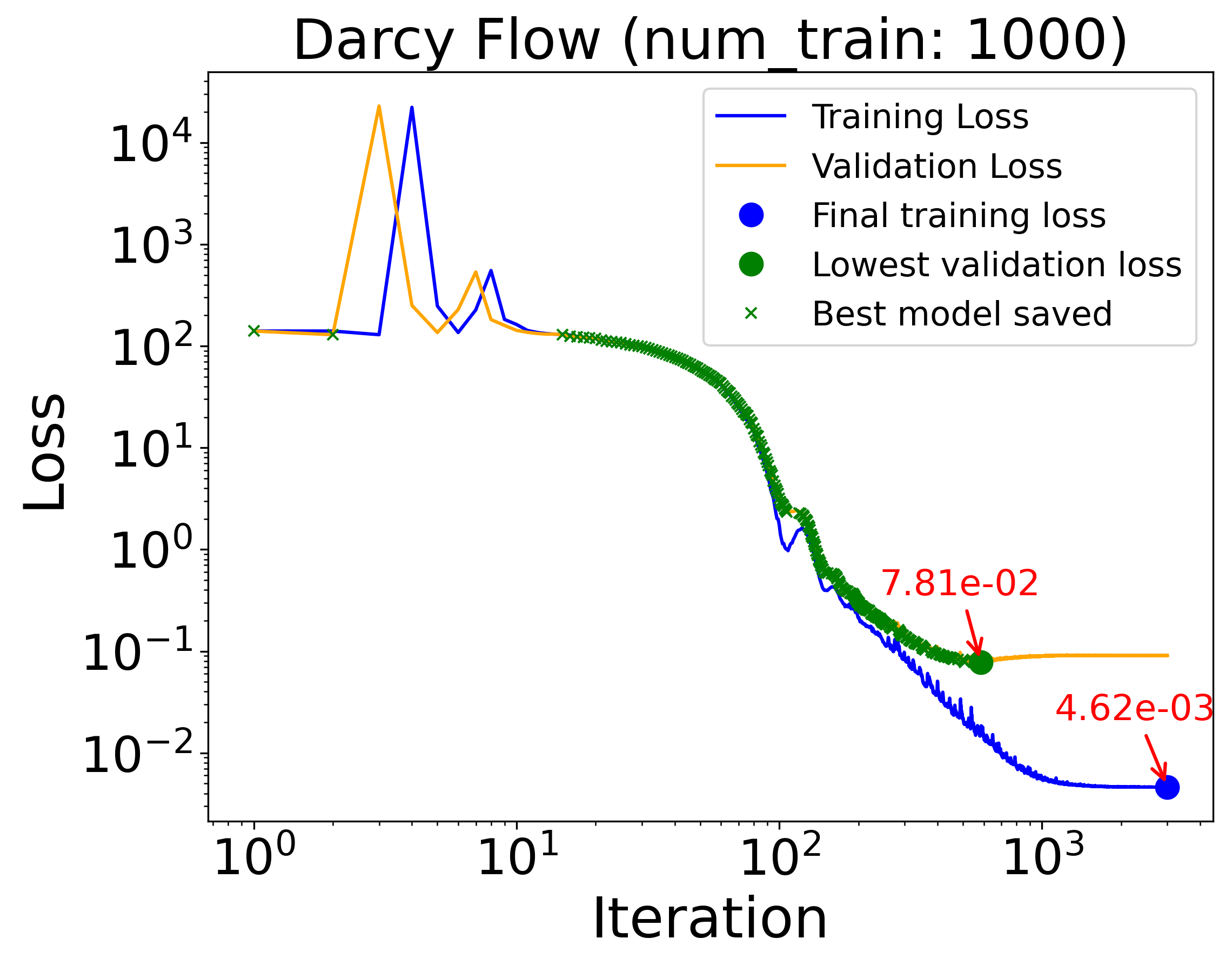}
    \includegraphics[width=0.32\linewidth]{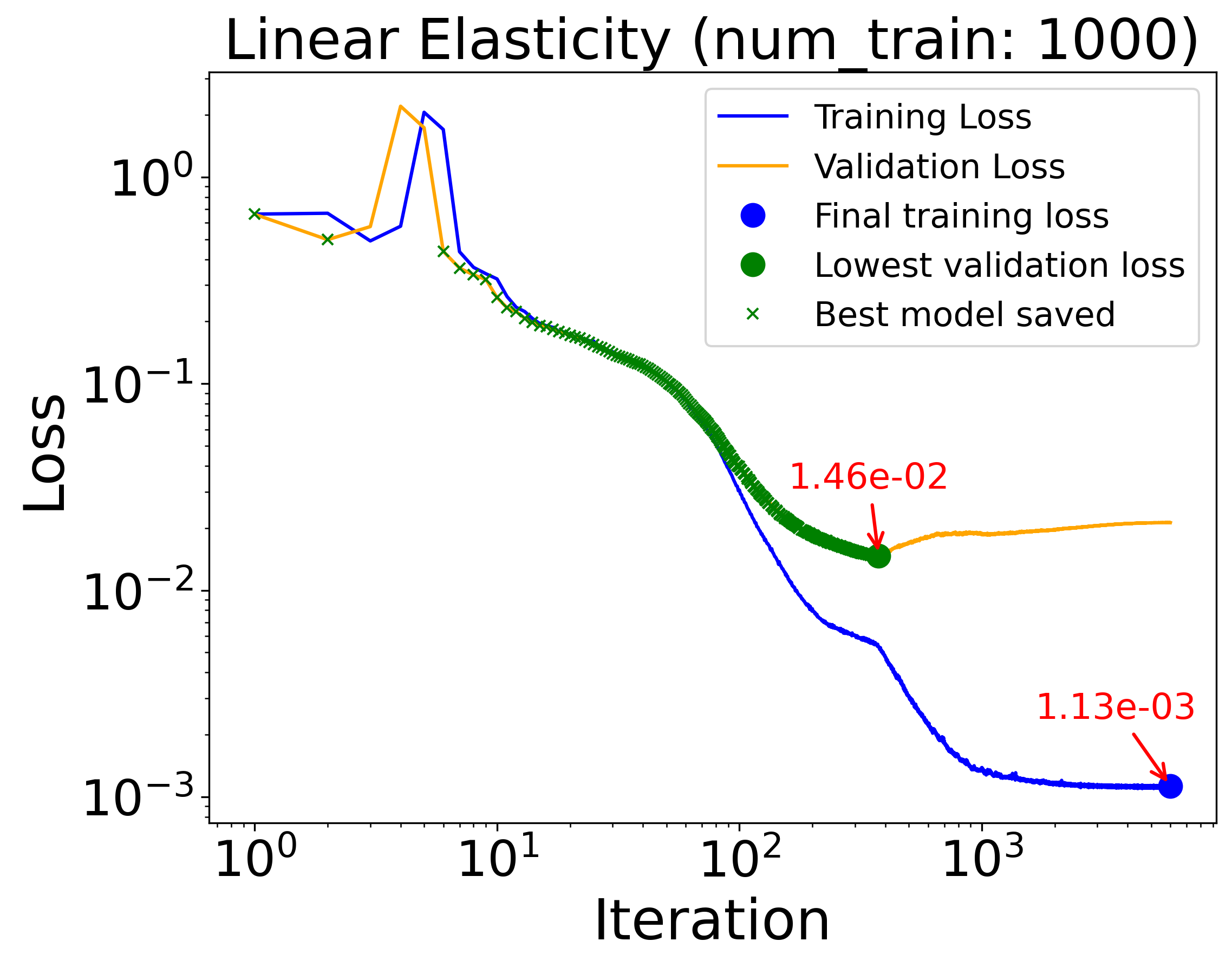}

    \includegraphics[width=0.32\linewidth]{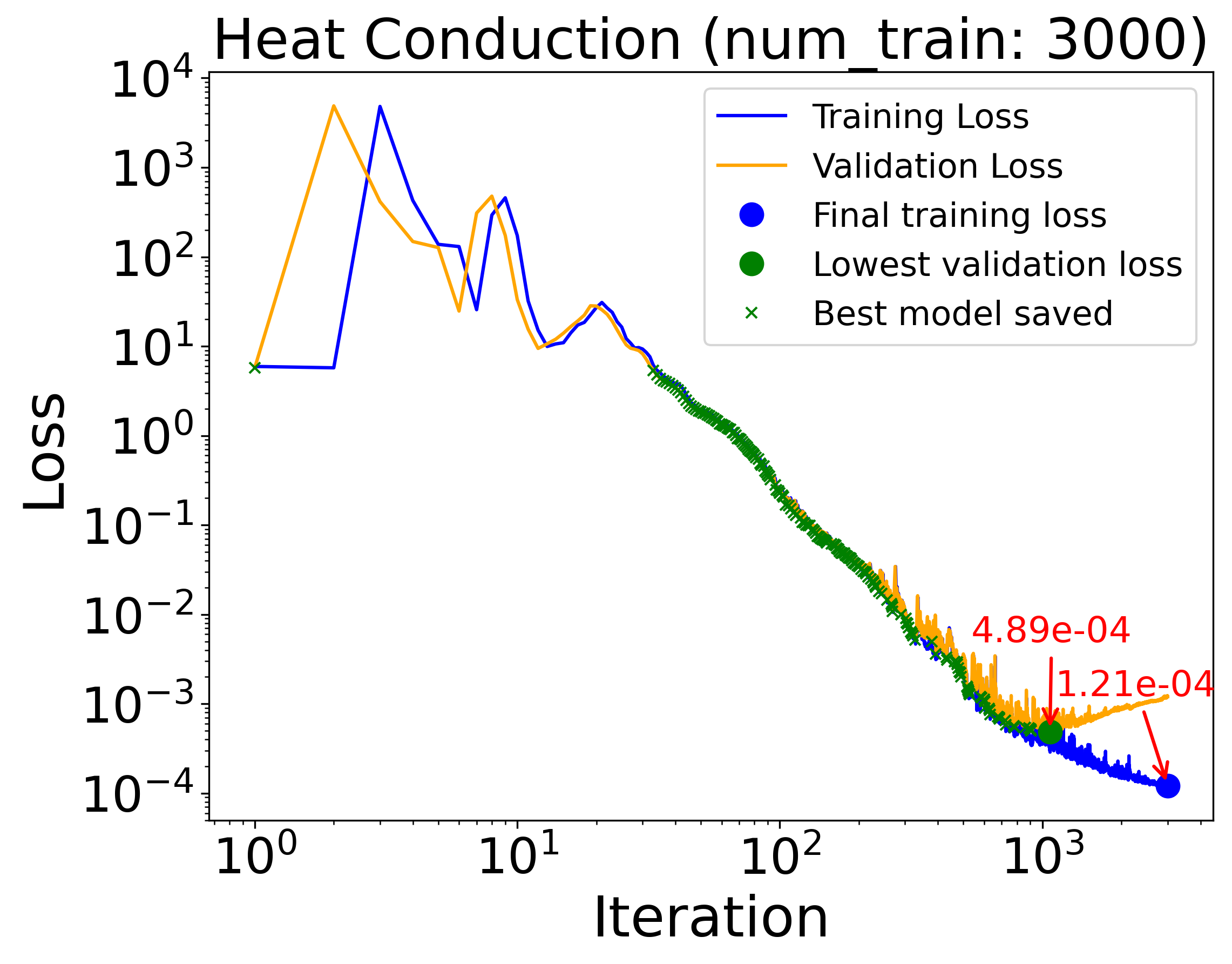}
    \includegraphics[width=0.32\linewidth]{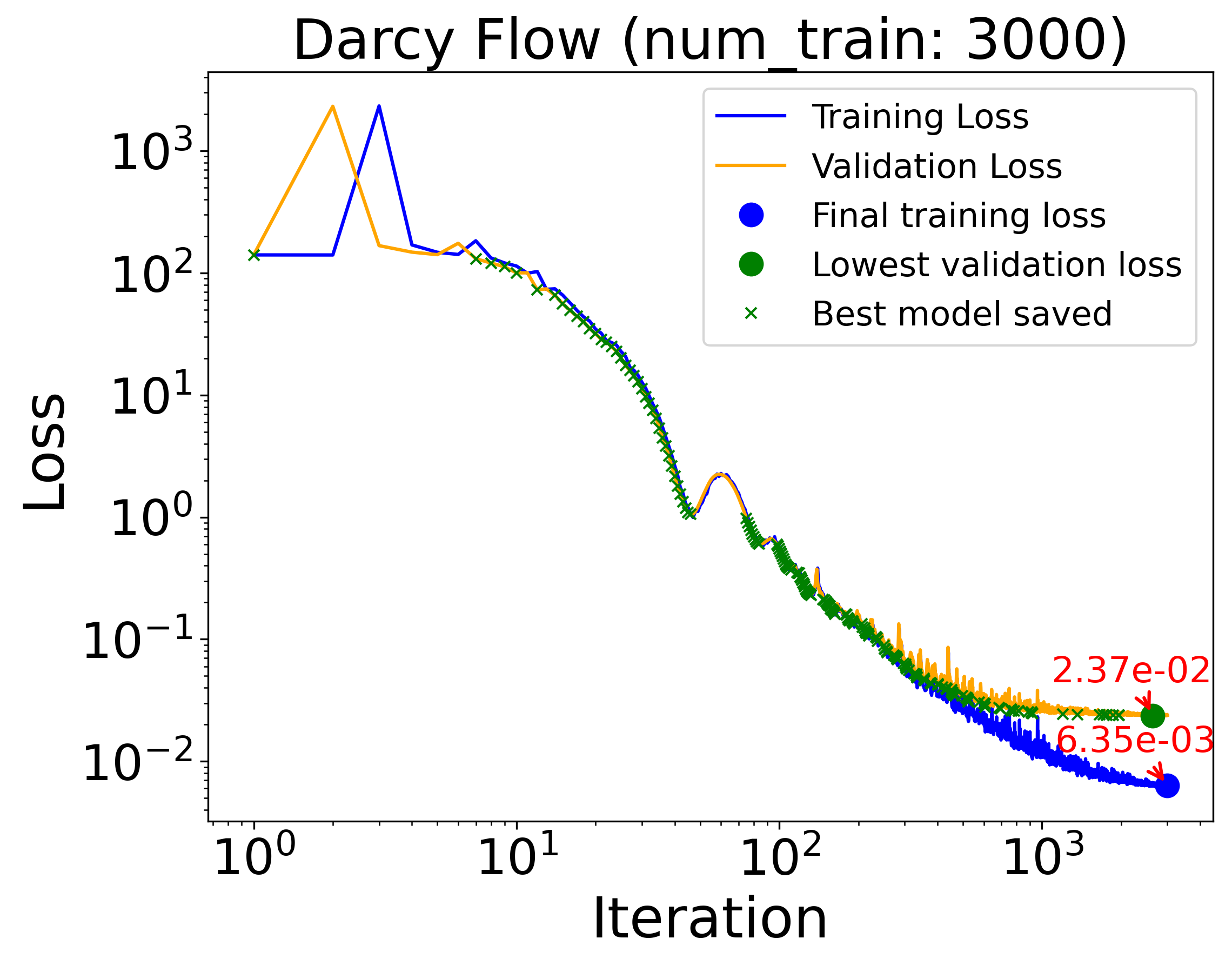}
    \includegraphics[width=0.32\linewidth]{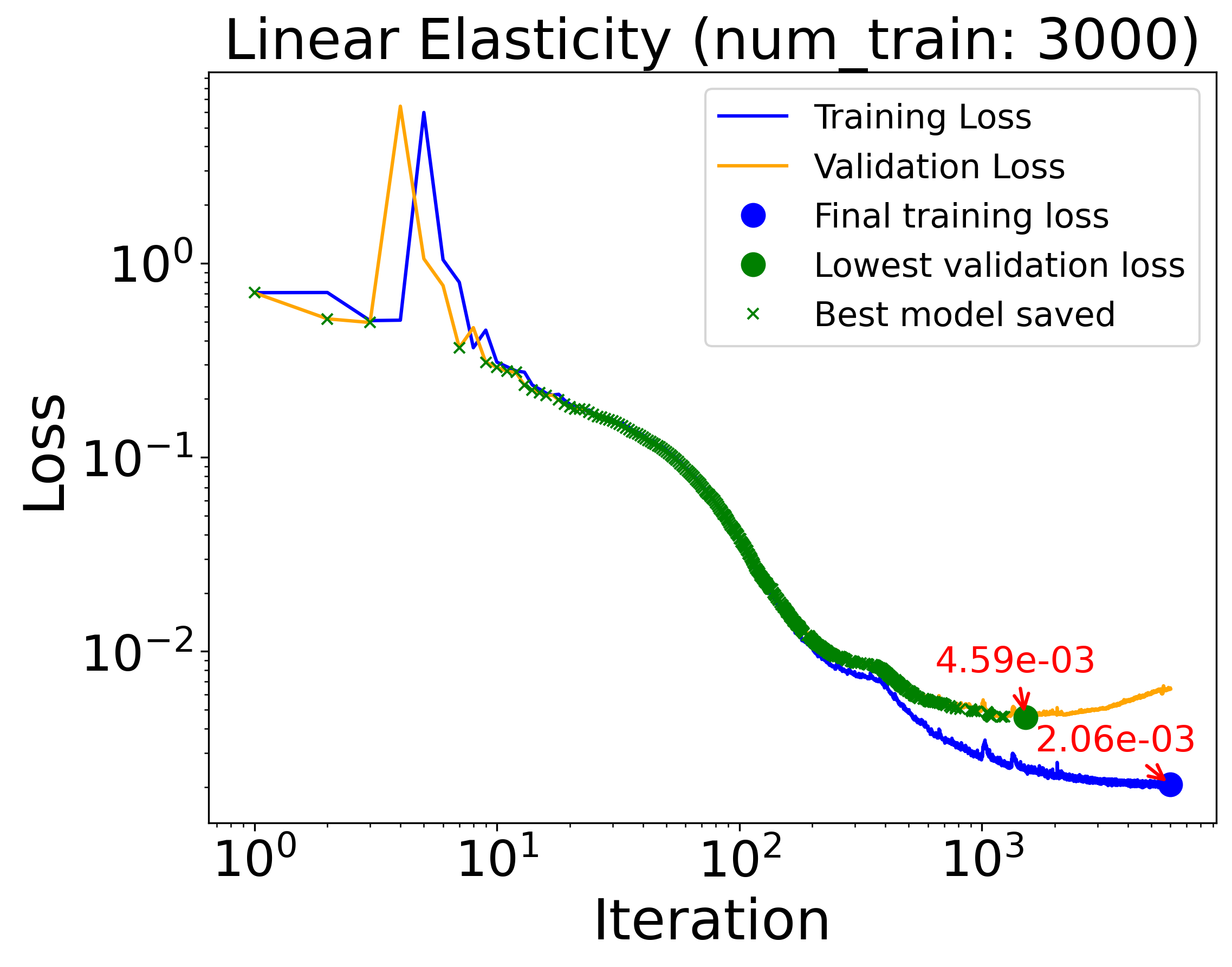}

    \caption{Empirical loss decay over optimization iterations on the training set with 1000 samples (top) and 3000 samples (bottom), and on the validation set with 500 samples, for the three problems. During training, model checkpoints are updated whenever the validation loss decreases, and the final model is selected as the checkpoint with the lowest validation loss.}
    \label{fig:loss_history}
\end{figure}

To demonstrate \cref{thm:final-H-error}, we show in~\cref{fig:error_and_loss_vs_num_train} how three quantities decay with increasing number of training samples from 16 to 64, 256, 1024, and 4096. Specifically, we report (1) the RBNO loss $\EE_{\pp\sim\mu} \big[\mathcal{L}(s_r(\pp;\hat{\theta});\pp)\big]$ at the RBNO solution $s_r(\pp;\hat{\theta})$, (2) the mean square error between the RBNO solution and the RB solution $\EE_{\pp\sim \mu}\big[||s_r(\pp) - s_r(\pp;\hat{\theta})||^2_\HH\big]$, and (3) the mean square error between the RBNO solution and the ``ground truth'' FE solution (with RT$_3\times$CG$_4$ elements) $\EE_{\pp\sim \mu}\big[||\bar{s}_h(\pp) - s_r(\pp;\hat{\theta})||^2_\HH\big]$, 
where the expectation is evaluated with 500 test samples.

\begin{figure}[!htb]
    \centering
    \includegraphics[width=0.32\linewidth]{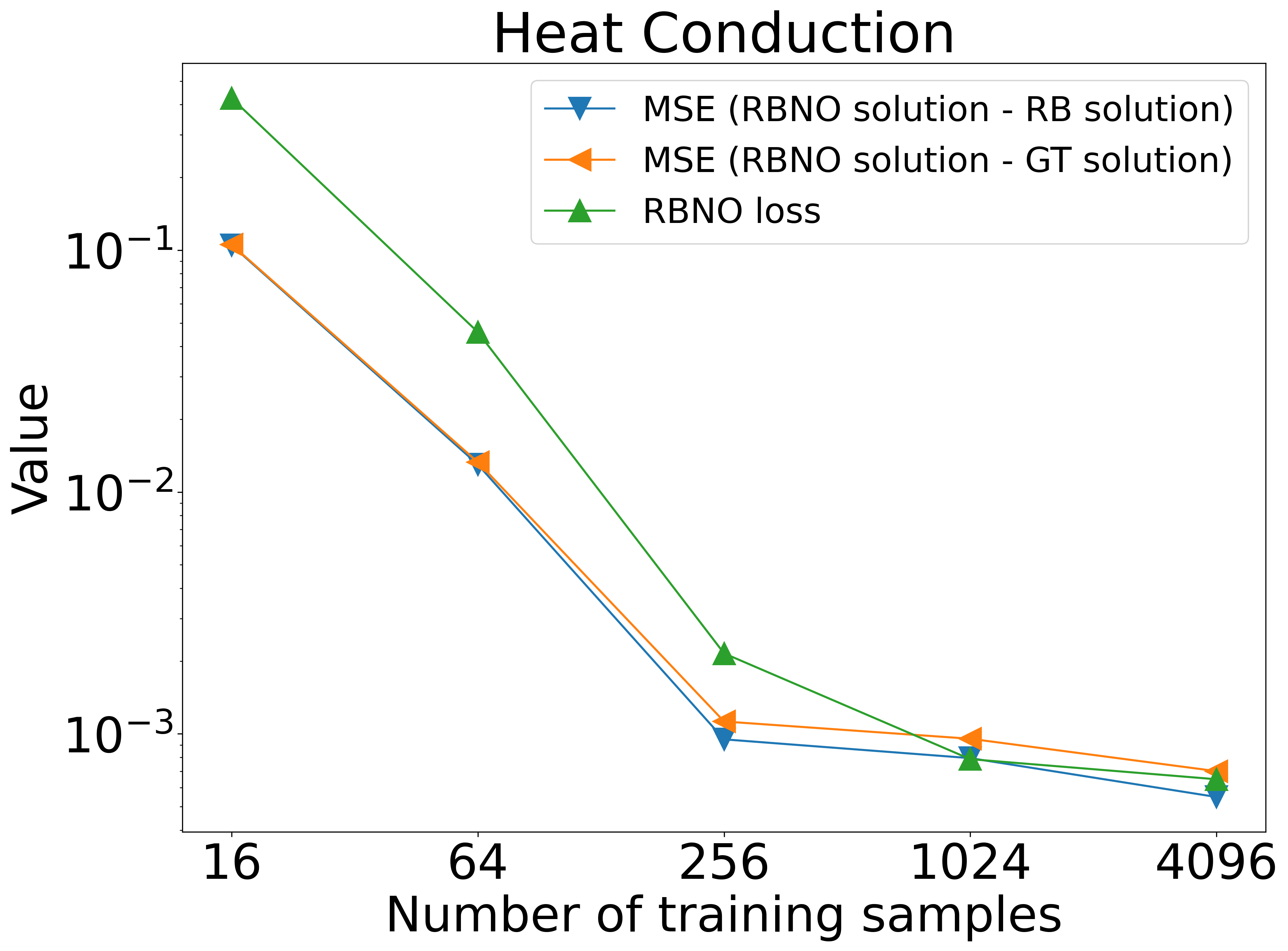}
    \includegraphics[width=0.32\linewidth]{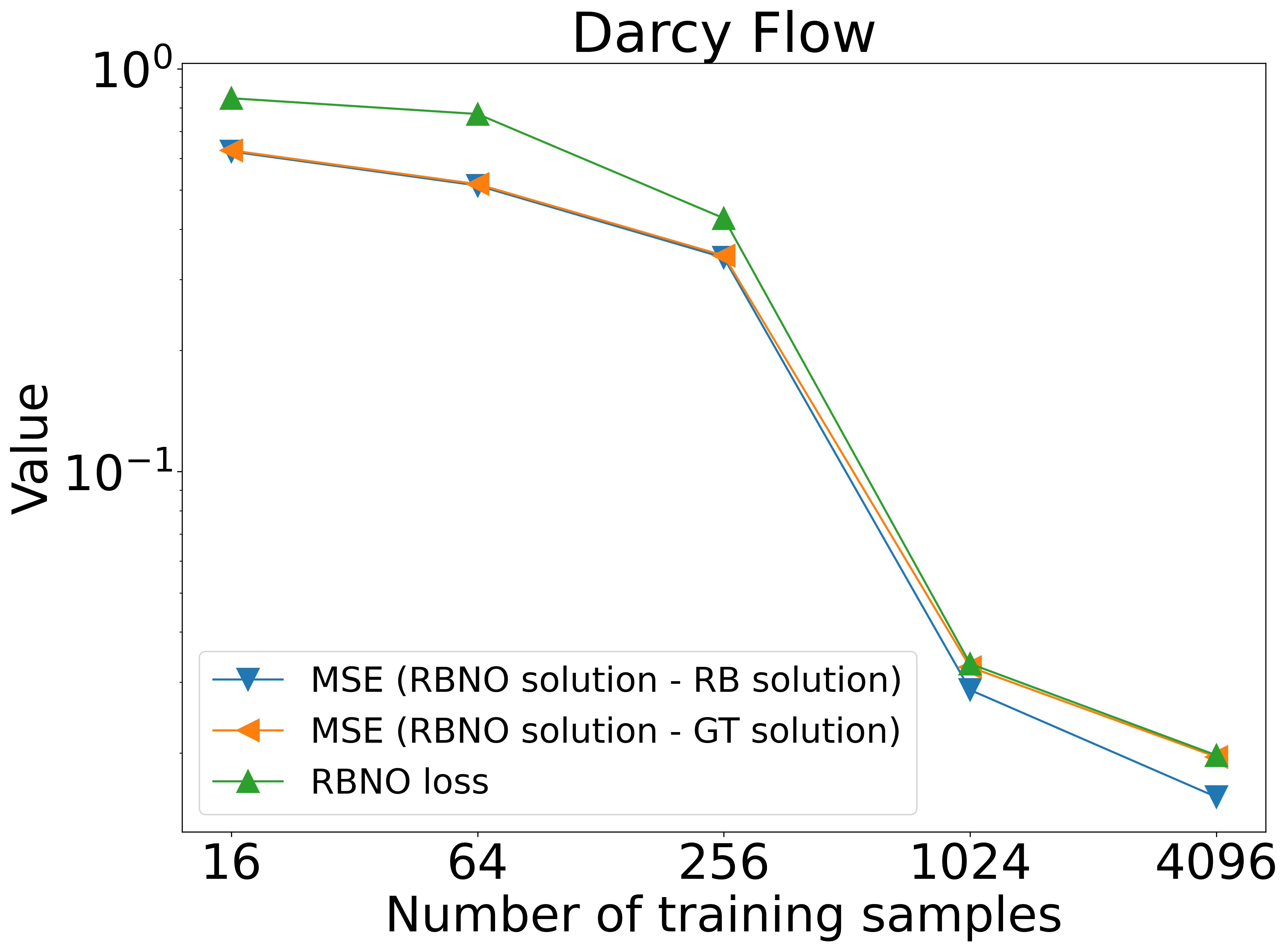}
    \includegraphics[width=0.32\linewidth]{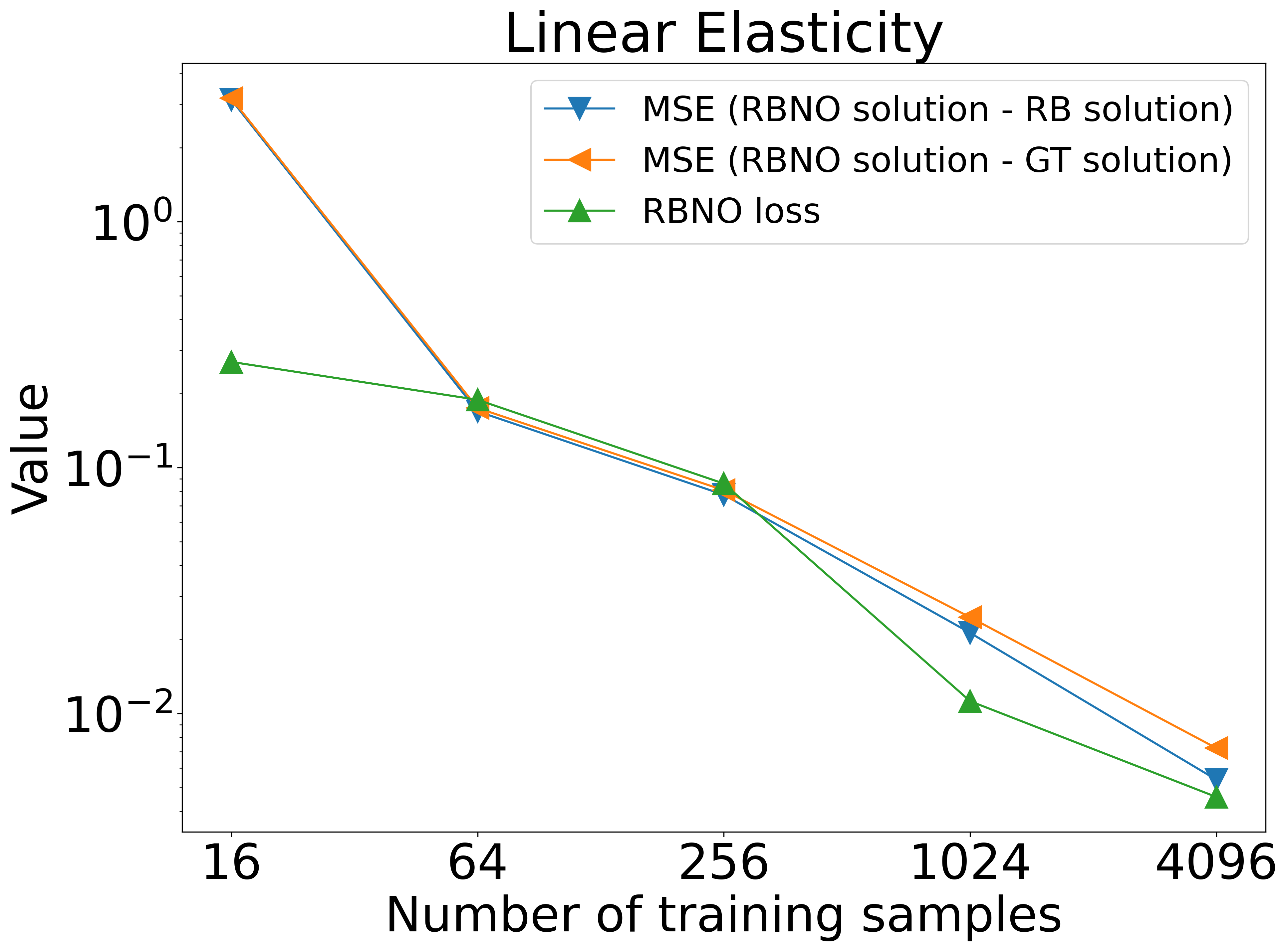}
    \caption{{Decay of the empirical RBNO loss $\EE_{\pp\sim\mu} \big[\mathcal{L}(s_r(\pp;\hat{\theta});\pp)\big]$ at the RBNO solution $s_r(\pp;\hat{\theta})$, the empirical mean square error between the RBNO solution and the RB solution $\EE_{\pp\sim \mu}\big[||s_r(\pp) - s_r(\pp;\hat{\theta})||^2_\HH\big]$,
    and the empirical mean square error between the RBNO solution and the ``ground truth'' FE solution (with RT$_3\times$CG$_4$ elements) $\EE_{\pp\sim \mu}\big[||\bar{s}_h(\pp) - s_r(\pp;\hat{\theta})||^2_\HH\big]$ 
    with increasing number of training samples from 16 to 64, 256, 1024, and 4096, where the expectation is evaluated with 500 test samples.}}
    \label{fig:error_and_loss_vs_num_train}
\end{figure}

In \cref{fig:ratio_comparison_finer_labels}, we show the histograms of the ratios between the RBNO error (compared to the ``ground truth'' FE solution) and the square root of the RBNO residual loss for the three problems with 500 test samples. The ratios cluster near~$1$, indicating that the residual loss is a good estimator of the error.

\begin{figure}[!htb]
        \centering
        \includegraphics[width=0.32\linewidth]{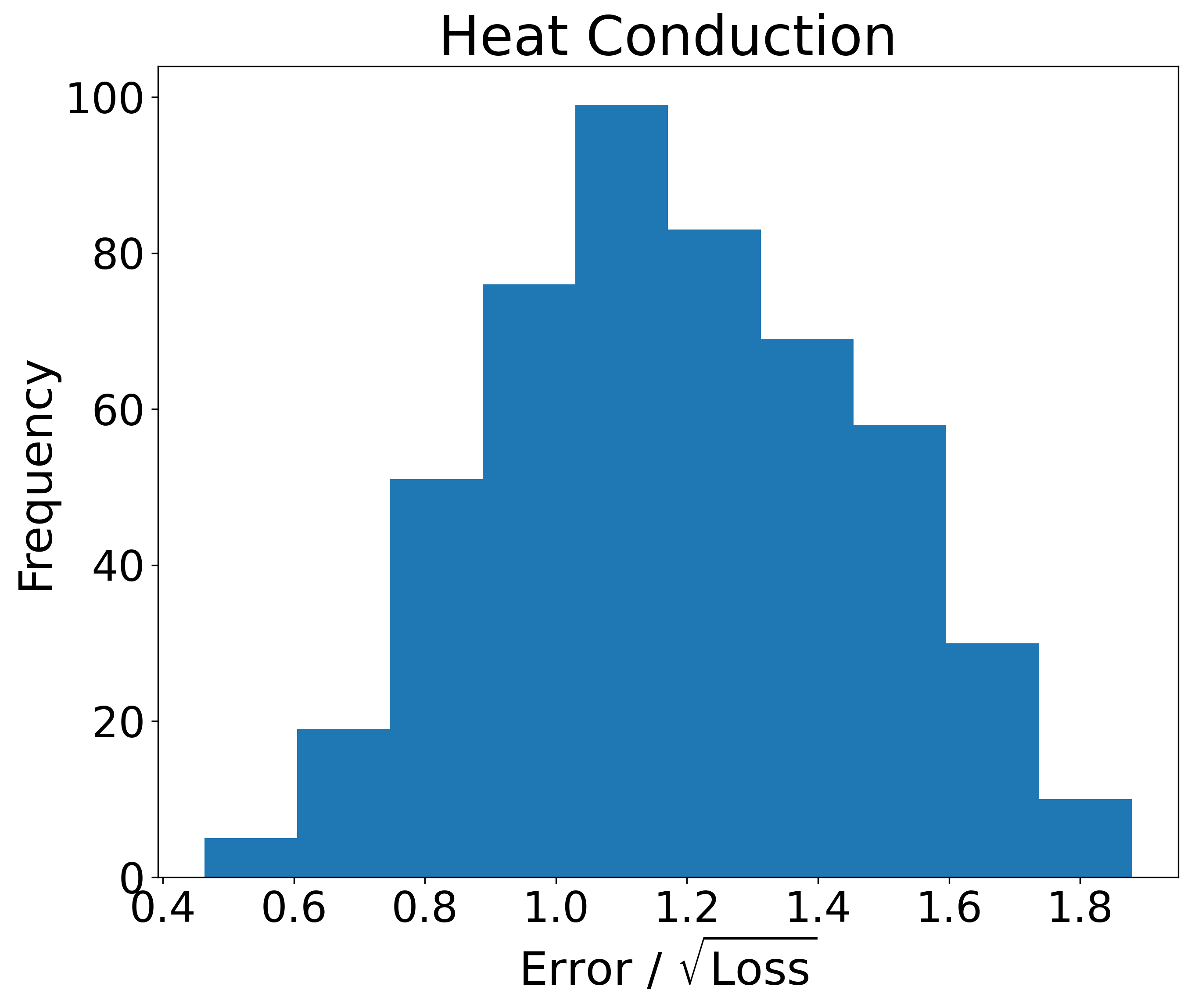}
        \includegraphics[width=0.32\linewidth]{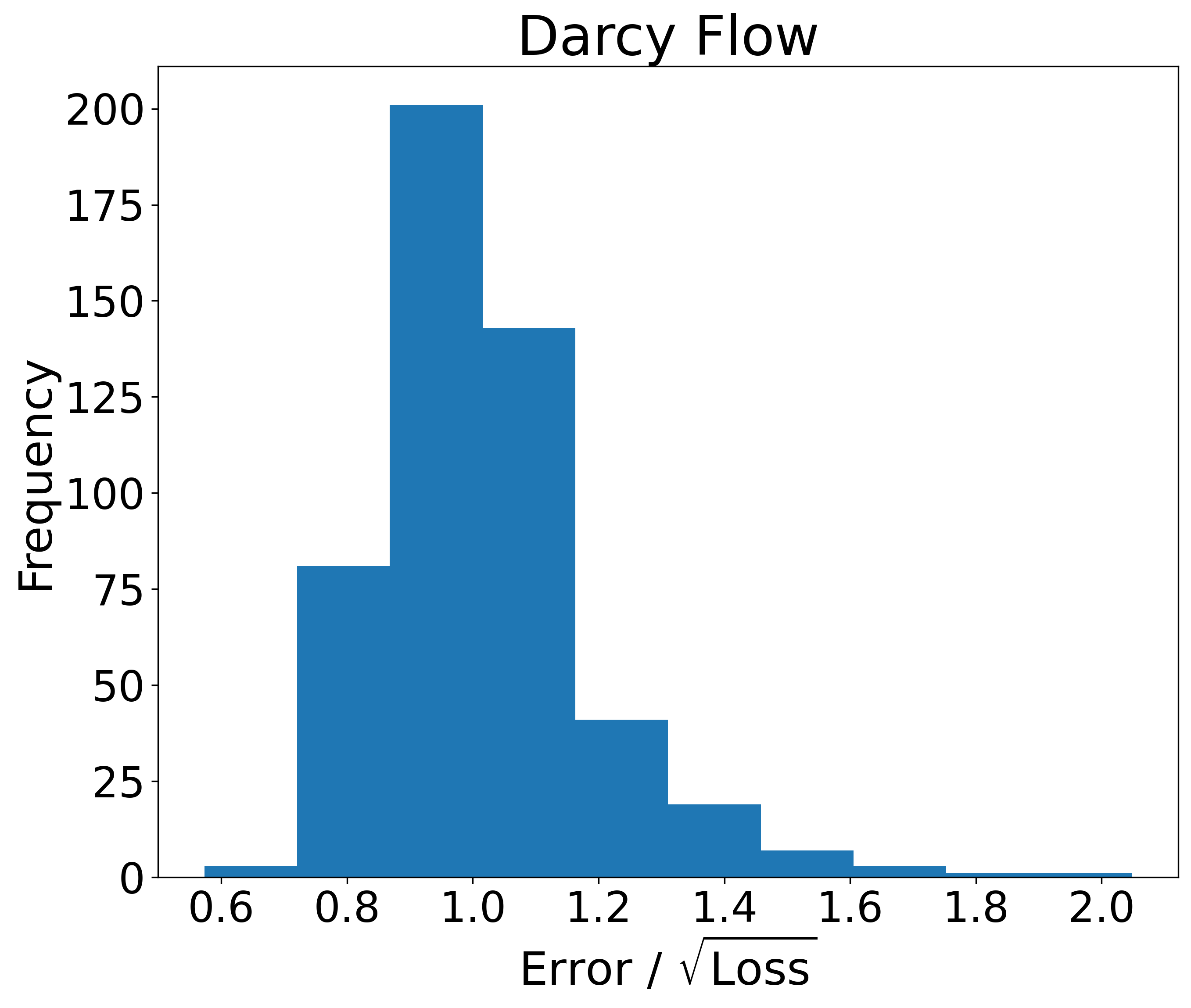}
        \includegraphics[width=0.32\linewidth]{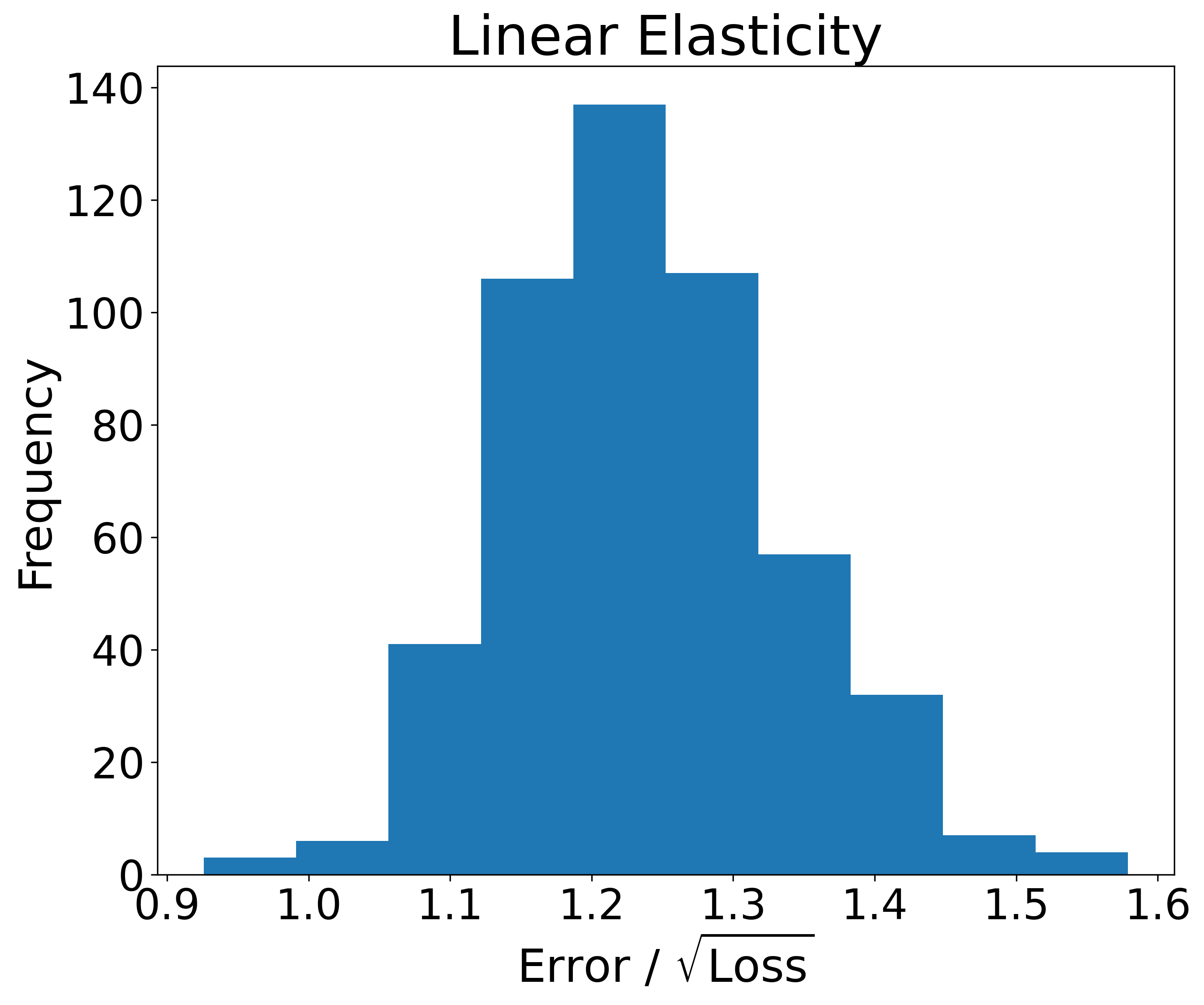}
    \caption{Histograms of the ratio $\text{Error}/\sqrt{\text{Loss}}$, between the RBNO error $||\bar{s}_h(\pp) - s_r(\pp;\hat{\theta})||_\HH$ (with ``ground truth'' FE solution $\bar{s}_h(\pp)$ using RT$_3\times$CG$_4$ elements) and the square root of the RBNO residual loss $\mathcal{L}(s_r(\pp;\hat{\theta});\pp)$, over 500 test samples for the three problems.}
    \label{fig:ratio_comparison_finer_labels}
\end{figure}

\subsection{Comparison to other neural operators}\label{sec:main_results}

{In this section, we compare RBNO with two popular neural operators, detailed below. To place these comparisons in proper perspective, we note a key distinction in objectives. The central goal of RBNO is to directly output the physically relevant quantities, both $u$ and $\sigma$, e.g., temperature/pressure and thermal/flow flux in stationary diffusion or displacement and stress in linear elasticity. In many scenarios, these are the variables of primary interest, and their errors in PDE-compliant norms can be directly measured by the FOSLS residual loss. In contrast, the neural operators discussed here typically provide pointwise evaluations of the solution $u$ at discrete sites within the computational domain. Consequently, any comparison requires post-processing the neural operator outputs, an additional step that may influence the final solution quality (see detailed investigations in \ref{sec:projection-details}). To ensure a fair and rigorous comparison, we apply these neural operators to the solution components $u$ and $\sigma$ of the FOSLS formulation---a step rarely taken in standard practice---to promote variational correctness.}

The first neural operator is the PCA-based neural network (PCA-Net)~\cite{bhattacharya2021model}, which projects both the parameter field and the solution onto low-dimensional PCA bases and then learns a map between their PCA coefficients using a multilayer perceptron (MLP) with a mean squared error loss. The PCA bases are computed as the eigenvectors of empirical covariance of the pointwise evaluation of the training data $\pp \mapsto [\sc_h(\pp), \uc_h(\pp)]$ as in \cite{bhattacharya2021model}.
The second is the Fourier Neural Operator (FNO)~\cite{li2020fourier}, which lifts the input field to a higher-dimensional feature space and applies a sequence of spectral convolution layers: each layer takes a Fourier transform in space, multiplies a fixed number of retained modes by trainable complex weights, and applies an inverse transform followed by a nonlinearity. This architecture learns the operator in frequency space and can be evaluated on arbitrary discretizations. In our experiments, FNO is trained with a relative $L_2$ loss on the pointwise finite element solution $[\sc_h, \uc_h]$ at the mesh grid points. Details of their architectures and training setups are given in \ref{sec:training-details}. To compare them with RBNO in terms of relative errors in $\LL = L_2\times L_2$ and $\HH = H(\rdiv)\times H^1$ norms, we first construct piecewise linear (CG$_1\times$ CG$_1$) FE functions from their output values, which are predictions at the grid points, on the same mesh. The errors are then computed with respect to the corresponding CG$_1\times$ CG$_1$ FE representations of the high-fidelity solutions, which are obtained by evaluating the solutions at the grid points and mapping them one-to-one to the CG$_1$ degrees of freedom. Residual losses are computed using a standard FE assembly. The results are shown in \cref{tab:approximation_relative_error}, with the smallest error or loss for each problem highlighted in bold. We observe that RBNO achieves both the smallest errors in $\HH$-norm and the smallest residual loss, while FNO attains the smallest error in $\LL$-norm. Notably, RBNO, despite not being trained on pointwise evaluation data, still yields a relatively small error in $\LL$-norm. In contrast, both PCA-Net and FNO, which are trained on pointwise evaluation data, produce large relative errors in the $\HH$-norm and violate the PDE residual with significantly large values. Constructing RT$_1\times$CG$_2$ FE representations from the pointwise outputs by projection in the $\HH$-norm does not reduce the error or loss. A more detailed investigation of the large errors and losses, as well as the performance of PCA-Net and RBNO trained using the same RB coefficients, is provided in \ref{sec:projection-details} and \ref{sec:ablation_study}.

\begin{table}[!htb]
\footnotesize
\centering
\begin{tabular}{|l|
  S[table-format=1.2e-2]@{(}S[table-format=1.2e-2]@{\!)\;\;}|
  S[table-format=1.2e-2]@{(}S[table-format=1.2e-2]@{\!)\;\;}|
  S[table-format=1.2e-2]@{(}S[table-format=1.2e-2]@{\!)\;\;}|}
\hline
\textbf{Method} &
\multicolumn{2}{c|}{\textbf{Relative error in $\LL$-norm}} &
\multicolumn{2}{c|}{\textbf{Relative error in $\HH$-norm}} &
\multicolumn{2}{c|}{\textbf{Residual loss}} \\ \hline

\multicolumn{7}{|c|}{\textbf{Heat Conduction}} \\ \hline
PCA-Net & 1.27e-01 & 2.52e-02 & 3.32e-01 & 5.75e-02 & 1.84e+01 & 3.01e+00 \\
FNO & \bmnum{1.82e-02} & \bmnum{1.66e-03} & 1.46e-01 & 1.53e-02 & 1.90e+01 & 4.12e+00 \\
RBNO & 2.43e-02 & 9.35e-03 & \bmnum{1.86e-02} & \bmnum{6.36e-03} & \bmnum{5.33e-04} & \bmnum{1.44e-03} \\
\hline

\multicolumn{7}{|c|}{\textbf{Darcy Flow}} \\ \hline
PCA-Net  & 1.79e-01 & 7.53e-02 & 1.13e-01 & 2.35e-02 & 2.59e+00 & 2.03e+00 \\
FNO      & \bmnum{3.30e-02} & \bmnum{1.18e-02} & 1.60e-01 & 3.81e-02 & 5.10e+00 & 2.34e+00 \\
RBNO     & 6.39e-02 & 2.26e-02 & \bmnum{1.23e-02} & \bmnum{3.64e-03} & \bmnum{2.72e-02} & \bmnum{2.61e-02} \\
\hline

\multicolumn{7}{|c|}{\textbf{Linear Elasticity}} \\ \hline
PCA-Net & 1.34e-01 & 1.03e-01 & 3.27e-01 & 1.14e-01 & 2.46e+01 & 6.11e+00 \\
FNO     & \bmnum{1.17e-02} & \bmnum{3.34e-03} & 4.30e-01 & 6.11e-02 & 2.82e+01 & 6.45e+00 \\
RBNO    & 2.49e-02 & 5.89e-03 & \bmnum{3.91e-02} & \bmnum{8.68e-03} & \bmnum{4.87e-03} & \bmnum{3.56e-03} \\
\hline
\end{tabular}
\caption{Comparison on the mean (and standard deviation, estimated over 500 test samples) of the relative errors in $\LL_2 = L_2 \times L_2$-norm and $\HH = H(\rdiv)\times H^1$-norm, and residual loss for three neural networks on three problems.}
\label{tab:approximation_relative_error}
\end{table}

\section{Conclusions, limitations, and future work}
\label{sec:conclusions_limitations_and_future_work}

This work presented a variationally correct operator learning framework for constructing residual-based loss functions to train accurate surrogates for parameter-to-solution maps of stationary diffusion and linear elasticity models. Our central design principle is \emph{variational correctness}: the residual loss is equivalent (up to constants) to the solution error in the model-induced norm. We achieved this by utilizing FOSLS formulations and incorporating mixed Dirichlet–Neumann boundary conditions via variational lifts, thereby avoiding boundary penalty terms that typically degrade norm equivalence.

Computationally, we bridged the gap between continuous theory and practical implementation by establishing a rigorous link between the discrete loss and the associated discretization errors. We proved that the total error is bounded by a sum of distinct components: finite element discretization bias, reduced basis projection error, neural network approximation error, statistical estimation error, and optimization error. To strictly enforce the function space conformity required by the FOSLS objective while mitigating high-dimensional output effects, we employed the RBNO architecture trained using a computationally efficient RB loss function. Numerical experiments validated this error decomposition and convergence analysis, confirming that the variationally correct residual serves as a tight, computable \emph{a posteriori} error estimator.

While less reliant on data than purely data-driven regression, our method still requires a moderate number of high-fidelity solutions to construct the POD reduced basis. For this strategy to be effective, the reduced dimension must grow slowly as target accuracy increases, a requirement tied to the robust decay of Kolmogorov $n$-widths. While such decay is expected for the dissipative models treated here , our experiments indicate that the required dimension may vary substantially under less benign parameter representations. Furthermore, this favorable decay does not extend to dispersive or transport-dominated problems (e.g., Helmholtz or wave-type regimes), which will require alternative approximation concepts.

Handling boundary conditions remains a delicate issue in scientific machine learning. We addressed this by incorporating mixed boundary conditions via auxiliary lifts, preserving variational correctness without ``wrong penalty terms.'' This comes at the cost of solving two parameter-independent auxiliary boundary value problems, assuming parameter-independent boundary data. For parameter-dependent boundary data, this approach must be adapted; a promising direction is to jointly approximate (or train) the auxiliary lift solutions alongside the main surrogate while maintaining stability.

While we restricted our attention to stationary models, extending this framework to time-dependent PDEs appears viable for parabolic problems, where Kolmogorov $n$-widths decay in a parameter-robust manner and space–time least-squares formulations are available~\cite{KF1,GS}. In contrast, while analogous formulations exist for wave equations~\cite{KF2}, their dispersive nature typically precludes the existence of comparable low-dimensional bases, necessitating new hybrid representations and stability-aware learning objectives.

Finally, we aim to further develop RBNO as a surrogate model equipped with rigorous a posteriori error estimates for downstream tasks such as inverse problems~\cite{stuart2010inverse, cao2025derivative}, data assimilation~\cite{Si2025LatentEnSF,xiao2024ld}, optimization under uncertainty~\cite{chen2023performance,luo2025efficient,yao2025derivative}, and optimal experimental design~\cite{wu2023large, go2025accurate, go2025sequential}. Since success in these domains often relies on derivative or sensitivity information~\cite{o2024derivative,qiu2024derivative,yao2025derivative}, a critical direction for future work is the development of RBNO surrogates to provide accurate derivative information.

\section*{Acknowledgments} 

This work has been supported in part by NSF grants DMS-2245111, DMS-2245097, DMS-2038080, as well as by the German Research Foundation grant CRC 1481.

During the preparation of this work, the authors used ChatGPT and Google Gemini in order to polish the writing. After using this tool/service, the authors reviewed and edited the content as needed and take full responsibility for the content of the published article.

\bibliographystyle{elsarticle-num}
\bibliography{references}

\appendix 

\section{On the Riesz representation of $f \in \U'$}
\label{app:decomposition}
When admitting in \eqref{eq:poisson_strong_primal} a general functional $f$ in $\U'$ we have to ensure that
it does not interfere with the specific Neumann boundary condition on $\Gamma_N$. Indeed, in a weak formulation Neumann boundary conditions appear as functionals in $\U'$. We have therefore restricted
$f$ in \eqref{eq:poisson_strong_primal} to be of the form $f= f_2 + \div\,f_1$ where $f_2\in L_2(\Omega)$
and $f_1\in L_2(\Omega;\R^d)$ is ``flux-free'', in the sense that,
\begin{equation}
\label{fluxfree}
(\div\,f_1,v)_\Omega = -(f_1,\nabla v)_\Omega,\quad f(v) = (f_2,v)_\Omega + (\div\,f_1)(v),\quad \forall\, v\in\U.
\end{equation}

To explain why such a requirement is reasonable we briefly describe how a (generally non-unique) decomposition \eqref{fluxfree}
may come about.

 To that end, we'll make use of the following Weyl-type
decomposition of a general $ F\in L_2:= L_2(\Omega;\R^d)$:
\begin{equation}
\label{Weyl}
 F = \nabla p \oplus \zeta\quad \mbox{for some}\quad p\in\U,\,\,\zeta \in \Sigma_0,
\end{equation}
where we define 
\begin{equation}
\label{Sigma}
\Sigma:= \{\eta\in H(\div;\Omega): \eta \cdot n|_{\Gamma_N} =0\},\quad \Sigma_0:= \{\zeta\in \Sigma: \zeta \cdot n|_{\Gamma_N}=0, \, \div\,\zeta =0\}.
\end{equation}
In particular, we then have, by our assumptions on $\zeta$ in \eqref{Weyl},
\begin{align}
\label{Ftrace}
( F,\nabla v)_\Omega&= (\nabla p,\nabla v)_\Omega + (\zeta,\nabla v)_\Omega = (\nabla p,\nabla v)_\Omega\nonumber\\
&=- (\div\nabla p)(v) + \langle \nabla p \cdot n,v\rangle_{\Gamma_N}\nonumber\\
&= - (\div F)(v)+ \langle \nabla p \cdot n,v\rangle_{\Gamma_N},\quad v\in\U.
\end{align}
In other words, we can see the ``trace-effect'' when taking a weak divergence of an $L_2$-field. To suppress the flux term, consider the subspace 
$$
\U_{\Gamma_N} := \{ v\in \U: \nabla v \cdot n|_{\Gamma_N}=0\}\subset \U.
$$
Thus, whenever, $ F\in L_2$ has a decomposition \eqref{Weyl} with $p\in \U_{\Gamma_N}$, then
\eqref{Ftrace} says that
\begin{equation}
\label{without}
(F,\nabla v)_\Omega = -(\div\,F)(v),\quad v\in \U,
\end{equation}
which explains the notion {\em flux-free} in \eqref{fluxfree}.

To arrive at the desired decomposition of $f$ in \eqref{eq:poisson_strong_primal},  consider the {\em Riesz-lift}
of $f$ with respect to the full $H^1$-inner product on $\U$:
  find $r_f\in \U$ such that 
\begin{equation}
\label{RiesL}   
(\nabla r_f,\nabla v)_\Omega +(r_f,v)_\Omega = -f(v) 
,\quad v\in \U.
\end{equation}
which has a unique solution $r_f\in \U$. 
Since 
\begin{equation}
\label{divf}
f(v)= (\div \nabla r_f)(v) - (r_f,v)_\Omega - \langle  \nabla r_f \cdot n,v\rangle_{\Gamma_N}  ,\quad v\in \U,
\end{equation}
upon setting
$$
f_2:= -r_f\in L_2(\Omega),\quad  f_1:=\nabla r_f\in L_2(\Omega;\R^d),
$$
we can see that the decomposition \eqref{divf} does give rise to the desired property \eqref{fluxfree}, provided that $r_f\in \U_{\Gamma_N}$, and how the implied relation $\langle f_1 \cdot n,v\rangle_{\Gamma_N}=0$, $v\in \U$, is then to be understood.   

\section{Stability and norm equivalence} 
\label{sec:appendix}

\subsection{Stationary diffusion}
\label{app:PoissonNormEquiv}

\begin{lemma}
\label{lem:PoissonNormEquiv}
Assume that $0<\alpha\le \pp(x)\le \beta<\infty$ holds for fixed $\alpha,\beta$. Then
the mapping $\cB_\pp([\tau,v])= {\tau -\pp\nabla v\choose -\div\,\tau}$ is an isomorphism
from $\HH:= H_{0,\Gamma_N}(\div;\Omega)\times H^1_{0,\Gamma_D}(\Omega)$ onto $\LL_2:= L_2(\Omega;\R^{d+1})$. In particular, the norm equivalence
\begin{equation}
\label{Hnormeq}
c\|[\tau,v]\|_{H(\div;\Omega)\times H^1(\Omega)} \le \|\cB_\pp([\tau,v])\|_{\LL_2}\le C \|[\tau,v]\|_{H(\div;\Omega)\times H^1(\Omega)},\quad [\tau,v]\in \HH,
\end{equation}
holds for constants $c,C$ depending only on $\alpha,\beta$ and the Poincar\'{e} constant $C_P$ in $\|  v\|_{L_2(\Omega)}\le C_P\|\nabla v\|_{L_2(\Omega)}$, $v\in \U:= H^1_{0,\Gamma_D}(\Omega)$. In fact, denoting by $\kappa(\cB_\pp):= \|\cB_\pp\|_{\HH\to\LL_2}
\|\cB_\pp^{-1}\|_{\LL_2\to\HH}$ the condition number of $\cB_\pp$ we have $\kappa(\cB_\pp)\le
C\beta/\alpha$ where $C$ is a constant depending only on $C_P$. In these terms
\begin{equation}
\label{cond}
c\ge c_1\big(\sup_{\pp\in\PP}\kappa(\cB_\pp))^{-1} \quad \mbox{while}\quad C\le (2+\beta^2)^{1/2},
\end{equation}
where $c_1$ depends only on $C_P$. In other words, the FOSLS bilinear form
\begin{equation}
\label{FOSLSbl}
   ([\tau,v],[\eta,s])_{\X_\pp} := (\tau-\pp\nabla v,\eta-\pp\nabla s)_{L^2(\Omega)}
                              + (\div\tau,\div\eta)_{L^2(\Omega)},
\end{equation}
is an equivalent inner product on $\HH$, inducing the   norm $\|[\tau,v]\|^2_{\X_\pp} := ([\tau,v],[\tau,v])_{\X_\pp} = \|\cB_\pp[\tau,v]\|^2_{L^2(\Omega;\R^{d+1})}$.
\end{lemma}

\begin{proof}
    Proofs of this result can, in essence,  be found in the literature, see \cite{FOSLS,bochev2009least}. For the convenience of the reader, we present a short version with an eye to the parameter dependence.
    \paragraph{Upper bound.} Simply use triangle inequality and Cauchy-Schwarz to obtain 
    \begin{align*}
    \|\cB_\pp([\tau,v])\|_{\LL_2}&\le \|\tau\|_{L_2(\Omega;\R^d)}+ \|\div\,\tau\|_{L_2(\Omega)}
    + \|\pp\nabla v\|_{L_2(\Omega)}\le\sqrt{2}\|\tau\|_{H(\div;\Omega)}  +\beta \|v\|_{H^1(\Omega)}\\
    & \le (2+\beta^2)^{1/2}\|[\tau,v]\|_{H(\div;\Omega)\times H^1(\Omega)},
    \end{align*}
    confirming the second part in \eqref{cond}. 
    
    \paragraph{Lower bound.} Using integration by parts (since $\tau\cdot n|_{\Gamma_N}=0$ and $v|_{\Gamma_D}=0$), and the fact that, by assumption on $\PP$, $\alpha\,\|\nabla v\|_{L^2}^2\le (\pp\nabla v,\nabla v)$, we obtain
    \begin{align*}
      \|\nabla v\|_{L^2}^2&\le \alpha^{-1} (\pp\nabla v,\nabla v)  = -\alpha^{-1}\big(\big(\tau-\pp\nabla v,\nabla v\big)+ (\div\tau, v)\big) \\
     &\le \alpha^{-1}\Big\{\|\tau-\pp\nabla v\|_{L^2}\,\|\nabla v\|_{L^2}+\|\div\tau\|_{L^2}\,\|v\|_{L^2}\Big\} \\
     &\le\alpha^{-1} \|\cB_\pp([\tau,v])\|_{\LL_2}\|v\|_{H^1(\Omega)}\le \alpha^{-1}(1+C_P^2)^{1/2}\|\cB_\pp([\tau,v])\|_{\LL_2}\|\nabla v\|_{L_2},
    \end{align*}
    so that, in particular, 
    \begin{equation}
    \label{nablav}
    \|\nabla v\|_{L_2}\le \alpha^{-1}(1+C_P^2)^{1/2}\|\cB_\pp([\tau,v])\|_{\LL_2}.
    \end{equation}
    Thus, by Poincar\'{e}'s inequality,
    \begin{equation}
    \label{P2}
    \|v\|^2_{H^1}\le (1+C_P^2)\|\nabla v\|^2_{L_2}\le \frac{(1+C_P^2)^2}{\alpha^2}\|\cB_\pp([\tau,v])\|^2_{\LL_2}.
    \end{equation}
     Moreover, again by \eqref{nablav} and using $\|\pp\nabla v\|_{L_2}\le \beta \|\nabla v\|_{L_2}$, 
    \begin{align}
    \label{A1}
    \|\tau\|_{H(\div)}&\le  \|\tau\|_{L_2}+ \|\div\,\tau\|_{L_2} \le \|\tau-\pp\nabla v\|_{L_2}+\|\pp\nabla v\|_{L_2} +  \|\div\,\tau\|_{L_2}\nonumber\\
    &\le (\sqrt{2}+A)\|\cB_\pp([\tau,v])\| _{\LL_2}   
    ,\quad \mbox{where }\, A:=   \frac{\beta(1+C_P^2)^{1/2}}{\alpha} .
     \end{align}
    Combining \eqref{P2} and \eqref{A1} yields the lower bound with 
    $$
    c= \Big(\Big(\sqrt{2}+ \frac{\beta(1+C_P^2)^{1/2}}{\alpha}\Big)^2 + \frac{(1+C_P^2)^2}{\alpha^2}
    \Big)^{-1/2},
    $$
    from which \eqref{Hnormeq} and the first inequality in \eqref{cond} follow.
    
    To complete the proof, note that \eqref{Hnormeq} already implies boundedness and injectivity of $\cB_\pp$, so that the remainder of the assertion follows from the Open Mapping Theorem, once we have confirmed surjectivity of $\cB_\pp$. To that end, note that we have incidentally shown that the bilinear form $B(\cdot,\cdot):\HH\times \HH\to \R$, defined by $B_\pp([\sigma,u],[\tau;v])
    := (\cB_\pp([\sigma,u]),\cB_\pp([\tau,v])_\Omega$ is coercive. Since $\cB_\pp$ is bounded, $B(\cdot,\cdot)$
    is also bounded (with constant $\|\cB_\pp\|^2_{\HH\to\LL_2}\le 2+\beta^2$.  Lax-Milgram shows that the normal equations: for any $(y_1,y_2)\in \LL_2$, find $[\sigma,u]\in \HH$ such that
    $$
     ([\sigma,u],[\tau,v])_{\X_\pp}= ([y_1,y_2], \cB_\pp([\tau,v]))_\Omega\quad \forall\, [\tau,v]\in\HH,
    $$
    is well-posed and stable. Now suppose $\cB_\pp$ is not surjective. Since $\cB_\pp$ is closed, its range 
    is then a closed subspace of $\LL_2$. Then there exists a $y=[y_1,y_2]\in \LL_2\setminus \{0\}$ such that 
    $(\cB_\pp [\tau,v],y)_\Omega =0$ for all $[\tau,v]\in \HH$. On the other hand, we know from the above
    observation that there exists a unique $0\neq [\sigma_y,u_y]\in\HH$ such that
    $$
    (\cB_\pp[(\sigma_y,u_y],\cB_\pp[\eta,w])_\Omega = (y,\cB_\pp([\eta,w])_\Omega,\quad [\eta,w]\in\HH.
    $$
    By assumption, the right hand side vanishes for all $[\eta,w]\in\HH$, while for $[\eta,w]=[\sigma_y,u_y]$, we have $(\cB_\pp[\sigma_y,u_y],\cB_\pp[\sigma_y,u_y])_\Omega
    = \|\cB_\pp [\sigma_y,u_y]\|^2_{\LL_2}\ge c^2 \|[\sigma_y,u_y]\|^2_{H(\div;\Omega)\times H^1(\Omega)}>0$, a contradiction.   This contradiction confirms surjectivity and completes the proof. 
\end{proof}

\subsection{Linear elasticity}
\label{sec:equivalence-elasticity}

Assume is $\Omega$ is Lipschitz and $\Gamma_D$ has positive measure. Define for $[\tenbar[2]{\tau},\tenbar[1]{v}]\in \HH:= H_{0,\Gamma_N}(\rdiv;\Omega;\R^{d\times d})\times H^1_{0,\Gamma_D}(\Omega;\R^d)$
\begin{equation*}
  \cB_\pp[\tenbar[2]{\tau},\tenbar[1]{v}]
  := {\cC_\pp^{-1/2}(\tenbar[2]{\tau}- \cC_\pp\tenbar[2]{\e}(\tenbar[1]{v}))\choose -\tenbar[1]{\rdiv}\,\tenbar[2]{\tau}}
  \in L_2(\Omega;\R^{d\times d})\times L_2(\Omega;\R^d)=:\LL_2.
\end{equation*}
As before, suppressing at times the reference to domain and range in $L_2$-, $H^1$-, and $H(\div)$-norms, we introduce the elasticity FOSLS bilinear form 
\begin{equation*}
  ([\tenbar[2]{\tau},\tenbar[1]{v}],[\tenbar[2]{\eta},\tenbar[1]{s}])_{\X_\pp}
  := (\cC_\pp^{-1/2}(\tenbar[2]{\tau}- \cC_\pp\tenbar[2]{\e}(\tenbar[1]{v})),\,\cC_\pp^{-1/2}(\tenbar[2]{\eta}- \cC_\pp\tenbar[2]{\e}(\tenbar[1]{s})))_{\Omega}
   + (\tenbar[1]{\rdiv}\,\tenbar[2]{\tau},\,\tenbar[1]{\rdiv}\,\tenbar[2]{\eta})_{L_2(\Omega)},
\end{equation*}
and set $\|[\tenbar[2]{\tau},\tenbar[1]{v}]\|_{\X_\pp}^2 := ([\tenbar[2]{\tau},\tenbar[1]{v}],[\tenbar[2]{\tau},\tenbar[1]{v}])_{\X_\pp} = \|\cB_\pp[\tenbar[2]{\tau},\tenbar[1]{v}]\|^2_{\LL_2}$.

\begin{lemma}
Assume the stiffness tensor $\cC_\pp$ is uniformly symmetric positive definite (SPD) with bounds independent of $\pp\in\PP$, i.e., there exist positive constants $c_0,c_1$ such that
\begin{equation*}
  c_0\,\tenbar[2]{\e}(\tenbar[1]{w}):\tenbar[2]{\e}(\tenbar[1]{w})
  \le \tenbar[2]{\e}(\tenbar[1]{w}):\cC_\pp\,\tenbar[2]{\e}(\tenbar[1]{w})
  \le c_1\,\tenbar[2]{\e}(\tenbar[1]{w}):\tenbar[2]{\e}(\tenbar[1]{w})\quad \forall\,\tenbar[1]{w}\in H^1(\Omega;\R^d).
\end{equation*}
Then for all $[\tenbar[2]{\tau},\tenbar[1]{v}]\in H_{0,\Gamma_N}(\rdiv;\Omega;\R^{d\times d})\times H^1_{0,\Gamma_D}(\Omega;\R^d)$,
\begin{equation}
\label{elasteq}
  c\,\big(\|\tenbar[2]{\tau}\|_{H(\rdiv)}^2 + \|\tenbar[1]{v}\|_{H^1}^2\big)
  \le \|[\tenbar[2]{\tau},\tenbar[1]{v}]\|_{\X_\pp}^2
  \le C\,\big(\|\tenbar[2]{\tau}\|_{H(\rdiv)}^2 + \|\tenbar[1]{v}\|_{H^1}^2\big),
\end{equation}
with constants $c,C>0$ depending only on $c_0,c_1$ and the domain. Moreover, $\cB_\pp$ is an isomorphism from $\HH=H_{0,\Gamma_N}(\rdiv;\Omega;\R^{d\times d})\times H^1_{0,\Gamma_D}(\Omega;\R^d)$ onto $\LL_2=L_2(\Omega;\R^{d\times d})\times L_2(\Omega;\R^d)$.
\end{lemma}

\begin{proof}
    Let $c_0,c_1>0$ be the uniform SPD bounds for $\cC_\pp$ and let $C_P=C_P(\Omega,\Gamma_D)$ and $C_K=C_K(\Omega,\Gamma_D)$ be the Poincar\'e and Korn constants so that $\|\tenbar[1]{v}\|_{L^2}\le C_P\|\nabla \tenbar[1]{v}\|_{L^2}$ and $\|\nabla \tenbar[1]{v}\|_{L^2}\le C_K\|\tenbar[2]{\e}(\tenbar[1]{v})\|_{L^2}$ for all $\tenbar[1]{v}\in H^1_{0,\Gamma_D}(\Omega;\R^d)$ (e.g., \cite{boffi2013mixed}).
    The following arguments are analogous to the diffusion case, beginning with \eqref{elasteq}.

    \paragraph{Upper bound.} Triangle- and Cauchy Schwarz inequalities provide as before
    \begin{align*}
     \|\cB_\pp[\tenbar[2]{\tau},\tenbar[1]{v}]\|_{\LL_2}
     &\le \|\cC_\pp^{-1/2}\tenbar[2]{\tau}\|_{L_2}+ \|\cC^{1/2}_\pp\tenbar[2]{\e}(\tenbar[1]{v}))\|_{L_2}
       + \|\tenbar[1]{\rdiv}\,\tenbar[2]{\tau}\|_{L_2}  
      \le c_0^{-1}\sqrt{2} \|\tenbar[2]{\tau}\|_{H(\div)} + c_1\|\tenbar[2]{\rgrad}\tenbar[1]{v})\|_{L_2}\\
     &\le (2 c_0  + c_1 )^{1/2}\|[\tenbar[2]{\tau},\tenbar[1]{v}]\|_{H(\div)\times H^1},
    \end{align*}
    confirming the upper bound in \eqref{elasteq}.
    
    \paragraph{Lower bound.} Along the same lines as in the previous section, one can prove
    $$
    \|\tenbar[2]{\e}(\tenbar[1]{v})\|_{L_2}\lesssim  \frac{C_P}{c_0^{1/2}}  \|\cB_\pp[\tenbar[2]{\tau},\tenbar[1]{v}]\|_{\LL_2},
    $$
    (and likewise for $\|\tenbar[2]{\rgrad}{(v)}\|_{L_2}$) with an absolute proportionality constant to conclude further
    \begin{equation*}
        \label{H1elast}
    \|\tenbar[1]{v}\|_{H^1}\lesssim \frac{C_KC_P}{c_0^{1/2}}   \|\cB_\pp[\tenbar[2]{\tau},\tenbar[1]{v}]\|_{\LL_2},\quad \tenbar[1]{v}\in \U:=H^1_{0,\Gamma_D}(\Omega;\R^d),
    \end{equation*}
    as well as 
    $$
    \|\tenbar[2]{\tau}\|_{H(\div)}\lesssim c_1^{1/2} \Big(1+\frac{c_1^{1/2}C_K C_P}{c_0^{1/2}}\Big)
     \|\cB_\pp[\tenbar[2]{\tau},\tenbar[1]{v}]\|_{\LL_2}.
    $$
    Assembling these bounds confirms the lower bound as well with an analogous dependence of $c$ on 
    $(c_1/c_0)$, describing the square of the condition number of $\cB_\pp$.
    
    While \eqref{elasteq} so far establishes bijectivity of $\cB_\pp$ as a mapping from $\HH$ to its range
    in $\LL_2$, the same argument as in the previous section shows that this range is indeed all of $\LL_2$,
    which by the Inverse Mapping Theorem completes the proof.
\end{proof}

\section{Error estimates for finite element approximations}
\label{sec:FE-Error-Estimate}

\subsection{Proof of \cref{thm:FELSPoisson} for the stationary diffusion problem}
\label{app:proofFELSPoisson}
\begin{proof}
    Again, for the convenience of the reader, we recall some relevant facts that can be found, for instance,
    in \cite{boffi2013mixed} concerning the FE error estimate for the first-order system least-squares (FOSLS) formulation of the diffusion problem. Recall the fiber-residual loss from \eqref{eq:lossPoisson}
    \begin{equation*}
    \mathcal{L}([\tilde{\sigma}^\circ,\tilde u^\circ];\pp)
      := \|\tilde{\sigma}^\circ - ( \pp\nabla \tilde u^\circ + \pp\nabla w - z + f_1) \|^2_{L^2(\Omega)}
       + \|\div \tilde{\sigma}^\circ+ f_2\|^2_{L^2(\Omega)}.
    \end{equation*}
    Let $S:=[\pp\nabla w - z + f_1,\, f_2]$ denote the source vector. Then
    \begin{equation*}
    \mathcal{L}([\tilde{\sigma}^\circ,\tilde u^\circ];\pp) = \|\cB_\pp[\tilde{\sigma}^\circ,\tilde u^\circ] - S\|_{L^2(\Omega;\R^{d+1})}^2
      = \|\tilde{\sigma}^\circ - ( \pp\nabla \tilde u^\circ + \pp\nabla w - z + f_1) \|^2_{L^2(\Omega)}
       + \|\div \tilde{\sigma}^\circ+ f_2\|^2_{L^2(\Omega)}.
    \end{equation*}
    Let $[\sc,\uc]=[\sc(\pp),\uc(\pp)]$ denote the exact fiber solution, which satisfies $\cB_\pp[\sc,\uc]=S$ so that $\mathcal{L}([\sc,\uc];\pp)=0$ (hence is the unique minimizer of the loss). For conforming FE spaces $\Sigma_h\subset H_{0,\Gamma_N}(\div;\Omega)$ and $\U_h\subset H^1_{0,\Gamma_D}(\Omega)$, the Galerkin FOSLS solution $[\sc_h,\uc_h]=[\sc_h(\pp),\uc_h(\pp)]\in\Sigma_h\times\U_h$ minimizes for each $\pp\in\PP$,  $\cL(\cdot;\pp)$ over $\Sigma_h\times\U_h$  and satisfies the corresponding normal equations in $\Sigma_h
    \times \U_h$.
    
    Since $[\sc,\uc]$ solves the normal equation, 
    one also has for any   $[\tilde{\sigma}^\circ,\tilde u^\circ]\in \HH$  
    \begin{equation*}
      \cL([\tilde{\sigma}^\circ,\tilde u^\circ];\pp)
      = \|\cB_\pp[\tilde{\sigma}^\circ,\tilde u^\circ]-\cB_\pp[\sc,\uc]\|_{\LL_2}^2
      = \|\cB_\pp[\tilde{\sigma}^\circ-\sc,\tilde u^\circ-\uc]\|_{\LL_2}^2
      = \|[\tilde{\sigma}^\circ-\sc,\tilde u^\circ-\uc]\|^2_{\X_\pp}.
    \end{equation*}
    Since $[\sc_h,\uc_h]$ solves the normal equations in $\Sigma_h\times \U_h$, Galerkin orthogonality with
    respect to the inner product $(\cdot,\cdot)_{\X_\pp}$ and the resulting  best-approximation property
    yields 
    \begin{equation*}
      \|[\sc-\sc_h,\uc-\uc_h]\|_{\X_\pp} = \min_{[\tau_h,v_h]\in\Sigma_h\times\U_h}\|[\sc-\tau_h,\uc-v_h]\|_{\X_\pp}.
    \end{equation*}
    In what follows, to simplify notation, we occasionally write $L_2$ without specifying the domain and range, when this is clear from the context.
    Let $\Pi^{\rm RT}_k: H(\div;\Omega)\to\Sigma_h$ denote the canonical Raviart–Thomas interpolant, 
    (see \cite{boffi2013mixed}, section 2.5.2, page 109)
    and $I_m: H^1(\Omega)\to\U_h$ an $H^1$ quasi-interpolant. 
    Choosing $[\tau_h,v_h]=[\Pi^{\rm RT}_k\sc, I_m\uc]$ and expanding the $\X_\pp$-norm, yields
    \begin{align*}
     \|[\sc-\sc_h,\uc-\uc_h]\|_{\X_\pp}
     &\le \|[\sc-\Pi^{\rm RT}_k\sc,\uc-I_m\uc]\|_{\X_\pp} \\
     &\le \|\sc-\Pi^{\rm RT}_k\sc\|_{L_2}
          + \beta\|\nabla(\uc-I_m\uc)\|_{L_2}
          + \|\div(\sc-\Pi^{\rm RT}_k\sc)\|_{L_2}.
    \end{align*}
    Standard RT and $H^1$ interpolation estimates on shape-regular meshes (see, e.g., \cite[Proposition 2.5.4]{boffi2013mixed}) give, assuming   regularity-orders $s_{\sigma},s_{\div}\ge 0$ and $s_u\ge1$ for $\sc,\div\,\sc, \uc$,
    respectively,
    \begin{align*}
      \|\sc-\Pi^{\rm RT}_k\sc\|_{L_2} &\lesssim h^{\min(k+1,s_{\sigma})}\,\|\sc\|_{H^{s_{\sigma}}(\Omega)}, \\
      \|\div(\sc-\Pi^{\rm RT}_k\sc)\|_{L_2 } &\lesssim h^{\min(k+1,s_{\div})}\,\|\div\sc\|_{H^{s_{\div}}(\Omega)}, \\
      \|\nabla(\uc-I_m\uc)\|_{L_2} &\lesssim h^{\min(m,s_u-1)}\,\|\uc\|_{H^{s_u}(\Omega)}.
    \end{align*}
    where, generally $s_\div= s_{\sigma}-1$. Note that these estimates are therefore relevant when $ s_{\sigma}-1$ is at least $k+1$, which takes advantage of the particular properties of the Raviart-Thomas finite element spaces.
     
    Combining these bounds, using $\|\cdot\|_{\X_\pp}\eqsim \|\cdot\|_{H(\div)\times H^1}$, we obtain the a priori bound stated in \cref{thm:FELSPoisson}:
    \begin{align*}
      \cL([\sc_h,\uc_h];\pp)
      &= \|[\sc-\sc_h,\uc-\uc_h]\|_{\X_\pp}^2 \\
      &\lesssim h^{2\min(k+1,s_{\sigma})}\,\|\sc\|_{H^{s_{\sigma}}}^2
            + h^{2\min(k+1,s_{\div})}\,\|\div\sc\|_{H^{s_{\div}}}^2
            + h^{2\min(m,s_u-1)}\,\|\uc\|_{H^{s_u}}^2.
    \end{align*}
    In particular, when $s_{\sigma}, s_{\div}\geq k+1$,  and $s_u\geq m+1$, it reduces to
    \(\cL([\sc_h,\uc_h];\pp)\lesssim h^{2(k+1)} + h^{2m}\). Thus, choosing $m=k+1$ balances the contributions, yielding $\cL([\sc_h,\uc_h];\pp)\lesssim h^{2(k+1)}$.   
\end{proof}

\begin{remark}[approximate lifts consistency]
In practice, in the evaluation of the FE loss $\cL([\sc_h,\uc_h];\pp)$, we replace the variational lifts $(w,z)$ by their $\text{CG}_m$ approximations $(w_h, z_h)$, leading to the approximate FE loss $\cL_h([\sc_h,\uc_h];\pp)$. It is straightforward to show (by $(a+b)^2 \leq 2a^2 + 2b^2$) that 
$$
\cL_h([\sc_h,\uc_h];\pp) \leq 2 \cL([\sc_h,\uc_h];\pp) + 2 ||\delta||_{L^2}^2 \text{ and } \cL([\sc_h,\uc_h];\pp) \leq 2 \cL_h([\sc_h,\uc_h];\pp) + 2 ||\delta||_{L^2}^2,
$$
where the error $\delta := \pp\nabla(w-w_h) - (z-z_h)$ satisfies $\|\delta\|_{L^2}^2 \lesssim h^{2m}$~\cite[Proposition 2.5.4]{boffi2013mixed}) as the lifts are harmonic. This confirms that the approximate lifts do not degrade the asymptotic accuracy with $m \geq k+1$. 
\end{remark}

\subsection{Error estimates of FE approximations for linear elasticity}
\label{app:proofFELSElasticity}
\begin{theorem}\label{thm:FELSElasticity}
Assume $\Omega$ is Lipschitz and $\cC_\pp$ is uniformly bounded with bounds independent of $\pp\in\PP$. Let $\Sigma_h=(\mathrm{RT}_k^\circ)^d$ and $\U_h=(\mathrm{CG}_m^\circ)^d$ on a shape-regular mesh of size $h$, with $k\ge0$, $m\ge1$. Let $[\tenbar[2]{\sigma}^\circ,\tenbar[1]{u}^\circ]$ be the exact  solution and $[\tenbar[2]{\sigma}^\circ_h,\tenbar[1]{u}^\circ_h]$ the Galerkin solution in $\Sigma_h\times\U_h$. Then, for each fixed $\pp\in\PP$, and $[\tenbar[2]{\sigma}^\circ(\pp),\tenbar[1]{u}^\circ(\pp)]=[\tenbar[2]{\sigma}^\circ,\tenbar[1]{u}^\circ]$, $[\tenbar[2]{\sigma}^\circ_h(\pp),\tenbar[1]{u}^\circ_h(\pp)]=[\tenbar[2]{\sigma}^\circ_h,\tenbar[1]{u}^\circ_h]$,
one has the error–residual equivalence
\begin{equation}\label{eq:FE-loss-equiv-el}
  \mathcal{L}([\tenbar[2]{\sigma}^\circ_h,\tenbar[1]{u}^\circ_h];\pp)
  \eqsim \|\tenbar[2]{\sigma}^\circ-\tenbar[2]{\sigma}^\circ_h\|_{H(\rdiv;\Omega;\R^{d\times d})}^2
        + \|\tenbar[1]{u}^\circ-\tenbar[1]{u}^\circ_h\|_{H^1(\Omega;\R^d)}^2,
\end{equation}
with equivalence constants, depending only on the uniform SPD bounds for $\cC_\pp$ and the domain.

Moreover, if $\tenbar[2]{\sigma}^\circ\in H^{s_{\sigma}}(\Omega;\R^{d\times d})$, $\tenbar[1]{\rdiv}\,\tenbar[2]{\sigma}^\circ\in H^{s_{\div}}(\Omega;\R^d)$ with $s_{\sigma},s_{\div}\ge0$, and $\tenbar[1]{u}^\circ\in H^{s_u}(\Omega;\R^d)$ with $s_u\ge1$, then
\begin{equation}\label{eq:FELS-bound-Elasticity}
  \mathcal{L}([\tenbar[2]{\sigma}^\circ_h,\tenbar[1]{u}^\circ_h];\pp)
  \lesssim h^{2\min(k+1,s_{\sigma})}\,\|\tenbar[2]{\sigma}^\circ\|_{H^{s_{\sigma}}}^2
        + h^{2\min(k+1,s_{\div})}\,\|\tenbar[1]{\rdiv}\,\tenbar[2]{\sigma}^\circ\|_{H^{s_{\div}}}^2
        + h^{2\min(m,s_u-1)}\,\|\tenbar[1]{u}^\circ\|_{H^{s_u}}^2.
\end{equation}
In particular, if $\tenbar[2]{\sigma}^\circ\in H^{k+1}$, $\tenbar[1]{\rdiv}\,\tenbar[2]{\sigma}^\circ\in H^{k+1}$, and $\tenbar[1]{u}^\circ\in H^{m+1}$, then
\(\mathcal{L}([\tenbar[2]{\sigma}^\circ_h,\tenbar[1]{u}^\circ_h];\pp)\lesssim h^{2(k+1)} + h^{2m}\), and the balanced choice $m=k+1$ yields the optimal scaling
\(\mathcal{L}([\tenbar[2]{\sigma}^\circ_h,\tenbar[1]{u}^\circ_h];\pp)\lesssim h^{2(k+1)}\).
\end{theorem}

\begin{proof}
    By the same orthogonal projection argument in $(\cdot,\cdot)_{\X_\pp}$ defined in \ref{sec:equivalence-elasticity}, the Galerkin solution $[\tenbar[2]{\sigma}^\circ_h,\tenbar[1]{u}^\circ_h]$ satisfies the best-approximation property
    \begin{equation*}
     \|[\tenbar[2]{\sigma}^\circ-\tenbar[2]{\sigma}^\circ_h,\tenbar[1]{u}^\circ-\tenbar[1]{u}^\circ_h]\|_{\X_\PP}
     = \min_{[\tenbar[2]{\tau}_h,\tenbar[1]{v}_h]\in\Sigma_h\times\U_h}
        \|[\tenbar[2]{\sigma}^\circ-\tenbar[2]{\tau}_h,\tenbar[1]{u}^\circ-\tenbar[1]{v}_h]\|_{\X_\pp}.
    \end{equation*}
    Choosing interpolation operators $\Pi_h$ for $H(\rdiv)$-conforming tensor fields and $I_h$ for $H^1$-conforming vector fields with standard approximation properties, and using the stability lemma, we obtain the a priori bound \eqref{eq:FELS-bound-Elasticity} whenever $\Sigma_h$ affords order-$k+1$ approximation in $L_2$ for tensors and their divergences, and $\U_h$ affords order-$m$ in $H^1$. This holds, for instance, for componentwise RT$_k$ spaces for stresses and CG$_m$ for displacements; see \cite[Ch.~2]{boffi2013mixed}. Balancing $m=k+1$ yields the optimal rate $\cL([\tenbar[2]{\sigma}_h,\tenbar[1]{u}^\circ_h];\pp)\lesssim h^{2(k+1)}$.  
\end{proof}

\section{Experimental details}

\subsection{Visualization of the sparsity pattern of weight}
In~\cref{fig:sparsity_patterns}, we mark by blue dots the entries of the weight $W_{\pp}$ whose magnitudes exceed $10^{-10}$. We observe that the nonzero entries are predominantly concentrated along and near the main diagonal. In addition to this banded structure, finer diagonal patterns offset from the main diagonal are visible. These arise because each finite element contributes a structured local stencil, resulting in repeated small dense sub-blocks in the global matrix.

\begin{figure}[!htbp]
\centering

\begin{subfigure}[b]{0.28\textwidth}
  \centering
  \includegraphics[width=\linewidth]{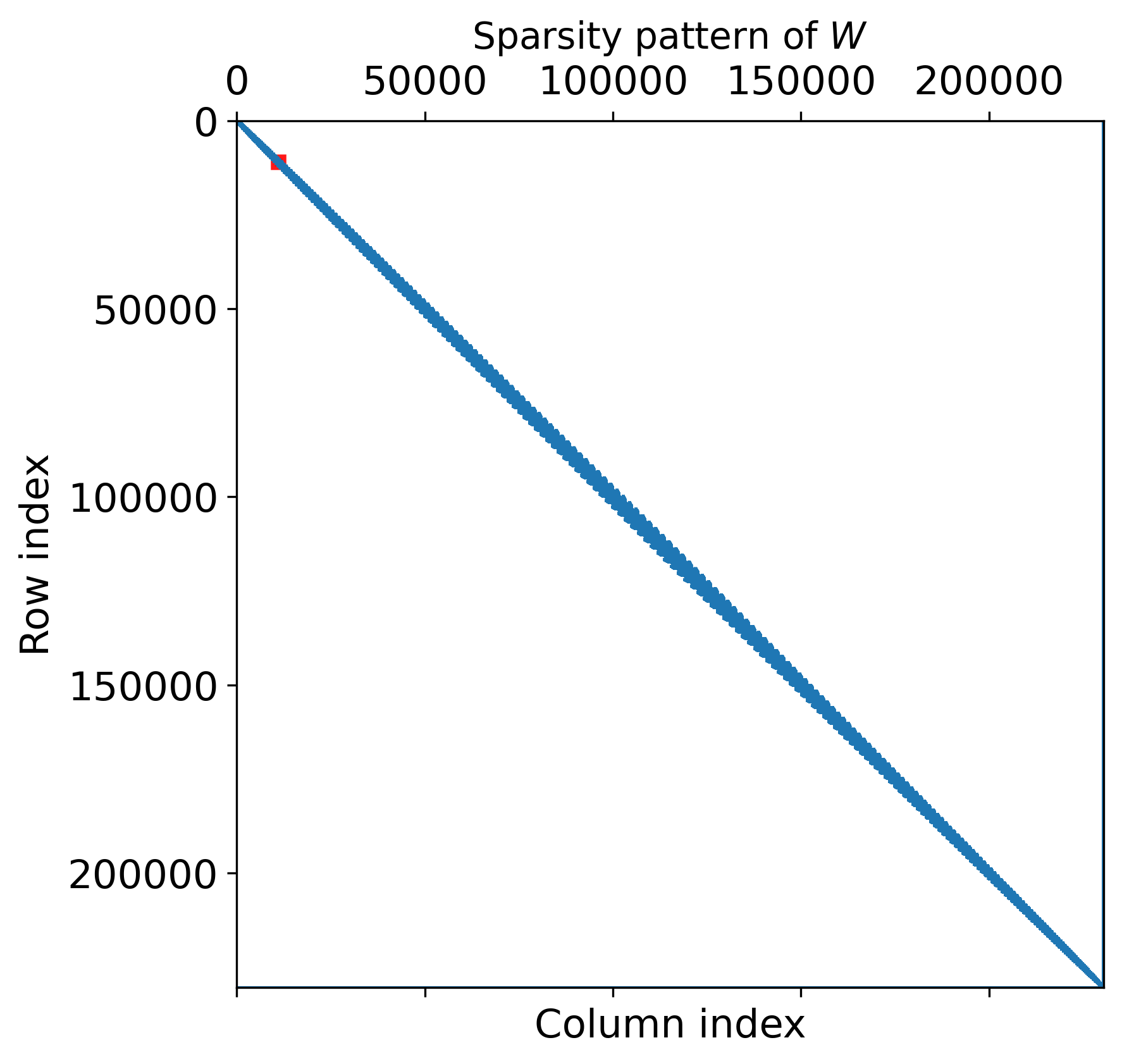}
  \subcaption{Heat conduction: full}
  \label{fig:poisson_full}
\end{subfigure}
\hspace{1.5em}
\begin{subfigure}[b]{0.27\textwidth}
  \centering
  \includegraphics[width=\linewidth]{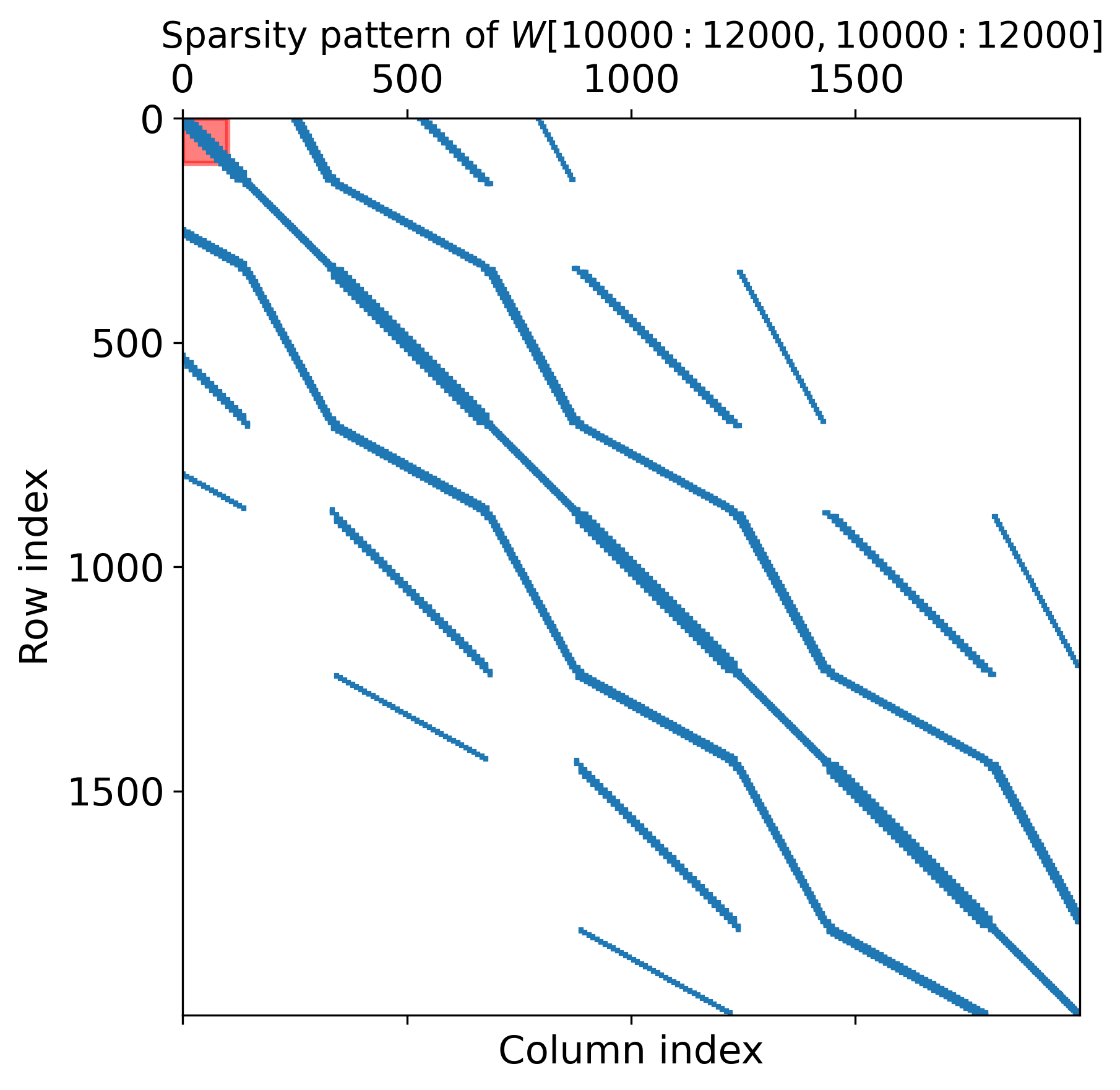}
  \subcaption{Heat conduction: 10{,}000–12{,}000}
  \label{fig:poisson_10k_12k}
\end{subfigure}
\hspace{1.5em}
\begin{subfigure}[b]{0.26\textwidth}
  \centering
  \includegraphics[width=\linewidth]{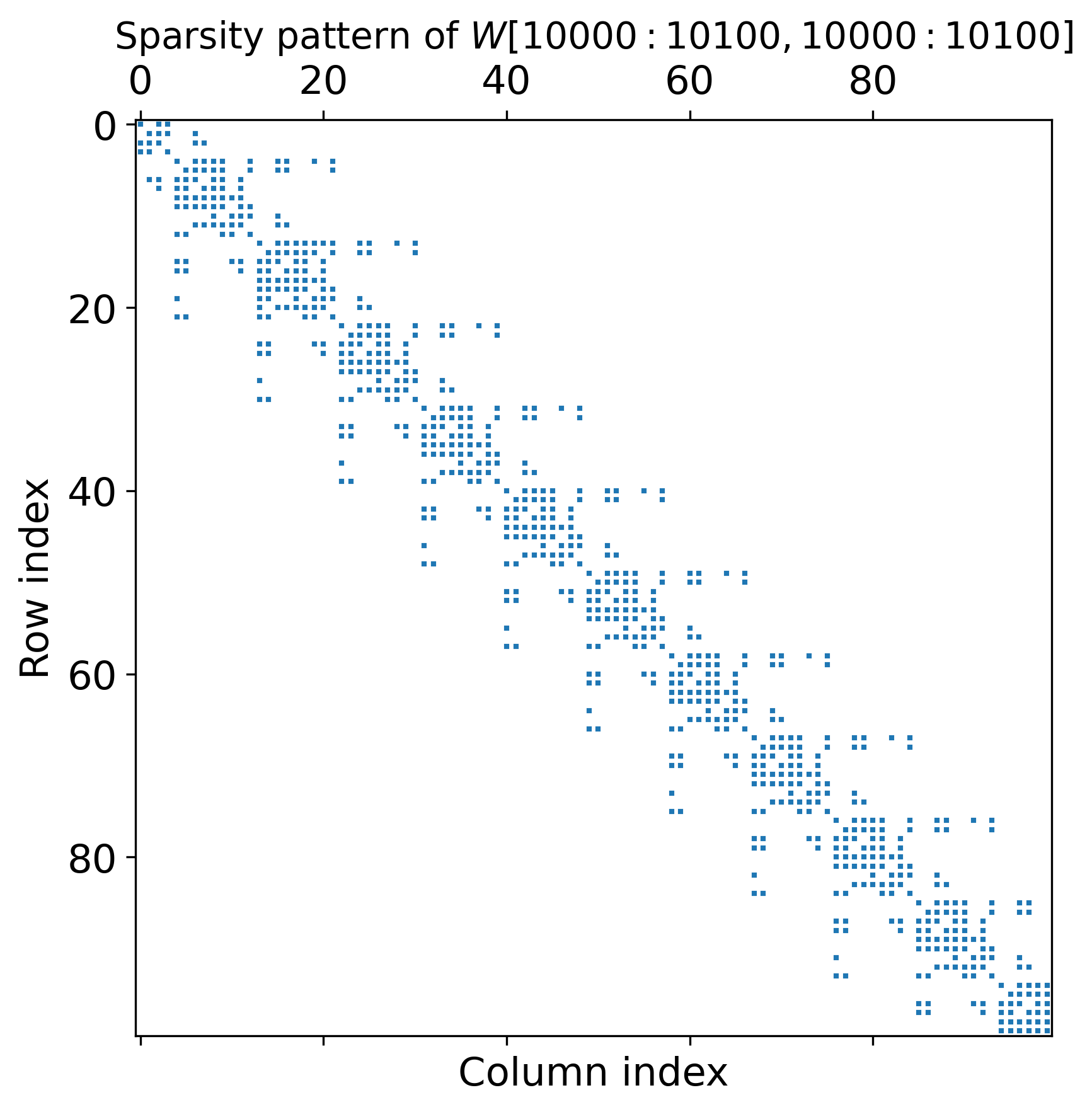}
  \subcaption{\tiny Heat conduction: 10{,}000–10{,}100}
  \label{fig:poisson_10k_10k1}
\end{subfigure}

\par\medskip

\begin{subfigure}[b]{0.28\textwidth}
  \centering
  \includegraphics[width=\linewidth]{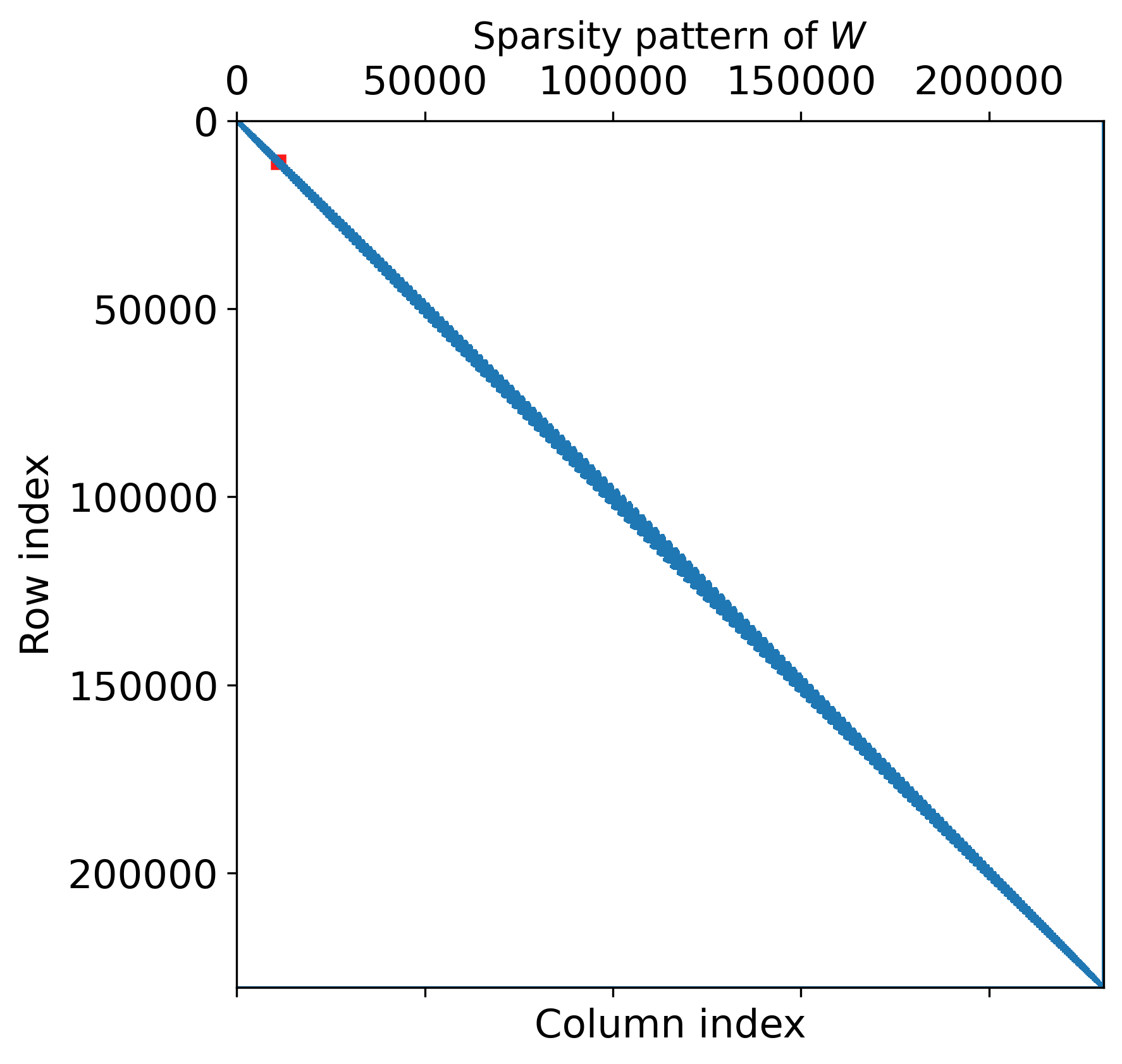}
  \subcaption{Darcy flow: full}
  \label{fig:poisson_full}
\end{subfigure}
\hspace{1.5em}
\begin{subfigure}[b]{0.27\textwidth}
  \centering
  \includegraphics[width=\linewidth]{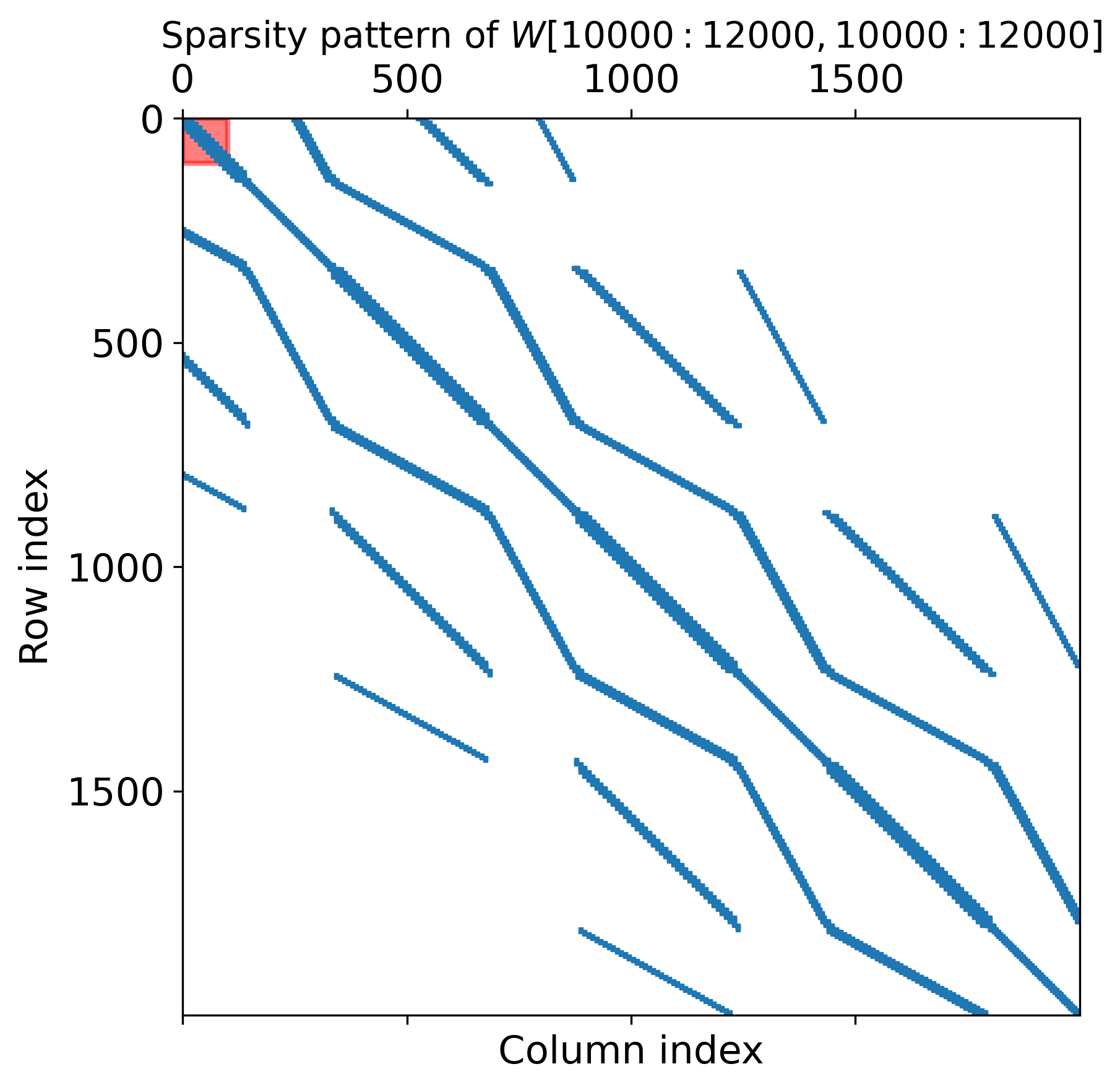}
  \subcaption{Darcy flow: 10{,}000–12{,}000}
  \label{fig:poisson_10k_12k}
\end{subfigure}
\hspace{1.5em}
\begin{subfigure}[b]{0.26\textwidth}
  \centering
  \includegraphics[width=\linewidth]{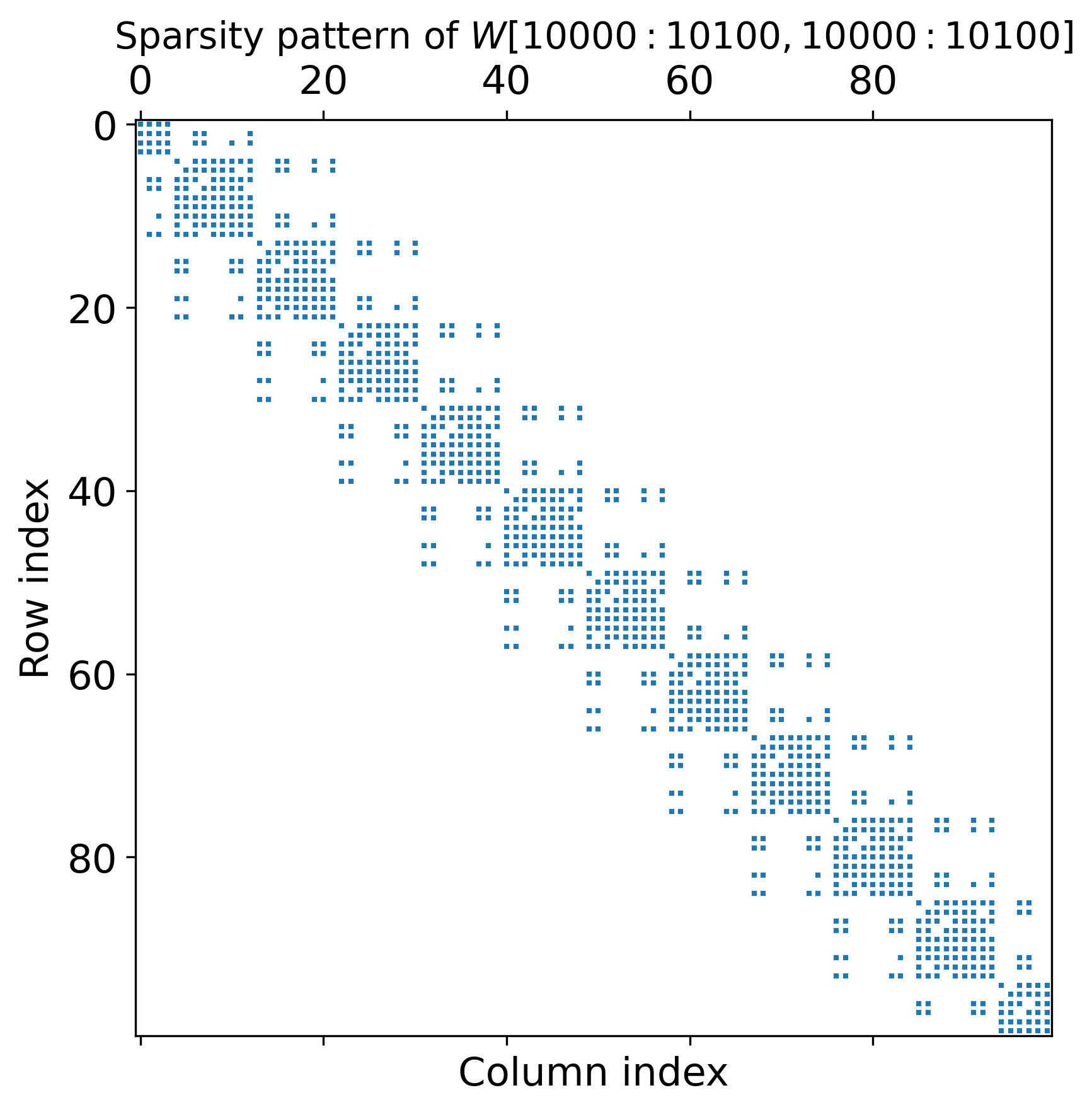}
  \subcaption{Darcy flow: 10{,}000–10{,}100}
  \label{fig:poisson_10k_10k1}
\end{subfigure}

\par\medskip

\begin{subfigure}[b]{0.28\textwidth}
  \centering
  \includegraphics[width=\linewidth]{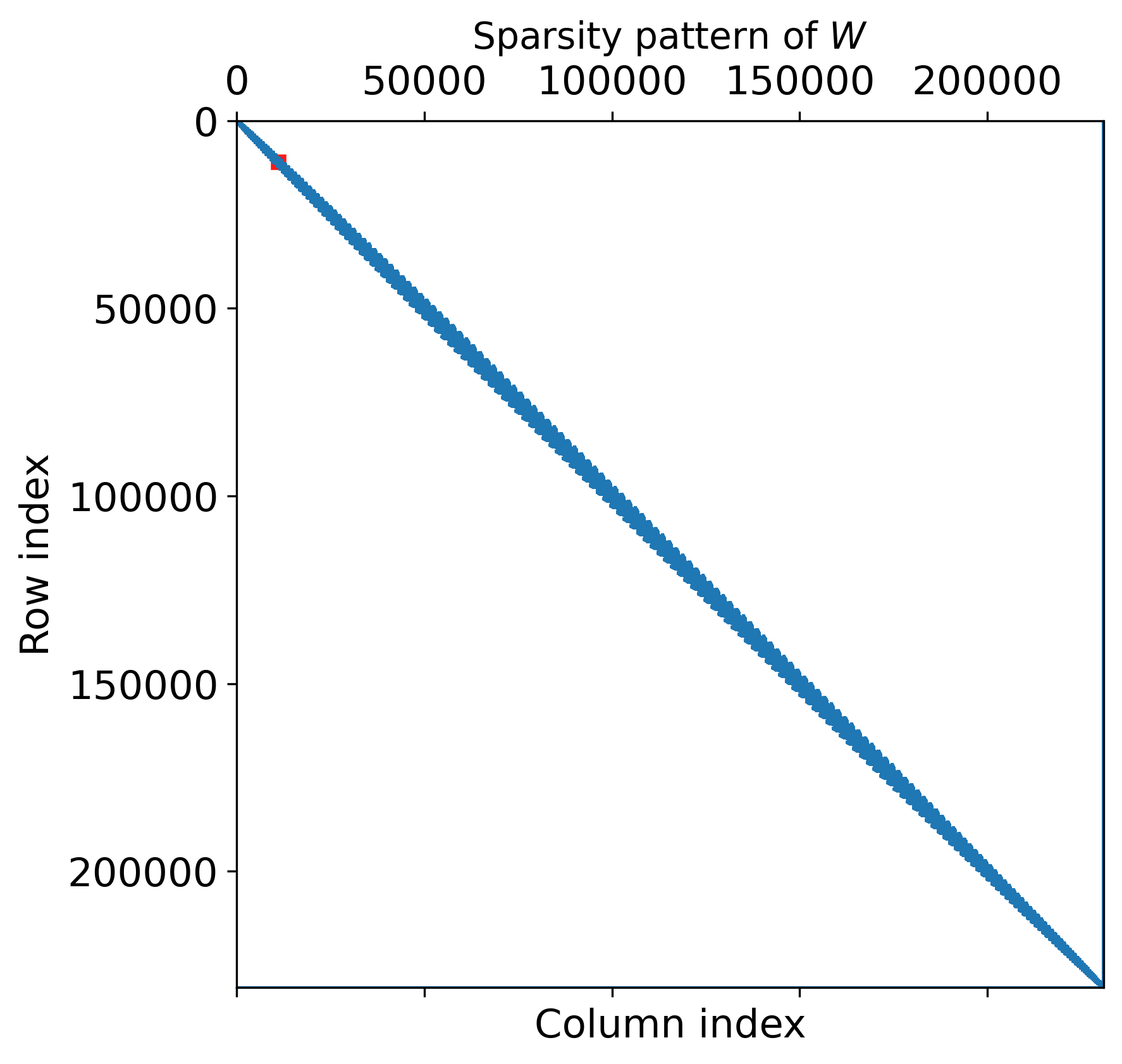}
  \subcaption{Linear elasticity: full}
  \label{fig:elasticity_full}
\end{subfigure}
\hspace{1.5em}
\begin{subfigure}[b]{0.27\textwidth}
  \centering
  \includegraphics[width=\linewidth]{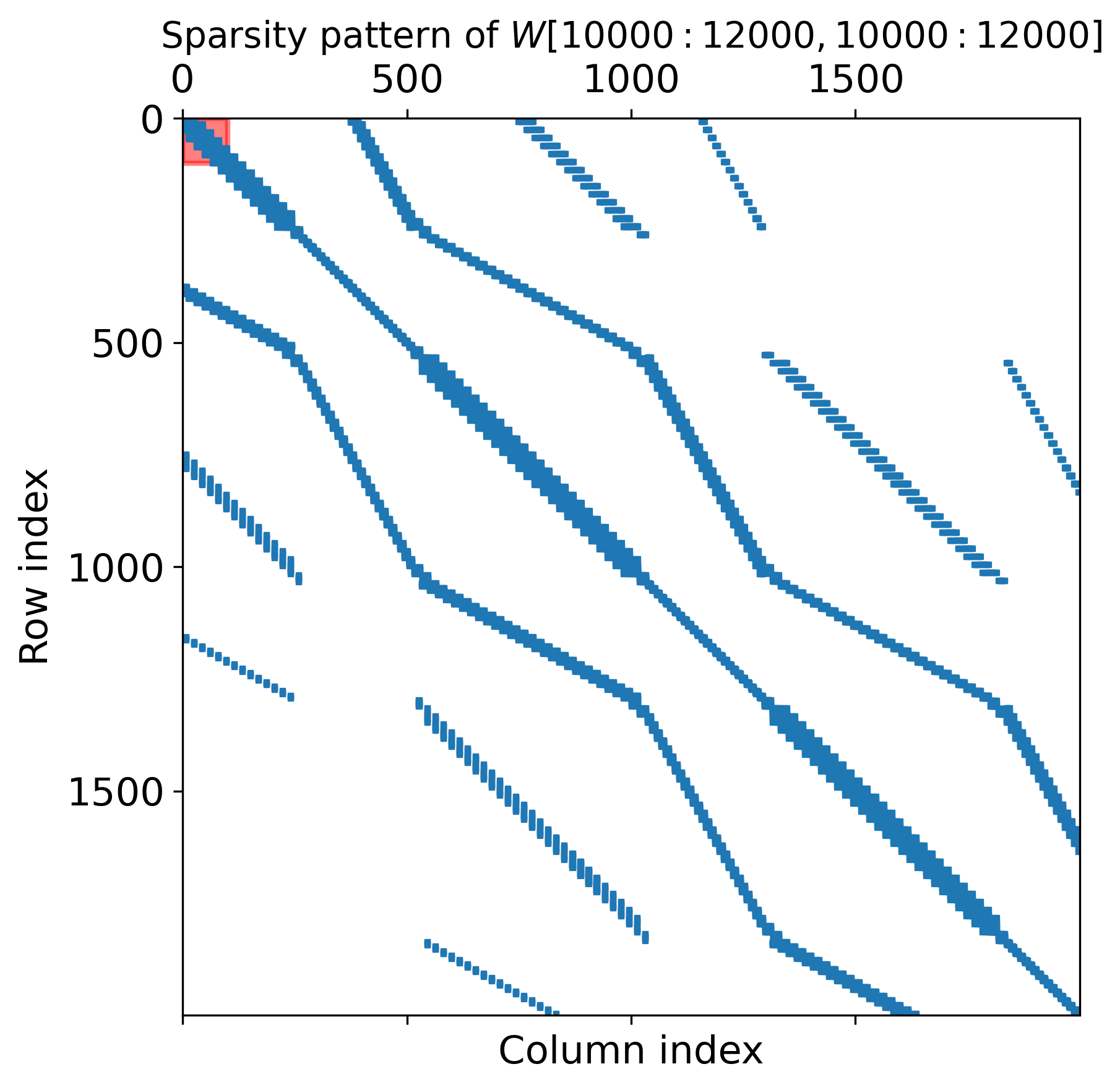}
  \subcaption{Linear elasticity: 10{,}000–12{,}000}
  \label{fig:elasticity_10k_12k}
\end{subfigure}
\hspace{1.5em}
\begin{subfigure}[b]{0.26\textwidth}
  \centering
  \includegraphics[width=\linewidth]{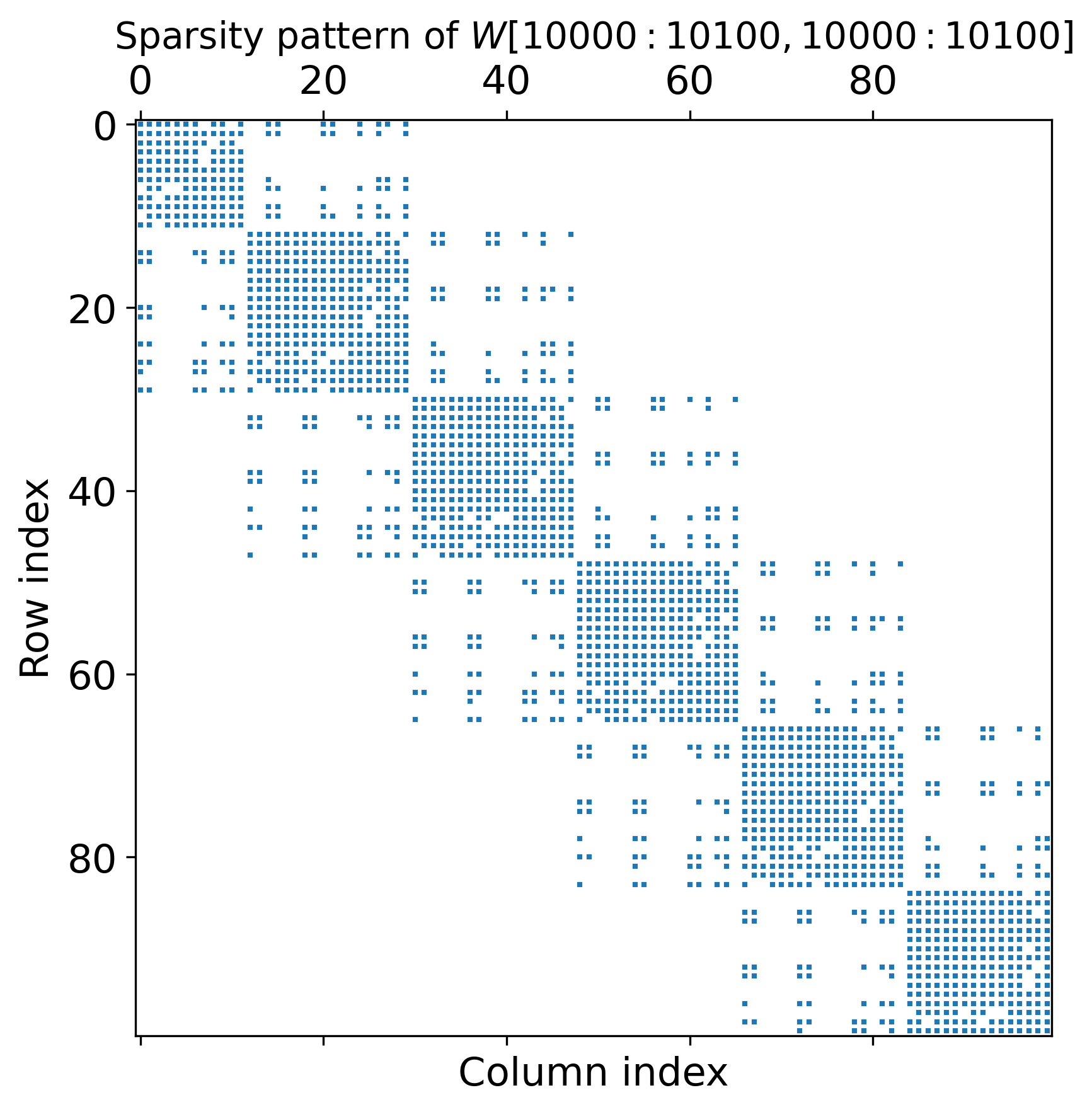}
  \subcaption{Linear elasticity: 10{,}000–10{,}100}
  \label{fig:elasticity_10k_10k1}
\end{subfigure}

\caption{Sparsity patterns of the weights $W_{\pp}$ (threshold at $10^{-10}$) shown at three zoom levels: first column -- full matrix $W_{\pp}$; second column -- submatrix $W_{\pp}[10000:12000, 10000:12000]$ (highlighted by the red patch in the first column); third column -- submatrix $W_{\pp}[10000:10100, 10000:10100]$ (highlighted by the red patch in the second column). Results are presented for three problem setups: heat conduction (top row), Darcy flow (middle row), and linear elasticity (bottom row).} 
\label{fig:sparsity_patterns}
\end{figure}

\clearpage
\subsection{Visualization of reduced weight}
In~\cref{fig:reduced_weight_grid}, we visualize the entries of a representative reduced weights $W_{\pp}^r$ for the three problem setups. The number of POD basis functions used is $128$ for heat conduction,  $512$ for Darcy flow, and $512$ for the elasticity problem.
\begin{figure}[!htbp]
\centering

\begin{subfigure}[b]{0.33\textwidth}
  \centering
  \includegraphics[width=\linewidth]{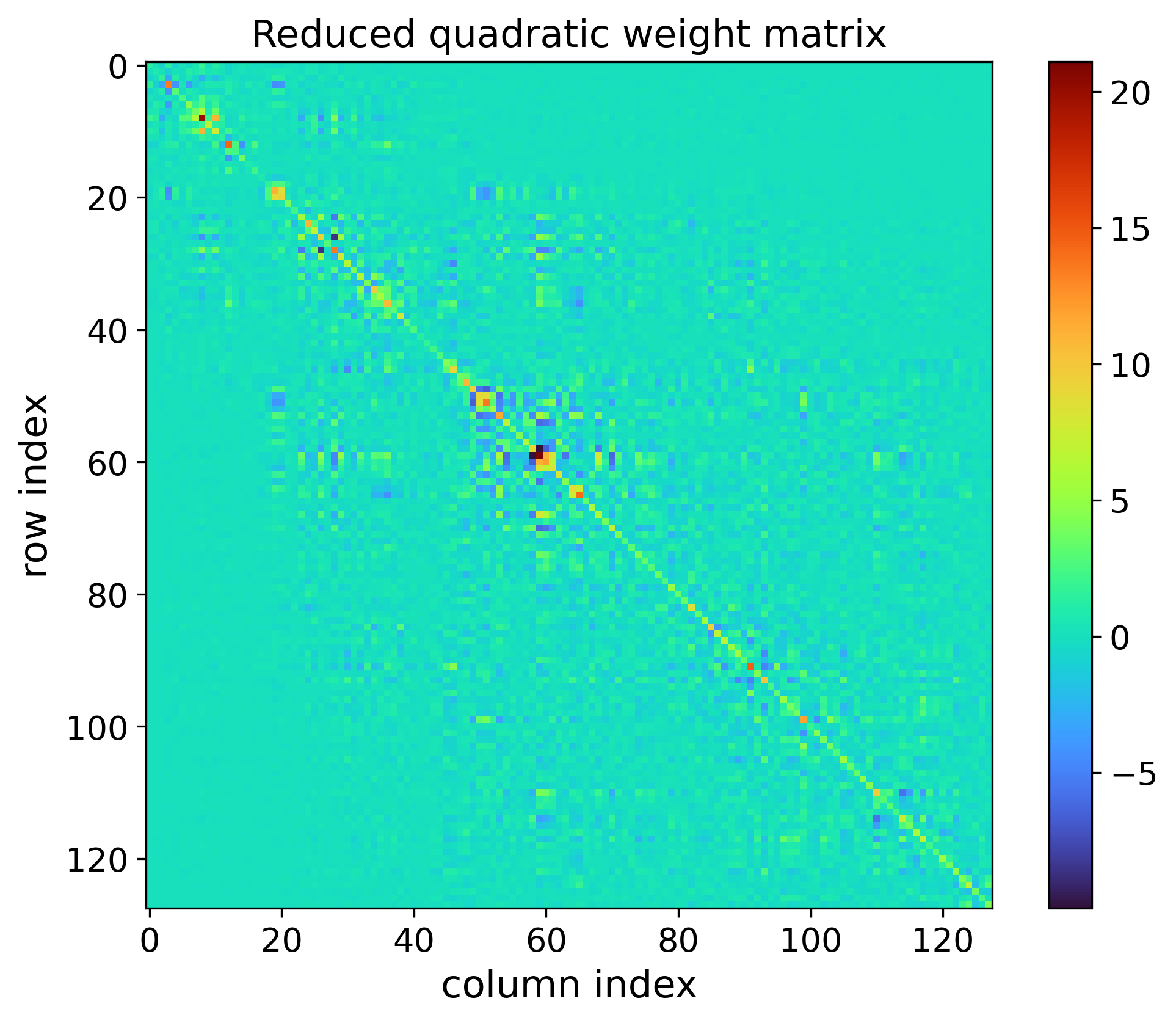}
  \subcaption{Heat conduction: full}
  \label{fig:reduced_weight_poisson_full}
\end{subfigure}
\hfill
\begin{subfigure}[b]{0.32\textwidth}
  \centering
  \includegraphics[width=\linewidth]{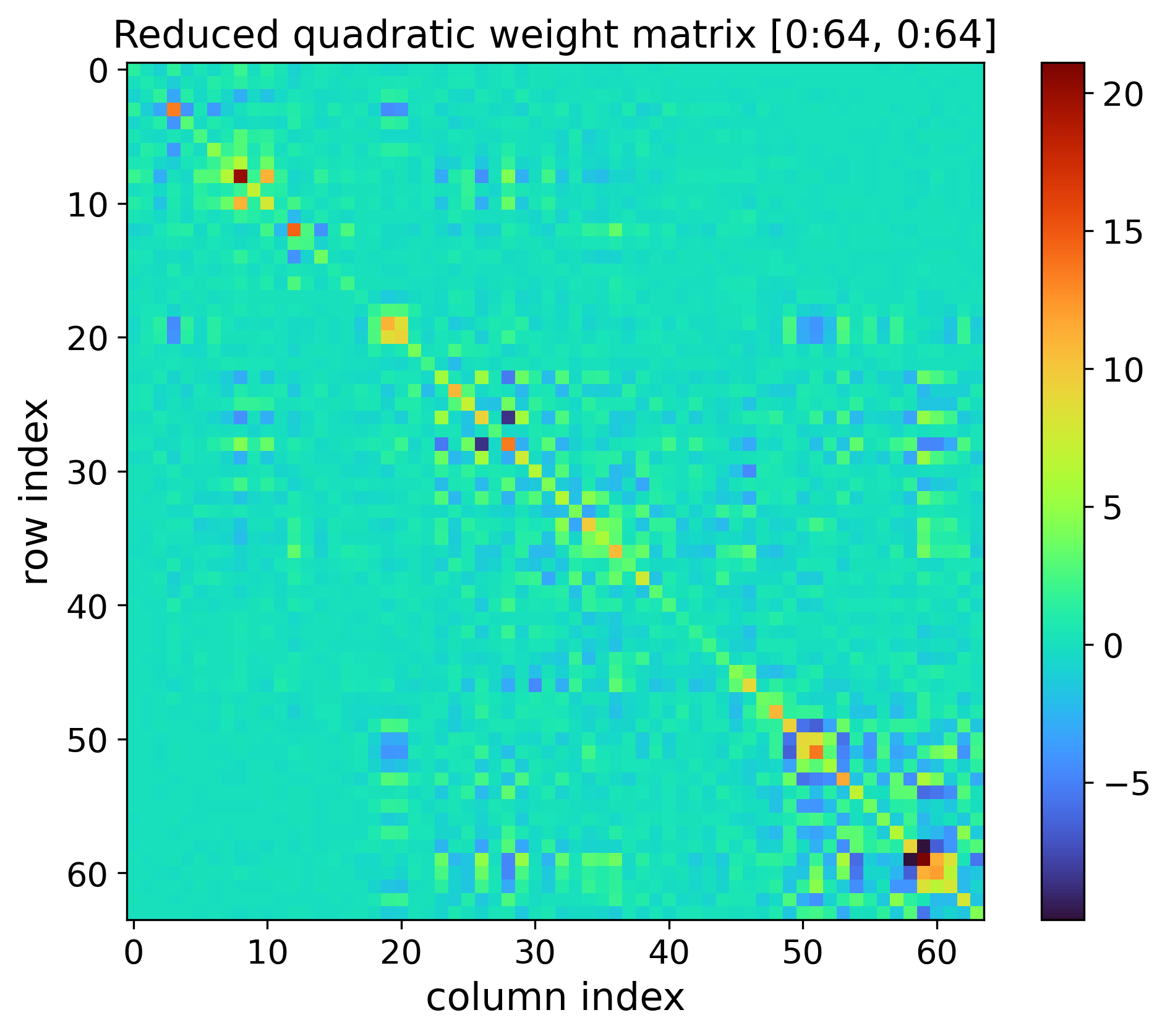}
  \subcaption{Heat conduction: 0–64}
  \label{fig:reduced_weight_poisson_0_64}
\end{subfigure}
\hfill
\begin{subfigure}[b]{0.32\textwidth}
  \centering
  \includegraphics[width=\linewidth]{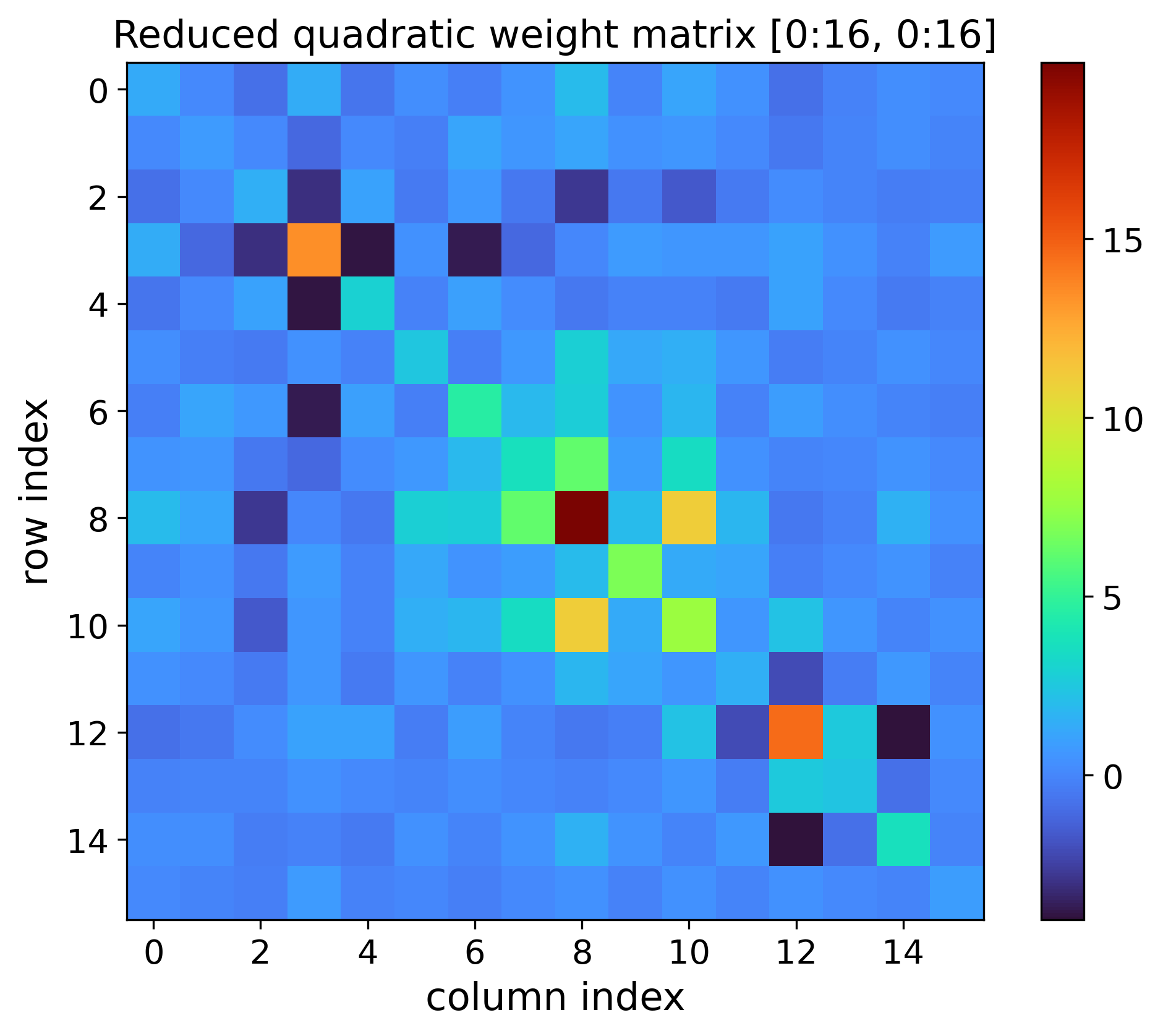}
  \subcaption{Heat conduction: 0–16}
  \label{fig:reduced_weight_poisson_0_16}
\end{subfigure}

\par\medskip

\begin{subfigure}[b]{0.33\textwidth}
  \centering
  \includegraphics[width=\linewidth]{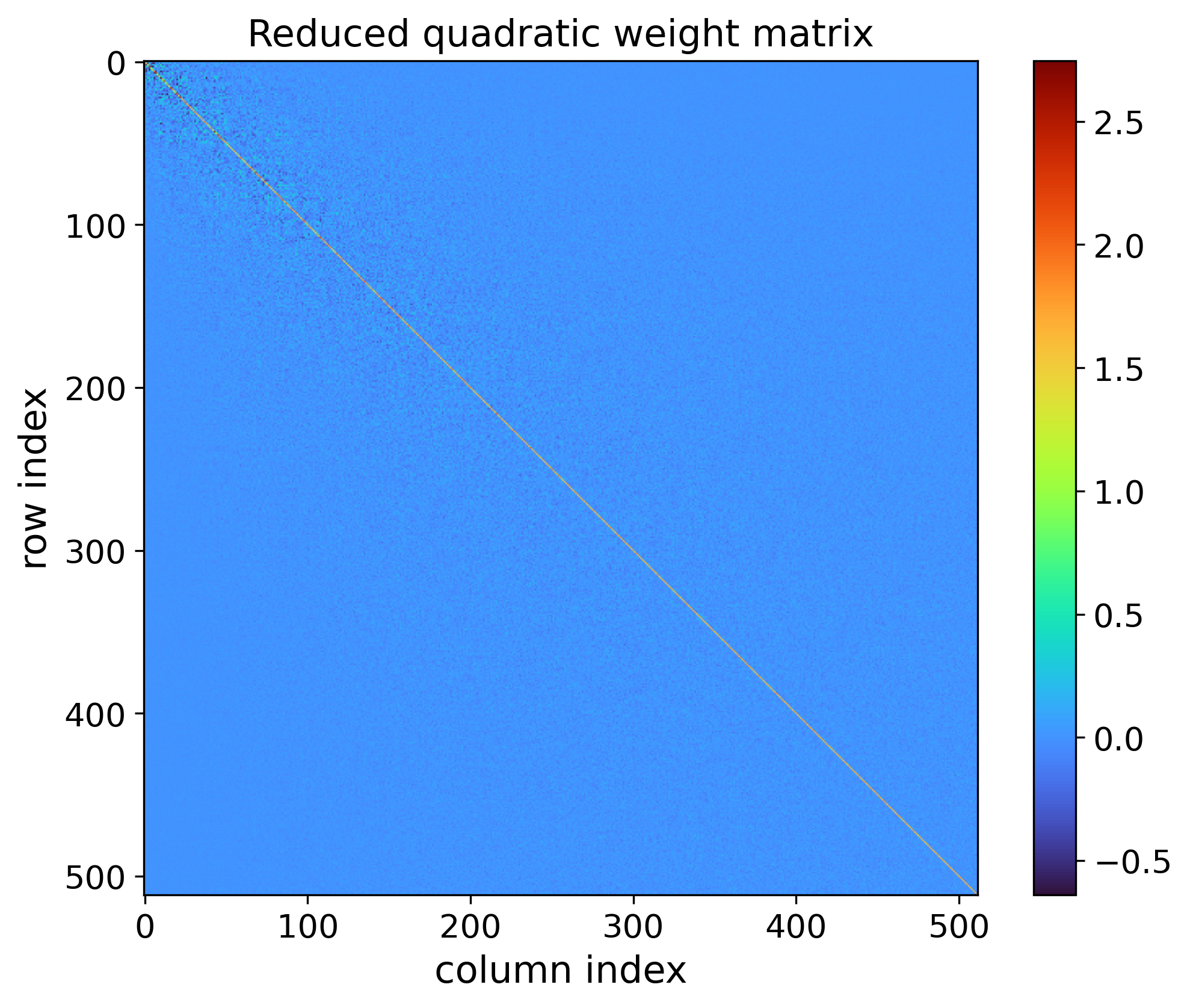}
  \subcaption{Darcy flow: full}
  \label{fig:reduced_weight_poisson_full}
\end{subfigure}
\hfill
\begin{subfigure}[b]{0.32\textwidth}
  \centering
  \includegraphics[width=\linewidth]{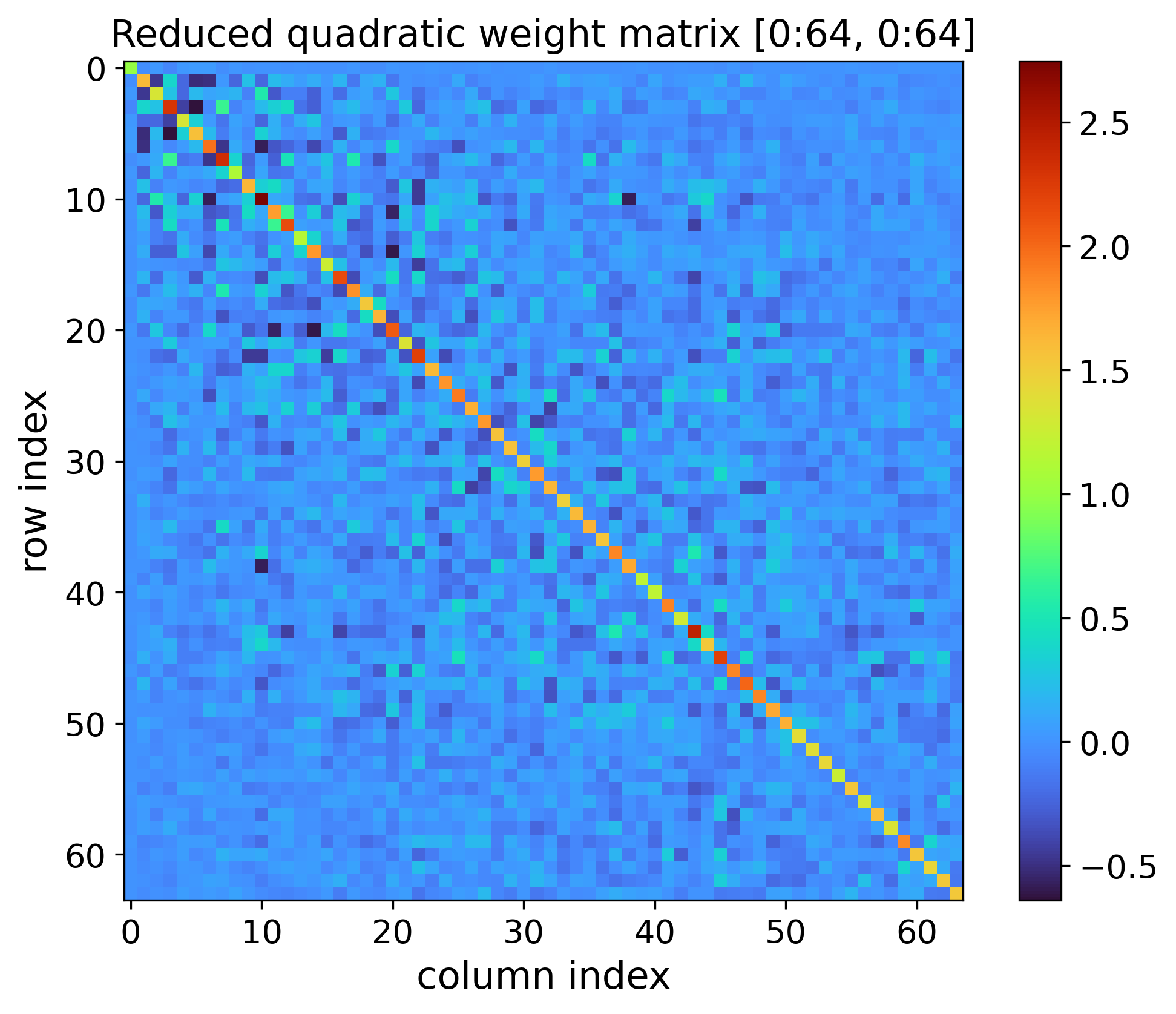}
  \subcaption{Darcy flow: 0–64}
  \label{fig:reduced_weight_poisson_0_64}
\end{subfigure}
\hfill
\begin{subfigure}[b]{0.32\textwidth}
  \centering
  \includegraphics[width=\linewidth]{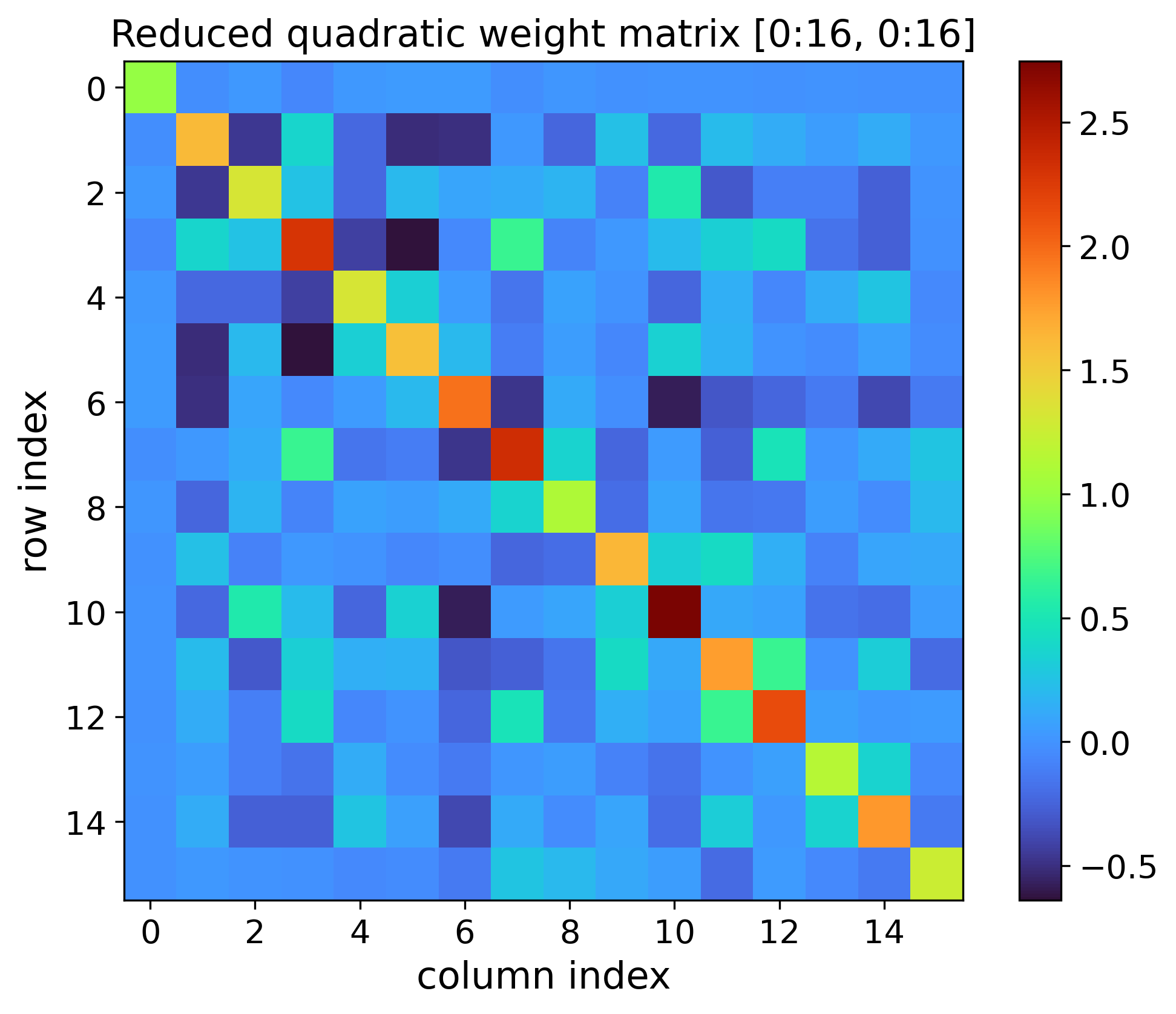}
  \subcaption{Darcy flow: 0–16}
  \label{fig:reduced_weight_poisson_0_16}
\end{subfigure}

\par\medskip

\begin{subfigure}[b]{0.33\textwidth}
  \centering
  \includegraphics[width=\linewidth]{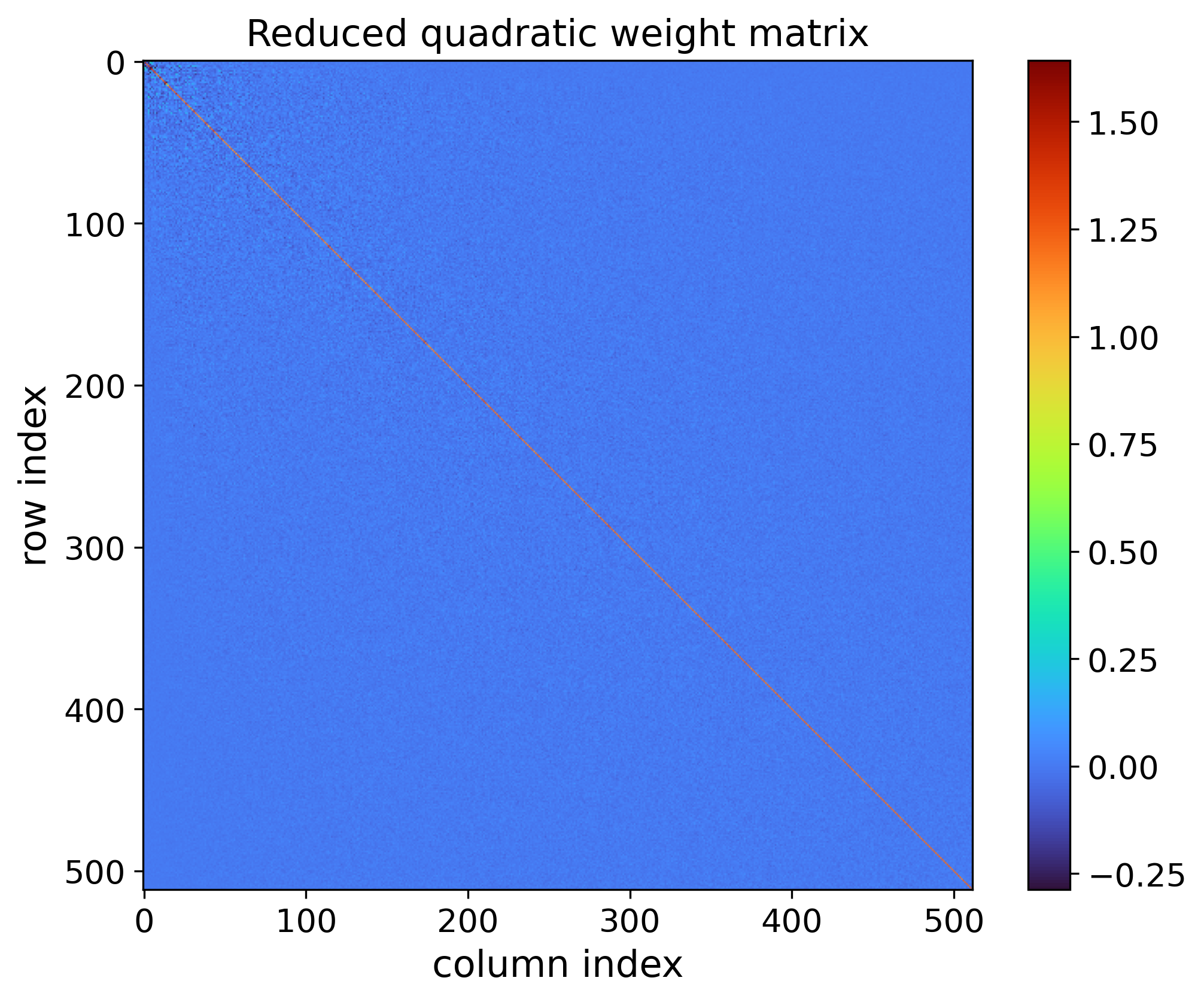}
  \subcaption{Linear elasticity: full}
  \label{fig:reduced_weight_elasticity_full}
\end{subfigure}
\hfill
\begin{subfigure}[b]{0.32\textwidth}
  \centering
  \includegraphics[width=\linewidth]{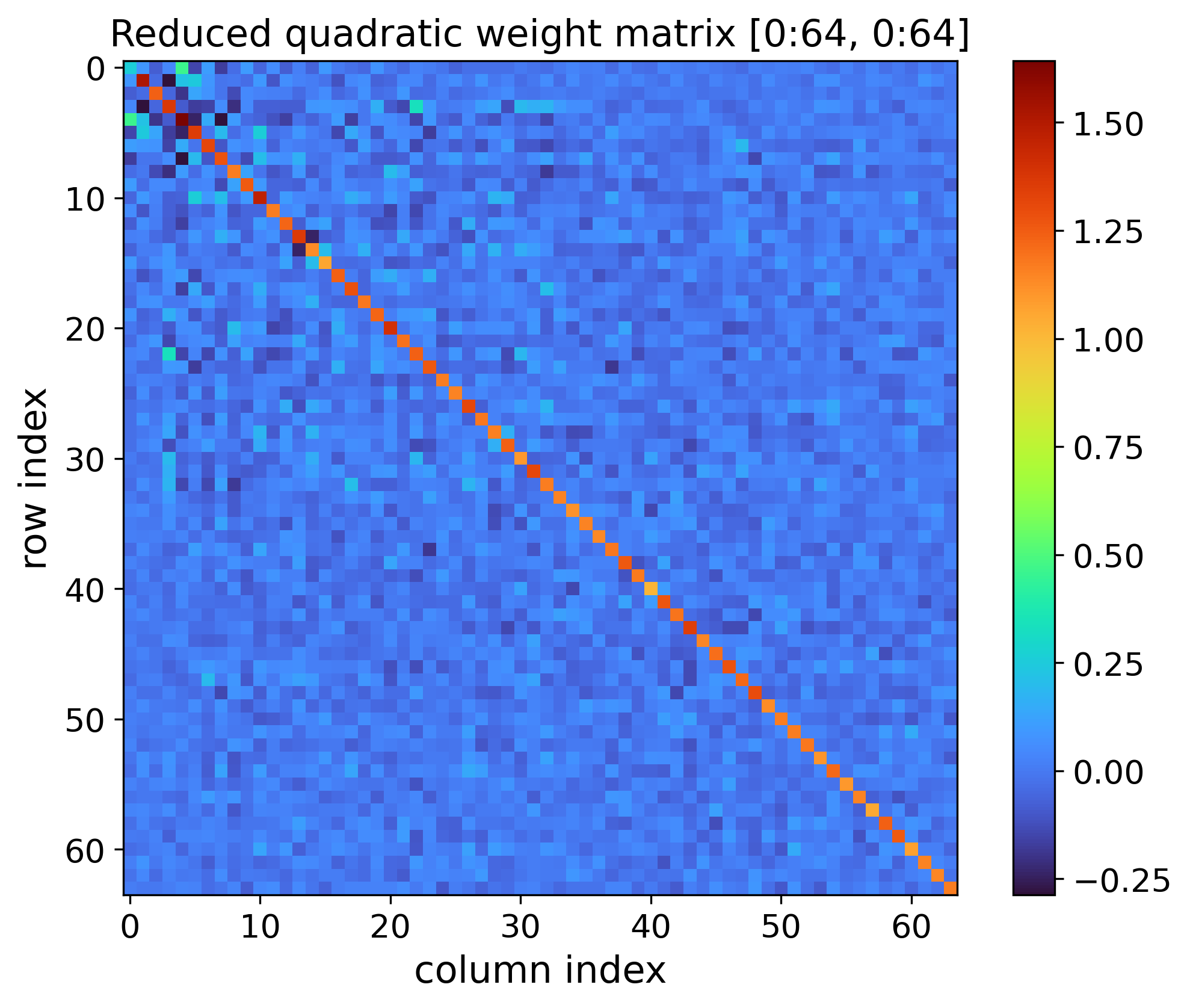}
  \subcaption{Linear elasticity: 0–64}
  \label{fig:reduced_weight_elasticity_0_64}
\end{subfigure}
\hfill
\begin{subfigure}[b]{0.32\textwidth}
  \centering
  \includegraphics[width=\linewidth]{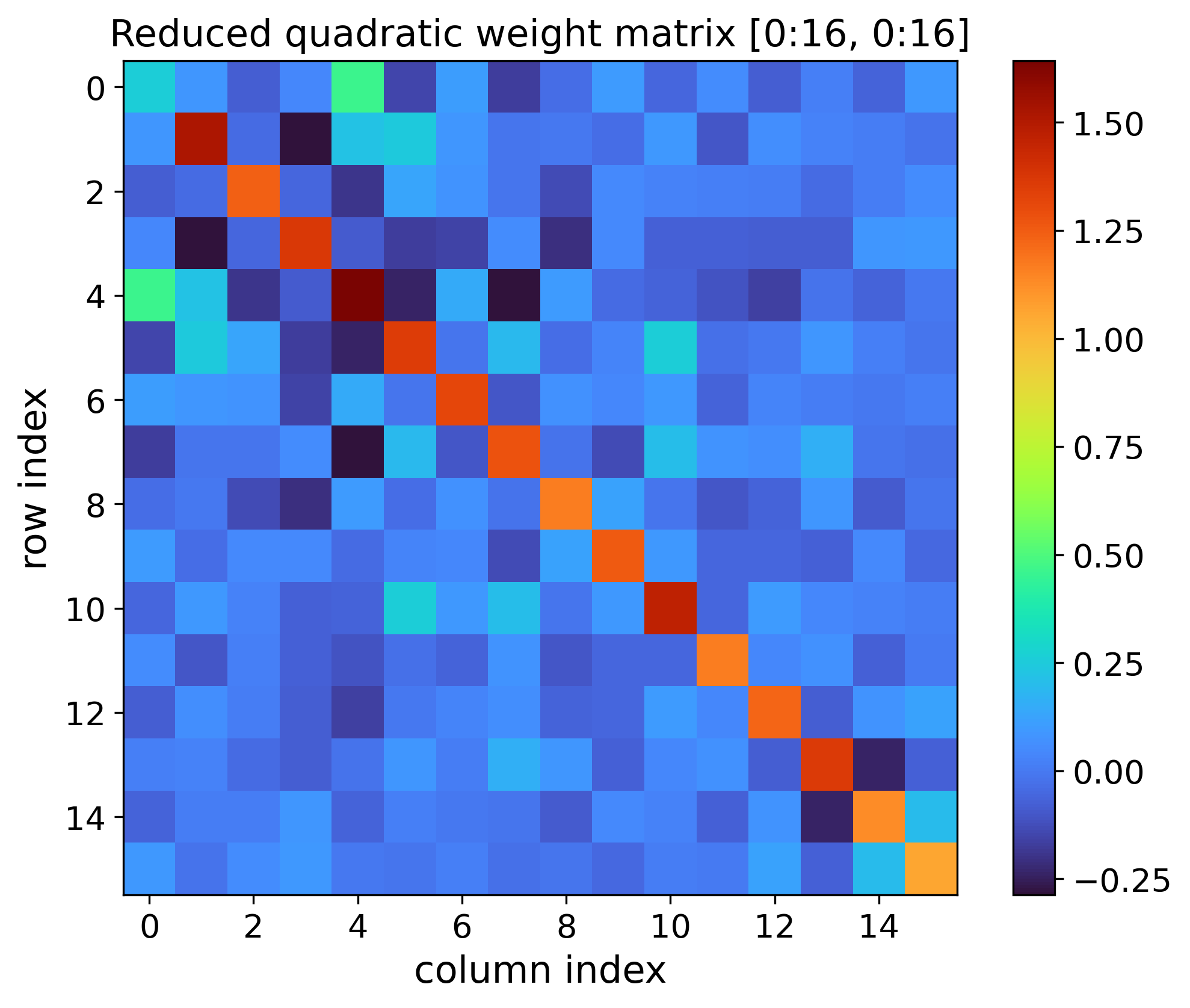}
  \subcaption{Linear elasticity: 0–16}
  \label{fig:reduced_weight_elasticity_0_16}
\end{subfigure}

\caption{Reduced weights $W_\pp^r$ in \eqref{eq:lossPoissonReduced} of heat conduction (top row), Darcy flow (middle row), and linear elasticity (bottom row) problems at three zoom levels, indicating dense matrix with dominant diagonal entries.}
\label{fig:reduced_weight_grid}
\end{figure}

\clearpage
\subsection{Architecture and training details}
\label{sec:training-details}
We report the configurations of network architectures and training details of RBNO, PCA-Net and FNO in~\cref{tab:architecture_training_details_rbno,tab:architecture_training_details_pcanet,tab:architecture_training_details_pcanet,tab:architecture_training_details_fno} along with training time and inference time in~\cref{tab:train_inference_time}. For the PCA-Net and FNO, pointwise evaluations on grid points of high-fidelity solutions are used as training labels: $500$ solutions for heat conduction, and $1000$ for Darcy flow and linear elasticity. These numbers match those used to construct the POD basis in RBNO, which is the only stage where RBNO accesses high-fidelity solutions. All models are trained for $3000$ iterations for diffusion and $6000$ iterations for elasticity, with a batch size that fits in GPU memory. As in RBNO, training and validation losses are monitored at each iteration, and the model weights corresponding to the lowest validation loss are saved.

\vspace{3em}

\begin{table}[!ht]
\centering
\begin{small}
\begin{tabular}{@{}ccc@{}}
\toprule
\textbf{Layer Type} & \textbf{Configuration} & \textbf{Output Shape (C, H, W)} \\
\midrule
\multicolumn{3}{c}{\textbf{Stationary Diffusion}} \vspace{0.5em} \\
Input & 1 channel image & (1, 129, 129) \\
Conv2d + act. & kernel=5, stride=2, out\_ch=64 & (64, 63, 63) \\
Conv2d + act. & kernel=5, stride=2, out\_ch=128 & (128, 30, 30) \\
Conv2d + act. & kernel=5, stride=3, out\_ch=256 & (256, 9, 9) \\
Conv2d + act. & kernel=3, stride=1, out\_ch=512 & (512, 7, 7) \\
Flatten & - & (512$\times$7$\times$7) \\
Linear + act. & 512$\times$7$\times$7 $\rightarrow$ 512 & (512) \\
Linear & 512 $\rightarrow$ output\_dim & (output\_dim) \\
\# trainable parameters & \multicolumn{2}{c}{15.1 M} \\
\midrule
\multicolumn{3}{c}{\textbf{Linear Elasticity}} \vspace{0.5em} \\
Input & 1 channel image & (1, 65, 129) \\
Conv2d + act. & kernel=5, stride=2, out\_ch=64 & (64, 31, 63) \\
Conv2d + act. & kernel=5, stride=2, out\_ch=128 & (128, 14, 30) \\
Conv2d + act. & kernel=5, stride=3, out\_ch=256 & (256, 4, 9) \\
Conv2d + act. & kernel=3, stride=1, out\_ch=512 & (512, 2, 7) \\
Flatten & - & (512$\times$2$\times$7) \\
Linear + act. & 512$\times$2$\times$7 $\rightarrow$ 512 & (512) \\
Linear & 512 $\rightarrow$ output\_dim & (output\_dim) \\
\# trainable parameters & \multicolumn{2}{c}{6.1 M} \\
\midrule 
initialization & \multicolumn{2}{c}{xavier\_uniform} \\
activation & \multicolumn{2}{c}{LeakyReLU (negative\_slope=0.01)} \\
\midrule
\multicolumn{3}{c}{\textbf{Training details}} \vspace{0.5em} \\
batch size & \multicolumn{2}{c}{1000 (diffusion), 500 (elasticity)} \\
optimizer & \multicolumn{2}{c}{SOAP~\cite{vyas2024soap} (lr=5e-3, betas=(.95, .95), weight\_decay=.01, precondition\_frequency=5)} \\
lr scheduler & \multicolumn{2}{c}{StepLR (step\_size=50, gamma=0.95 (heat conduction), 0.9 (Darcy flow \& elasticity))} \\
\bottomrule
\end{tabular}
\end{small}
\caption{Architecture and training details for RBNO}
\label{tab:architecture_training_details_rbno}
\end{table}

\begin{table}[!ht]
\begin{center}
\begin{small}
    \begin{tabular}{@{}cccc@{}}
    \toprule
    \multirow{2}{*}{\( \text{} \)} & \multicolumn{3}{c}{\( \text{Dataset} \)} \\
    \cmidrule(l){2-4}
                    & Heat Conduction  & Darcy Flow &  Linear Elasticity\\
    \midrule    
    input dim     & $20$     & $512$  & $512$ \\
    output dim      & $128$    &   $512$        & $512$      \\
    hidden dim     & $(1024, 2048, 1024)$      &    $(1024, 2048, 1024)$   & $(1024, 2048, 1024)$     \\
    activation function        & SELU    &    SELU      & SELU      \\
    \# trainable parameters  & 4.4 M    &    5.2 M      & 5.2 M      \\
    \midrule
    \multicolumn{4}{c}{\textbf{Training details}} \vspace{0.5em} \\
    batch size               & $500$       &  $1000$    & $500$   \\ 
    AdamW~\cite{loshchilov2017decoupled} (lr, weight decay)  & $(10^{-3}, 10^{-2})$     &     $(10^{-3}, 10^{-2})$             & $(10^{-4},10^{-2})$  \\
    StepLR (gamma, step size)  & $(0.9,50)$               &     $(0.9,50)$          & $(0.99, 100)$  \\
    \bottomrule
    \end{tabular}
\end{small}        
\end{center}
\caption{Architecture and training details for PCA-Net}
\label{tab:architecture_training_details_pcanet}
\end{table}

\begin{table}[!ht]
\begin{center}
\begin{small}
    \begin{tabular}{@{}ccc@{}}
    \toprule
    \multirow{2}{*}{\( \text{} \)} & \multicolumn{2}{c}{\( \text{Dataset} \)} \\
    \cmidrule(l){2-3}
                    & Stationary Diffusion  & Linear Elasticity\\
    \midrule    
    number of modes           & $$(32,32)$$       & $(32,32)$ \\
    in channels      & $3$              & $3$      \\
    out channels     & $3$              & $6$      \\
    hidden channels  & $128$             & $128$     \\
    number of layers        & $4$              & $4$      \\
    lifting channel ratio   & $2$              & $2$    \\
    projection channel ratio & $2$             & $2$    \\
    activation function      & GELU            & GELU    \\
    \# trainable parameters &  35.9 M            & 35.9 M    \\
    \midrule
    \multicolumn{3}{c}{\textbf{Training details}} \vspace{0.5em} \\
    batch size               & $10$            & $10$   \\ 
    AdamW~\cite{loshchilov2017decoupled} (lr, weight decay)  & $(10^{-3}, 10^{-2})$                       & $(10^{-3},10^{-2})$  \\
    StepLR (gamma, step size)  & $(0.9,50)$                              & $(0.9,50)$  \\
    \bottomrule
    \end{tabular}
\end{small}        
\end{center}
\caption{Architecture and training details for FNO}
\label{tab:architecture_training_details_fno}
\end{table}

\begin{table}[h!]
\centering
\footnotesize
\begin{tabular}{|c|l|S[table-format=2.1]|S[table-format=2.1]|}
\hline
\textbf{Problem} & \textbf{Method} & {\textbf{Training time (min)}} & {\textbf{Inference time/sample (ms)}} \\ \hline
\multirow{5}{*}{Heat Conduction} 
& PCA-Net & 0.2 & 0.2 \\ 
& FNO & 15 & 3.2 \\ 
& RBNO (RB coef.) & 7 & 2.6 \\ 
& RBNO (residual) & 15 & 2.6 \\ 
& RBNO (both) & 19 & 2.6 \\ \hline
\multirow{5}{*}{Darcy Flow} 
& PCA-Net & 0.2 & 1.0 \\ 
& FNO & 15 & 3.2 \\ 
& RBNO (RB coef.) & 8 & 6.0 \\ 
& RBNO (residual) & 14 & 6.0 \\ 
& RBNO (both) & 19 & 6.0 \\ \hline
\multirow{5}{*}{Linear Elasticity} 
& PCA-Net & 1.8 & 1.0 \\ 
& FNO & 16 & 1.6 \\ 
& RBNO (RB coef.) & 10 & 6.0 \\ 
& RBNO (residual) & 18 & 6.0 \\ 
& RBNO (both) & 23 & 6.0 \\ \hline
\end{tabular}
\caption{Comparisons of training and inference time for different methods.}
\label{tab:train_inference_time}
\end{table}

\clearpage
\subsection{$H(\rdiv)\times H^1$ reconstruction from CG$_1\times$CG$_1$ representations}
\label{sec:projection-details}

Starting from a pointwise evaluation of the approximate solution $(\hat{\sigma}^{\circ}, \hat{u}^{\circ})$ and its FE representation with CG$_1\times$CG$_1$ elements, we construct an $H(\mathrm{div})\times H^1$-conforming approximation by projecting into the space $\HH_h=$RT$^\circ_1\times$CG$^\circ_2$ in the $H(\mathrm{div})\times H^1$ norm, enforcing the appropriate homogeneous boundary conditions. 

For the diffusion problems, we solve for $(\sigma^{\circ},u^{\circ}) \in \HH_h$:
\begin{align*}
    &\quad (\sigma^{\circ}, \tau)_{\Omega} + (\div{\sigma^{\circ}}, \div{\tau})_{\Omega} + (u^{\circ}, v)_{\Omega} + (\rgrad{u^{\circ}}, \rgrad{v})_{\Omega} \\
    &=
    (\hat{\sigma}^{\circ}, \tau)_{\Omega} + (\div{\hat{\sigma}^{\circ}}, \div{\tau})_{\Omega} + (\hat{u}^{\circ}, v)_{\Omega} + (\rgrad{\hat{u}^{\circ}}, \rgrad{v})_{\Omega}, \quad (\tau, v)\in \HH_h, 
\end{align*}
with $\sigma^{\circ}\cdot n = 0$ on $\Gamma_N$ and $u^{\circ}=0$ on $\Gamma_D$. 

For the linear elasticity problem, we solve for $(\tenbar[2]{\sigma^{\circ}}, \tenbar[1]{u^{\circ}}) \in \HH_h =(\text{RT}^\circ_1)^2\times(\text{CG}^\circ_2)^2$:
\begin{align*}
    &\quad (\tenbar[1]{\rdiv} \tenbar[2]{\sigma^{\circ}}, \tenbar[1]{\rdiv} \tenbar[2]{\tau^{\circ}})_{\Omega} + (\tenbar[2]{\sigma^{\circ}}, \cC_{\bar{\pp}}^{-1} \tenbar[2]{\tau^{\circ}})_{\Omega} + 
    (\tenbar[2]{\e}(\tenbar[1]{u^\circ}), \cC_{\bar{\pp}}\tenbar[2]{\e}(\tenbar[1]{ v^\circ}))_{\Omega} \\
    &= (\tenbar[1]{\rdiv} \tenbar[2]{\hat{\sigma}^{\circ}}, \tenbar[1]{\rdiv} \tenbar[2]{\tau^{\circ}})_{\Omega} + (\tenbar[2]{\hat{\sigma}^{\circ}}, \cC_{\bar{\pp}}^{-1} \tenbar[2]{\tau^{\circ}})_{\Omega} + 
    (\tenbar[2]{\e}(\tenbar[1]{\hat{u}^\circ}), \cC_{\bar{\pp}}\tenbar[2]{\e}(\tenbar[1]{ v^\circ}))_{\Omega}, \quad (\tenbar[2]{\tau^{\circ}}, \tenbar[1]{ v^\circ})\in \HH_h,
\end{align*}
with $\tenbar[2]{\sigma}^{\circ}\cdot \tenbar[1]{n} = (0, 0)^{\top}$ on $\Gamma_N$ and $\tenbar[1]{u^\circ}=(0, 0)^{\top}$ on $\Gamma_D$.

We then compare the residual loss of the original FE solution with RT$_1\times$CG$_2$ elements, its CG$_1\times$CG$_1$ representation, and the reconstructed RT$_1\times$CG$_2$ solution obtained via this projection. We also compare their relative errors in $\LL = L_2\times L_2$-norm, $L_2 \times H^1$-norm, and $\HH = H(\mathrm{div})\times H^1$-norm. We report the corresponding results in~\cref{tab:evaluation_reconstruction_error_merged}.

\begin{table}[!ht]
\centering
\begin{tabular}{|l|S[table-format=1.2e-2, scientific-notation=true]|S[table-format=1.2e-2, scientific-notation=true]|S[table-format=1.2e-2, scientific-notation=true]|}
\hline
\textbf{Metric} & \textbf{Original RT$_1\times $CG$_2$} & \textbf{CG$_1\times $CG$_1$} & \textbf{Recons. RT$_1\times $CG$_2$} \\ \hline

\multicolumn{4}{|c|}{\textbf{Heat Conduction}} \\ \hline
$\| \sigma^{\circ} - ( \pp\nabla u^{\circ} + \pp\nabla w - \bz + \bbf_1) \|^2_{L^2}$  & 1.12e-04 & 8.26e-03 & 8.26e-03 \\ 
$\|\div  \sigma^{\circ} + f_2\|^2_{L^2}$                                             & 7.78e-10 & 2.07e+01 & 2.07e+01 \\ 
Relative error (L2-L2)                                                            & {-}      & 1.71e-01 & 1.71e-01 \\ 
Relative error (L2-H1)                                                      & {-}      & 1.39e-01 & 1.39e-01 \\ 
Relative error (Hdiv-H1)                                                      & {-}      & 3.81e+00 & 3.81e+00 \\ \hline

\multicolumn{4}{|c|}{\textbf{Darcy Flow}} \\ \hline
$\| \sigma^{\circ} - ( \pp\nabla u^{\circ} + \pp\nabla w - \bz + \bbf_1) \|^2_{L^2}$      & 6.13e-06 & 1.19e-03 & 1.19e-03 \\ 
$\|\div  \sigma^{\circ}+ f_2\|^2_{L^2}$                                                 & 4.00e-03 & 2.80e+00 & 2.80e+00 \\ 
Relative error (L2-L2)                                                                            & {-}      & 2.12e-03 & 2.12e-03 \\ 
Relative error (L2-H1)                                                                       & {-}      & 1.05e-02 & 1.05e-02 \\ 
Relative error (Hdiv-H1)                                                                       & {-}      & 1.40e-01 & 1.40e-01 \\ \hline

\multicolumn{4}{|c|}{\textbf{Linear Elasticity}} \\ \hline
$\|\cC^{-1/2}(\tenbar[2]{\sigma^{\circ}} + \tenbar[2]{z}) - \cC^{1/2}(\tenbar[2]{\varepsilon}(\tenbar[1]{u}^{\circ} + \tenbar[1]{w}))\|_{{L^2}}^2$   & 7.75e-04 & 6.44e-03 & 7.48e-03 \\ 
$\| \tenbar[1]{\rdiv} \tenbar[2]{\sigma^{\circ}} + \tenbar[1]{f} \|_{{L^2}}^2 $                                                                      & 6.77e-06 & 2.88e+01 & 2.88e+01 \\ 
Relative error (L2-L2)                                                                                                                        & {-}      & 8.99e-03 & 8.49e-03 \\ 
Relative error (L2-H1)                                                                                                                   & {-}      & 4.55e-02 & 4.58e-02 \\
Relative error (Hdiv-H1)                                                                                                                   & {-}      & 2.99e+00 & 2.99e+00 \\ \hline
\end{tabular}
\caption{Reconstruction quality for one test sample on the two residual loss terms and relative errors for the original RT$_1\times$CG$_2$ solution, the intermediate CG$_1\times$CG$_1$ representation, and the reconstructed RT$_1\times$CG$_2$ from CG$_1\times$CG$_1$ representation.}
\label{tab:evaluation_reconstruction_error_merged}
\end{table}

Across all three problems, we observe that the CG$_1\times$CG$_1$ representation violates the divergence-related loss terms (second row in \cref{tab:evaluation_reconstruction_error_merged} for each problem).
 Projecting this representation into the RT$_1\times$CG$_2$ space does not reduce the residual loss or the errors in $\HH$-norm (rigth column), even though the corresponding vectors of degrees of freedom remain close in the Euclidean sense, as visualized in~\cref{fig:dof_comparison}. 
We have also tested DOLFINx's \texttt{interpolate} function and mesh refinement for transferring CG$_1\times$CG$_1$ representations into RT$_1\times$CG$_2$; the qualitative conclusions are unchanged.

\begin{figure}[!ht]
    \centering
    \begin{subfigure}[b]{0.42\linewidth}
        \centering
        \includegraphics[width=\linewidth]{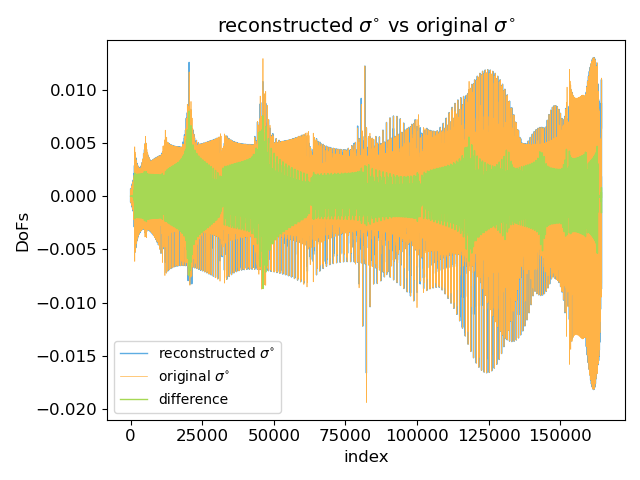}
        \caption{DoFs of $\sigma^{\circ}$ (heat conduction)}
        \label{fig:dof_comparison_sigma_poisson1}
    \end{subfigure}
    \hspace{1.5em}
    \begin{subfigure}[b]{0.42\linewidth}
        \centering
        \includegraphics[width=\linewidth]{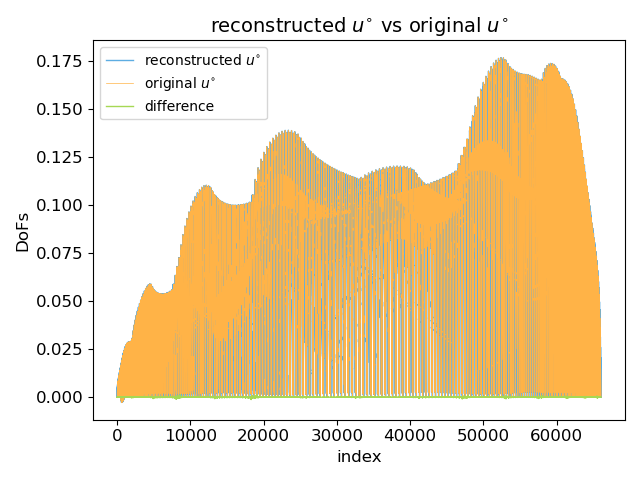}
        \caption{DoFs of $u^{\circ}$ (heat conduction)}
        \label{fig:dof_comparison_u_poisson1}
    \end{subfigure}

    \vspace{2mm} 

    \begin{subfigure}[b]{0.42\linewidth}
        \centering
        \includegraphics[width=\linewidth]{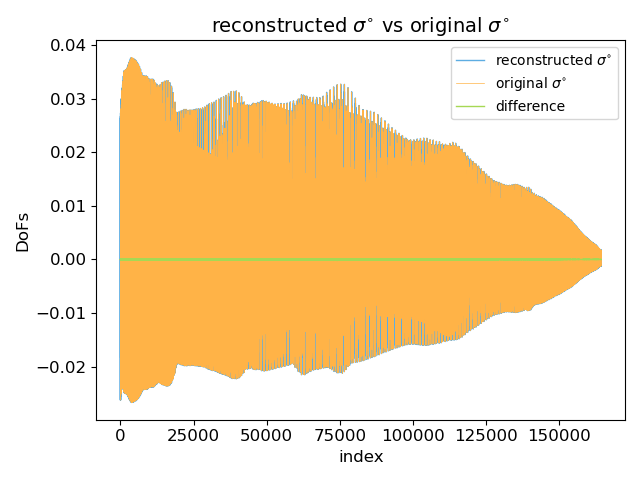}
        \caption{DoFs of $\sigma^{\circ}$ (Darcy flow)}
        \label{fig:dof_comparison_sigma_poisson2}
    \end{subfigure}
    \hspace{1.5em}
    \begin{subfigure}[b]{0.42\linewidth}
        \centering
        \includegraphics[width=\linewidth]{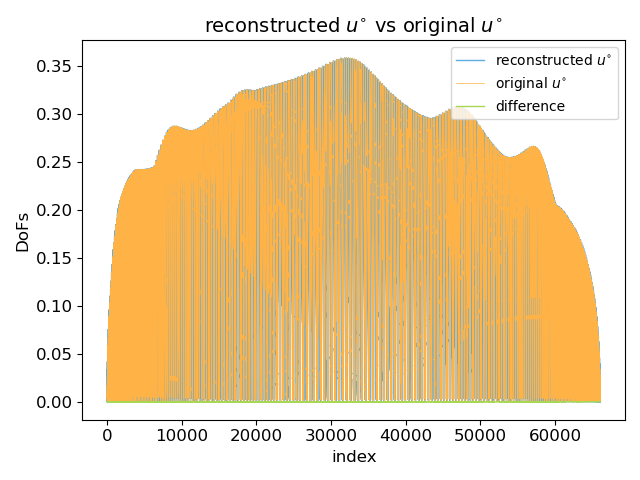}
        \caption{DoFs of $u^{\circ}$ (Darcy flow)}
        \label{fig:dof_comparison_u_poisson2}
    \end{subfigure}

    \vspace{2mm} 

    \begin{subfigure}[b]{0.42\linewidth}
        \centering
        \includegraphics[width=\linewidth]{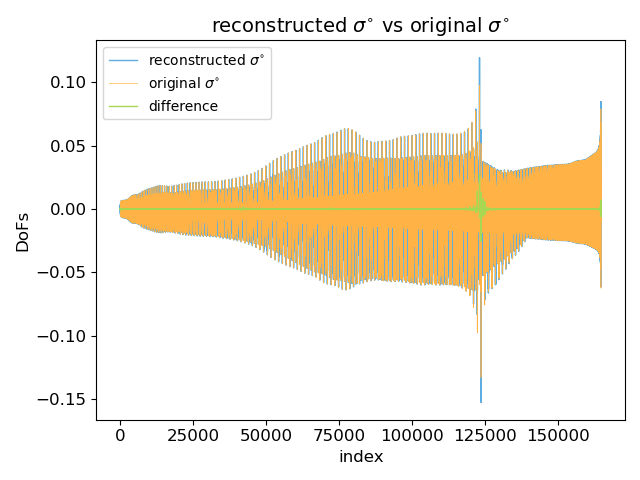}
        \caption{DoFs of $\sigma^{\circ}$ (Linear Elasticity)}
        \label{fig:dof_comparison_sigma_elasticity}
    \end{subfigure}
    \hspace{1.5em}
    \begin{subfigure}[b]{0.42\linewidth}
        \centering
        \includegraphics[width=\linewidth]{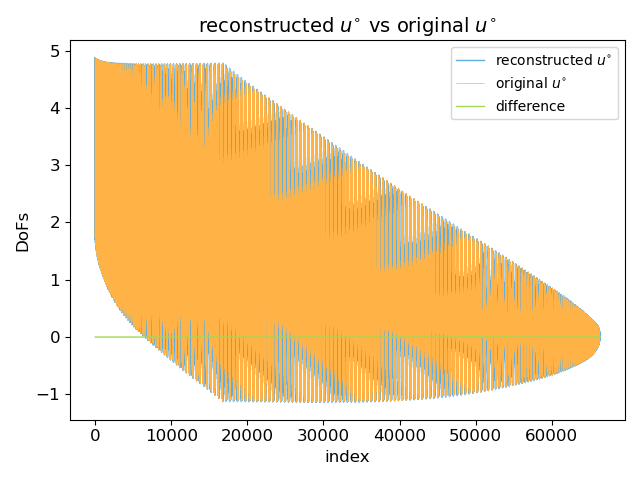}
        \caption{DoFs of $u^{\circ}$ (Linear Elasticity)}
        \label{fig:dof_comparison_u_elasticity}
    \end{subfigure}

    \caption{Comparison of DOFs and their reconstructed versions for the diffusion and elasticity problems.}
    \label{fig:dof_comparison}
\end{figure}

\subsection{Additional results (expansion of \cref{tab:approximation_relative_error})}
\label{sec:ablation_study}
In addition to the results shown in ~\cref{tab:approximation_relative_error}, we consider training with an MSE loss that matches the neural network predictions of the RB coefficients to their corresponding labels; this approach is denoted as RBNO (RB coef.). We also investigate whether adding this loss to the original RBNO training loss (the residual loss) during training can further reduce the error and loss over the test samples; this variant is denoted as RBNO (both).

For PCA-Net and FNO, the networks are trained to match pointwise evaluations at grid points to the corresponding labels and output pointwise approximations at test time. To compare these approximations with the high-fidelity solution, we propose two approaches. First, we construct a CG$_1\times$CG$_1$ representation of the approximation by mapping the pointwise evaluations to the degrees of freedom of the CG$_1\times$CG$_1$ function. The first comparison method then evaluates this approximation against the CG$_1\times$CG$_1$ representation of the high-fidelity solution (obtained by evaluating the solution on grid points and mapping to CG$_1\times$CG$_1$ degrees of freedom). In the second method, we project the CG$_1\times$CG$_1$ approximation into the RT$_1\times$CG$_2$ space in the $\HH = H(\rdiv)\times H^1$-norm and compare it with the high-fidelity solution (in the RT$_1\times$CG$_2$ FE space). We indicate which reference function space is used, [CG$_1\times$CG$_1$] or [RT$_1\times$CG$_2$], in the results. The results for PCA-Net and FNO in table~\cref{tab:approximation_relative_error} are now denoted as PCA-Net [CG$_1\times$CG$_1$] and FNO [CG$_1\times$CG$_1$]. The residual losses are evaluated using the corresponding FE representations of the approximations. In addition, we train PCA-Net to match the RB coefficients, rather than the PCA coefficients computed from solution evaluations at grid points. In this variant, we use the same basis as in RBNO to approximate the solution; this approach is denoted as PCA-Net (RB coef.).

We can see that RBNO with residual loss provides the best overall performance in terms of approximation accuracy and residual loss across all three problems. Adding the RB-coefficient MSE loss yields slight improvements for the diffusion problem, but in the elasticity problem, it even degrades performance. Standard data-driven models based on pointwise evaluations (FNO, PCA-Net) perform reasonably in $\LL$-norm, but struggle with $\HH$-norm and physics consistency (indicated by residual loss). While the PCA-Net trained on RB coefficient data produces a slightly larger $\LL$-norm error than its pointwise counterpart (except for the heat conduction problem, where it is even lower), it generally achieves lower residual loss and $\HH$-norm error (except for the $\HH$-norm error in the linear elasticity problem); nevertheless, it is still outperformed by~RBNO. 

\begin{table}[!ht]
\footnotesize
\centering
\begin{tabular}{|l|
  S[table-format=1.2e-2]@{(}S[table-format=1.2e-2]@{\!)\;\;}|
  S[table-format=1.2e-2]@{(}S[table-format=1.2e-2]@{\!)\;\;}|
  S[table-format=1.2e-2]@{(}S[table-format=1.2e-2]@{\!)\;\;}|}
\hline
\textbf{Method} &
\multicolumn{2}{c|}{\textbf{Relative error in $\LL$-norm}} &
\multicolumn{2}{c|}{\textbf{Relative error in $\HH$-norm}} &
\multicolumn{2}{c|}{\textbf{Residual loss}} \\ \hline

\multicolumn{7}{|c|}{\textbf{Heat Conduction}} \\ \hline
PCA-Net [CG$_1\times$CG$_1$] & 1.27e-01 & 2.52e-02 & 3.32e-01 & 5.75e-02 & 1.84e+01 & 3.01e+00 \\
PCA-Net [RT$_1\times$CG$_2$] & 2.09e-01 & 1.99e-02 & 3.58e+00 & 2.95e-01 & 1.84e+01 & 3.01e+00 \\
PCA-Net (RB coef.) & 1.23e-01 & 2.52e-02 & 8.39e-02 & 1.65e-02 & 3.38e-02 & 2.66e-02 \\
FNO [CG$_1\times$CG$_1$] & \bmnum{1.82e-02} & \bmnum{1.66e-03} & 1.46e-01 & 1.53e-02 & 1.90e+01 & 4.12e+00 \\
FNO [RT$_1\times$CG$_2$] & 1.70e-01 & 1.57e-02 & 3.64e+00 & 3.98e-01 & 1.90e+01 & 4.12e+00 \\
RBNO (RB coef.) & 3.25e-02  & 1.98e-02 & 2.01e-02 & 1.25e-02 & 3.27e-03 & 2.50e-02 \\
RBNO (residual) & 2.43e-02 & 9.35e-03 & 1.86e-02 & 6.36e-03 & 5.33e-04 & 1.44e-03 \\
RBNO (both) & 2.25e-02 & 7.69e-03 & \bmnum{1.71e-02} & \bmnum{5.38e-03} & \bmnum{4.55e-04} & \bmnum{2.67e-04} \\ \hline

\multicolumn{7}{|c|}{\textbf{Darcy Flow}} \\ \hline
PCA-Net [CG$_1\times$CG$_1$]  & 1.79e-01 & 7.53e-02 & 1.13e-01 & 2.35e-02 & 2.59e+00 & 2.03e+00 \\
PCA-Net [RT$_1\times$CG$_2$]  & 1.79e-01 & 7.53e-02 & 1.25e-01 & 1.31e-02 & 2.59e+00 & 2.03e+00 \\
PCA-Net (RB coef.) & 2.63e-01 & 1.10e-01 & 5.64e-02 & 2.01e-02 & 6.20e-01 & 9.40e-01 \\
FNO [CG$_1\times$CG$_1$]      & 3.30e-02 & 1.18e-02 & 1.60e-01 & 3.81e-02 & 5.10e+00 & 2.34e+00 \\
FNO [RT$_1\times$CG$_2$]      & \bmnum{3.25e-02} & \bmnum{1.18e-02} & 1.89e-01 & 3.55e-02 & 5.10e+00 & 2.34e+00 \\
RBNO (RB coef.)     & 8.33e-02  & 3.39e-02 & 1.65e-02 & 5.59e-03 & 5.08e-02  & 6.36e-02 \\
RBNO (residual)     & 6.39e-02 & 2.26e-02 & 1.23e-02 & 3.64e-03 & 2.72e-02 & 2.61e-02 \\
RBNO (both)  & 6.11e-02 & 2.01e-02 & \bmnum{1.22e-02} & \bmnum{3.72e-03} & \bmnum{2.66e-02} & \bmnum{2.38e-02} \\ \hline

\multicolumn{7}{|c|}{\textbf{Linear Elasticity}} \\ \hline
PCA-Net [CG$_1^2\times$CG$_1^2$]     & 1.34e-01 & 1.03e-01 & 3.27e-01 & 1.14e-01 & 2.46e+01 & 6.11e+00 \\
PCA-Net [RT$_1^2\times$CG$_2^2$]     & 1.34e-01 & 1.03e-01 & 2.62e+00 & 3.27e-01 & 2.46e+01 & 6.11e+00 \\
PCA-Net (RB coef.) & 2.90e-01 & 3.87e-02 & 6.12e-01 & 6.93e-02 & 1.51e+00  & 4.64e-01 \\
FNO [CG$_1^2\times$CG$_1^2$]        & \bmnum{1.17e-02} & \bmnum{3.34e-03} & 4.30e-01 & 6.11e-02 & 2.82e+01 & 6.45e+00 \\
FNO [RT$_1^2\times$CG$_2^2$]        & 2.22e-02 & 7.08e-03 & 2.78e+00 & 3.08e-01 & 2.81e+01 & 6.45e+00 \\
RBNO (RB coef.)              & 2.95e-02 & 6.56e-03 & 5.32e-02 & 8.93e-03 & 1.02e-02 & 3.62e-03 \\
RBNO (residual)            & 2.49e-02 & 5.89e-03 & \bmnum{3.91e-02} & \bmnum{8.68e-03} & \bmnum{4.87e-03} & \bmnum{3.56e-03} \\
RBNO (both)               & 3.43e-02 & 8.74e-03 & 4.86e-02 & 1.69e-02 & 6.62e-03 & 2.01e-02 \\ \hline
\end{tabular}
\caption{Comparison on the mean (and standard deviation, estimated over 500 test samples) of the relative errors in $\LL_2 = L_2 \times L_2$-norm and $\HH = H(\rdiv)\times H^1$-norm, and residual loss for PCA-Net, FNO and RBNO. }
\label{tab:approximation_relative_error_merged}
\end{table}

\clearpage
\subsection{Visualization of reference-prediction-difference}
\label{sec:visualization}

We visualize the reference–prediction difference for a random test sample in~\cref{fig:ref_pred_diff_poisson_setup1} (heat conduction), ~\cref{fig:ref_pred_diff_poisson_setup2} (Darcy flow), and~\cref{fig:ref_pred_diff_elasticity_u,fig:ref_pred_diff_elasticity_sigma} (elasticity). All functions are evaluated on a uniform grid, so the figures should be interpreted as showing the CG$_1\times$CG$_1$ representation of the solutions.

\begin{figure}[!h]
    \centering
    \includegraphics[width=\linewidth, valign=t]{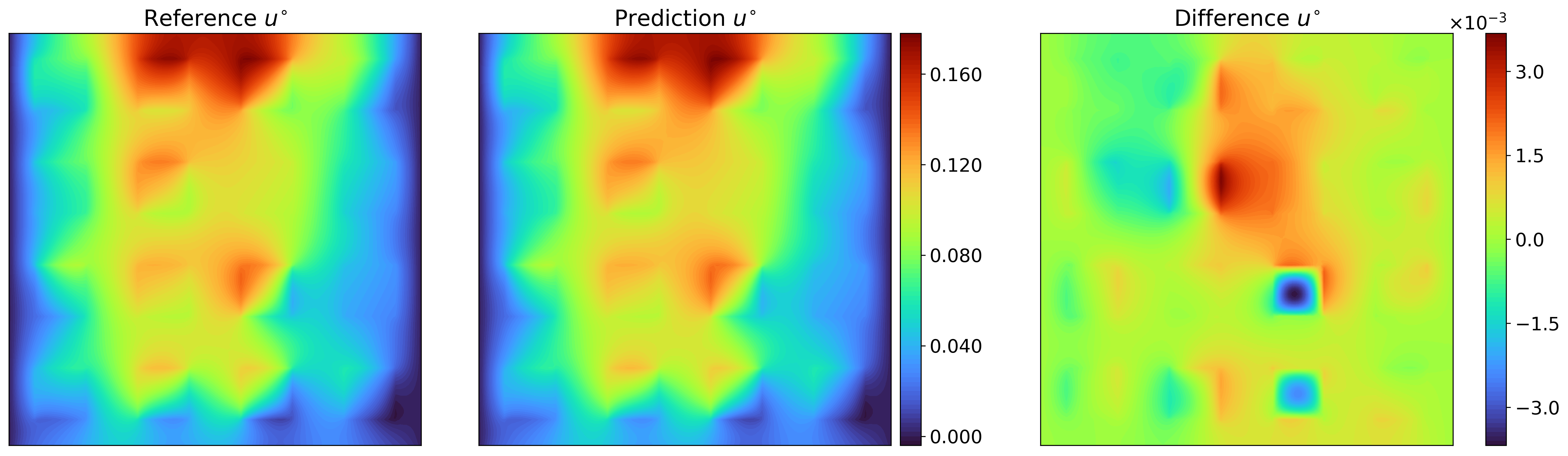}
    \includegraphics[width=\linewidth, valign=t]{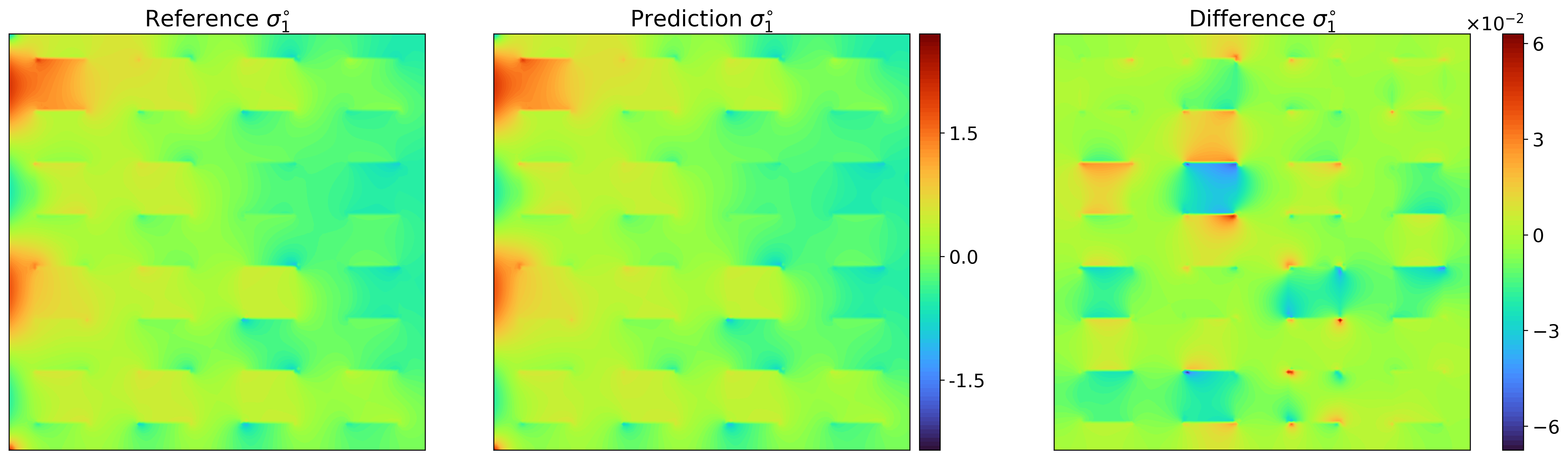}
    \includegraphics[width=\linewidth, valign=t]{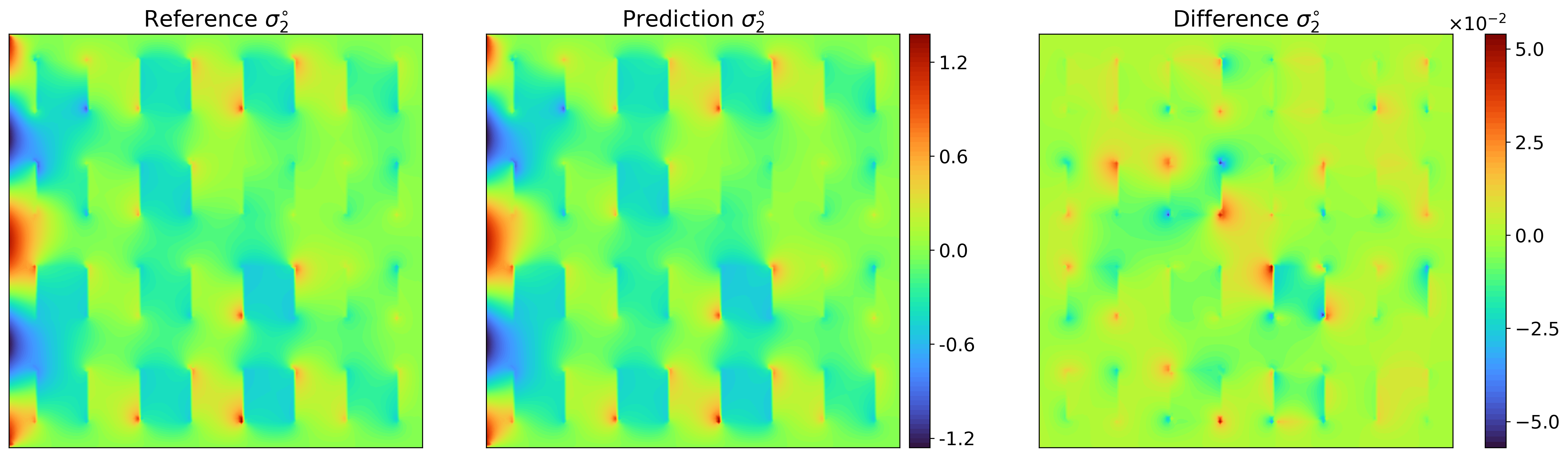}
    \caption{Reference-prediction-difference (left-middle-right) of $u^{\circ}$ (top) and $\sigma^{\circ}=(\sigma_1^{\circ}, \sigma_2^{\circ})$ (middle, bottom) [heat conduction].}
    \label{fig:ref_pred_diff_poisson_setup1}
\end{figure}

\begin{figure}[!h]
    \centering
    \includegraphics[width=\linewidth, valign=l]{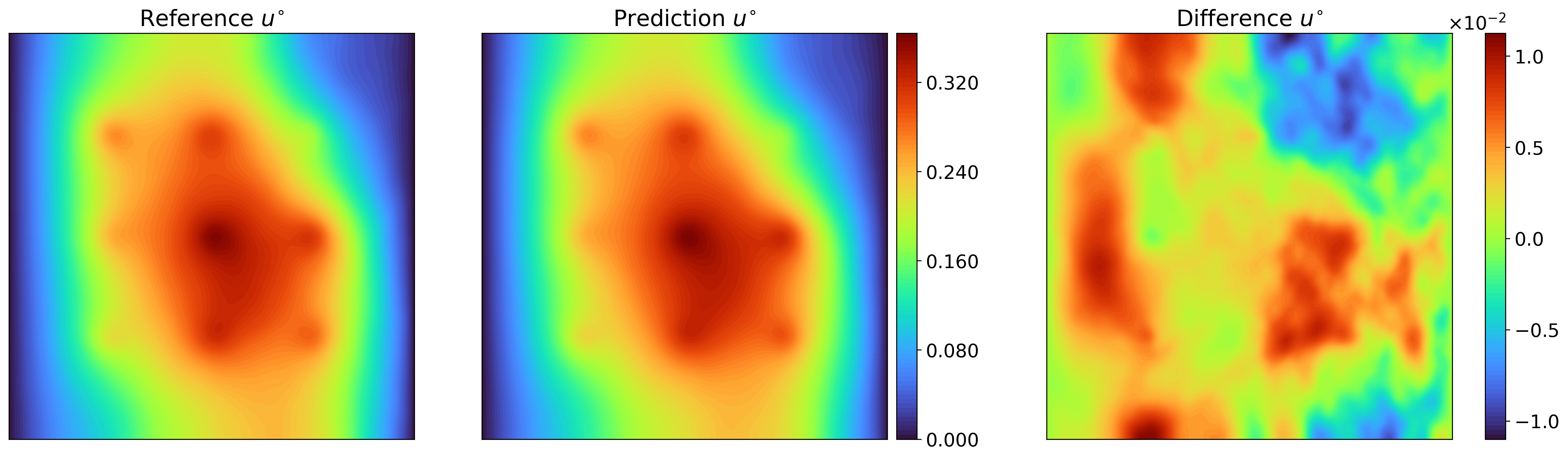}
    \includegraphics[width=\linewidth, valign=l]{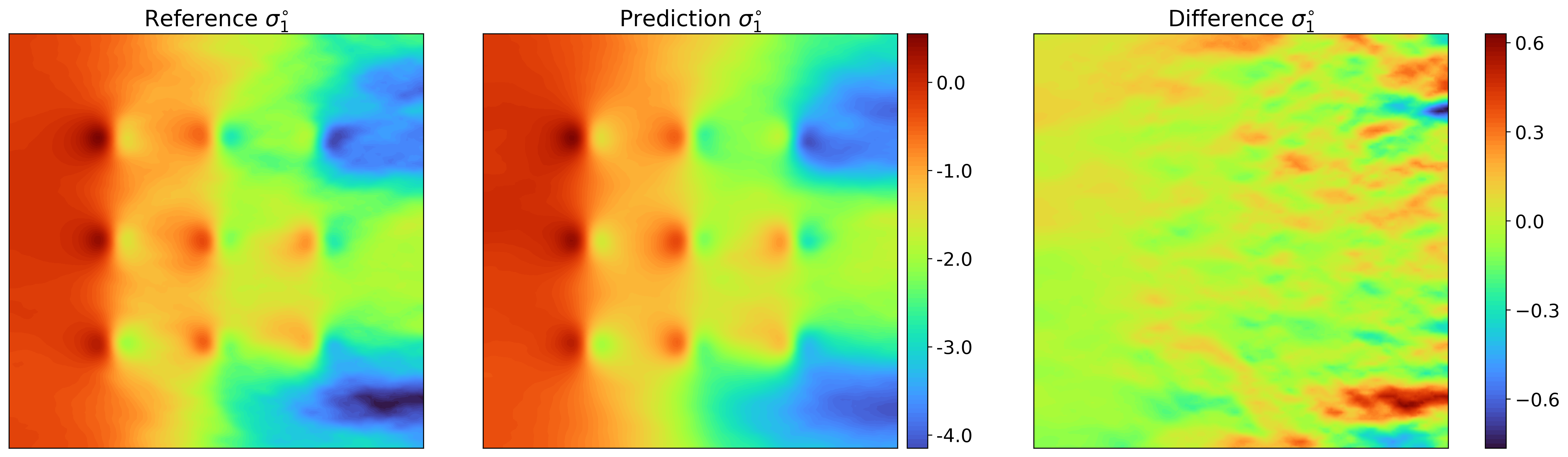}
    \includegraphics[width=\linewidth, valign=l]{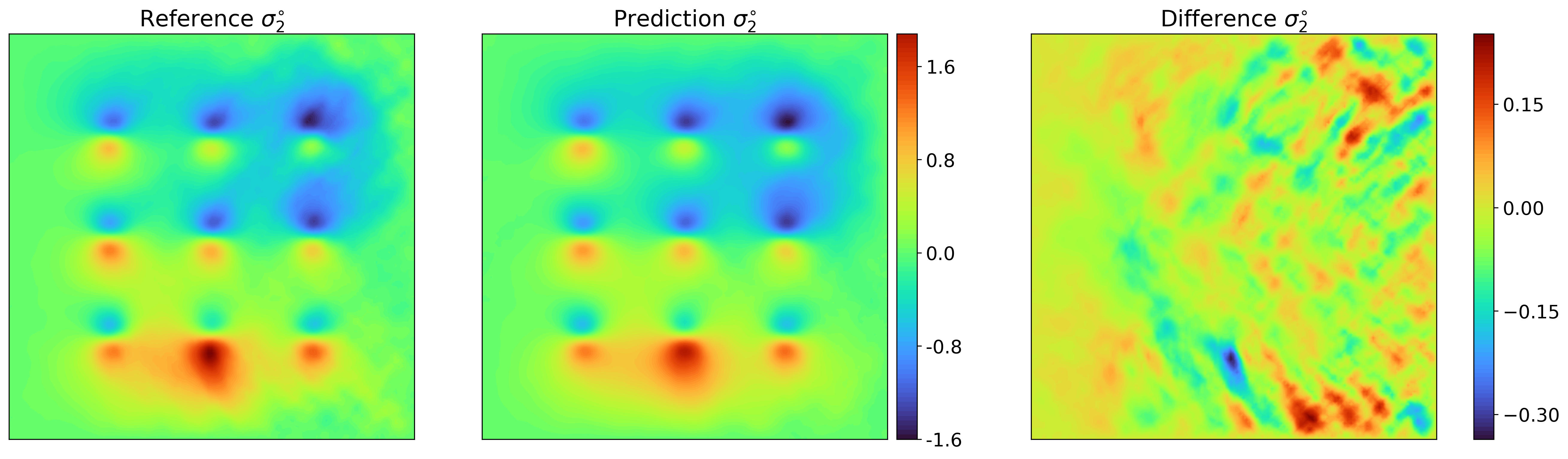}
  \caption{Reference-approximation difference (left-middle-right) of $u^{\circ}$ (top) and $\sigma^{\circ}=(\sigma_1^{\circ}, \sigma_2^{\circ})$ (middle, bottom) [Darcy flow].}
    \label{fig:ref_pred_diff_poisson_setup2}
\end{figure}

\begin{figure}[!h]
    \centering
    \includegraphics[width=0.45\linewidth]{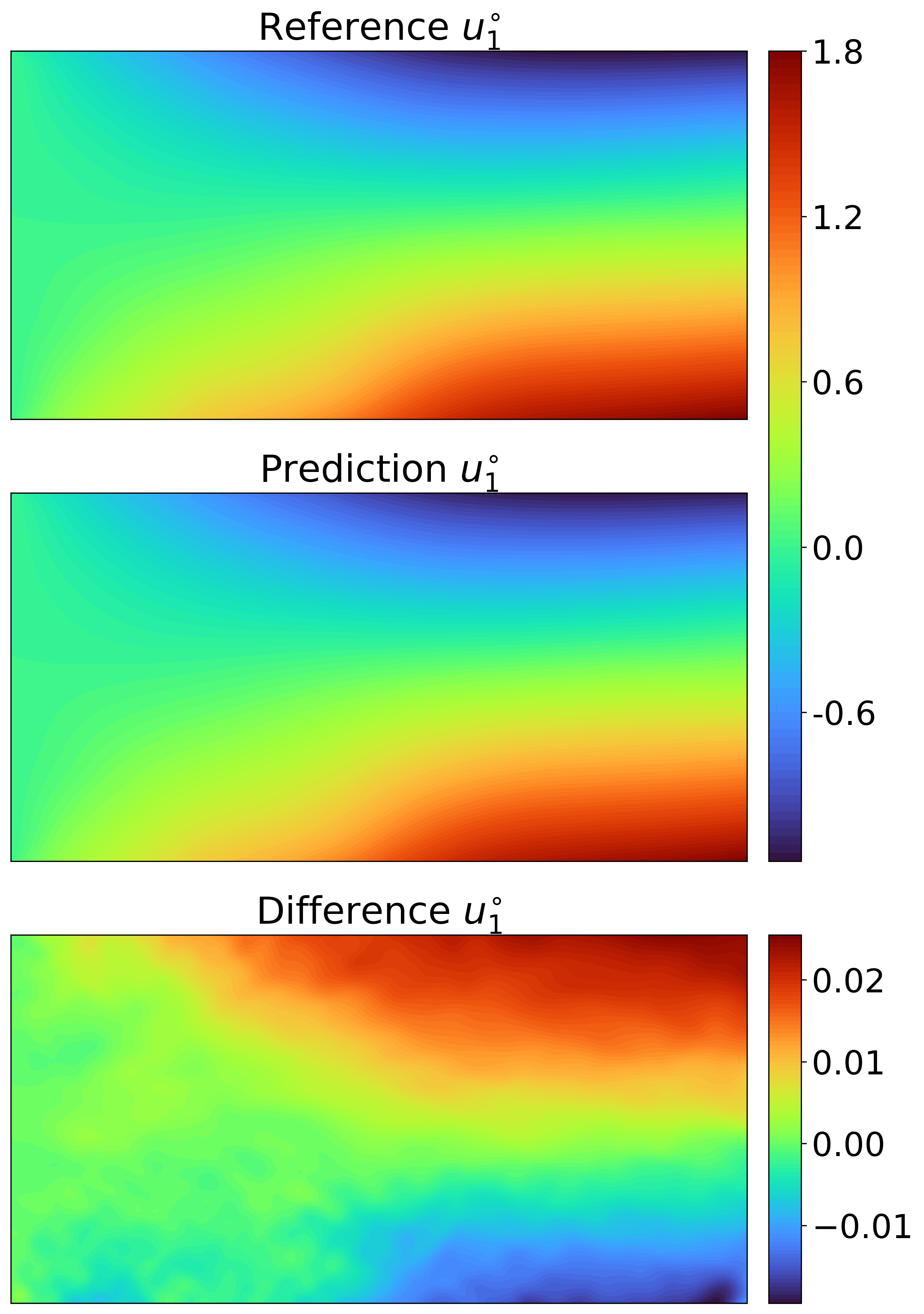}
    \includegraphics[width=0.46\linewidth]{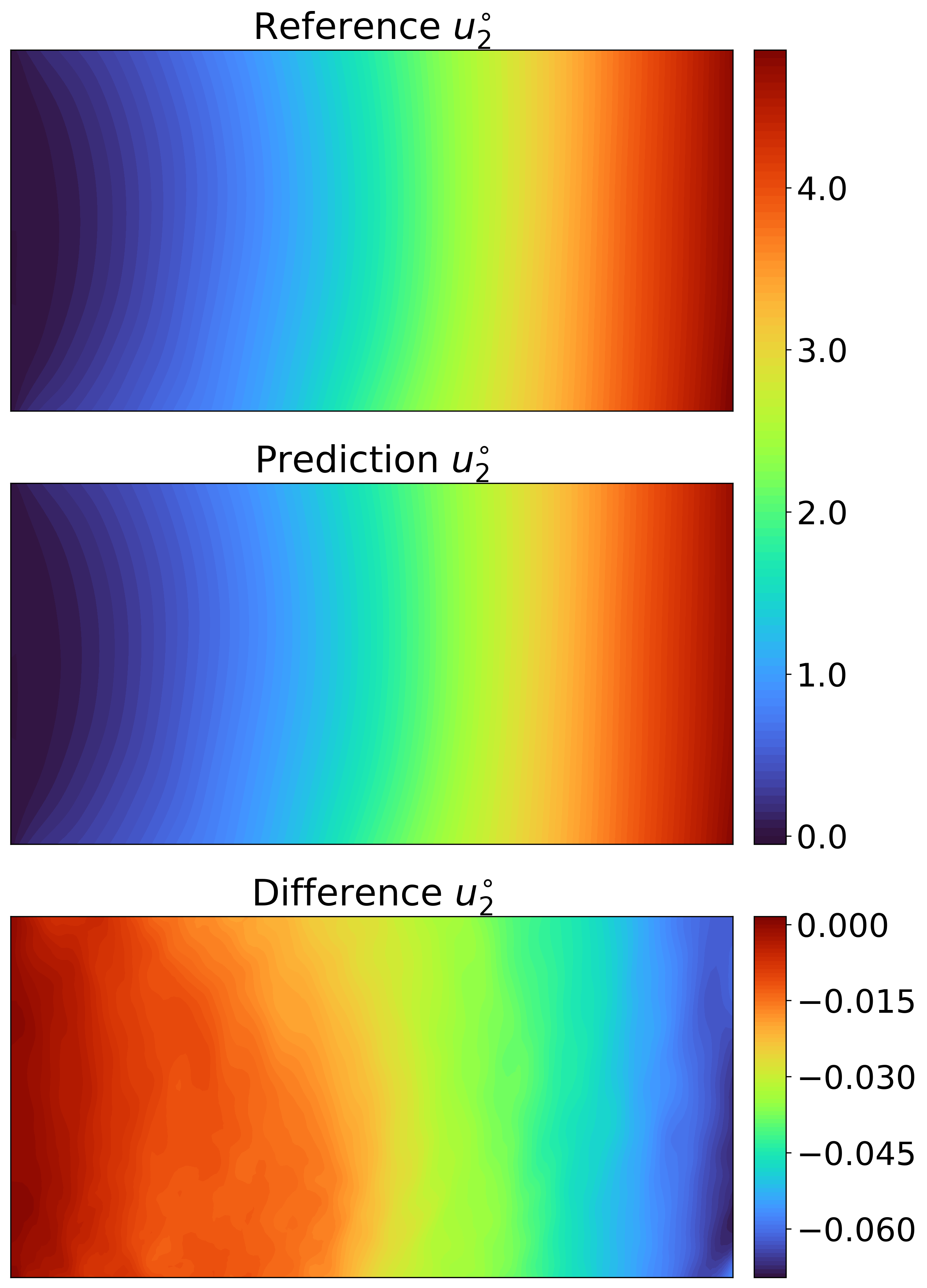}
    \caption{Reference-prediction-difference (top-middle-bottom) of $u_1$ and $u_2$ (left and right) [linear elasticity].}
    \label{fig:ref_pred_diff_elasticity_u}
\end{figure}

\begin{figure}[!ht]
    \centering
    \includegraphics[width=0.223\linewidth, valign=t]{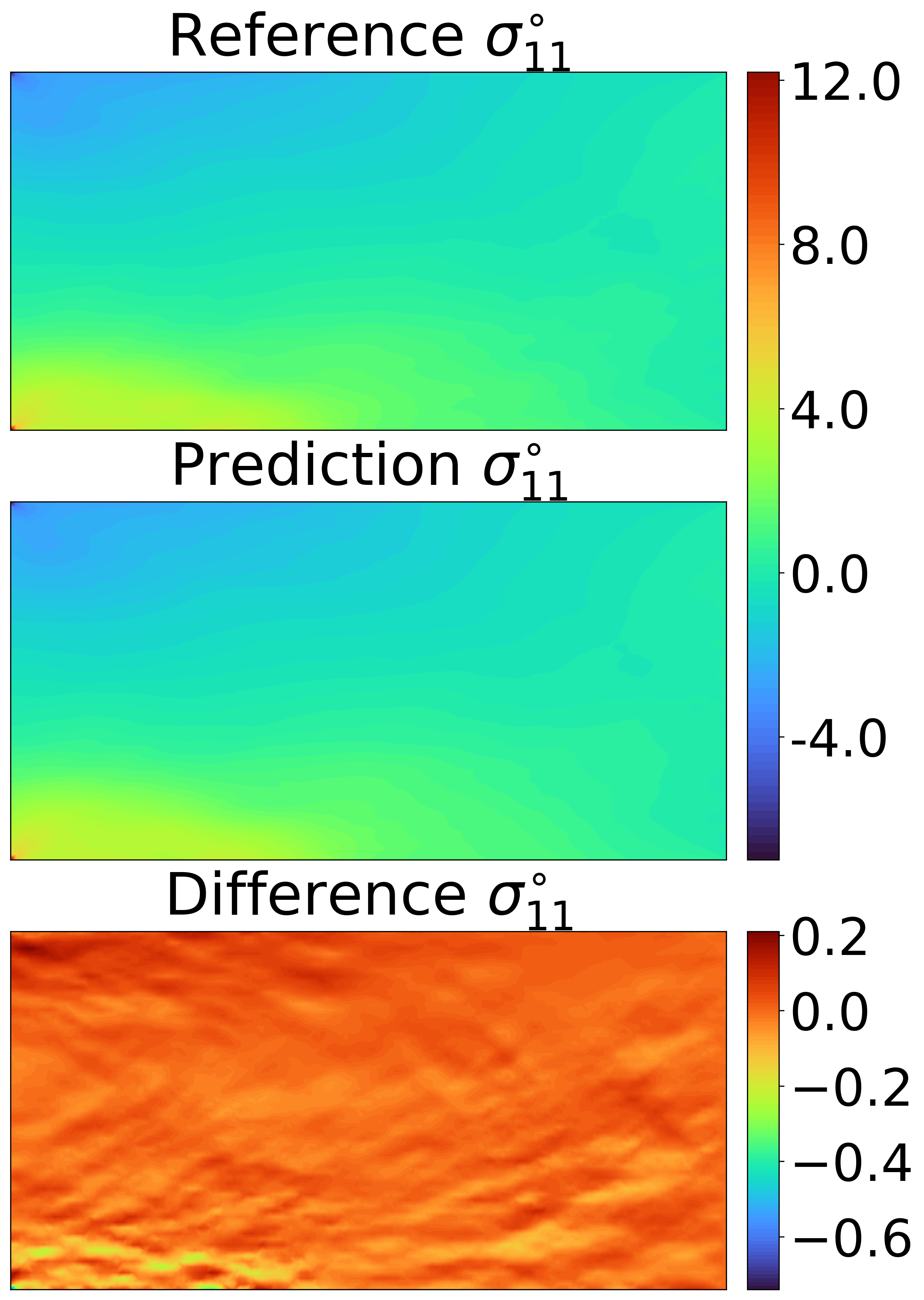}
    \includegraphics[width=0.23\linewidth, valign=t]{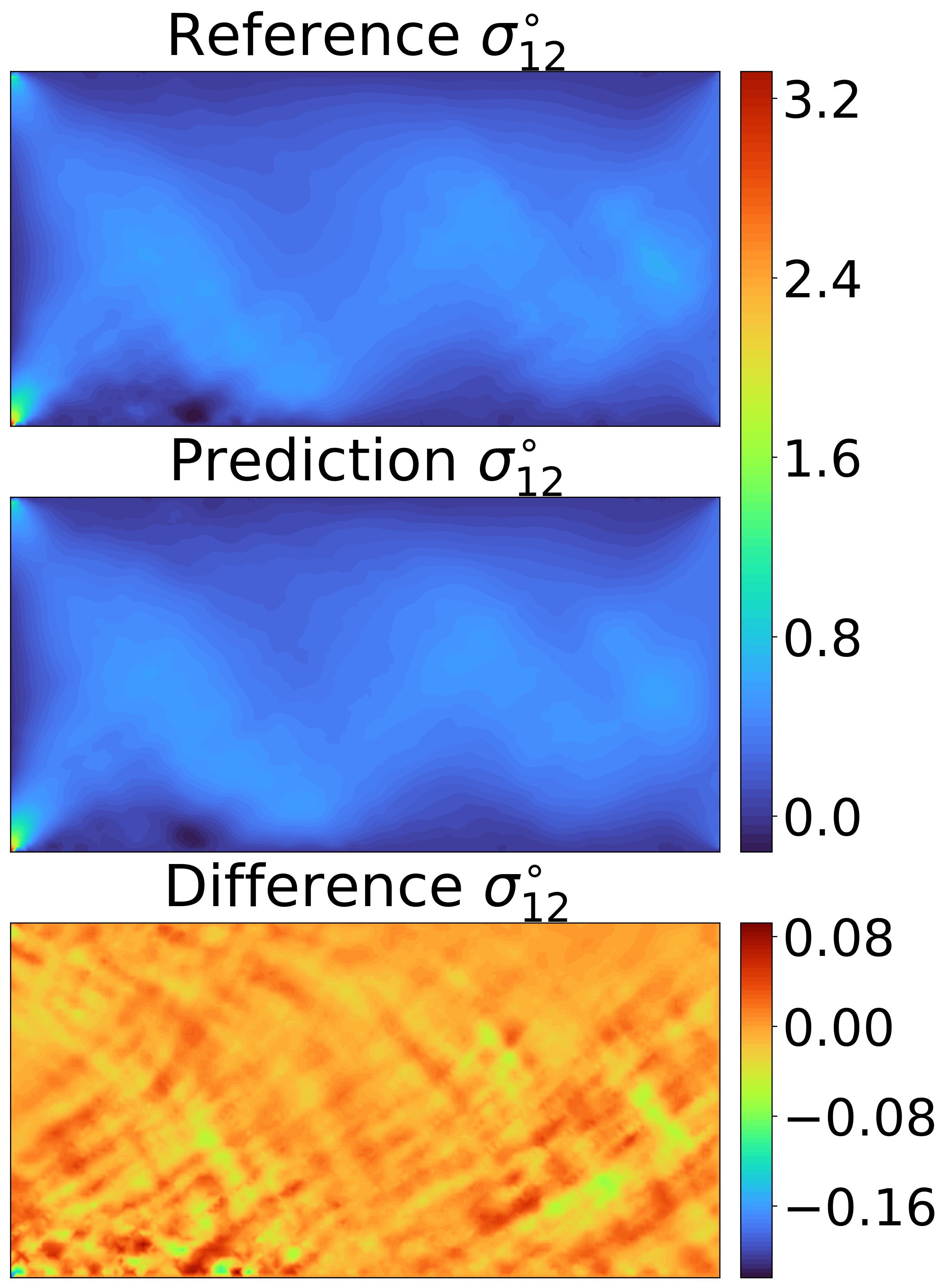}
    \includegraphics[width=0.23\linewidth, valign=t]{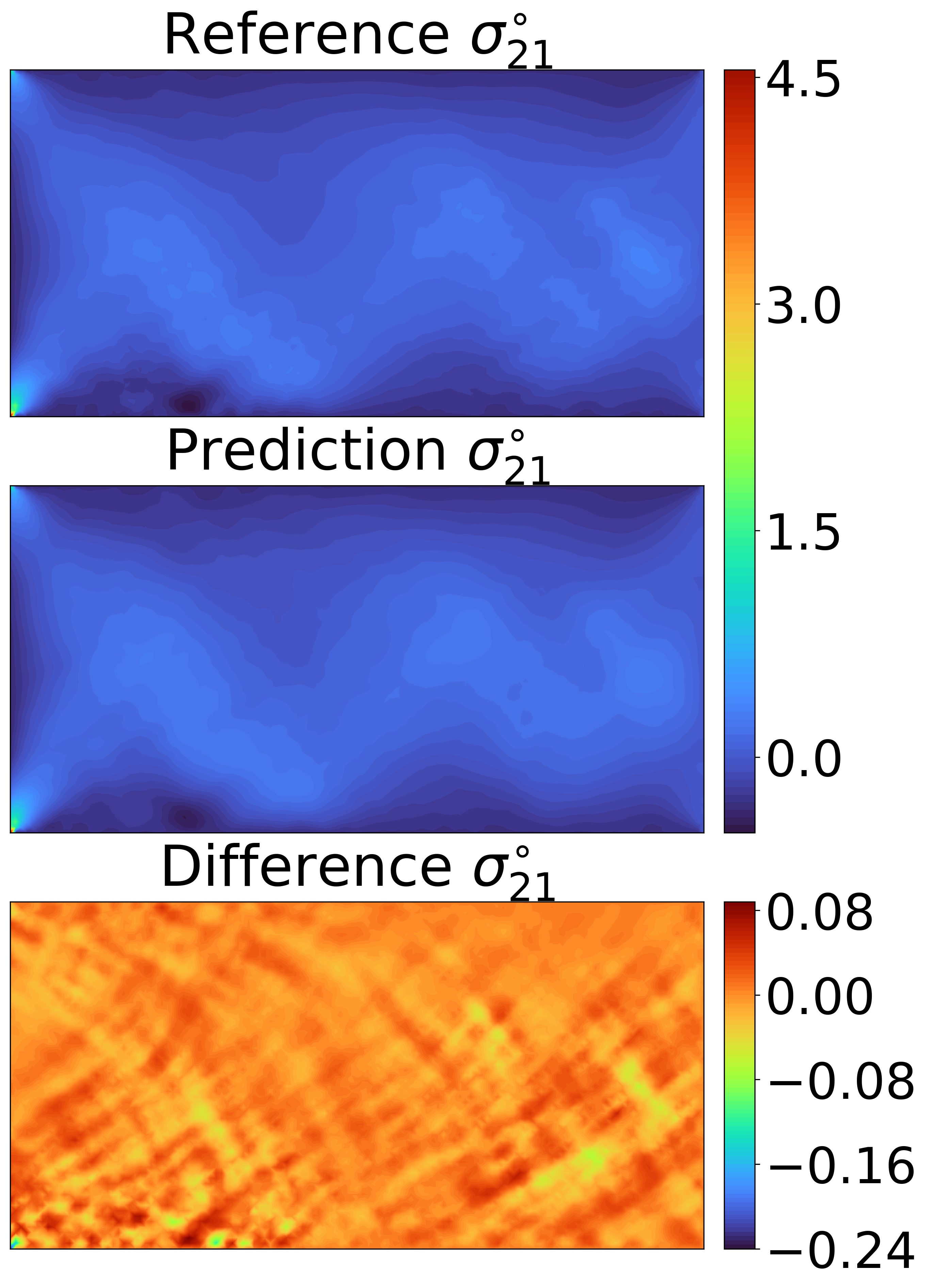}
    \includegraphics[width=0.23\linewidth, valign=t]{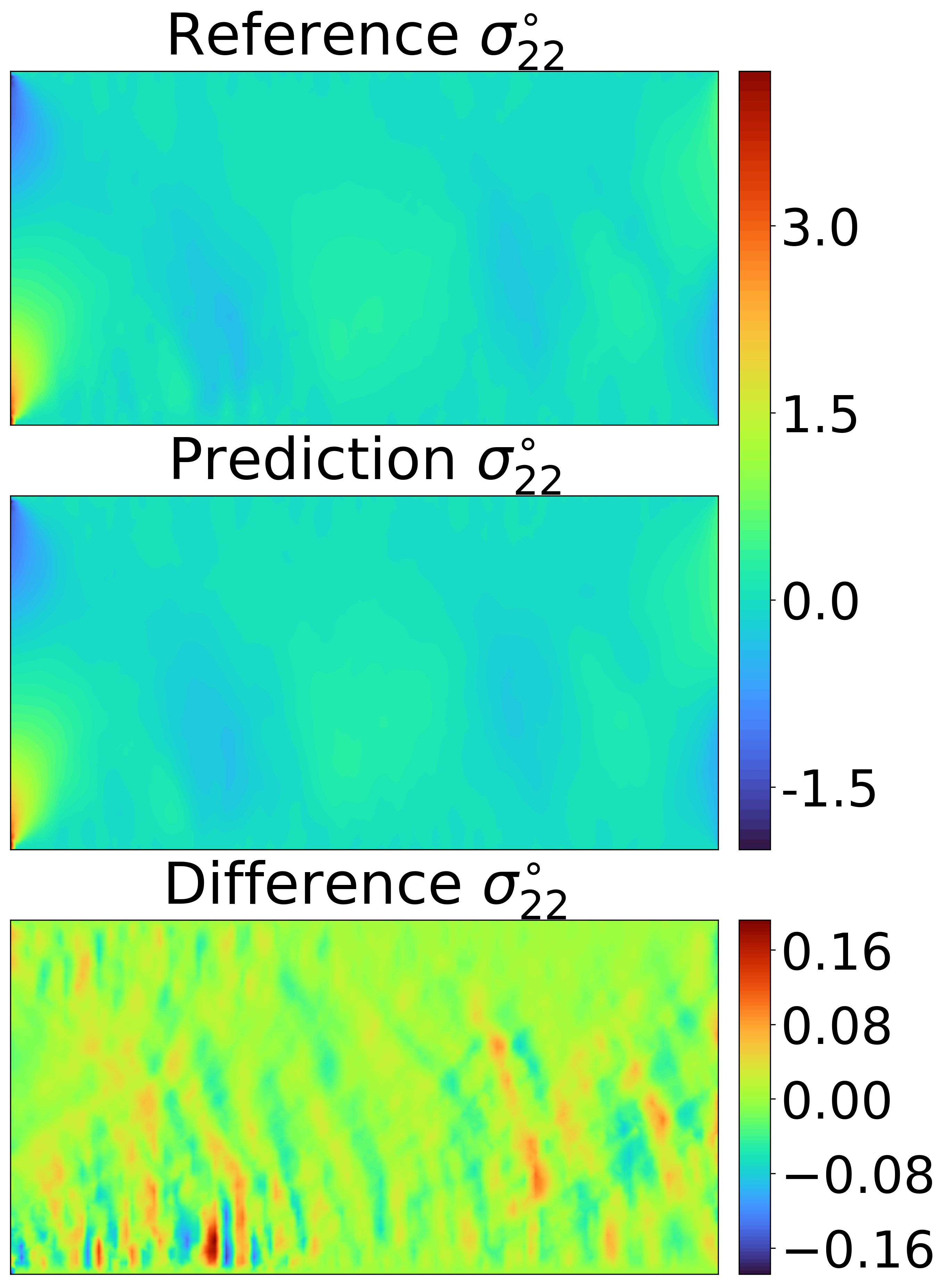}
    \caption{Reference-prediction-difference (top-middle-bottom) of $\sigma_{11}, \sigma_{12}, \sigma_{21}$ and $\sigma_{22}$ (from left to right) [linear elasticity].}
    \label{fig:ref_pred_diff_elasticity_sigma}
\end{figure}

\end{document}